\documentclass[12pt,english]{amsart}
\pdfoutput=1
\usepackage{geometry}
\usepackage[T1]{fontenc}
\usepackage[utf8]{inputenc}
\usepackage[english]{babel}
\usepackage{latexsym}
\usepackage{mathrsfs}
\usepackage{enumerate}
\usepackage{amsmath}
\usepackage{amssymb}
\usepackage{mathtools}
\usepackage{braket}
\usepackage{amsthm}
\usepackage{mathrsfs}
\usepackage{graphicx}
\usepackage{tikz}
\usepackage{tikz-cd}
\usepackage{frontespizio}
\usepackage{shellesc}
\usepackage{quoting}
\usepackage{url}
\usepackage{hyperref}
\usepackage{tabularx}
\hypersetup{colorlinks=true,linkcolor=black}
\usepackage{booktabs}
\usepackage{spectralsequences}

\usepackage[autostyle,italian=guillemets]{csquotes}
\usepackage[backend=biber, style=numeric, bibstyle=numeric,doi=false,isbn=false,url=false,eprint=false]{biblatex}

\theoremstyle{plain} 
\newtheorem{thm}{Theorem}[section] 
\newtheorem{corollario}[thm]{Corollary} 
\newtheorem{lem}[thm]{Lemma} 
\newtheorem{prop}[thm]{Proposition} 

\theoremstyle{definition} 
\newtheorem{defn}{Definition}[section]
\newtheorem*{es}{Example}
\newtheorem*{costr}{Construction}
\newtheorem{oss}[thm]{Remark} 

\theoremstyle{remark}

\newcommand{\blind}{0}
\newcommand{\M}{\mathcal{M}}
\newcommand{\numberset}{\mathbb}
\newcommand{\N}{\numberset{N}}
\newcommand{\F}{\numberset{F}}
\newcommand{\R}{\numberset{R}}
\newcommand{\C}{\numberset{C}}
\newcommand{\Z}{\numberset{Z}}
\newcommand{\Q}{\numberset{Q}}
\DeclarePairedDelimiter{\abs}{\lvert}{\rvert}
\DeclarePairedDelimiter{\norma}{\lVert}{\rVert}

\date {}

\begin{document}
	\author[T. Rossi]{Tommaso Rossi}
	\address{Universit\`a di Roma Tor Vergata}
	\email{tommaso.rossi118@gmail.com}
	\author[P. Salvatore]{Paolo Salvatore}
	\address{Universit\`a di Roma Tor Vergata}
	\email{salvator@mat.uniroma2.it}
		
	

	\title[Chain level Koszul duality beteen Grav and Hycom]{Chain level Koszul duality between the Gravity and Hypercommutative operads}
	\maketitle
	
	\begin{abstract}
		Let $\overline{\M}_{0,n+1}$ be the moduli space of genus zero stable curves with $(n+1)$-marked points. The collection $\overline{\M}\coloneqq\{\overline{\M}_{0,n+1}\}_{n\geq 2}$ forms an operad in the category of complex projective varieties; its homology $Hycom\coloneqq H_*(\overline{\M})$ is called the Hypercommutative operad. In this paper we construct a chain model for the hypercommutative operad, i.e. an operad of chain complexes $C_*^{dual}(\overline{\M})$ which is weakly equivalent to the operad of singular chains $C_*(\overline{\M})$. We prove that $C_*^{dual}(\overline{\M})$ is the linear dual of the bar construction $B(grav)$, where $grav$ is a chain model of the gravity operad based on cacti without basepoint. This shows that the Gravity and Hypercommutative operad are Koszul dual also at the chain level, refining a previous result of Getzler \cite{Getzler}. The construction is topological, since $C_*^{dual}(\overline{\M})(n)$ is the cellular complex associated to a regular CW-decomposition of $\overline{\M}_{0,n+1}$.
	\end{abstract}
	
	\tableofcontents
	
	\section{Introduction}
	Let $\M_{0,n+1}$ be the moduli space of genus zero Riemann surfaces with $n+1$ marked points and consider its Deligne-Mumford compactification $\overline{\M}_{0,n+1}$. There are two operads which are related to these moduli spaces: the Hypercommutative and the Gravity operad. The Hypercommutative operad $Hycom$ was studied by Barannikov-Kontsevich \cite{Barannikov-Kontsevich} and Kontsevich-Manin \cite{Kontsevich-Manin}, and its space of $n$-ary operations is given by $H_*(\overline{\M}_{0,n+1})$. The operad stucture actually exists at the topological level: indeed for any $i=1,\dots,n+1$ there are gluing maps
	\[
	\circ_i: \overline{\M}_{0,n+1}\times \overline{\M}_{0,m+1}\to \overline{\M}_{0,n+m}
	\]
	which endow the collection of complex algebraic varieties $\overline{\M}\coloneqq\{\overline{\M}_{0,n+1}\}_{n\geq 2}$ with a (cyclic) operad structure. Important examples of Hypercommutative algebras are the following:
	\begin{itemize}
		\item The cohomology of any compact Calabi-Yau manifold, see Barannikov-Kontsevich \cite{Barannikov-Kontsevich}.
		\item  Kontsevich and Manin have shown in \cite{Kontsevich-Manin} through Gromov-Witten invariants that the rational homology of a smooth projective algebraic variety is a hypercommutative algebra.		
	\end{itemize}
	Informally we can think of an Hypercommutative algebra as a generalization of a commutative algebra, because in addition to the commutative product $(x_1,x_2)$ we also have higher symmetric operations $(x_1,\dots,x_n)$, $n\in\N$, which satisfy a higher version of the associativity relation.
	
	On the other hand the Gravity operad $Grav$ was introduced by Getzler in \cite{Getzler94} as a suboperad of $H_*(\mathcal{D}_2)$, where $\mathcal{D}_2$ is the little two disks operad. The space of arity $n$ operations is given by $sH_*(\M_{0,n+1})$, the (shifted) homology of the moduli space of genus zero marked curves. In this case the operad structure is more subtle since it involves some homological transfers. In contrast to $Hycom$ the Gravity operad is not the homology of a topological operad, however there are models for $Grav$ at the level of spectra (Westerland \cite{Westerland}) and at the level of chain complexes (see Getzler-Kapranov \cite{Getzler-Kapranov} and Ward \cite{Ward}). Remarkable examples of gravity algebras are:
	\begin{itemize}
		\item $H_*^{S^1}(X)$, where $X$ is an algebra over the framed little two disk operad $f\mathcal{D}_2$ (see \cite{Westerland}).
		\item If $M$ is a closed oriented manifold, then the string topology operations on $H_*^{S^1}(LM)$ assemble to a gravity algebra structure (see \cite{Chas} and \cite{Westerland}).
		\item If $A$ is a Frobenius algebra, then $HC^*(A)$ is a gravity algebra. Further examples along these lines can be found in \cite{Ward}.
	\end{itemize}
	Informally we can think of a Gravity algebra as a generalization of a (shifted) Lie algebra, because in addition to a Lie bracket $\{x_1,x_2\}$ (of degree one) we also have higher antisymmetric operations $\{x_1,\dots,x_n\}$, $n\in\N$, which satisfy a higher version of the Jacobi relation.
	The analogy between  Hypercommutative/Gravity algebras and commutative/Lie algebras does not stop here: the Commutative and Lie operad are Koszul dual to each other, and similarly
	the Hypercommutative and the Gravity operad are Koszul dual (Getzler \cite{Getzler}). 
	\medskip	
	
	In this paper we construct a chain model for $Hycom$, i.e an operad of chain complexes $C_*^{dual}(\overline{\M})$ which is weakly equivalent to the operad of singular chains $C_*(\overline{\M})$. The upshot is that $C_*^{dual}(\overline{\M})$ is much smaller than  $C_*(\overline{\M})$ (in the sense that $C_*^{dual}(\overline{\M})(n)$ is finite dimensional in each degree for any $n\in \N$) and therefore it is more suitable for computations.
	Moreover it turns out that $C_*^{dual}(\overline{\M})$ can be identified (up to a shift in degrees) with the linear dual of the operadic bar costruction $B(grav)$, where $grav$ is a chain model for the Gravity operad which is very close to that of Ward \cite{Ward}. This proves Koszul duality between the Hypercommutative and the Gravity operad at the level of chains (Theorem \ref{thm: chain level Koszul duality}), refining the result of Getzler \cite{Getzler}.
	
	\medskip
	The ideas behind the construction of $C_*^{dual}(\overline{\M})$ are very close to that of \cite{Salvatore}: for each choice of positive numbers $a_1,\dots,a_n>0$ the second author constructed (open) cell decompositions of $F_n(\C)/\C\rtimes \R^{>0}$, the quotient of the ordered configuration space by the action of $\C\rtimes \R^{>0}$ by translations and dilations. Then he proved that these decompositions can be extended to regular CW-decompositions of $FM(n)$, the Fulton-MacPherson compactification of $F_n(\C)/\C\rtimes \R^{>0}$ (see Axelrod-Singer \cite{Axelrod-Singer}). The combinatorics arising from this construction is the following: a cell is represented by a tree with edges of two colors and vertices labelled by cells of the cacti operad (see \cite{Salvatore}, McClure-Smith \cite{McClure-Smith1} and Kontsevich-Soibelman \cite{Kontsevich-Soibelman}). A subtle point of this construction is that the CW-decomposition obtained in this way depends on the initial choice of the weights $a_1,\dots,a_n$ and there is not any coherent choice of weights that makes all the operad compositions
	\[
	\circ_i:FM(n)\times FM(m)\to FM(n+m-1)
	\]
	send a product of cells into a cell. However it is possible to get rid of the weights, obtaining regular CW-decompositions of $FM(n)$, $n\in \N$, such that each $\circ_i$ sends a product of cells into a cell. Therefore one obtains a cell decomposition of the Fulton-MacPherson operad $FM$, such that the operad of cellular chains $C_*^{cell}(FM)$ is precisely the cobar-bar construction $\Omega B (Cact)$ of the cacti operad (see Section \ref{sec: cactus models} for an overview of cacti and related operads).  
	
	Similarly, for any choice of weights $a_1,\dots,a_n>0$, we construct an open cell decomposition of $\M_{0,n+1}\cong F_n(\C)/\C\rtimes \C^*$ (see Section \ref{sec: combinatorial models for the open moduli space}), where $\C\rtimes \C^*$ acts by translations, rotations and dilations on the ordered configuration space $F_n(\C)$. The intuition is the following and goes back to Nakamura \cite{Nakamura} and the paper \cite{Salvatore} by the second author: given a configuration of $n$ points $(z_1,\dots,z_n)\in F_n(\C)$, think of each $z_i$ as a negative electric charge of value $-a_i$. Each charge generates a radial electric field and by superposition we obtain a total electric field
	\[
	E(z)\coloneqq \sum_{i=1}^n-a_i\frac{z-z_i}{\abs{z-z_i}^2}
	\]
	The electric field is conservative, and it is not difficult to find an explicit potential for $E(z)$: indeed it turns out that 
	\[
	E(z)=-\nabla U(z)
	\]
	where $U(z)\coloneqq \log \abs{h(z)}$ and $h(z)\coloneqq\prod_{i=1}^n(z-z_i)^{a_i}$. Now look at the flow lines of $E(z)$, i.e. curves $\gamma(t)$ such that $\gamma'(t)=E(\gamma(t))$: most of them start at one point $z_i$ of the configuration $(z_1,\dots,z_n)$ and go to infinity; conversely, there are some flow lines which are of finite length, namely those that connects two zeros of $E(z)$ or one zero to a point $z_i$ of the configuration. Looking at these special flow lines we can associate to each point $(z_1,\dots,z_n)\in F_n(\C)$ a tree with black and white vertices (correponding respectively to zeroes of $E(z)$ and points of the configuration) with labels (encoding the value of $U(z)$ at each black vertex and the angles between two consecutive flow lines incident to a white vertex). See Figure \ref{fig:associazione albero bw to a configuration} for a picture.
	\begin{figure}
		\centering

		\tikzset{every picture/.style={line width=0.75pt}} 
		
		\begin{tikzpicture}[x=0.75pt,y=0.75pt,yscale=-1,xscale=1]
			
			\draw [color={rgb, 255:red, 155; green, 155; blue, 155 }  ,draw opacity=0.5 ]   (155.93,212.36) .. controls (177.33,200.68) and (222.21,167.49) .. (239.73,130.52) ;
			\draw [color={rgb, 255:red, 155; green, 155; blue, 155 }  ,draw opacity=0.5 ]   (206.65,348.47) .. controls (192.32,330.2) and (198.86,286.2) .. (206.52,251.28) ;
			\draw [color={rgb, 255:red, 155; green, 155; blue, 155 }  ,draw opacity=0.5 ]   (206.52,251.28) .. controls (218.32,214.2) and (270.86,192.79) .. (300.05,185.01) ;
			\draw    (206.52,251.28) -- (246.38,259.06) ;
			\draw    (179.28,259.06) -- (206.52,251.28) ;
			\draw    (179.28,259.06) -- (155.93,212.36) ;
			\draw    (134,176.51) -- (155.93,212.36) ;
			\draw [color={rgb, 255:red, 155; green, 155; blue, 155 }  ,draw opacity=0.5 ]   (74.32,241.44) .. controls (95.73,239.49) and (121.03,233.66) .. (155.93,212.36) ;
			\draw  [fill={rgb, 255:red, 0; green, 0; blue, 0 }  ,fill opacity=1 ] (151.19,212.36) .. controls (151.19,209.74) and (153.31,207.62) .. (155.93,207.62) .. controls (158.54,207.62) and (160.67,209.74) .. (160.67,212.36) .. controls (160.67,214.98) and (158.54,217.1) .. (155.93,217.1) .. controls (153.31,217.1) and (151.19,214.98) .. (151.19,212.36) -- cycle ;
			\draw  [fill={rgb, 255:red, 0; green, 0; blue, 0 }  ,fill opacity=1 ] (201.78,251.28) .. controls (201.78,248.66) and (203.9,246.54) .. (206.52,246.54) .. controls (209.14,246.54) and (211.26,248.66) .. (211.26,251.28) .. controls (211.26,253.89) and (209.14,256.01) .. (206.52,256.01) .. controls (203.9,256.01) and (201.78,253.89) .. (201.78,251.28) -- cycle ;
			\draw [color={rgb, 255:red, 155; green, 155; blue, 155 }  ,draw opacity=0.5 ]   (134,176.51) .. controls (171.62,216.14) and (204.7,159.71) .. (218.32,116.9) ;
			\draw [color={rgb, 255:red, 155; green, 155; blue, 155 }  ,draw opacity=0.5 ]   (134,176.51) .. controls (175.51,188.9) and (176.21,152.14) .. (189.84,109.33) ;
			\draw [color={rgb, 255:red, 155; green, 155; blue, 155 }  ,draw opacity=0.5 ]   (76.27,218.09) .. controls (113.24,221.98) and (134.65,216.14) .. (134,176.51) ;
			\draw [color={rgb, 255:red, 155; green, 155; blue, 155 }  ,draw opacity=0.5 ]   (86,181.12) .. controls (103.51,190.85) and (107.4,194.74) .. (134,176.51) ;
			\draw [color={rgb, 255:red, 155; green, 155; blue, 155 }  ,draw opacity=0.5 ]   (134,176.51) .. controls (152.16,165.55) and (154.11,138.31) .. (152.16,120.79) ;
			\draw [color={rgb, 255:red, 155; green, 155; blue, 155 }  ,draw opacity=0.5 ]   (134,176.51) .. controls (136.59,151.93) and (130.75,153.87) .. (121.03,136.36) ;
			\draw [color={rgb, 255:red, 155; green, 155; blue, 155 }  ,draw opacity=0.5 ]   (134,176.51) .. controls (111.3,167.49) and (113.24,173.33) .. (97.67,155.82) ;
			\draw [color={rgb, 255:red, 155; green, 155; blue, 155 }  ,draw opacity=0.5 ]   (179.28,259.06) .. controls (194.97,212.25) and (241.67,183.06) .. (284.48,163.6) ;
			\draw [color={rgb, 255:red, 155; green, 155; blue, 155 }  ,draw opacity=0.5 ]   (179.28,259.06) .. controls (161.89,206.41) and (233.89,173.33) .. (259.19,136.36) ;
			\draw [color={rgb, 255:red, 155; green, 155; blue, 155 }  ,draw opacity=0.5 ]   (80.16,260.9) .. controls (111.3,255.06) and (148.27,221.98) .. (179.28,259.06) ;
			\draw [color={rgb, 255:red, 155; green, 155; blue, 155 }  ,draw opacity=0.5 ]   (93.78,290.09) .. controls (124.92,284.25) and (148.27,251.17) .. (179.28,259.06) ;
			\draw [color={rgb, 255:red, 155; green, 155; blue, 155 }  ,draw opacity=0.5 ]   (189.84,349.22) .. controls (167.73,317.33) and (204.7,272.57) .. (179.28,259.06) ;
			\draw [color={rgb, 255:red, 155; green, 155; blue, 155 }  ,draw opacity=0.5 ]   (156.05,337.89) .. controls (152.16,302.86) and (179.4,285.35) .. (179.28,259.06) ;
			\draw [color={rgb, 255:red, 155; green, 155; blue, 155 }  ,draw opacity=0.5 ]   (121.03,319.28) .. controls (142.43,288.14) and (159.94,288.14) .. (179.28,259.06) ;
			\draw [color={rgb, 255:red, 155; green, 155; blue, 155 }  ,draw opacity=0.5 ]   (246.38,259.06) .. controls (204.7,227.82) and (278.64,218.09) .. (307.83,210.31) ;
			\draw [color={rgb, 255:red, 155; green, 155; blue, 155 }  ,draw opacity=0.5 ]   (230,336.79) .. controls (198.86,297.87) and (218.32,257.01) .. (246.38,259.06) ;
			\draw [color={rgb, 255:red, 155; green, 155; blue, 155 }  ,draw opacity=0.5 ]   (266.97,317.33) .. controls (243.62,301.76) and (220.27,274.52) .. (246.38,259.06) ;
			\draw [color={rgb, 255:red, 155; green, 155; blue, 155 }  ,draw opacity=0.5 ]   (246.38,259.06) .. controls (255.29,229.76) and (282.54,229.12) .. (309.78,229.28) ;
			\draw [color={rgb, 255:red, 155; green, 155; blue, 155 }  ,draw opacity=0.5 ]   (246.38,259.06) .. controls (272.81,256.36) and (284.48,260.25) .. (302,260.25) ;
			\draw [color={rgb, 255:red, 155; green, 155; blue, 155 }  ,draw opacity=0.5 ]   (246.38,259.06) .. controls (257.24,274.52) and (268.92,288.14) .. (282.54,297.87) ;
			\draw  [fill={rgb, 255:red, 255; green, 255; blue, 255 }  ,fill opacity=1 ] (241.64,259.06) .. controls (241.64,256.44) and (243.76,254.32) .. (246.38,254.32) .. controls (249,254.32) and (251.12,256.44) .. (251.12,259.06) .. controls (251.12,261.68) and (249,263.8) .. (246.38,263.8) .. controls (243.76,263.8) and (241.64,261.68) .. (241.64,259.06) -- cycle ;
			\draw  [fill={rgb, 255:red, 255; green, 255; blue, 255 }  ,fill opacity=1 ] (129.27,176.51) .. controls (129.27,173.9) and (131.39,171.78) .. (134,171.78) .. controls (136.62,171.78) and (138.74,173.9) .. (138.74,176.51) .. controls (138.74,179.13) and (136.62,181.25) .. (134,181.25) .. controls (131.39,181.25) and (129.27,179.13) .. (129.27,176.51) -- cycle ;
			\draw  [fill={rgb, 255:red, 255; green, 255; blue, 255 }  ,fill opacity=1 ] (174.54,259.06) .. controls (174.54,256.44) and (176.66,254.32) .. (179.28,254.32) .. controls (181.9,254.32) and (184.02,256.44) .. (184.02,259.06) .. controls (184.02,261.68) and (181.9,263.8) .. (179.28,263.8) .. controls (176.66,263.8) and (174.54,261.68) .. (174.54,259.06) -- cycle ;
			\draw    (497.3,251.28) -- (537.16,259.06) ;
			\draw    (470.06,259.06) -- (497.3,251.28) ;
			\draw    (470.06,259.06) -- (446.71,212.36) ;
			\draw    (424.78,176.51) -- (446.71,212.36) ;
			\draw  [fill={rgb, 255:red, 0; green, 0; blue, 0 }  ,fill opacity=1 ] (441.97,212.36) .. controls (441.97,209.74) and (444.09,207.62) .. (446.71,207.62) .. controls (449.32,207.62) and (451.44,209.74) .. (451.44,212.36) .. controls (451.44,214.98) and (449.32,217.1) .. (446.71,217.1) .. controls (444.09,217.1) and (441.97,214.98) .. (441.97,212.36) -- cycle ;
			\draw  [fill={rgb, 255:red, 0; green, 0; blue, 0 }  ,fill opacity=1 ] (492.56,251.28) .. controls (492.56,248.66) and (494.68,246.54) .. (497.3,246.54) .. controls (499.92,246.54) and (502.04,248.66) .. (502.04,251.28) .. controls (502.04,253.89) and (499.92,256.01) .. (497.3,256.01) .. controls (494.68,256.01) and (492.56,253.89) .. (492.56,251.28) -- cycle ;
			\draw  [fill={rgb, 255:red, 255; green, 255; blue, 255 }  ,fill opacity=1 ] (532.42,259.06) .. controls (532.42,256.44) and (534.54,254.32) .. (537.16,254.32) .. controls (539.77,254.32) and (541.9,256.44) .. (541.9,259.06) .. controls (541.9,261.68) and (539.77,263.8) .. (537.16,263.8) .. controls (534.54,263.8) and (532.42,261.68) .. (532.42,259.06) -- cycle ;
			\draw  [fill={rgb, 255:red, 255; green, 255; blue, 255 }  ,fill opacity=1 ] (420.04,176.51) .. controls (420.04,173.9) and (422.17,171.78) .. (424.78,171.78) .. controls (427.4,171.78) and (429.52,173.9) .. (429.52,176.51) .. controls (429.52,179.13) and (427.4,181.25) .. (424.78,181.25) .. controls (422.17,181.25) and (420.04,179.13) .. (420.04,176.51) -- cycle ;
			\draw  [fill={rgb, 255:red, 255; green, 255; blue, 255 }  ,fill opacity=1 ] (465.32,259.06) .. controls (465.32,256.44) and (467.44,254.32) .. (470.06,254.32) .. controls (472.67,254.32) and (474.8,256.44) .. (474.8,259.06) .. controls (474.8,261.68) and (472.67,263.8) .. (470.06,263.8) .. controls (467.44,263.8) and (465.32,261.68) .. (465.32,259.06) -- cycle ;
			\draw [line width=0.75]    (327.17,235.71) -- (400.18,235.71) ;
			\draw [shift={(403.18,235.71)}, rotate = 180] [fill={rgb, 255:red, 0; green, 0; blue, 0 }  ][line width=0.08]  [draw opacity=0] (10.72,-5.15) -- (0,0) -- (10.72,5.15) -- (7.12,0) -- cycle    ;
			\draw [shift={(327.17,235.71)}, rotate = 180] [color={rgb, 255:red, 0; green, 0; blue, 0 }  ][line width=0.75]    (0,5.59) -- (0,-5.59)   ;
			
			\draw (107.66,172.12) node [anchor=north west][inner sep=0.75pt]   [align=left] {$\displaystyle z_{1}$};
			\draw (156.26,262.25) node [anchor=north west][inner sep=0.75pt]   [align=left] {$\displaystyle z_{2}$};
			\draw (264.29,237.01) node [anchor=north west][inner sep=0.75pt]   [align=left] {$\displaystyle z_{3}$};
			\draw (403.96,158.62) node [anchor=north west][inner sep=0.75pt]   [align=left] {$\displaystyle 1$};
			\draw (453.98,264.2) node [anchor=north west][inner sep=0.75pt]   [align=left] {$\displaystyle 2$};
			\draw (548.47,256.47) node [anchor=north west][inner sep=0.75pt]   [align=left] {$\displaystyle 3$};
			\draw (159.79,188.13) node [anchor=north west][inner sep=0.75pt]   [align=left] {$\displaystyle b_{1}$};
			\draw (203.41,224.05) node [anchor=north west][inner sep=0.75pt]   [align=left] {$\displaystyle b_{2}$};

		\end{tikzpicture}
		
		\caption{This picture shows how to associate a labelled tree with black and white vertices to a point $(z_1,z_2,z_3)\in F_3(\C)$. On the left we see the flow lines of the electric field $E(z)$: some of them are of finite length (the black ones) and connect points of the configuration (white vertices) to zeroes of $E(z)$ (black vertices).}
		\label{fig:associazione albero bw to a configuration}
	\end{figure}
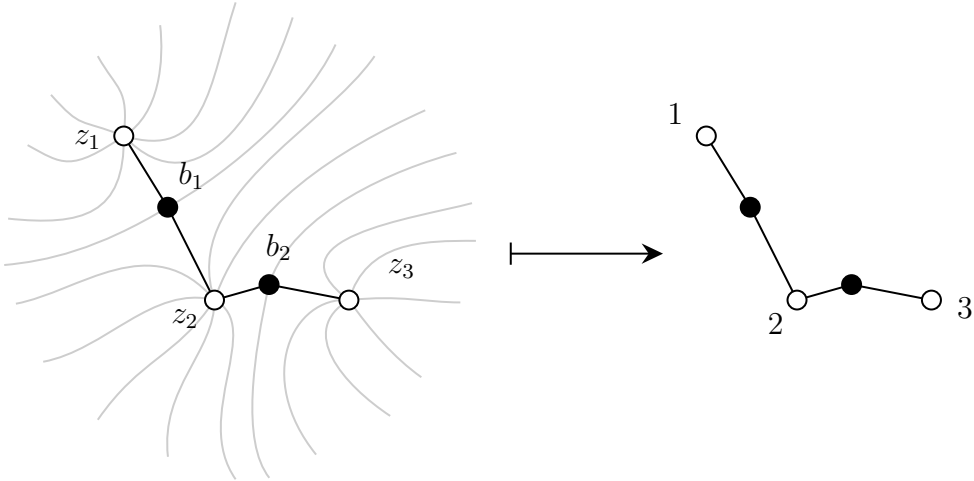
	This association factors through the action of $\C\rtimes \C^*$ and gives a homeomorphism (Theorem \ref{thm:omeo tra spazio di moduli e labelled trees}) 
	\[
	\phi_{a_1,\dots,a_n}:\M_{0,n+1}\cong \frac{F_n(\C)}{\C\rtimes \C^*}\to Tr_n
	\]
	where $Tr_n$ is the space of labelled trees (Definition \ref{defn: space of labelled trees}). Following the path of \cite{Salvatore} we define a regular CW-complex $N_n^{\sigma}(\mathcal{C}/S^1)$, called the \emph{space of nested cacti} (Paragraph \ref{subsec: space of nested cacti}) and we prove that $Tr_n$ is homeomorphic to a subspace of $N_n^{\sigma}(\mathcal{C}/S^1)$. The combinatorics of the cellular structure is similar to the case of $FM$, indeed the cells of $N_n^{\sigma}(\mathcal{C}/S^1)$ are indexed by trees with $n$ leaves and vertices labelled by cacti without basepoint. It is important to note that the cellular complexes of the spaces of unbased cacti form an operad $grav$ which is a chain model for the gravity operad (see Paragraph \ref{subsec: gravity operad}). To sum up we have an embedding
	\[
	\phi_{a_1,\dots,a_n}:\M_{0,n+1}\to Tr_n\subseteq N_n^{\sigma}(\mathcal{C}/S^1)
	\]
	and in Paragraph \ref{subsec: CW decomposition of Deligne mumford} we show that it extends to a homeomorphism
	\[
	\overline{\phi}_{a_1,\dots,a_n}:\overline{\M}_{0,n+1}\to N_n^{\sigma}(\mathcal{C}/S^1)
	\]
	giving a regular CW-decomposition of $\overline{\M}_{0,n+1}$ (Theorem \ref{thm:omeo tra la compattificazione e lo spazio dei nested cactus}). It is easy to see that these CW-decomposition are not compatible with the operad structure, in the sense that the composition maps $\circ_i:\overline{\M}_{0,n+1}\times \overline{\M}_{0,m+1}\to\overline{\M}_{0,n+m}$ are not even cellular maps. However there is some nice combinatorics going on: the Gravity operad has a chain model $grav$ based on cacti without base point (see Section \ref{sec: cactus models} or \cite{Ward}). It turns out that the collection of cellular complexes $C_*^{cell}(\overline{\M})\coloneqq\{C_*^{cell}(\overline{\M}_{0,n+1})\}_{n\geq 2}$ can be identified, up to a shift in degrees, with the operadic bar costruction $B(grav)$. Therefore it has a natural cooperad structure. Passing to cochains, i.e. to the dual cell decomposition by Poincarè duality, we get an operad in chain complexes which we call 
	\[
	C_*^{dual}(\overline{\M})\coloneqq\{C_*^{dual}(\overline{\M}_{0,n+1})\}_{n\geq 2}
	\]
	These chain complexes do not depends anymore on the choice of weights, and finally we prove that $C_*^{dual}(\overline{\M})$ is a chain model for the Hypercommutative operad, i.e. there is a zig-zag of quasi isomorphisms between $C_*^{dual}(\overline{\M})$ and the operad of singular chains $C_*(\overline{\M})$ (Theorem \ref{thm: chain model for hycomm}).
	
	\subsection{Two open problems:} 
	\begin{enumerate}
		\item The original purpose of this paper was to prove that $\overline{\M}$ is a cellular operad. This means that there are CW-decompositions of $\overline{\M}_{0,n+1}$, $n\in\N$, such that the operadic compositions are cellular maps. The dual cells described in Section \ref{sec:operad of dual cells} are good candidates for being such decompositions, but there are some issues due to the fact that they depend on the initial choice of some weights $a_1,\dots,a_n>0$. See Paragraph \ref{subsec:an open problem} for a wider overview on this problem.
		\item In \cite{Getzler} Getzler proved that the Hypercommutative and the Gravity operad are Koszul dual as operad of graded vector spaces. In this paper we lift this result to the category of dg-operads (Theorem \ref{thm: chain level Koszul duality}). The ultimate generalization would be to prove Koszul duality in the category of spectra: indeed both the Hypercommutative and Gravity operad have models in this category. In the first case one just take $\Sigma_+^{\infty}\overline{\M}$; in the second case we do not have a topological operad to suspend, but it possible to construct the gravity operad by taking the homotopy fixed points $(\Sigma_+^{\infty}\mathcal{D}_2)^{hS^1}$ (Westerland, \cite{Westerland}). If $\mathcal{P}$ is an operad in spectra, the Koszul dual in this context is usally denoted by $\mathcal{K}$ and it is obtained as $\mathcal{K}(\mathcal{P})\coloneqq F_{\mathbb{S}}(B(\mathcal{P}),\mathbb{S})$. Here $\mathbb{S}$ is the sphere spectrum,  $F_{\mathbb{S}}(-,-)$ is the internal hom functor and $B(-)$ is the operadic bar construction of Ching \cite{Ching}. In \cite[Section 4]{Ward3} Ward conjectures that $\mathcal{K}(\Sigma_+^{\infty}\overline{\M})$ is equivalent to  $\Lambda^{-2}(\Sigma_+^{\infty}\mathcal{D}_2)^{hS^1}$, where $\Lambda$ is the operadic suspension. In this paper we verify the analogue of this conjecture in the category of chain complexes (Theorem \ref{thm: chain level Koszul duality}). Observe that in the case of the little $n$ disks operad $\mathcal{D}_n$ the analogues of these results are known. At the level of chain complexes Fresse \cite{Fresse} proved that the Koszul dual of $\mathcal{D}_n$ is equivalent to its $n$-fold desuspension. The corresponding statement at the level of spectra is due to Ching and Salvatore \cite{Salvatore-Ching} who proved that $\mathcal{K}(\Sigma^{\infty}_+\mathcal{D}_n)\simeq \Lambda^{-n} \Sigma^{\infty}_+\mathcal{D}_n$, where $\Lambda$ is the operadic suspension. Ward suggests that combining this result with the main theorem of \cite{Ward3} one should get  that $\mathcal{K}(\Sigma_+^{\infty}\overline{\M})$ is equivalent to an operadic suspension of $(\Sigma_+^{\infty}\mathcal{D}_2)^{hS^1}$. 
	\end{enumerate}
		\subsection{Summary of the paper:} we start by reviewing the main facts about the Gravity and Hypercommutative operads (Section \ref{sec: operads and moduli spaces}). Then we recall the construction of an open cell decomposition of $\M_{0,n+1}$  which goes back to \cite{Salvatore} and \cite{Nakamura} (Section \ref{sec: combinatorial models for the open moduli space}). Section \ref{sec: cactus models} is an overview about cacti and their role in operad theory. In Section \ref{sec: combinatorial models for Deligne-Mumford} we extend the open cell decomposition of Section \ref{sec: combinatorial models for the open moduli space} to a regular CW-decomposition of $\overline{\M}_{0,n+1}$. Then we discuss the compatibility of these cells with the operad structure of $\overline{\M}$: in Section \ref{sec:operad of dual cells} we construct explicitly the dual cell decomposition of $\overline{\M}_{0,n+1}$ and we define an operad structure on the collection of cellular chain complexes $C_*^{dual}(\overline{\M})$. We prove in Section \ref{sec: chain model for hycomm} that $C_*^{dual}(\overline{\M})$ is weakly equivalent to the operad of singular chains $C_*(\overline{\M})$. Using the previous results we get that the Hypercommutative and Gravity operad are Koszul dual at the level of chains (Section \ref{sec: Koszul duality}, Theorem \ref{thm: chain level Koszul duality}).

	\subsection{Acknowledgments:} this paper is part of the first author's PhD thesis, written under the supervision of the second author. This was funded by the PhD program of the University of Roma Tor Vergata and by the MIUR Excellence Project MatMod@TOV awarded to the
	Department of Mathematics, University of Roma Tor Vergata,  CUP E83C18000100006.

	\section{Operads and moduli spaces}\label{sec: operads and moduli spaces}
	In this section we just recall the definition of the Hypercommutative and Gravity operads, stating the main results already present in the literature.
	
	\subsection{The Hypercommutative operad}\label{subsec: the hypercommutative operad}
	
	\subsubsection{Pointed stable curves} \label{subsec: pointed stable curves}
	Let $\M_{0,n+1}$ be the moduli space of genus zero Riemann surfaces with $n+1$ marked points. Grothendieck and Knudsen \cite{Deligne}, \cite{Knudsen} defined a canonical compactification $\overline{\M}_{0,n+1}$ of $\mathcal{M}_{0,n+1}$: $\overline{\M}_{0,n+1}$ is the moduli space of stable $(n + 1)$-pointed curves of genus $0$, i.e. the data $(C,p_0,\dots,p_n)$ of a (possibly reducible) algebraic curve $C$ with at most nodal singularities and smooth points $p_0,\dots,p_n\in C$ (all distinct) such that:
	\begin{itemize}
		\item Each irreducible component of $C$ is isomorphic to $\C P^1$.
		\item  The dual graph is a tree. Recall that the dual graph has one vertex for each irreducible components of $C$, edges corresponding to  intersection points and half edges for each marked point.
		\item \textbf{Stability condition:} 
		each irreducible component of $C$ has at least three special points, where a special point means either one of the $p_i$, $i=0,\dots,n$ or a singular point.
	\end{itemize}
	Figure \ref{fig:esempi di curve stabili} shows some examples of stable and not stable curves.
	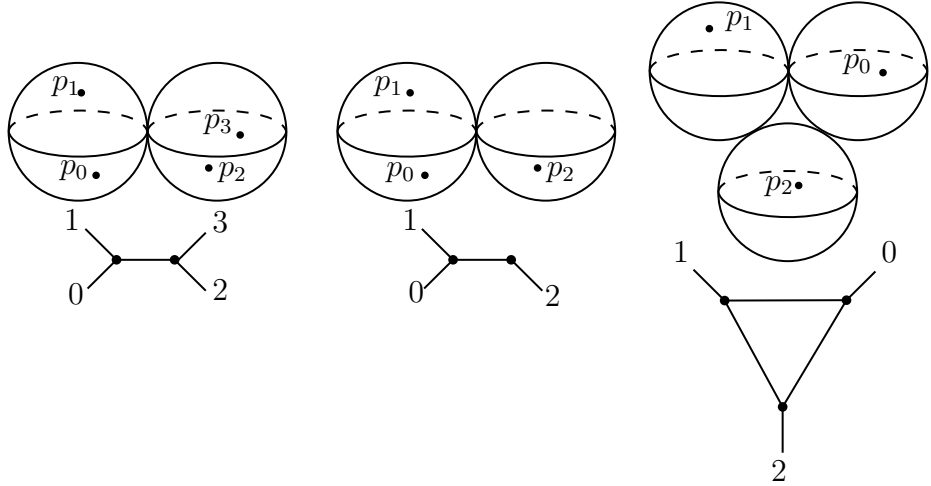
\begin{figure}
		\centering

		\tikzset{every picture/.style={line width=0.75pt}} 
		
		\begin{tikzpicture}[x=0.75pt,y=0.75pt,yscale=-1,xscale=1]
			
			\draw   (31.83,121.25) .. controls (31.83,102.11) and (47.35,86.6) .. (66.49,86.6) .. controls (85.63,86.6) and (101.14,102.11) .. (101.14,121.25) .. controls (101.14,140.39) and (85.63,155.9) .. (66.49,155.9) .. controls (47.35,155.9) and (31.83,140.39) .. (31.83,121.25) -- cycle ;
			\draw    (31.83,121.25) .. controls (36.29,137.46) and (96.23,138.39) .. (101.14,121.25) ;
			\draw  [dash pattern={on 4.5pt off 4.5pt}]  (31.83,121.25) .. controls (32.6,107.96) and (97.15,106.11) .. (101.14,121.25) ;
			\draw   (101.14,121.25) .. controls (101.14,102.11) and (116.66,86.6) .. (135.8,86.6) .. controls (154.94,86.6) and (170.45,102.11) .. (170.45,121.25) .. controls (170.45,140.39) and (154.94,155.9) .. (135.8,155.9) .. controls (116.66,155.9) and (101.14,140.39) .. (101.14,121.25) -- cycle ;
			\draw    (101.14,121.25) .. controls (105.6,137.46) and (165.53,138.39) .. (170.45,121.25) ;
			\draw  [dash pattern={on 4.5pt off 4.5pt}]  (101.14,121.25) .. controls (101.55,114.19) and (119.95,110.36) .. (137.71,110.59) .. controls (153.39,110.79) and (168.58,114.15) .. (170.45,121.25) ;
			\draw   (351.79,90.82) .. controls (351.79,71.68) and (367.31,56.17) .. (386.45,56.17) .. controls (405.59,56.17) and (421.1,71.68) .. (421.1,90.82) .. controls (421.1,109.96) and (405.59,125.48) .. (386.45,125.48) .. controls (367.31,125.48) and (351.79,109.96) .. (351.79,90.82) -- cycle ;
			\draw    (351.79,90.82) .. controls (356.25,107.03) and (416.19,107.96) .. (421.1,90.82) ;
			\draw  [dash pattern={on 4.5pt off 4.5pt}]  (351.79,90.82) .. controls (352.56,77.53) and (417.11,75.68) .. (421.1,90.82) ;
			\draw   (421.1,90.82) .. controls (421.1,71.68) and (436.62,56.17) .. (455.76,56.17) .. controls (474.9,56.17) and (490.41,71.68) .. (490.41,90.82) .. controls (490.41,109.96) and (474.9,125.48) .. (455.76,125.48) .. controls (436.62,125.48) and (421.1,109.96) .. (421.1,90.82) -- cycle ;
			\draw    (421.1,90.82) .. controls (425.56,107.03) and (485.5,107.96) .. (490.41,90.82) ;
			\draw  [dash pattern={on 4.5pt off 4.5pt}]  (421.1,90.82) .. controls (421.51,83.76) and (439.91,79.93) .. (457.67,80.16) .. controls (473.35,80.36) and (488.54,83.72) .. (490.41,90.82) ;
			\draw   (386.06,151.68) .. controls (386.06,132.54) and (401.58,117.02) .. (420.72,117.02) .. controls (439.86,117.02) and (455.37,132.54) .. (455.37,151.68) .. controls (455.37,170.82) and (439.86,186.33) .. (420.72,186.33) .. controls (401.58,186.33) and (386.06,170.82) .. (386.06,151.68) -- cycle ;
			\draw    (386.06,151.68) .. controls (390.52,167.89) and (450.46,168.81) .. (455.37,151.68) ;
			\draw  [dash pattern={on 4.5pt off 4.5pt}]  (386.06,151.68) .. controls (386.47,144.62) and (404.87,140.79) .. (422.63,141.02) .. controls (438.31,141.22) and (453.5,144.58) .. (455.37,151.68) ;
			\draw  [fill={rgb, 255:red, 0; green, 0; blue, 0 }  ,fill opacity=1 ] (66.72,101.66) .. controls (66.72,100.89) and (67.34,100.27) .. (68.1,100.27) .. controls (68.87,100.27) and (69.48,100.89) .. (69.48,101.66) .. controls (69.48,102.42) and (68.87,103.04) .. (68.1,103.04) .. controls (67.34,103.04) and (66.72,102.42) .. (66.72,101.66) -- cycle ;
			\draw  [fill={rgb, 255:red, 0; green, 0; blue, 0 }  ,fill opacity=1 ] (74.1,143.15) .. controls (74.1,142.39) and (74.71,141.77) .. (75.48,141.77) .. controls (76.24,141.77) and (76.86,142.39) .. (76.86,143.15) .. controls (76.86,143.91) and (76.24,144.53) .. (75.48,144.53) .. controls (74.71,144.53) and (74.1,143.91) .. (74.1,143.15) -- cycle ;
			\draw  [fill={rgb, 255:red, 0; green, 0; blue, 0 }  ,fill opacity=1 ] (146.02,122.86) .. controls (146.02,122.1) and (146.64,121.48) .. (147.4,121.48) .. controls (148.16,121.48) and (148.78,122.1) .. (148.78,122.86) .. controls (148.78,123.63) and (148.16,124.25) .. (147.4,124.25) .. controls (146.64,124.25) and (146.02,123.63) .. (146.02,122.86) -- cycle ;
			\draw  [fill={rgb, 255:red, 0; green, 0; blue, 0 }  ,fill opacity=1 ] (130.34,139.46) .. controls (130.34,138.7) and (130.96,138.08) .. (131.73,138.08) .. controls (132.49,138.08) and (133.11,138.7) .. (133.11,139.46) .. controls (133.11,140.22) and (132.49,140.84) .. (131.73,140.84) .. controls (130.96,140.84) and (130.34,140.22) .. (130.34,139.46) -- cycle ;
			\draw  [fill={rgb, 255:red, 0; green, 0; blue, 0 }  ,fill opacity=1 ] (380.23,69.38) .. controls (380.23,68.62) and (380.84,68) .. (381.61,68) .. controls (382.37,68) and (382.99,68.62) .. (382.99,69.38) .. controls (382.99,70.15) and (382.37,70.77) .. (381.61,70.77) .. controls (380.84,70.77) and (380.23,70.15) .. (380.23,69.38) -- cycle ;
			\draw  [fill={rgb, 255:red, 0; green, 0; blue, 0 }  ,fill opacity=1 ] (424.8,148.27) .. controls (424.8,147.5) and (425.42,146.88) .. (426.18,146.88) .. controls (426.94,146.88) and (427.56,147.5) .. (427.56,148.27) .. controls (427.56,149.03) and (426.94,149.65) .. (426.18,149.65) .. controls (425.42,149.65) and (424.8,149.03) .. (424.8,148.27) -- cycle ;
			\draw   (195.96,121.25) .. controls (195.96,102.11) and (211.48,86.6) .. (230.62,86.6) .. controls (249.76,86.6) and (265.27,102.11) .. (265.27,121.25) .. controls (265.27,140.39) and (249.76,155.9) .. (230.62,155.9) .. controls (211.48,155.9) and (195.96,140.39) .. (195.96,121.25) -- cycle ;
			\draw    (195.96,121.25) .. controls (200.42,137.46) and (260.35,138.39) .. (265.27,121.25) ;
			\draw  [dash pattern={on 4.5pt off 4.5pt}]  (195.96,121.25) .. controls (196.73,107.96) and (261.28,106.11) .. (265.27,121.25) ;
			\draw   (265.27,121.25) .. controls (265.27,102.11) and (280.79,86.6) .. (299.93,86.6) .. controls (319.07,86.6) and (334.58,102.11) .. (334.58,121.25) .. controls (334.58,140.39) and (319.07,155.9) .. (299.93,155.9) .. controls (280.79,155.9) and (265.27,140.39) .. (265.27,121.25) -- cycle ;
			\draw    (265.27,121.25) .. controls (269.73,137.46) and (329.66,138.39) .. (334.58,121.25) ;
			\draw  [dash pattern={on 4.5pt off 4.5pt}]  (265.27,121.25) .. controls (265.68,114.19) and (284.08,110.36) .. (301.84,110.59) .. controls (317.52,110.79) and (332.71,114.15) .. (334.58,121.25) ;
			\draw  [fill={rgb, 255:red, 0; green, 0; blue, 0 }  ,fill opacity=1 ] (230.85,101.66) .. controls (230.85,100.89) and (231.47,100.27) .. (232.23,100.27) .. controls (233,100.27) and (233.61,100.89) .. (233.61,101.66) .. controls (233.61,102.42) and (233,103.04) .. (232.23,103.04) .. controls (231.47,103.04) and (230.85,102.42) .. (230.85,101.66) -- cycle ;
			\draw  [fill={rgb, 255:red, 0; green, 0; blue, 0 }  ,fill opacity=1 ] (238.23,143.15) .. controls (238.23,142.39) and (238.84,141.77) .. (239.61,141.77) .. controls (240.37,141.77) and (240.99,142.39) .. (240.99,143.15) .. controls (240.99,143.91) and (240.37,144.53) .. (239.61,144.53) .. controls (238.84,144.53) and (238.23,143.91) .. (238.23,143.15) -- cycle ;
			\draw  [fill={rgb, 255:red, 0; green, 0; blue, 0 }  ,fill opacity=1 ] (294.47,139.46) .. controls (294.47,138.7) and (295.09,138.08) .. (295.85,138.08) .. controls (296.62,138.08) and (297.24,138.7) .. (297.24,139.46) .. controls (297.24,140.22) and (296.62,140.84) .. (295.85,140.84) .. controls (295.09,140.84) and (294.47,140.22) .. (294.47,139.46) -- cycle ;
			\draw  [fill={rgb, 255:red, 0; green, 0; blue, 0 }  ,fill opacity=1 ] (466.9,91.51) .. controls (466.9,90.75) and (467.52,90.13) .. (468.28,90.13) .. controls (469.05,90.13) and (469.67,90.75) .. (469.67,91.51) .. controls (469.67,92.28) and (469.05,92.9) .. (468.28,92.9) .. controls (467.52,92.9) and (466.9,92.28) .. (466.9,91.51) -- cycle ;
			\draw  [fill={rgb, 255:red, 0; green, 0; blue, 0 }  ,fill opacity=1 ] (83.76,185.66) .. controls (83.76,184.6) and (84.61,183.75) .. (85.67,183.75) .. controls (86.72,183.75) and (87.58,184.6) .. (87.58,185.66) .. controls (87.58,186.71) and (86.72,187.56) .. (85.67,187.56) .. controls (84.61,187.56) and (83.76,186.71) .. (83.76,185.66) -- cycle ;
			\draw    (85.67,185.66) -- (114.67,185.66) ;
			\draw  [fill={rgb, 255:red, 0; green, 0; blue, 0 }  ,fill opacity=1 ] (112.76,185.66) .. controls (112.76,184.6) and (113.61,183.75) .. (114.67,183.75) .. controls (115.72,183.75) and (116.58,184.6) .. (116.58,185.66) .. controls (116.58,186.71) and (115.72,187.56) .. (114.67,187.56) .. controls (113.61,187.56) and (112.76,186.71) .. (112.76,185.66) -- cycle ;
			\draw  [fill={rgb, 255:red, 0; green, 0; blue, 0 }  ,fill opacity=1 ] (251.76,185.66) .. controls (251.76,184.6) and (252.61,183.75) .. (253.67,183.75) .. controls (254.72,183.75) and (255.58,184.6) .. (255.58,185.66) .. controls (255.58,186.71) and (254.72,187.56) .. (253.67,187.56) .. controls (252.61,187.56) and (251.76,186.71) .. (251.76,185.66) -- cycle ;
			\draw    (253.67,185.66) -- (282.67,185.66) ;
			\draw  [fill={rgb, 255:red, 0; green, 0; blue, 0 }  ,fill opacity=1 ] (280.76,185.66) .. controls (280.76,184.6) and (281.61,183.75) .. (282.67,183.75) .. controls (283.72,183.75) and (284.58,184.6) .. (284.58,185.66) .. controls (284.58,186.71) and (283.72,187.56) .. (282.67,187.56) .. controls (281.61,187.56) and (280.76,186.71) .. (280.76,185.66) -- cycle ;
			\draw  [fill={rgb, 255:red, 0; green, 0; blue, 0 }  ,fill opacity=1 ] (448.14,205.96) .. controls (448.14,204.91) and (449,204.05) .. (450.05,204.05) .. controls (451.11,204.05) and (451.96,204.91) .. (451.96,205.96) .. controls (451.96,207.02) and (451.11,207.87) .. (450.05,207.87) .. controls (449,207.87) and (448.14,207.02) .. (448.14,205.96) -- cycle ;
			\draw   (389.23,206.15) -- (450.05,205.96) -- (418.18,259.64) -- cycle ;
			\draw  [fill={rgb, 255:red, 0; green, 0; blue, 0 }  ,fill opacity=1 ] (416.27,259.64) .. controls (416.27,258.59) and (417.13,257.73) .. (418.18,257.73) .. controls (419.24,257.73) and (420.09,258.59) .. (420.09,259.64) .. controls (420.09,260.69) and (419.24,261.55) .. (418.18,261.55) .. controls (417.13,261.55) and (416.27,260.69) .. (416.27,259.64) -- cycle ;
			\draw  [fill={rgb, 255:red, 0; green, 0; blue, 0 }  ,fill opacity=1 ] (387.32,206.15) .. controls (387.32,205.09) and (388.17,204.24) .. (389.23,204.24) .. controls (390.28,204.24) and (391.14,205.09) .. (391.14,206.15) .. controls (391.14,207.2) and (390.28,208.06) .. (389.23,208.06) .. controls (388.17,208.06) and (387.32,207.2) .. (387.32,206.15) -- cycle ;
			\draw    (114.67,185.66) -- (130.23,201.21) ;
			\draw    (116.58,185.66) -- (130.23,172.01) ;
			\draw    (71.23,200.1) -- (85.67,185.66) ;
			\draw    (70.11,170.1) -- (85.67,185.66) ;
			\draw    (239.23,200.1) -- (253.67,185.66) ;
			\draw    (238.11,170.1) -- (253.67,185.66) ;
			\draw    (282.67,185.66) -- (298.23,201.21) ;
			\draw    (373.67,190.59) -- (389.23,206.15) ;
			\draw    (418.18,259.64) -- (418.18,282.68) ;
			\draw    (450.05,205.96) -- (464.49,191.52) ;
			
			\draw (52,92) node [anchor=north west][inner sep=0.75pt]   [align=left] {$\displaystyle p_{1}$};
			\draw (213,92) node [anchor=north west][inner sep=0.75pt]   [align=left] {$\displaystyle p_{1}$};
			\draw (388.45,59.17) node [anchor=north west][inner sep=0.75pt]   [align=left] {$\displaystyle p_{1}$};
			\draw (56,134) node [anchor=north west][inner sep=0.75pt]   [align=left] {$\displaystyle p_{0}$};
			\draw (128,111) node [anchor=north west][inner sep=0.75pt]   [align=left] {$\displaystyle p_{3}$};
			\draw (219,135) node [anchor=north west][inner sep=0.75pt]   [align=left] {$\displaystyle p_{0}$};
			\draw (447,81) node [anchor=north west][inner sep=0.75pt]   [align=left] {$\displaystyle p_{0}$};
			\draw (135,135) node [anchor=north west][inner sep=0.75pt]   [align=left] {$\displaystyle p_{2}$};
			\draw (299,135) node [anchor=north west][inner sep=0.75pt]   [align=left] {$\displaystyle p_{2}$};
			\draw (408,142) node [anchor=north west][inner sep=0.75pt]   [align=left] {$\displaystyle p_{2}$};
			\draw (58,159) node [anchor=north west][inner sep=0.75pt]   [align=left] {$\displaystyle 1$};
			\draw (60,196) node [anchor=north west][inner sep=0.75pt]   [align=left] {$\displaystyle 0$};
			\draw (132,160) node [anchor=north west][inner sep=0.75pt]   [align=left] {$\displaystyle 3$};
			\draw (132,194) node [anchor=north west][inner sep=0.75pt]   [align=left] {$\displaystyle 2$};
			\draw (227,159) node [anchor=north west][inner sep=0.75pt]   [align=left] {$\displaystyle 1$};
			\draw (230,195) node [anchor=north west][inner sep=0.75pt]   [align=left] {$\displaystyle 0$};
			\draw (298,197) node [anchor=north west][inner sep=0.75pt]   [align=left] {$\displaystyle 2$};
			\draw (411,285) node [anchor=north west][inner sep=0.75pt]   [align=left] {$\displaystyle 2$};
			\draw (362,175) node [anchor=north west][inner sep=0.75pt]   [align=left] {$\displaystyle 1$};
			\draw (466,175) node [anchor=north west][inner sep=0.75pt]   [align=left] {$\displaystyle 0$};

		\end{tikzpicture}

		\caption{In this picture we see three nodal curves with marked points, and below their dual graphs: the curve on the left is stable, while the others are not. The middle one does not satisfy the stability condition, while the curve on the right is not stable since its dual graph is not a tree.}
		\label{fig:esempi di curve stabili}
	\end{figure}
	It turns out that $\overline{\M}_{0,n+1}$ is a smooth complex projective variety of dimension $n-2$ and contains $\M_{0,n+1}$ as an open dense subset. An explicit construction of $\overline{\M}_{0,n+1}$ by a sequence of complex blow ups can be found in \cite{Markl-Stasheff}. 
	
	\subsubsection{The Deligne-Mumford operad}
	
	For any $m,n\in\N$ and $i=1,\dots,n$ we can define \emph{gluing maps}
	\begin{align*}
		\circ_i:\overline{\M}_{0,n+1}\times \overline{\M}_{0,m+1}&\to \overline{\M}_{0,n+m}\\
		(C_1,p_0,\dots,p_n)\times(C_2,q_0,\dots,q_m)&\mapsto (C,p_0,\dots,p_{i-1},q_1,\dots,q_m,p_{i+1},\dots,p_n)
	\end{align*}
	where $C$ is the curve obtained from $C_1\sqcup C_2$ identifying $p_i$ and $q_0$ and introducing a nodal singularity.
	
	\begin{defn}
		Let us denote by $\overline{\M}(n)\coloneqq\overline{\M}_{0,n+1}$. $\Sigma_n$ acts on $\overline{\M}(n)$ by permuting the marked points $p_1,\dots,p_n$ of a stable curve $(C,p_0,\dots,p_n)\in\overline{\M}_{0,n+1}$. The gluing maps defined above define an operad structure on the collection $\overline{\M}\coloneqq\{\overline{\M}(n)\}_{n\geq 2}$. We call $\overline{\M}$ the \textbf{Deligne-Mumford operad}.
	\end{defn}
	\begin{oss}
		We have a natural action of $\Sigma_{n+1}$ on $\overline{\M}(n)\coloneqq\overline{\M}_{0,n+1}$ by permuting all the marked points, turning $\overline{\M}$ into a \emph{cyclic operad}.
	\end{oss}
	
	\subsubsection{The Hypercommutative operad:}
	since homology is a lax monoidal functor, the operadic structure of $\overline{\M}$ induce an operadic structure on $H_*(\overline{\M};\Z)$, giving an operad in graded abelian groups:
	\begin{defn}
		The \textbf{Hypercommutative operad} $Hycom$ is a (cyclic) operad in graded abelian groups whose arity $n$ operations are given by
		\[
		Hycom(n)\coloneqq H_*(\overline{\M}_{0,n+1};\Z)
		\]
		As an operad, it is generated by (graded) symmetric operations of degree $2(n-2)$
		\[
		(a_1,\dots,a_n)\in H_*(\overline{\M}_{0,n+1};\Z) \quad n\geq 2
		\]
		which satisfy the following \emph{generalized associativity relations}:
		\begin{equation}\label{eq: hypercommutative algebra equations}
			\sum_{S_1\sqcup S_2=\{1,\dots,n\}}\pm ((a,b,x_{S_1}),c,x_{S_2})=  \sum_{S_1\sqcup S_2=\{1,\dots,n\}}\pm (a,(b,c,x_{S_1}),x_{S_2})
		\end{equation}

		where if $S=\{s_1,\dots,s_k\}\subseteq\{1,\dots,n\}$, $x_S$ is an abbreviation for $x_{s_1},\dots,x_{s_k}$. The sign in front of each summand is determined by the Koszul sign rule. We report below some explicit examples of these relations:
		\begin{itemize}
			\item $n=0$: we get $((a,b),c)=(a,(b,c))$, i.e. the binary operation is associative (and commutative).
			\item $n=1$: $((a,b),c,d)+(-1)^{\abs{c}\abs{d}}((a,b,d),c)=(a,(b,c),d)+(a,(b,c,d))$.
		\end{itemize}
		
		Further details can be found in \cite{Getzler} and \cite{Kontsevich-Manin}.
	\end{defn}
	\begin{oss}
		Geometrically, the $n$-ary operation $(a_1,\dots,a_n)$ corresponds to the fundamental class $[\overline{\M}_{0,n+1}]$.
	\end{oss}
	\begin{defn}
		An \textbf{hypercommutative algebra}  (in the category of chain complexes) is just an algebra over the Hypercommutative operad. Explicitly, it is a chain complex $(A,d_A)$ equipped by (graded) symmetric products $(-,\cdots,-):A^{\otimes n}\to A$ of degree $2(n-2)$, such that Equation \ref{eq: hypercommutative algebra equations} is satisfied for any choice of variables $a,b,c,x_1,\dots,x_n\in A$.
	\end{defn}
	We end this paragraph by giving some examples of algebras over $Hycom$:
	\begin{itemize}
		\item A commutative algebra is a special case of Hypercommutative algebra, where all the higher products $(-,\dots,-):A^{\otimes n}\to A$ with $n\geq 3$ are zero.
		\item Barannikov and Kontsevich defined a canonical hypercommutative algebra structure on the cohomology of any compact Calabi-Yau manifold, see \cite{Barannikov-Kontsevich}.
		\item  Kontsevich and Manin have shown in \cite{Kontsevich-Manin} through Gromov-Witten invariants that the rational homology of a smooth projective algebraic variety is a hypercommutative algebra.
		
	\end{itemize}

	\subsection{The Gravity operad}\label{subsec: gravity operad}
	The Gravity operad was introduced by Getzler in \cite{Getzler94} and \cite{Getzler}. In this paragraph we recall the main facts about the Gravity operad and the algebras over it. There are also chain model versions of the Gravity operad, described in the paper \cite{Westerland} by Westerland and in \cite{Getzler-Kapranov} by Getzler-Kapranov. A comparison between these definitions has been written by Dupont and Horel in \cite{Dupont-Horel}. 
	In what follows all the homology groups are taken with integer coefficients, unless otherwise stated. To ease the notation we sometimes write $H_*(X)$ instead of $H_*(X;\Z)$.
	
	\begin{defn}
		
		Consider the graded abelian group $Grav(n)\coloneqq s H_*(\mathcal{M}_{0,n+1};\Z)$, where $s$ is the degree shift. The collection $Grav\coloneqq\{Grav(n)\}_{n\geq2}$ forms an operad of graded abelian groups, called the \textbf{Gravity operad}. The composition maps can be described as follows: first observe that the orbit space $F_n(\C)/S^1$ is homotopy equivalent to $\mathcal{M}_{0,n+1}$. The quotient map $p:F_n(\C)\to F_n(\C)/S^1$ has a section given by
		\begin{align*}
			j: F_n(\C)/S^1&\to F_n(\C)\\
			[z_1,\dots,z_n]&\mapsto \left( \frac{z_2-z_1}{\abs{z_2-z_1}}\right) ^{-1}\cdot(z_1,\dots,z_n)
		\end{align*}
		Therefore $p$ is a trivial $S^1$-principal bundle and the transfer map
		\[
		\tau: H_*(F_n(\C)/S^1)\to  H_*(F_n(\C))
		\]
		is injective. Now let $\circ_i:H_*(F_n(\C))\otimes H_*(F_m(\C))\to H_*(F_{n+m-1}(\C))$ be the map induced in homology by the $i$-th composition of the little two disk operad. Getzler observed that given two classes $a\in H_*(F_n(\C)/S^1)$, $b\in H_*(F_m(\C)/S^1)$, the composition $\tau(a)\circ_i\tau(b)$ is the transfer of a unique class of $H_*(F_{n+m-1}(\C)/S^1)$. Therefore we can define $a\circ_i b\coloneqq \tau^{-1}(\tau(a)\circ_i\tau(b))$ and we get the following commutative diagram:
		\[
		\begin{tikzcd}
			& H_*(F_n(\C))\otimes H_*(F_m(\C))\arrow[r,"\circ_i"]& H_*(F_{n+m-1}(\C))\\
			&  H_*(F_n(\C)/S^1)\otimes H_*(F_m(\C)/S^1)\arrow[u,"\tau\otimes \tau"]\arrow[r, dashed,"\circ_i"] & H_*(F_{n+m-1}(\C)/S^1) \arrow[u,"\tau"]
		\end{tikzcd}
		\]
		The $i$-th operadic composition of the gravity operad is then defined to be the dashed arrow of the above diagram. Note that this map raises the degree by $1$. If we shift the graded vector spaces $H_*(F_n(\C)/S^1)$ by one we get a map of degree zero, justifying the shifting term in the definition of the Gravity operad. 
	\end{defn}
	\begin{oss}
		The action of $\Sigma_{n+1}$ on $\M_{0,n+1}$ by relabelling the points induces an action in homology, making $Grav$ a \emph{cyclic operad}.
	\end{oss}
	Unlike many familiar operads, the Gravity operad is not generated by a finite number of operations. However, it has a nice presentation with infinitely many generators:
	\begin{thm}[Getzler \cite{Getzler94}]
		As an operad $Grav$ is generated by (graded) symmetric operations of degree one
		\[
		\{a_1,\dots,a_n\}\in Grav(n) \quad \text{ for } n\geq 2
		\]
		Geometrically, $\{a_1,\dots,a_n\}$ corresponds to the generator of $H_0(\mathcal{M}_{0,n+1},\Z)$. These operations (called brackets) satisfy the so called generalized Jacobi relations: for any $k\geq 2$ and $l\in \N$
		\begin{equation}\label{eq: generalized jacobi relations}
			\sum_{1\leq i<j\leq k}(-1)^{\epsilon(i,j)} \{\{a_i,a_j\},a_1,\dots,\hat{a}_i,\dots,\hat{a}_j,\dots,a_k,b_1,\dots,b_l\}=
			\{\{a_1,\dots,a_k\},b_1,\dots,b_l\} 
		\end{equation}
		where the right hand term is interpreted as zero if $l=0$ and $\epsilon(i,j)=(\abs{a_1}+\dots+ \abs{a_{i-1}})\abs{a_i}+(\abs{a_1}+\dots+\abs{a_{j-1}})\abs{a_j}+ \abs{a_i}\abs{a_j}$.
	\end{thm}
	Some explicit examples of this relations are:
	\begin{itemize}
		\item \textbf{Jacobi relation} ($k=3$, $l=0$):
		\[
		\{\{a_1,a_2\},a_3\}+(-1)^{\abs{a_2}\abs{a_3}}\{\{a_1,a_3\},a_2\}+(-1)^{\abs{a_1}(\abs{a_2}+\abs{a_3})}\{\{a_2,a_3\},a_1\}=0
		\]
		This shows that the binary bracket $\{-,-\}$ is a Lie bracket (of degree one).
		\item $k=3$, $l=1$:
		\[
		\{\{a_1,a_2\},a_3,b_1\}\pm\{\{a_1,a_3\},a_2,b_1\}\pm\{\{a_2,a_3\},a_1,b_1\}=\{\{a_1,a_2,a_3\},b_1\}
		\]
	\end{itemize}
	\begin{defn}
		A \textbf{Gravity algebra} (in the category of chain complexes) is an algebra over the Gravity operad. To be explicit, it is a chain complex $(A,d_A)$ together with graded symmetric chain maps $\{-,\dots,-\}:A^{\otimes k}\to A$ of degree one such that for $k\geq 3$, $l\geq 0$ and $a_1,\dots,a_k,b_1,\dots,b_l\in A$ Equation \ref{eq: generalized jacobi relations} is satisfied.
	\end{defn}
	\begin{oss}
		We can think of a Gravity algebra as a generalization of a (differential graded) Lie algebra. Indeed we have a binary Lie bracket (of degree one) and many other higher brackets which satisfy a higher version of the usual Jacobi relation.
	\end{oss}
	\begin{oss}
		The definition of Gravity operad we just gave is the one contained in \cite{Getzler94}. In the paper \cite{Getzler} E. Getzler calls \emph{Gravity operad} the operadic suspension $\Lambda Grav$, i.e. the operad which encode the algebraic structure on the suspension of a gravity algebra. More explicitly, its arity $n$ space is 
		\[
		(\Lambda Grav)(n)\coloneqq s^{n-2}sgn\otimes H_*(\M_{0,n+1};\Z)
		\]
		where $sgn$ denote the sign representation of $\Sigma_n$. Because of the sign representation, the generators of $\Lambda Grav$ become (graded) antisymmetric operations (of degree $n-2$), so the binary bracket is actually a Lie bracket. 
	\end{oss}
	\begin{oss}(Koszul duality)\label{oss: koszul duality}
		An old result by D. Quillen \cite{Quillen} estabilished a duality between commutative and Lie algebras. The main result of \cite{Getzler} is an analogue of Quillen's result, replacing commutative algebras by Hypercommutative algebras and Lie algebras by $\Lambda Grav$-algebras. This can be summarized by saying that		$Hycom$ and  $\Lambda Grav$ are Koszul dual to each other. Here Koszul duality must be interpreted in the sense of Ginzburg-Kapranov \cite{Ginzburg-Kapranov}.
	\end{oss}
	Nice examples of gravity algebras arise in string topology and cyclic cohomology. We refere to Westerland \cite{Westerland} and Ward \cite{Ward} for further details.
	
	\section{An open cell decomposition of $\mathcal{M}_{0,n+1}$}\label{sec: combinatorial models for the open moduli space}
	In this section  we review very quickly a combinatorial model for the moduli space $\mathcal{M}_{0,n+1}$ based on trees due to the second author. We will skip all the details, which can be found in \cite{Salvatore}. Before starting let us observe that $\M_{0,n+1}$ can be seen as a quotient of the ordered configuration space $F_n(\C)$:
	\begin{prop}
		Let $\C\rtimes \C^*$ act on $F_n(\C)$ by translations, dilations and rotations. Then $\M_{0,n+1}$ is homeomorphic to $F_n(\C)/\C\rtimes\C^*$.
	\end{prop}
	\begin{proof}
		A point in $\M_{0,n+1}$ is just a configuration of points $(p_0,p_1,\dots,p_n)$ in the Riemann sphere up to biholomorphisms. If we rotate suitably the sphere we can suppose that $p_0$ is the point at the infinity $\infty\in \C\cup\{\infty\}$. Deleting this point and using the stereographic projection we obtain a configuration of $n$ points in the complex plane $\numberset{C}$ up to translation, dilatations and rotations, and this proves the claim.
	\end{proof}
	It will be useful to keep in mind this identification for the rest of this work. The idea to get a combinatorial model of $\mathcal{M}_{0,n+1}\cong F_n(\mathbb{C})/\numberset{C}\rtimes\C^*$ is the following: given a configuration of $n$ points $(z_1,\dots,z_n)\in F_n(\C)$, think of each $z_i$ as a negative electric charge of value $-a_i$, where each $a_i$ is a fixed real number strictly greater than $0$ and $\sum_{1=1}^n a_i=1$ (normalization condition). Each charge generates a radial electric field and by superposition we obtain a total electric field
	\[
	E(z)\coloneqq \sum_{i=1}^n-a_i\frac{z-z_i}{\abs{z-z_i}^2}
	\]
	The electric field is conservative, and it is not difficult to find an explicit potential for $E(z)$: indeed it turns out that 
	\[
	E(z)=-\nabla U(z)
	\]
	where $U(z)\coloneqq \log \abs{h(z)}$ and $h(z)\coloneqq\prod_{i=1}^n(z-z_i)^{a_i}$. Now look at the flow lines of $E(z)$, i.e. curves $\gamma(t)$ such that $\gamma'(t)=E(\gamma(t))$: most of them start at one point $z_i$ of the configuration $(z_1,\dots,z_n)$ and go to infinity; conversely, there are some flow lines which are of finite length, namely those that connects two zeros of $E(z)$ or one zero to a point $z_i$ of the configuration. Starting from these special flow lines we can obtain what we will call an \emph{admissible tree} $T$ with $n$ leaves:
	\begin{defn}
		An \textbf{admissible tree} $T$ with $n$-leaves is a tree (connected graph without closed paths) such that:
		\begin{itemize}
			\item Its vertices are colored with two colors, black and white. We will write $V(T)=W\sqcup B$, where $W$ is the set of white vertices, and $B$ is the set of black vertices. Moreover we require that there are exactly $n$ white vertices and that they are labelled by the numbers $\{1,\dots,n\}$. White vertices are also called \textbf{leaves}.
			\item The edges of $T$ are oriented, i.e. they are ordered couples $e=(v,w)\in V(T)\times V(T)$. $v$ is called the source of $e$, $w$ is called the target. We say that an edge $e$ is incident to a vertex $v$  if $v$ is either the source or the target of $e$.
			\item The set $E_v$ of edges incident to a fixed vertex $v\in V(T)$ is equipped with a cyclic ordering. In other words, $T$ is a \textbf{ribbon graph}.
		\end{itemize}
		Moreover we require the following conditions to be satisfied:
		\begin{enumerate}
			\item A white vertex is not allowed to be a source.
			\item If $(v,w)\in E(T)$ is an edge, then $(w,v)\notin E(T)$, i.e. an edge can not be inverted.
			\item Any black vertex $v\in B$ is the source of at least two distinct edges. Moreover it can not be the target of two edges that are next to each other in the cyclic ordering of $E_v$.
		\end{enumerate}
		We will denote by $T_n$ the set of isomorphism classes of admissible trees with $n$-leaves. See Figure \ref{fig:esempi di black and white trees} for some examples.
	\end{defn}
	
	In our case, the black vertices are the zeros of $E(z)$, the white vertices are $\{z_1,\dots,z_n\}$, and the oriented edges are given by the flow lines of finite length (which are naturally oriented). The cyclic order on the set of incident edges $E_v$ is induced by a fixed orientation of the plane (say counterclockwise). See Figure \ref{fig:associare un albero a una configurazione} for an example.
	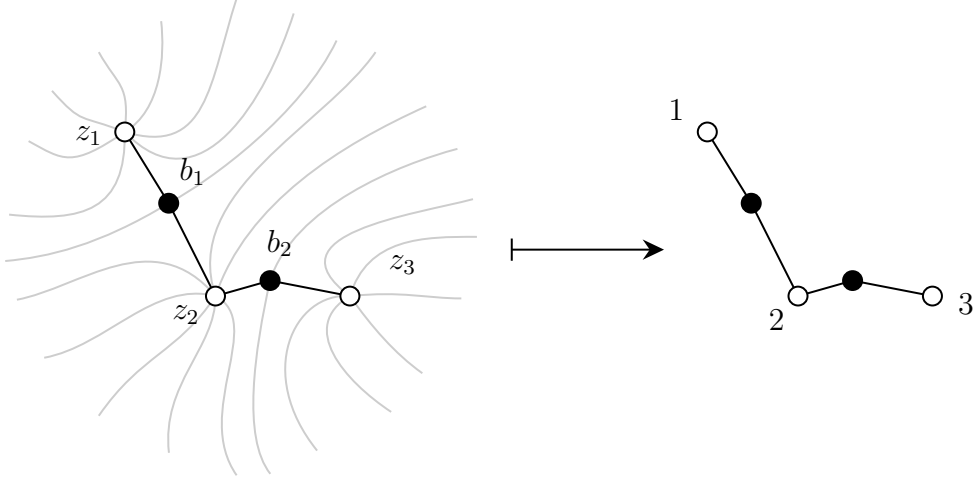
\begin{figure}
		\centering

		\tikzset{every picture/.style={line width=0.75pt}} 
		
		\begin{tikzpicture}[x=0.75pt,y=0.75pt,yscale=-1,xscale=1]
			
			\draw [color={rgb, 255:red, 155; green, 155; blue, 155 }  ,draw opacity=0.5 ]   (155.93,212.36) .. controls (177.33,200.68) and (222.21,167.49) .. (239.73,130.52) ;
			\draw [color={rgb, 255:red, 155; green, 155; blue, 155 }  ,draw opacity=0.5 ]   (206.65,348.47) .. controls (192.32,330.2) and (198.86,286.2) .. (206.52,251.28) ;
			\draw [color={rgb, 255:red, 155; green, 155; blue, 155 }  ,draw opacity=0.5 ]   (206.52,251.28) .. controls (218.32,214.2) and (270.86,192.79) .. (300.05,185.01) ;
			\draw    (206.52,251.28) -- (246.38,259.06) ;
			\draw    (179.28,259.06) -- (206.52,251.28) ;
			\draw    (179.28,259.06) -- (155.93,212.36) ;
			\draw    (134,176.51) -- (155.93,212.36) ;
			\draw [color={rgb, 255:red, 155; green, 155; blue, 155 }  ,draw opacity=0.5 ]   (74.32,241.44) .. controls (95.73,239.49) and (121.03,233.66) .. (155.93,212.36) ;
			\draw  [fill={rgb, 255:red, 0; green, 0; blue, 0 }  ,fill opacity=1 ] (151.19,212.36) .. controls (151.19,209.74) and (153.31,207.62) .. (155.93,207.62) .. controls (158.54,207.62) and (160.67,209.74) .. (160.67,212.36) .. controls (160.67,214.98) and (158.54,217.1) .. (155.93,217.1) .. controls (153.31,217.1) and (151.19,214.98) .. (151.19,212.36) -- cycle ;
			\draw  [fill={rgb, 255:red, 0; green, 0; blue, 0 }  ,fill opacity=1 ] (201.78,251.28) .. controls (201.78,248.66) and (203.9,246.54) .. (206.52,246.54) .. controls (209.14,246.54) and (211.26,248.66) .. (211.26,251.28) .. controls (211.26,253.89) and (209.14,256.01) .. (206.52,256.01) .. controls (203.9,256.01) and (201.78,253.89) .. (201.78,251.28) -- cycle ;
			\draw [color={rgb, 255:red, 155; green, 155; blue, 155 }  ,draw opacity=0.5 ]   (134,176.51) .. controls (171.62,216.14) and (204.7,159.71) .. (218.32,116.9) ;
			\draw [color={rgb, 255:red, 155; green, 155; blue, 155 }  ,draw opacity=0.5 ]   (134,176.51) .. controls (175.51,188.9) and (176.21,152.14) .. (189.84,109.33) ;
			\draw [color={rgb, 255:red, 155; green, 155; blue, 155 }  ,draw opacity=0.5 ]   (76.27,218.09) .. controls (113.24,221.98) and (134.65,216.14) .. (134,176.51) ;
			\draw [color={rgb, 255:red, 155; green, 155; blue, 155 }  ,draw opacity=0.5 ]   (86,181.12) .. controls (103.51,190.85) and (107.4,194.74) .. (134,176.51) ;
			\draw [color={rgb, 255:red, 155; green, 155; blue, 155 }  ,draw opacity=0.5 ]   (134,176.51) .. controls (152.16,165.55) and (154.11,138.31) .. (152.16,120.79) ;
			\draw [color={rgb, 255:red, 155; green, 155; blue, 155 }  ,draw opacity=0.5 ]   (134,176.51) .. controls (136.59,151.93) and (130.75,153.87) .. (121.03,136.36) ;
			\draw [color={rgb, 255:red, 155; green, 155; blue, 155 }  ,draw opacity=0.5 ]   (134,176.51) .. controls (111.3,167.49) and (113.24,173.33) .. (97.67,155.82) ;
			\draw [color={rgb, 255:red, 155; green, 155; blue, 155 }  ,draw opacity=0.5 ]   (179.28,259.06) .. controls (194.97,212.25) and (241.67,183.06) .. (284.48,163.6) ;
			\draw [color={rgb, 255:red, 155; green, 155; blue, 155 }  ,draw opacity=0.5 ]   (179.28,259.06) .. controls (161.89,206.41) and (233.89,173.33) .. (259.19,136.36) ;
			\draw [color={rgb, 255:red, 155; green, 155; blue, 155 }  ,draw opacity=0.5 ]   (80.16,260.9) .. controls (111.3,255.06) and (148.27,221.98) .. (179.28,259.06) ;
			\draw [color={rgb, 255:red, 155; green, 155; blue, 155 }  ,draw opacity=0.5 ]   (93.78,290.09) .. controls (124.92,284.25) and (148.27,251.17) .. (179.28,259.06) ;
			\draw [color={rgb, 255:red, 155; green, 155; blue, 155 }  ,draw opacity=0.5 ]   (189.84,349.22) .. controls (167.73,317.33) and (204.7,272.57) .. (179.28,259.06) ;
			\draw [color={rgb, 255:red, 155; green, 155; blue, 155 }  ,draw opacity=0.5 ]   (156.05,337.89) .. controls (152.16,302.86) and (179.4,285.35) .. (179.28,259.06) ;
			\draw [color={rgb, 255:red, 155; green, 155; blue, 155 }  ,draw opacity=0.5 ]   (121.03,319.28) .. controls (142.43,288.14) and (159.94,288.14) .. (179.28,259.06) ;
			\draw [color={rgb, 255:red, 155; green, 155; blue, 155 }  ,draw opacity=0.5 ]   (246.38,259.06) .. controls (204.7,227.82) and (278.64,218.09) .. (307.83,210.31) ;
			\draw [color={rgb, 255:red, 155; green, 155; blue, 155 }  ,draw opacity=0.5 ]   (230,336.79) .. controls (198.86,297.87) and (218.32,257.01) .. (246.38,259.06) ;
			\draw [color={rgb, 255:red, 155; green, 155; blue, 155 }  ,draw opacity=0.5 ]   (266.97,317.33) .. controls (243.62,301.76) and (220.27,274.52) .. (246.38,259.06) ;
			\draw [color={rgb, 255:red, 155; green, 155; blue, 155 }  ,draw opacity=0.5 ]   (246.38,259.06) .. controls (255.29,229.76) and (282.54,229.12) .. (309.78,229.28) ;
			\draw [color={rgb, 255:red, 155; green, 155; blue, 155 }  ,draw opacity=0.5 ]   (246.38,259.06) .. controls (272.81,256.36) and (284.48,260.25) .. (302,260.25) ;
			\draw [color={rgb, 255:red, 155; green, 155; blue, 155 }  ,draw opacity=0.5 ]   (246.38,259.06) .. controls (257.24,274.52) and (268.92,288.14) .. (282.54,297.87) ;
			\draw  [fill={rgb, 255:red, 255; green, 255; blue, 255 }  ,fill opacity=1 ] (241.64,259.06) .. controls (241.64,256.44) and (243.76,254.32) .. (246.38,254.32) .. controls (249,254.32) and (251.12,256.44) .. (251.12,259.06) .. controls (251.12,261.68) and (249,263.8) .. (246.38,263.8) .. controls (243.76,263.8) and (241.64,261.68) .. (241.64,259.06) -- cycle ;
			\draw  [fill={rgb, 255:red, 255; green, 255; blue, 255 }  ,fill opacity=1 ] (129.27,176.51) .. controls (129.27,173.9) and (131.39,171.78) .. (134,171.78) .. controls (136.62,171.78) and (138.74,173.9) .. (138.74,176.51) .. controls (138.74,179.13) and (136.62,181.25) .. (134,181.25) .. controls (131.39,181.25) and (129.27,179.13) .. (129.27,176.51) -- cycle ;
			\draw  [fill={rgb, 255:red, 255; green, 255; blue, 255 }  ,fill opacity=1 ] (174.54,259.06) .. controls (174.54,256.44) and (176.66,254.32) .. (179.28,254.32) .. controls (181.9,254.32) and (184.02,256.44) .. (184.02,259.06) .. controls (184.02,261.68) and (181.9,263.8) .. (179.28,263.8) .. controls (176.66,263.8) and (174.54,261.68) .. (174.54,259.06) -- cycle ;
			\draw    (497.3,251.28) -- (537.16,259.06) ;
			\draw    (470.06,259.06) -- (497.3,251.28) ;
			\draw    (470.06,259.06) -- (446.71,212.36) ;
			\draw    (424.78,176.51) -- (446.71,212.36) ;
			\draw  [fill={rgb, 255:red, 0; green, 0; blue, 0 }  ,fill opacity=1 ] (441.97,212.36) .. controls (441.97,209.74) and (444.09,207.62) .. (446.71,207.62) .. controls (449.32,207.62) and (451.44,209.74) .. (451.44,212.36) .. controls (451.44,214.98) and (449.32,217.1) .. (446.71,217.1) .. controls (444.09,217.1) and (441.97,214.98) .. (441.97,212.36) -- cycle ;
			\draw  [fill={rgb, 255:red, 0; green, 0; blue, 0 }  ,fill opacity=1 ] (492.56,251.28) .. controls (492.56,248.66) and (494.68,246.54) .. (497.3,246.54) .. controls (499.92,246.54) and (502.04,248.66) .. (502.04,251.28) .. controls (502.04,253.89) and (499.92,256.01) .. (497.3,256.01) .. controls (494.68,256.01) and (492.56,253.89) .. (492.56,251.28) -- cycle ;
			\draw  [fill={rgb, 255:red, 255; green, 255; blue, 255 }  ,fill opacity=1 ] (532.42,259.06) .. controls (532.42,256.44) and (534.54,254.32) .. (537.16,254.32) .. controls (539.77,254.32) and (541.9,256.44) .. (541.9,259.06) .. controls (541.9,261.68) and (539.77,263.8) .. (537.16,263.8) .. controls (534.54,263.8) and (532.42,261.68) .. (532.42,259.06) -- cycle ;
			\draw  [fill={rgb, 255:red, 255; green, 255; blue, 255 }  ,fill opacity=1 ] (420.04,176.51) .. controls (420.04,173.9) and (422.17,171.78) .. (424.78,171.78) .. controls (427.4,171.78) and (429.52,173.9) .. (429.52,176.51) .. controls (429.52,179.13) and (427.4,181.25) .. (424.78,181.25) .. controls (422.17,181.25) and (420.04,179.13) .. (420.04,176.51) -- cycle ;
			\draw  [fill={rgb, 255:red, 255; green, 255; blue, 255 }  ,fill opacity=1 ] (465.32,259.06) .. controls (465.32,256.44) and (467.44,254.32) .. (470.06,254.32) .. controls (472.67,254.32) and (474.8,256.44) .. (474.8,259.06) .. controls (474.8,261.68) and (472.67,263.8) .. (470.06,263.8) .. controls (467.44,263.8) and (465.32,261.68) .. (465.32,259.06) -- cycle ;
			\draw [line width=0.75]    (327.17,235.71) -- (400.18,235.71) ;
			\draw [shift={(403.18,235.71)}, rotate = 180] [fill={rgb, 255:red, 0; green, 0; blue, 0 }  ][line width=0.08]  [draw opacity=0] (10.72,-5.15) -- (0,0) -- (10.72,5.15) -- (7.12,0) -- cycle    ;
			\draw [shift={(327.17,235.71)}, rotate = 180] [color={rgb, 255:red, 0; green, 0; blue, 0 }  ][line width=0.75]    (0,5.59) -- (0,-5.59)   ;
			
			\draw (107.66,172.12) node [anchor=north west][inner sep=0.75pt]   [align=left] {$\displaystyle z_{1}$};
			\draw (156.26,262.25) node [anchor=north west][inner sep=0.75pt]   [align=left] {$\displaystyle z_{2}$};
			\draw (264.29,237.01) node [anchor=north west][inner sep=0.75pt]   [align=left] {$\displaystyle z_{3}$};
			\draw (403.96,158.62) node [anchor=north west][inner sep=0.75pt]   [align=left] {$\displaystyle 1$};
			\draw (453.98,264.2) node [anchor=north west][inner sep=0.75pt]   [align=left] {$\displaystyle 2$};
			\draw (548.47,256.47) node [anchor=north west][inner sep=0.75pt]   [align=left] {$\displaystyle 3$};
			\draw (159.79,188.13) node [anchor=north west][inner sep=0.75pt]   [align=left] {$\displaystyle b_{1}$};
			\draw (203.41,224.05) node [anchor=north west][inner sep=0.75pt]   [align=left] {$\displaystyle b_{2}$};

		\end{tikzpicture}
		
		\caption{This picture shows how to associate an admissible tree to a configuration of points $(z_1,\dots,z_n)\in \M_{0,n+1}$. The orientation of the grey flow lines is omitted in order to have a clearer picture: their source is the point $\infty\in\C\cup\{\infty\}$, their endpoint is one of the black/white vertices. }
		\label{fig:associare un albero a una configurazione}
	\end{figure}
	So, we have just assigned a combinatorial object (an admissible tree) to a configuration of points $(z_1,\dots,z_n)$. 
	\begin{oss}
		If we apply a rotation/dilatation/translation to the configuration $(z_1,\dots,z_n)$, all the flow lines are rotated/dilatated/translated but the resulting tree remains the same. So there is a well defined admissible tree associated to each point in $F_n(\mathbb{C})/\numberset{C}\rtimes\C^*$.
	\end{oss} 
	Now we would like to reconstruct the original configuration of charges from the associated tree. Since different points in $F_n(\mathbb{C})/\numberset{C}\rtimes\C^*$ can be associated to the same tree, we have to add some parameters to each admissible tree in order to reconstruct the original configuration. 
	\begin{defn}
		A \textbf{labelled tree} is an admissible tree $T$ equipped with the following parameters:
		\begin{enumerate}
			\item A function $f:B \to(0,1]$ such that:
			\begin{itemize}
				\item $max(f)=1$
				\item If there is an edge $e=(b,b')$ connecting two black vertices, then $f(b)\geq f(b')$.
			\end{itemize}
			\item For each white vertex $w$, a function $g:E_w\to[0,1]$ such that $\sum_{e\in E_w}g(e)=1$. We interpret $2\pi g(e)$ as the angle between the edge $e$ and the next edge in the cyclic ordering.
		\end{enumerate}
	\end{defn}
	\begin{oss}
		Given a configuration $(z_1,\dots,z_n)\in\M_{0,n+1}$ we associate to it a labelled tree as follows:
		\begin{itemize}
			\item The underlying admissible tree $T$ is the one constructed with the flow lines.
			\item Let $B$ be the set of zeroes of $E(z)$ and $M\coloneqq max_{b\in B}\abs{h(b)}$. Then we set 
			\begin{align*}
				f:B&\to (0,1]\\
				b&\mapsto \frac{\abs{h(b)}}{M}
			\end{align*}
			\item For any point $z_i$ in the configuration, let $E_{z_i}$ be the set of flow lines of finite length that are incident to $z_i$. Then we put
			\begin{align*}
				g:E_{z_i}&\to [0,1]\\
				e&\mapsto \frac{\theta(e)}{2\pi}
			\end{align*}
			where $\theta(e)$ is the angle between the flow line $e$ and the next one in the cyclic order.
		\end{itemize}
	\end{oss}
	It turns out that a configuration of $\M_{0,n+1}$ is uniquely determined by the correspoding labelled tree. To make this statement a little more precise we need a few definitions:
	\begin{figure}
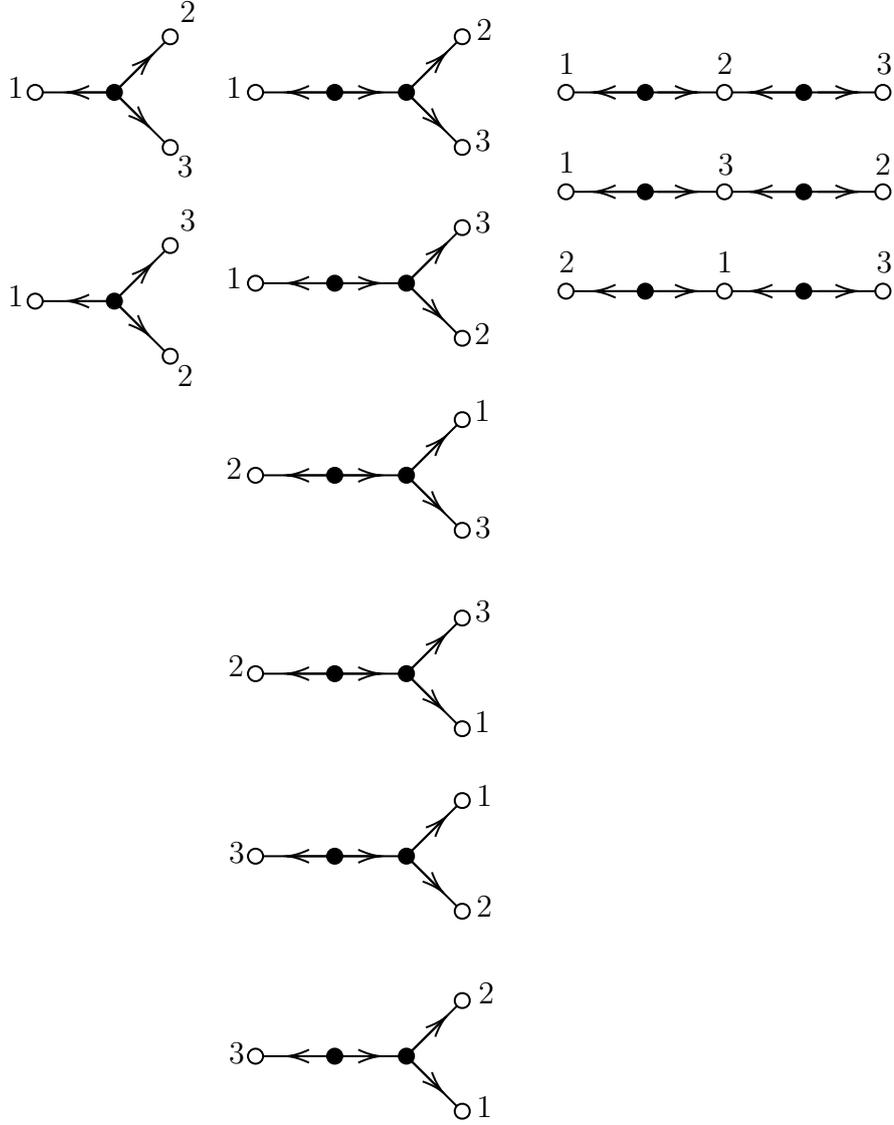

		\centering
		
		\tikzset{every picture/.style={line width=0.75pt}} 

		\tikzset{every picture/.style={line width=0.75pt}} 
		


		\caption{This picture shows all admissible trees with $3$ leaves.}
		\label{fig:esempi di black and white trees}
	\end{figure}
	
	\begin{defn}\label{defn: space of labelled trees}
		For each tree $T\in T_n$ with set of black vertices $B$, let $\sigma_B$ be the topological space of functions $f:B\to(0,1]$ such that:
		\begin{itemize}
			
			\item $max(f)=1$
			\item For any edge  $e=(b,b')$ of $T$ we have $f(b)\geq f(b')$
		\end{itemize}
	\end{defn}
	\begin{defn}
		For any white vertex $w\in W$ we denote by $\Delta_w$ the space of functions $g:E_w\to[0,1]$ such that $\sum_{e\in E_w}g(e)=1$. This space is a simplex of dimension $\abs{E_w}-1$.
	\end{defn}
	\begin{defn}
		The \textbf{stratum} associated to a tree $T\in T_n$ is the space 
		\[
		St_T=\sigma_B\times\prod_{w\in W}\Delta_w
		\]
		If we number the white vertices $w_i$ $i=1,\dots,m$, then we will denote by $g_i:E_{w_i}\to[0,1]$ an element of $\Delta_{w_i}$ and we will write $(f:B\to(0,1],(g_i)_i)_T$ for a generic element of $St_T$.
	\end{defn}
	\begin{defn}
		The space of \textbf{labelled trees with n leaves} is the quotient
		\[
		Tr_n\coloneqq (\bigcup_{T\in T_n}St_T)/\sim
		\]
		with respect to the following equivalence relation:
		\begin{enumerate}
			\item For $e=(b,b')$ and $f(b)=f(b')$ ($b$ and $b'$ are black vertices)
			\[
			(f:B\to(0,1],(g_i)_i)_T\sim (f':B/\{b,b'\}\to(0,1],(g_i)_i)_{T/e}
			\]
			where $T/e$ is the tree obtained from $T$ by collapsing the edge $e$ and $f'([x])\coloneqq f(x)$.
			\item If $g_i(e)=0$ for $e=(b,v_i)$ and the next edge in the cyclic ordering is $e'=(b',v_i)$, then
			\begin{itemize}
				\item For $f(b')=f(b)$ consider the tree $T'$ with black vertex set $B'=B/\{b,b'\}$ and set of edges $E'=E/\{e,e'\}$. Then we identify 
				\[
				(f:B\to(0,1],(g_i)_i)_T\sim (f':B'\to(0,1],(g'_i)_i)_{T'}
				\]
				with $g'_i([x])=g_i(x)$ for $x\in E-\{e\}$ and $f'$ induced by $f$.
				\item For $f(b')<f(b)$ consider the tree $T^+$ with the same vertices of $T$ but the edge $e=(b,v_i)$ replaced by $(b,b')$. Then
				\[
				(f:B\to(0,1],(g_i)_i)_T\sim (f:B\to(0,1], (g_i)_i)_{T^+}
				\]
				\item For $f(b')>f(b)$ consider the tree $T^-$ with the same vertices of $T$ but the edge $e'=(b',v_i)$ replaced by $(b',b)$. Then
				\[
				(f:B\to(0,1],(g_i)_i)_T\sim (f:B\to(0,1], (g'_i)_i)_{T^-}
				\]
				where $g'_i(e)=g_i(e')$ and $g'_i$ coincides with $g_i$ on the other edges incident to $v_i$.
			\end{itemize}
		\end{enumerate}
		See Figure \ref{relazione di equivalenza} for a concrete example of this equivalence relation.
		\begin{figure}
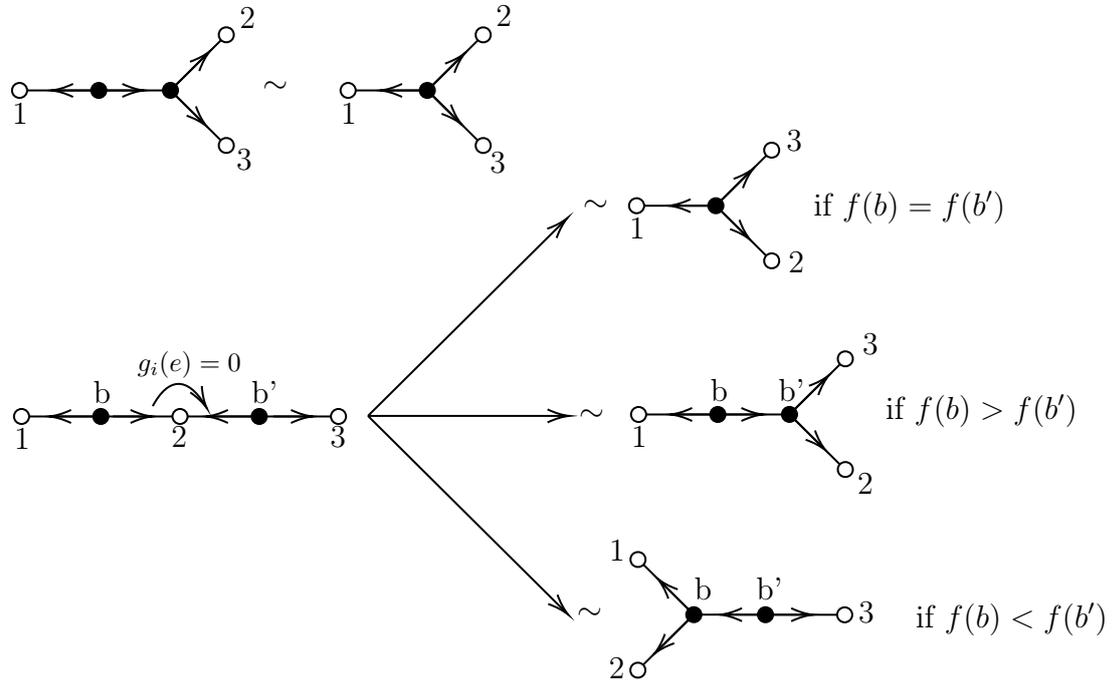

			\centering
			
			\tikzset{every picture/.style={line width=0.75pt}} 
			

			
			\caption{This picture shows some identifications that we need to do in the definition of $Tr_n$. }
			\label{relazione di equivalenza}
		\end{figure}
	\end{defn}
	\begin{es}
		The patient reader could easily verify that $Tr_3$ is homeomorphic to a sphere with three points deleted. The two trees of the first column of Figure \ref{fig:esempi di black and white trees} are $0$-dimensional cells. The six trees of the second column are $1$-dimensional (open) cells and the trees of the last column are  $2$-dimensional (open) cells. See Figure \ref{fig:decomposizione cellulare alberi con tre foglie} for a picture.
	\end{es}
	\begin{figure}
		\centering
		
		\tikzset{every picture/.style={line width=0.75pt}} 
		
		\begin{tikzpicture}[x=0.75pt,y=0.75pt,yscale=-1,xscale=1]
			
			\draw [color={rgb, 255:red, 155; green, 155; blue, 155 }  ,draw opacity=1 ] [dash pattern={on 4.5pt off 4.5pt}]  (67.4,234) .. controls (56.17,183.33) and (378.17,182.33) .. (381,234) ;
			\draw [color={rgb, 255:red, 155; green, 155; blue, 155 }  ,draw opacity=1 ]   (381,234) .. controls (370.17,297.33) and (70.17,285.33) .. (67.4,234) ;
			\draw [line width=1.5]    (118.52,265.99) .. controls (120.5,301.6) and (162.17,376) .. (224.2,390.8) ;
			\draw [line width=1.5]    (325.52,267.99) .. controls (319.17,320.56) and (289.17,381.56) .. (224.2,390.8) ;
			\draw  [color={rgb, 255:red, 155; green, 155; blue, 155 }  ,draw opacity=1 ] (67.4,234) .. controls (67.4,147.4) and (137.6,77.2) .. (224.2,77.2) .. controls (310.8,77.2) and (381,147.4) .. (381,234) .. controls (381,320.6) and (310.8,390.8) .. (224.2,390.8) .. controls (137.6,390.8) and (67.4,320.6) .. (67.4,234) -- cycle ;
			\draw  [fill={rgb, 255:red, 0; green, 0; blue, 0 }  ,fill opacity=1 ] (220.43,77.2) .. controls (220.43,75.12) and (222.12,73.43) .. (224.2,73.43) .. controls (226.28,73.43) and (227.97,75.12) .. (227.97,77.2) .. controls (227.97,79.28) and (226.28,80.97) .. (224.2,80.97) .. controls (222.12,80.97) and (220.43,79.28) .. (220.43,77.2) -- cycle ;
			\draw  [fill={rgb, 255:red, 0; green, 0; blue, 0 }  ,fill opacity=1 ] (220.43,390.8) .. controls (220.43,388.72) and (222.12,387.03) .. (224.2,387.03) .. controls (226.28,387.03) and (227.97,388.72) .. (227.97,390.8) .. controls (227.97,392.88) and (226.28,394.57) .. (224.2,394.57) .. controls (222.12,394.57) and (220.43,392.88) .. (220.43,390.8) -- cycle ;
			\draw [line width=1.5]  [dash pattern={on 5.63pt off 4.5pt}]  (224.2,80.97) .. controls (259.17,88.56) and (287.5,150.6) .. (286.52,197.99) ;
			\draw [line width=1.5]    (224.2,77.2) .. controls (137.17,88) and (113.5,212.6) .. (118.52,265.99) ;
			\draw [line width=1.5]    (224.2,77.2) .. controls (301.17,87) and (330.17,213) .. (325.52,267.99) ;
			\draw  [fill={rgb, 255:red, 255; green, 255; blue, 255 }  ,fill opacity=1 ][line width=1.5]  (322.94,270.74) .. controls (321.42,269.32) and (321.35,266.93) .. (322.77,265.41) .. controls (324.2,263.89) and (326.59,263.81) .. (328.11,265.24) .. controls (329.62,266.67) and (329.7,269.06) .. (328.27,270.57) .. controls (326.85,272.09) and (324.46,272.17) .. (322.94,270.74) -- cycle ;
			\draw [line width=1.5]  [dash pattern={on 5.63pt off 4.5pt}]  (286.52,197.99) .. controls (287.5,236.6) and (273.27,341.6) .. (227.97,390.8) ;
			\draw  [fill={rgb, 255:red, 255; green, 255; blue, 255 }  ,fill opacity=1 ][line width=1.5]  (283.94,200.74) .. controls (282.42,199.32) and (282.35,196.93) .. (283.77,195.41) .. controls (285.2,193.89) and (287.59,193.81) .. (289.11,195.24) .. controls (290.62,196.67) and (290.7,199.06) .. (289.27,200.57) .. controls (287.85,202.09) and (285.46,202.17) .. (283.94,200.74) -- cycle ;
			\draw  [fill={rgb, 255:red, 255; green, 255; blue, 255 }  ,fill opacity=1 ][line width=1.5]  (115.94,268.74) .. controls (114.42,267.32) and (114.35,264.93) .. (115.77,263.41) .. controls (117.2,261.89) and (119.59,261.81) .. (121.11,263.24) .. controls (122.62,264.67) and (122.7,267.06) .. (121.27,268.57) .. controls (119.85,270.09) and (117.46,270.17) .. (115.94,268.74) -- cycle ;

		\end{tikzpicture}

		\caption{The (open) cell decomposition of $Tr_3$. There are two $0$-cells (black vertices), six $1$-dimensional open cells (black edges) and three $2$-dimensional open cells.}
		\label{fig:decomposizione cellulare alberi con tre foglie}
	\end{figure}
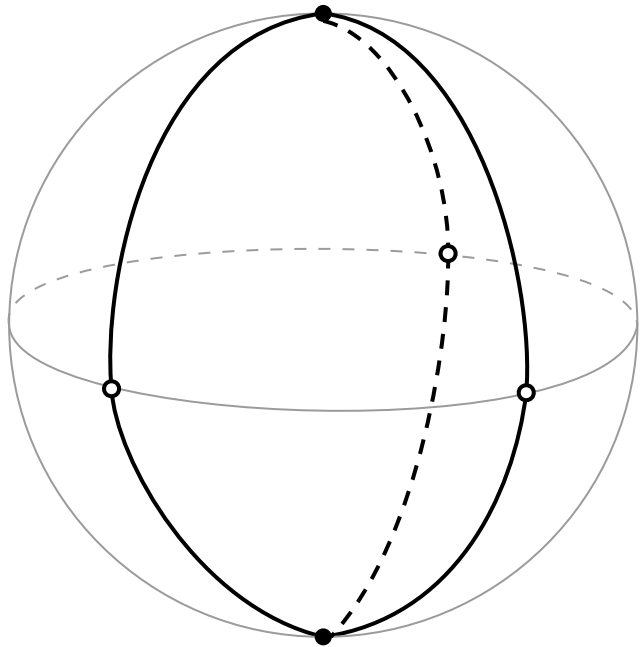
	The construction we just described, that assigns to a point in $\M_{0,n+1}$ a labelled tree with $n$-leaves, turns out to be a homeomorphism, as the second author proved in \cite{Salvatore}:
	\begin{thm}[ \cite{Salvatore}]\label{thm:omeo tra spazio di moduli e labelled trees}
		For any choice of $a_1,\dots,a_n>0$ such that $a_1+\dots+a_n=1$, there is a homeomorphism
		\[
		\phi_{a_1,\dots,a_n}: \mathcal{M}_{0,n+1}\to Tr_n
		\]
	\end{thm}
	\section{Operads and cacti}\label{sec: cactus models}
	In this section we recall some known facts about cacti spaces and operads. These notions will be useful in Section \ref{sec: combinatorial models for Deligne-Mumford} to construct regular CW-decompositions of $\overline{\M}_{0,n+1}$.
	\subsection{The space of cacti}
	A combinatorial construction of the space of cacti was introduced by McClure and Smith \cite{McClure-Smith2}, while geometric constructions are due to Voronov \cite{Voronov} and Kaufmann \cite{Kaufmann}. The second author compared the two approaches in \cite{Salvatore2}. Another good reference where the reader can find many connections with string topology is the book by Cohen, Hess, Voronov \cite{VoronovHess}. 
	\begin{defn}\label{def:cactus}[The space of cacti]
		Let $\mathcal{C}_n$ be the set of partitions $x$ of $S^1$ into $n$ closed $1$-manifolds $I_1(x),\dots,I_n(x)$ such that:
		\begin{itemize}
			\item They have equal measure.
			\item They have pairwise disjoint interiors.
			\item  There is no cyclically ordered sequence of points $(z_1,z_2,z_3,z_4)$ in $S^1$ such that $z_1,z_3\in I_j(x)$, $z_2,z_4\in I_k(x)$ and $j\neq k$.
		\end{itemize}
		We can equip this set with a topology by defining a metric on it: for any $x,y\in \mathcal{C}_n$ we set
		\[
		d(x,y)=\sum_{j=1}^n \mu(I_j(x)-\mathring{I}_j(y))
		\]
		where $\mu$ denotes the measure. We will call $\mathcal{C}_n$ (with this topology) the \textbf{space of based cacti with $n$-lobes}. See Figure \ref{fig:esempio di cactus} for an example. 
	\end{defn}
	\begin{oss}
		$\mathcal{C}_n$ is called the space of cacti for the following reason: given $x\in\mathcal{C}_n$, let us define a relation $\sim$ on $S^1$: two points $z_1,z_2\in S^1$ are  equivalent if there is an index $j\in\{0,\dots,n\}$ such that $z_1$ and $z_2$ are the boundary points of the same connected component of $S^1-\mathring{I}_j$. The quotient space $c(x)\coloneqq S^1/\sim$ by this relation is a pointed space (the base point is just the image of $1\in S^1$) called the \textbf{cactus} associated to $x$: topologically it is a configuration of $n$-circles in the plane, called lobes, whose dual graph is a tree. The dual graph is a graph with two kind of vertices: a white vertex for any lobe and a black vertex for any intersection point between two lobes. An edges connects a white vertex $w$ to a black vertex $b$ if $b$ represents the intersection point of the lobe corresponding to $w$ with some other lobe. See Figure \ref{fig:esempi di cactus} for some examples. Note that the lobes are the image of the $1$-manifolds $I_1(x),\dots,I_n(x)$ under the quotient map $S^1\to S^1/\sim$. In what follows we will freely identify a partition $x\in\mathcal{C}_n$ and its associated cactus $c(x)$.
	\end{oss}
	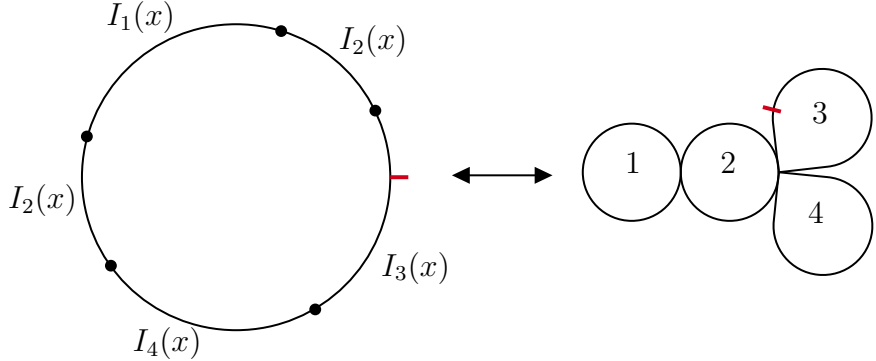
\begin{figure}
		\centering
		
		\tikzset{every picture/.style={line width=0.75pt}} 
		
		\begin{tikzpicture}[x=0.75pt,y=0.75pt,yscale=-1,xscale=1]
			
			\draw   (54.32,193.95) .. controls (54.32,151.45) and (88.78,117) .. (131.27,117) .. controls (173.77,117) and (208.23,151.45) .. (208.23,193.95) .. controls (208.23,236.45) and (173.77,270.9) .. (131.27,270.9) .. controls (88.78,270.9) and (54.32,236.45) .. (54.32,193.95) -- cycle ;
			\draw  [fill={rgb, 255:red, 0; green, 0; blue, 0 }  ,fill opacity=1 ] (66.32,238.45) .. controls (66.32,237.1) and (67.42,236) .. (68.77,236) .. controls (70.13,236) and (71.23,237.1) .. (71.23,238.45) .. controls (71.23,239.81) and (70.13,240.9) .. (68.77,240.9) .. controls (67.42,240.9) and (66.32,239.81) .. (66.32,238.45) -- cycle ;
			\draw  [fill={rgb, 255:red, 0; green, 0; blue, 0 }  ,fill opacity=1 ] (151.32,120.45) .. controls (151.32,119.1) and (152.42,118) .. (153.77,118) .. controls (155.13,118) and (156.23,119.1) .. (156.23,120.45) .. controls (156.23,121.81) and (155.13,122.9) .. (153.77,122.9) .. controls (152.42,122.9) and (151.32,121.81) .. (151.32,120.45) -- cycle ;
			\draw  [fill={rgb, 255:red, 0; green, 0; blue, 0 }  ,fill opacity=1 ] (198.32,160.45) .. controls (198.32,159.1) and (199.42,158) .. (200.77,158) .. controls (202.13,158) and (203.23,159.1) .. (203.23,160.45) .. controls (203.23,161.81) and (202.13,162.9) .. (200.77,162.9) .. controls (199.42,162.9) and (198.32,161.81) .. (198.32,160.45) -- cycle ;
			\draw  [fill={rgb, 255:red, 0; green, 0; blue, 0 }  ,fill opacity=1 ] (54.32,173.45) .. controls (54.32,172.1) and (55.42,171) .. (56.77,171) .. controls (58.13,171) and (59.23,172.1) .. (59.23,173.45) .. controls (59.23,174.81) and (58.13,175.9) .. (56.77,175.9) .. controls (55.42,175.9) and (54.32,174.81) .. (54.32,173.45) -- cycle ;
			\draw  [fill={rgb, 255:red, 0; green, 0; blue, 0 }  ,fill opacity=1 ] (168.32,260.45) .. controls (168.32,259.1) and (169.42,258) .. (170.77,258) .. controls (172.13,258) and (173.23,259.1) .. (173.23,260.45) .. controls (173.23,261.81) and (172.13,262.9) .. (170.77,262.9) .. controls (169.42,262.9) and (168.32,261.81) .. (168.32,260.45) -- cycle ;
			\draw [color={rgb, 255:red, 208; green, 2; blue, 27 }  ,draw opacity=1 ][line width=1.5]    (208.23,193.95) -- (217.23,193.95) ;
			\draw   (304.32,191.45) .. controls (304.32,177.95) and (315.27,167) .. (328.77,167) .. controls (342.28,167) and (353.23,177.95) .. (353.23,191.45) .. controls (353.23,204.96) and (342.28,215.9) .. (328.77,215.9) .. controls (315.27,215.9) and (304.32,204.96) .. (304.32,191.45) -- cycle ;
			\draw   (353.23,191.45) .. controls (353.23,177.95) and (364.17,167) .. (377.68,167) .. controls (391.18,167) and (402.13,177.95) .. (402.13,191.45) .. controls (402.13,204.96) and (391.18,215.9) .. (377.68,215.9) .. controls (364.17,215.9) and (353.23,204.96) .. (353.23,191.45) -- cycle ;
			\draw   (426.67,193.92) .. controls (440.22,195.29) and (450.1,207.38) .. (448.74,220.93) .. controls (447.37,234.49) and (435.28,244.37) .. (421.73,243) .. controls (408.17,241.64) and (398.29,229.54) .. (399.66,215.99) .. controls (401.31,199.63) and (402.13,191.45) .. (402.13,191.45) .. controls (402.13,191.45) and (410.31,192.28) .. (426.67,193.92) -- cycle ;
			\draw   (399.32,166.95) .. controls (397.77,153.41) and (407.49,141.19) .. (421.02,139.64) .. controls (434.56,138.09) and (446.78,147.81) .. (448.33,161.34) .. controls (449.88,174.87) and (440.16,187.1) .. (426.63,188.65) .. controls (410.3,190.52) and (402.13,191.45) .. (402.13,191.45) .. controls (402.13,191.45) and (401.19,183.28) .. (399.32,166.95) -- cycle ;
			\draw [color={rgb, 255:red, 208; green, 2; blue, 27 }  ,draw opacity=1 ][line width=1.5]    (394.23,158.68) -- (403.23,160.95) ;
			\draw    (242,193) -- (286.23,193) ;
			\draw [shift={(289.23,193)}, rotate = 180] [fill={rgb, 255:red, 0; green, 0; blue, 0 }  ][line width=0.08]  [draw opacity=0] (8.93,-4.29) -- (0,0) -- (8.93,4.29) -- cycle    ;
			\draw [shift={(239,193)}, rotate = 0] [fill={rgb, 255:red, 0; green, 0; blue, 0 }  ][line width=0.08]  [draw opacity=0] (8.93,-4.29) -- (0,0) -- (8.93,4.29) -- cycle    ;
			
			\draw (65,104) node [anchor=north west][inner sep=0.75pt]   [align=left] {$\displaystyle I_{1}( x)$};
			\draw (181,116) node [anchor=north west][inner sep=0.75pt]   [align=left] {$\displaystyle I_{2}( x)$};
			\draw (16,196) node [anchor=north west][inner sep=0.75pt]   [align=left] {$\displaystyle I_{2}( x)$};
			\draw (202.23,229.95) node [anchor=north west][inner sep=0.75pt]   [align=left] {$\displaystyle I_{3}( x)$};
			\draw (78.23,267.95) node [anchor=north west][inner sep=0.75pt]   [align=left] {$\displaystyle I_{4}( x)$};
			\draw (323.9,180.2) node [anchor=north west][inner sep=0.75pt]   [align=left] {$\displaystyle 1$};
			\draw (371.63,180.2) node [anchor=north west][inner sep=0.75pt]   [align=left] {$\displaystyle 2$};
			\draw (417.35,155.04) node [anchor=north west][inner sep=0.75pt]   [align=left] {$\displaystyle 3$};
			\draw (415.35,205.04) node [anchor=north west][inner sep=0.75pt]   [align=left] {$\displaystyle 4$};

		\end{tikzpicture}
		\caption{On the left there is an element $x\in\mathcal{C}_4$, on the right its associated cactus $c(x)$. The base point of the circle $S^1$ is depicted in red and corresponds to a base point on the cactus $c(x)$ (which we depict as a red spine).}
		\label{fig:esempio di cactus}
	\end{figure}
	\subsection{Cell decomposition} 
	To any point $x\in \mathcal{C}_n$ we can associate a sequence $(X_1,\dots,X_l)$ of numbers in $\{1,\dots,n\}$ by the following procedure: start from the point $1\in S^1$ and  move along the circle clockwise. The sequence $(X_1,\dots,X_l)$ is obtained by writing one after the other the indices of all the $1$-manifolds one encounters and has the following properties:
	\begin{itemize}
		\item All values between $1$ and $n$ appear.
		\item $X _i\neq X_{i+1}$ for every $i=1,\dots,l-1$.
		\item There is no subsequence of the form $(i,j,i,j)$ with $i\neq j$.
	\end{itemize}
	This sequence is just a combinatorial way to encode the shape of the cactus $c(x)$. By an abuse of notation we will also denote by $(X_1,\dots,X_l)$ the subspace of $\mathcal{C}_n$ consisting of all partitions whose associated sequence is $(X_1,\dots,X_l)$. This subspace turns out to be homeomorphic to a product of simplices. More precisely, if $m_i$ is the cardinality of $\{j\in\{1,\dots,l\}\mid X_j=i\}$, $i=1,\dots,n$, then 
	\[
	(X_1,\dots,X_l)\cong \prod_{i=1}^n\Delta^{m_i-1}
	\]
	For an example look at the cactus of Figure \ref{fig:esempio di cactus}: it belongs to the cell $(3,4,2,1,2,3)\cong\Delta^0\times \Delta^1\times\Delta^1\times\Delta^0$.
	Intuitively all the cacti $c(x)$ associated to partitions $x\in(X_1,\dots,X_l)$ have the same shape. So we will represent pictorially a cell by drawing a cactus, meaning that the cell contains all the partitions $x\in\mathcal{C}_n$ whose associated cactus $c(x)$ has that shape. From this point of view, the parameters
	of a cell $(X_1,\dots,X_l)$ can be thought as the lengths of the arcs between two consecutive \emph{marked points}, where a marked point is an intersection point of lobes or the base point. The boundary of a cell is obtained by collapsing some of these arcs. See Figure \ref{fig:esempi di cactus e loro bordi} for some examples. This gives a regular CW-decomposition of   $\mathcal{C}_n$.
	\begin{figure}
		\centering

		\tikzset{every picture/.style={line width=0.75pt}} 
		
		\begin{tikzpicture}[x=0.75pt,y=0.75pt,yscale=-1,xscale=1]
			
			\draw   (321.32,82.28) .. controls (321.32,80.19) and (323.01,78.51) .. (325.1,78.51) .. controls (327.18,78.51) and (328.87,80.19) .. (328.87,82.28) .. controls (328.87,84.36) and (327.18,86.05) .. (325.1,86.05) .. controls (323.01,86.05) and (321.32,84.36) .. (321.32,82.28) -- cycle ;
			\draw   (400.41,82.28) .. controls (400.41,80.19) and (402.1,78.51) .. (404.19,78.51) .. controls (406.27,78.51) and (407.96,80.19) .. (407.96,82.28) .. controls (407.96,84.36) and (406.27,86.05) .. (404.19,86.05) .. controls (402.1,86.05) and (400.41,84.36) .. (400.41,82.28) -- cycle ;
			\draw   (479.51,82.28) .. controls (479.51,80.19) and (481.19,78.51) .. (483.28,78.51) .. controls (485.36,78.51) and (487.05,80.19) .. (487.05,82.28) .. controls (487.05,84.36) and (485.36,86.05) .. (483.28,86.05) .. controls (481.19,86.05) and (479.51,84.36) .. (479.51,82.28) -- cycle ;
			\draw  [fill={rgb, 255:red, 0; green, 0; blue, 0 }  ,fill opacity=1 ] (360.87,82.28) .. controls (360.87,80.19) and (362.56,78.51) .. (364.64,78.51) .. controls (366.73,78.51) and (368.41,80.19) .. (368.41,82.28) .. controls (368.41,84.36) and (366.73,86.05) .. (364.64,86.05) .. controls (362.56,86.05) and (360.87,84.36) .. (360.87,82.28) -- cycle ;
			\draw  [fill={rgb, 255:red, 0; green, 0; blue, 0 }  ,fill opacity=1 ] (439.96,82.28) .. controls (439.96,80.19) and (441.65,78.51) .. (443.73,78.51) .. controls (445.82,78.51) and (447.51,80.19) .. (447.51,82.28) .. controls (447.51,84.36) and (445.82,86.05) .. (443.73,86.05) .. controls (441.65,86.05) and (439.96,84.36) .. (439.96,82.28) -- cycle ;
			\draw    (328.87,82.28) -- (360.87,82.28) ;
			\draw    (368.41,82.28) -- (400.41,82.28) ;
			\draw    (407.96,82.28) -- (439.96,82.28) ;
			\draw    (447.51,82.28) -- (479.51,82.28) ;
			\draw   (169.76,82.28) .. controls (169.76,73.12) and (177.19,65.69) .. (186.35,65.69) .. controls (195.51,65.69) and (202.93,73.12) .. (202.93,82.28) .. controls (202.93,91.44) and (195.51,98.86) .. (186.35,98.86) .. controls (177.19,98.86) and (169.76,91.44) .. (169.76,82.28) -- cycle ;
			\draw   (202.93,82.28) .. controls (202.93,73.12) and (210.35,65.69) .. (219.51,65.69) .. controls (228.67,65.69) and (236.1,73.12) .. (236.1,82.28) .. controls (236.1,91.44) and (228.67,98.86) .. (219.51,98.86) .. controls (210.35,98.86) and (202.93,91.44) .. (202.93,82.28) -- cycle ;
			\draw   (136.6,82.28) .. controls (136.6,73.12) and (144.02,65.69) .. (153.18,65.69) .. controls (162.34,65.69) and (169.76,73.12) .. (169.76,82.28) .. controls (169.76,91.44) and (162.34,98.86) .. (153.18,98.86) .. controls (144.02,98.86) and (136.6,91.44) .. (136.6,82.28) -- cycle ;
			\draw    (275,82) -- (307.17,82) ;
			\draw [shift={(309.17,82)}, rotate = 180] [color={rgb, 255:red, 0; green, 0; blue, 0 }  ][line width=0.75]    (10.93,-3.29) .. controls (6.95,-1.4) and (3.31,-0.3) .. (0,0) .. controls (3.31,0.3) and (6.95,1.4) .. (10.93,3.29)   ;
			\draw    (275,82) -- (248.17,82) ;
			\draw [shift={(246.17,82)}, rotate = 360] [color={rgb, 255:red, 0; green, 0; blue, 0 }  ][line width=0.75]    (10.93,-3.29) .. controls (6.95,-1.4) and (3.31,-0.3) .. (0,0) .. controls (3.31,0.3) and (6.95,1.4) .. (10.93,3.29)   ;
			\draw   (171.32,149.23) .. controls (164.23,150.66) and (157.28,145.89) .. (155.81,138.58) .. controls (154.33,131.28) and (158.89,124.19) .. (165.98,122.76) .. controls (173.07,121.33) and (180.02,126.09) .. (181.49,133.4) .. controls (183.27,142.22) and (184.16,146.63) .. (184.16,146.63) .. controls (184.16,146.63) and (179.88,147.5) .. (171.32,149.23) -- cycle ;
			\draw   (187.72,134.03) .. controls (187.72,134.03) and (187.72,134.03) .. (187.72,134.03) .. controls (189.68,127.06) and (197.09,123.06) .. (204.27,125.08) .. controls (211.44,127.11) and (215.67,134.39) .. (213.7,141.36) .. controls (211.74,148.32) and (204.33,152.32) .. (197.15,150.3) .. controls (188.49,147.86) and (184.16,146.63) .. (184.16,146.63) .. controls (184.16,146.63) and (185.35,142.43) .. (187.72,134.03) -- cycle ;
			\draw   (193.57,155.74) .. controls (193.57,155.74) and (193.57,155.74) .. (193.57,155.74) .. controls (193.57,155.74) and (193.57,155.74) .. (193.57,155.74) .. controls (198.77,160.78) and (198.79,169.2) .. (193.6,174.56) .. controls (188.42,179.91) and (180,180.18) .. (174.8,175.15) .. controls (169.6,170.12) and (169.59,161.69) .. (174.77,156.34) .. controls (181.04,149.87) and (184.17,146.64) .. (184.16,146.63) .. controls (184.17,146.64) and (187.3,149.67) .. (193.57,155.74) -- cycle ;
			\draw   (346.32,144.28) .. controls (346.32,142.19) and (348.01,140.51) .. (350.1,140.51) .. controls (352.18,140.51) and (353.87,142.19) .. (353.87,144.28) .. controls (353.87,146.36) and (352.18,148.05) .. (350.1,148.05) .. controls (348.01,148.05) and (346.32,146.36) .. (346.32,144.28) -- cycle ;
			\draw   (414.94,174.74) .. controls (413.42,173.32) and (413.35,170.93) .. (414.77,169.41) .. controls (416.2,167.89) and (418.59,167.81) .. (420.11,169.24) .. controls (421.62,170.67) and (421.7,173.06) .. (420.27,174.57) .. controls (418.85,176.09) and (416.46,176.17) .. (414.94,174.74) -- cycle ;
			\draw   (420.23,118.96) .. controls (418.77,120.45) and (416.38,120.48) .. (414.89,119.03) .. controls (413.4,117.57) and (413.37,115.19) .. (414.82,113.69) .. controls (416.27,112.2) and (418.66,112.17) .. (420.15,113.62) .. controls (421.65,115.08) and (421.68,117.46) .. (420.23,118.96) -- cycle ;
			\draw  [fill={rgb, 255:red, 0; green, 0; blue, 0 }  ,fill opacity=1 ] (385.87,144.28) .. controls (385.87,142.19) and (387.56,140.51) .. (389.64,140.51) .. controls (391.73,140.51) and (393.41,142.19) .. (393.41,144.28) .. controls (393.41,146.36) and (391.73,148.05) .. (389.64,148.05) .. controls (387.56,148.05) and (385.87,146.36) .. (385.87,144.28) -- cycle ;
			\draw    (353.87,144.28) -- (385.87,144.28) ;
			\draw    (414.77,169.41) -- (389.64,144.28) ;
			\draw    (389.64,144.28) -- (414.89,119.03) ;
			\draw    (275,143) -- (307.17,143) ;
			\draw [shift={(309.17,143)}, rotate = 180] [color={rgb, 255:red, 0; green, 0; blue, 0 }  ][line width=0.75]    (10.93,-3.29) .. controls (6.95,-1.4) and (3.31,-0.3) .. (0,0) .. controls (3.31,0.3) and (6.95,1.4) .. (10.93,3.29)   ;
			\draw    (275,143) -- (248.17,143) ;
			\draw [shift={(246.17,143)}, rotate = 360] [color={rgb, 255:red, 0; green, 0; blue, 0 }  ][line width=0.75]    (10.93,-3.29) .. controls (6.95,-1.4) and (3.31,-0.3) .. (0,0) .. controls (3.31,0.3) and (6.95,1.4) .. (10.93,3.29)   ;
			\draw (320,61) node [anchor=north west][inner sep=0.75pt]   [align=left] {$\displaystyle 1$};
			\draw (399,61) node [anchor=north west][inner sep=0.75pt]   [align=left] {$\displaystyle 2$};
			\draw (478.52,60.99) node [anchor=north west][inner sep=0.75pt]   [align=left] {$\displaystyle 3$};
			\draw (149,75) node [anchor=north west][inner sep=0.75pt]   [align=left] {$\displaystyle 1$};
			\draw (181,75) node [anchor=north west][inner sep=0.75pt]   [align=left] {$\displaystyle 2$};
			\draw (214,75) node [anchor=north west][inner sep=0.75pt]   [align=left] {$\displaystyle 3$};
			\draw (179,154.46) node [anchor=north west][inner sep=0.75pt]   [align=left] {$\displaystyle 3$};
			\draw (163,129.46) node [anchor=north west][inner sep=0.75pt]   [align=left] {$\displaystyle 1$};
			\draw (195,129.46) node [anchor=north west][inner sep=0.75pt]   [align=left] {$\displaystyle 2$};
			\draw (335,135) node [anchor=north west][inner sep=0.75pt]   [align=left] {$\displaystyle 1$};
			\draw (421,97) node [anchor=north west][inner sep=0.75pt]   [align=left] {$\displaystyle 2$};
			\draw (419.52,174.99) node [anchor=north west][inner sep=0.75pt]   [align=left] {$\displaystyle 3$};

		\end{tikzpicture}
		
		\caption{On the left there are some cacti (without base point), on the right the corresponding dual graphs.}
		\label{fig:esempi di cactus}
	\end{figure}
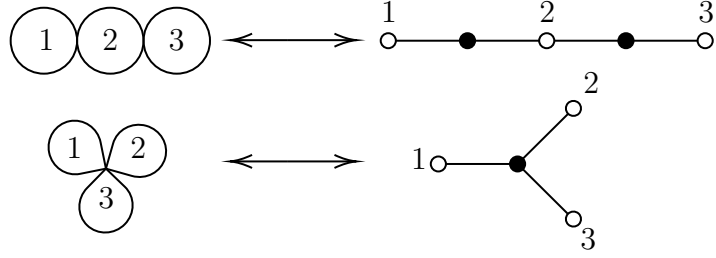
	
	\begin{figure}
		\centering
		
		\tikzset{every picture/.style={line width=0.75pt}} 
		
		\begin{tikzpicture}[x=0.75pt,y=0.75pt,yscale=-1,xscale=1]
			
			\draw   (64.76,278.28) .. controls (64.76,269.12) and (72.19,261.69) .. (81.35,261.69) .. controls (90.51,261.69) and (97.93,269.12) .. (97.93,278.28) .. controls (97.93,287.44) and (90.51,294.86) .. (81.35,294.86) .. controls (72.19,294.86) and (64.76,287.44) .. (64.76,278.28) -- cycle ;
			\draw   (97.93,278.28) .. controls (97.93,269.12) and (105.35,261.69) .. (114.51,261.69) .. controls (123.67,261.69) and (131.1,269.12) .. (131.1,278.28) .. controls (131.1,287.44) and (123.67,294.86) .. (114.51,294.86) .. controls (105.35,294.86) and (97.93,287.44) .. (97.93,278.28) -- cycle ;
			\draw   (31.6,278.28) .. controls (31.6,269.12) and (39.02,261.69) .. (48.18,261.69) .. controls (57.34,261.69) and (64.76,269.12) .. (64.76,278.28) .. controls (64.76,287.44) and (57.34,294.86) .. (48.18,294.86) .. controls (39.02,294.86) and (31.6,287.44) .. (31.6,278.28) -- cycle ;
			\draw   (235.32,330.23) .. controls (228.23,331.66) and (221.28,326.89) .. (219.81,319.58) .. controls (218.33,312.28) and (222.89,305.19) .. (229.98,303.76) .. controls (237.07,302.33) and (244.02,307.09) .. (245.49,314.4) .. controls (247.27,323.22) and (248.16,327.63) .. (248.16,327.63) .. controls (248.16,327.63) and (243.88,328.5) .. (235.32,330.23) -- cycle ;
			\draw   (251.72,315.03) .. controls (253.68,308.06) and (261.09,304.06) .. (268.27,306.08) .. controls (275.44,308.11) and (279.67,315.39) .. (277.7,322.36) .. controls (275.74,329.32) and (268.33,333.32) .. (261.15,331.3) .. controls (252.49,328.86) and (248.16,327.63) .. (248.16,327.63) .. controls (248.16,327.63) and (249.35,323.43) .. (251.72,315.03) -- cycle ;
			\draw   (257.57,336.74) .. controls (257.57,336.74) and (257.57,336.74) .. (257.57,336.74) .. controls (262.77,341.78) and (262.79,350.2) .. (257.6,355.56) .. controls (252.42,360.91) and (244,361.18) .. (238.8,356.15) .. controls (233.6,351.12) and (233.59,342.69) .. (238.77,337.34) .. controls (245.04,330.87) and (248.17,327.64) .. (248.16,327.63) .. controls (248.17,327.64) and (251.3,330.67) .. (257.57,336.74) -- cycle ;
			\draw  [color={rgb, 255:red, 0; green, 0; blue, 0 }  ,draw opacity=1 ] (120.8,92.51) .. controls (120.8,85.33) and (126.62,79.51) .. (133.8,79.51) .. controls (140.98,79.51) and (146.8,85.33) .. (146.8,92.51) .. controls (146.8,99.69) and (140.98,105.51) .. (133.8,105.51) .. controls (126.62,105.51) and (120.8,99.69) .. (120.8,92.51) -- cycle ;
			\draw  [color={rgb, 255:red, 0; green, 0; blue, 0 }  ,draw opacity=1 ] (94.8,92.51) .. controls (94.8,85.33) and (100.62,79.51) .. (107.8,79.51) .. controls (114.98,79.51) and (120.8,85.33) .. (120.8,92.51) .. controls (120.8,99.69) and (114.98,105.51) .. (107.8,105.51) .. controls (100.62,105.51) and (94.8,99.69) .. (94.8,92.51) -- cycle ;
			\draw  [color={rgb, 255:red, 0; green, 0; blue, 0 }  ,draw opacity=1 ] (374.8,92.51) .. controls (374.8,85.33) and (380.62,79.51) .. (387.8,79.51) .. controls (394.98,79.51) and (400.8,85.33) .. (400.8,92.51) .. controls (400.8,99.69) and (394.98,105.51) .. (387.8,105.51) .. controls (380.62,105.51) and (374.8,99.69) .. (374.8,92.51) -- cycle ;
			\draw  [color={rgb, 255:red, 0; green, 0; blue, 0 }  ,draw opacity=1 ] (348.8,92.51) .. controls (348.8,85.33) and (354.62,79.51) .. (361.8,79.51) .. controls (368.98,79.51) and (374.8,85.33) .. (374.8,92.51) .. controls (374.8,99.69) and (368.98,105.51) .. (361.8,105.51) .. controls (354.62,105.51) and (348.8,99.69) .. (348.8,92.51) -- cycle ;
			\draw  [color={rgb, 255:red, 0; green, 0; blue, 0 }  ,draw opacity=1 ] (250.8,163.51) .. controls (250.8,156.33) and (256.62,150.51) .. (263.8,150.51) .. controls (270.98,150.51) and (276.8,156.33) .. (276.8,163.51) .. controls (276.8,170.69) and (270.98,176.51) .. (263.8,176.51) .. controls (256.62,176.51) and (250.8,170.69) .. (250.8,163.51) -- cycle ;
			\draw  [color={rgb, 255:red, 0; green, 0; blue, 0 }  ,draw opacity=1 ] (224.8,163.51) .. controls (224.8,156.33) and (230.62,150.51) .. (237.8,150.51) .. controls (244.98,150.51) and (250.8,156.33) .. (250.8,163.51) .. controls (250.8,170.69) and (244.98,176.51) .. (237.8,176.51) .. controls (230.62,176.51) and (224.8,170.69) .. (224.8,163.51) -- cycle ;
			\draw  [fill={rgb, 255:red, 0; green, 0; blue, 0 }  ,fill opacity=1 ] (156.9,92.8) .. controls (156.9,91.25) and (158.15,90) .. (159.7,90) .. controls (161.25,90) and (162.5,91.25) .. (162.5,92.8) .. controls (162.5,94.35) and (161.25,95.6) .. (159.7,95.6) .. controls (158.15,95.6) and (156.9,94.35) .. (156.9,92.8) -- cycle ;
			\draw  [fill={rgb, 255:red, 0; green, 0; blue, 0 }  ,fill opacity=1 ] (335.5,92.8) .. controls (335.5,91.25) and (336.75,90) .. (338.3,90) .. controls (339.85,90) and (341.1,91.25) .. (341.1,92.8) .. controls (341.1,94.35) and (339.85,95.6) .. (338.3,95.6) .. controls (336.75,95.6) and (335.5,94.35) .. (335.5,92.8) -- cycle ;
			\draw   (159.7,92.8) .. controls (159.7,64.3) and (199.68,41.2) .. (249,41.2) .. controls (298.32,41.2) and (338.3,64.3) .. (338.3,92.8) .. controls (338.3,121.3) and (298.32,144.4) .. (249,144.4) .. controls (199.68,144.4) and (159.7,121.3) .. (159.7,92.8) -- cycle ;
			\draw  [color={rgb, 255:red, 0; green, 0; blue, 0 }  ,draw opacity=1 ] (250.8,23.51) .. controls (250.8,16.33) and (256.62,10.51) .. (263.8,10.51) .. controls (270.98,10.51) and (276.8,16.33) .. (276.8,23.51) .. controls (276.8,30.69) and (270.98,36.51) .. (263.8,36.51) .. controls (256.62,36.51) and (250.8,30.69) .. (250.8,23.51) -- cycle ;
			\draw  [color={rgb, 255:red, 0; green, 0; blue, 0 }  ,draw opacity=1 ] (224.8,23.51) .. controls (224.8,16.33) and (230.62,10.51) .. (237.8,10.51) .. controls (244.98,10.51) and (250.8,16.33) .. (250.8,23.51) .. controls (250.8,30.69) and (244.98,36.51) .. (237.8,36.51) .. controls (230.62,36.51) and (224.8,30.69) .. (224.8,23.51) -- cycle ;
			\draw  [color={rgb, 255:red, 208; green, 2; blue, 27 }  ,draw opacity=1 ][fill={rgb, 255:red, 208; green, 2; blue, 27 }  ,fill opacity=1 ] (119.3,87.01) .. controls (119.3,86.18) and (119.97,85.51) .. (120.8,85.51) .. controls (121.63,85.51) and (122.3,86.18) .. (122.3,87.01) .. controls (122.3,87.84) and (121.63,88.51) .. (120.8,88.51) .. controls (119.97,88.51) and (119.3,87.84) .. (119.3,87.01) -- cycle ;
			\draw  [color={rgb, 255:red, 208; green, 2; blue, 27 }  ,draw opacity=1 ][fill={rgb, 255:red, 208; green, 2; blue, 27 }  ,fill opacity=1 ] (223.3,163.51) .. controls (223.3,162.68) and (223.97,162.01) .. (224.8,162.01) .. controls (225.63,162.01) and (226.3,162.68) .. (226.3,163.51) .. controls (226.3,164.34) and (225.63,165.01) .. (224.8,165.01) .. controls (223.97,165.01) and (223.3,164.34) .. (223.3,163.51) -- cycle ;
			\draw  [color={rgb, 255:red, 208; green, 2; blue, 27 }  ,draw opacity=1 ][fill={rgb, 255:red, 208; green, 2; blue, 27 }  ,fill opacity=1 ] (373.3,97.01) .. controls (373.3,96.18) and (373.97,95.51) .. (374.8,95.51) .. controls (375.63,95.51) and (376.3,96.18) .. (376.3,97.01) .. controls (376.3,97.84) and (375.63,98.51) .. (374.8,98.51) .. controls (373.97,98.51) and (373.3,97.84) .. (373.3,97.01) -- cycle ;
			\draw  [color={rgb, 255:red, 208; green, 2; blue, 27 }  ,draw opacity=1 ][fill={rgb, 255:red, 208; green, 2; blue, 27 }  ,fill opacity=1 ] (275.3,23.51) .. controls (275.3,22.68) and (275.97,22.01) .. (276.8,22.01) .. controls (277.63,22.01) and (278.3,22.68) .. (278.3,23.51) .. controls (278.3,24.34) and (277.63,25.01) .. (276.8,25.01) .. controls (275.97,25.01) and (275.3,24.34) .. (275.3,23.51) -- cycle ;
			\draw  [color={rgb, 255:red, 208; green, 2; blue, 27 }  ,draw opacity=1 ][fill={rgb, 255:red, 208; green, 2; blue, 27 }  ,fill opacity=1 ] (79.85,261.69) .. controls (79.85,260.87) and (80.52,260.19) .. (81.35,260.19) .. controls (82.17,260.19) and (82.85,260.87) .. (82.85,261.69) .. controls (82.85,262.52) and (82.17,263.19) .. (81.35,263.19) .. controls (80.52,263.19) and (79.85,262.52) .. (79.85,261.69) -- cycle ;
			\draw   (237.76,224.28) .. controls (237.76,215.12) and (245.19,207.69) .. (254.35,207.69) .. controls (263.51,207.69) and (270.93,215.12) .. (270.93,224.28) .. controls (270.93,233.44) and (263.51,240.86) .. (254.35,240.86) .. controls (245.19,240.86) and (237.76,233.44) .. (237.76,224.28) -- cycle ;
			\draw   (270.93,224.28) .. controls (270.93,215.12) and (278.35,207.69) .. (287.51,207.69) .. controls (296.67,207.69) and (304.1,215.12) .. (304.1,224.28) .. controls (304.1,233.44) and (296.67,240.86) .. (287.51,240.86) .. controls (278.35,240.86) and (270.93,233.44) .. (270.93,224.28) -- cycle ;
			\draw   (204.6,224.28) .. controls (204.6,215.12) and (212.02,207.69) .. (221.18,207.69) .. controls (230.34,207.69) and (237.76,215.12) .. (237.76,224.28) .. controls (237.76,233.44) and (230.34,240.86) .. (221.18,240.86) .. controls (212.02,240.86) and (204.6,233.44) .. (204.6,224.28) -- cycle ;
			\draw  [color={rgb, 255:red, 208; green, 2; blue, 27 }  ,draw opacity=1 ][fill={rgb, 255:red, 208; green, 2; blue, 27 }  ,fill opacity=1 ] (236.26,218.78) .. controls (236.26,217.95) and (236.93,217.28) .. (237.76,217.28) .. controls (238.59,217.28) and (239.26,217.95) .. (239.26,218.78) .. controls (239.26,219.61) and (238.59,220.28) .. (237.76,220.28) .. controls (236.93,220.28) and (236.26,219.61) .. (236.26,218.78) -- cycle ;
			\draw   (237.76,274.28) .. controls (237.76,265.12) and (245.19,257.69) .. (254.35,257.69) .. controls (263.51,257.69) and (270.93,265.12) .. (270.93,274.28) .. controls (270.93,283.44) and (263.51,290.86) .. (254.35,290.86) .. controls (245.19,290.86) and (237.76,283.44) .. (237.76,274.28) -- cycle ;
			\draw   (270.93,274.28) .. controls (270.93,265.12) and (278.35,257.69) .. (287.51,257.69) .. controls (296.67,257.69) and (304.1,265.12) .. (304.1,274.28) .. controls (304.1,283.44) and (296.67,290.86) .. (287.51,290.86) .. controls (278.35,290.86) and (270.93,283.44) .. (270.93,274.28) -- cycle ;
			\draw   (204.6,274.28) .. controls (204.6,265.12) and (212.02,257.69) .. (221.18,257.69) .. controls (230.34,257.69) and (237.76,265.12) .. (237.76,274.28) .. controls (237.76,283.44) and (230.34,290.86) .. (221.18,290.86) .. controls (212.02,290.86) and (204.6,283.44) .. (204.6,274.28) -- cycle ;
			\draw  [color={rgb, 255:red, 208; green, 2; blue, 27 }  ,draw opacity=1 ][fill={rgb, 255:red, 208; green, 2; blue, 27 }  ,fill opacity=1 ] (269.43,269.78) .. controls (269.43,268.95) and (270.1,268.28) .. (270.93,268.28) .. controls (271.76,268.28) and (272.43,268.95) .. (272.43,269.78) .. controls (272.43,270.61) and (271.76,271.28) .. (270.93,271.28) .. controls (270.1,271.28) and (269.43,270.61) .. (269.43,269.78) -- cycle ;
			\draw  [color={rgb, 255:red, 208; green, 2; blue, 27 }  ,draw opacity=1 ][fill={rgb, 255:red, 208; green, 2; blue, 27 }  ,fill opacity=1 ] (275.43,312.78) .. controls (275.43,311.95) and (276.1,311.28) .. (276.93,311.28) .. controls (277.76,311.28) and (278.43,311.95) .. (278.43,312.78) .. controls (278.43,313.61) and (277.76,314.28) .. (276.93,314.28) .. controls (276.1,314.28) and (275.43,313.61) .. (275.43,312.78) -- cycle ;
			\draw    (139.1,278.28) -- (186.17,278.28) ;
			\draw [shift={(188.17,278.28)}, rotate = 180] [color={rgb, 255:red, 0; green, 0; blue, 0 }  ][line width=0.75]    (10.93,-3.29) .. controls (6.95,-1.4) and (3.31,-0.3) .. (0,0) .. controls (3.31,0.3) and (6.95,1.4) .. (10.93,3.29)   ;
			
			\draw (44,269) node [anchor=north west][inner sep=0.75pt]   [align=left] {$\displaystyle 1$};
			\draw (76,269) node [anchor=north west][inner sep=0.75pt]   [align=left] {$\displaystyle 2$};
			\draw (109,269) node [anchor=north west][inner sep=0.75pt]   [align=left] {$\displaystyle 3$};
			\draw (243,335.46) node [anchor=north west][inner sep=0.75pt]   [align=left] {$\displaystyle 3$};
			\draw (227,310.46) node [anchor=north west][inner sep=0.75pt]   [align=left] {$\displaystyle 1$};
			\draw (259,310.46) node [anchor=north west][inner sep=0.75pt]   [align=left] {$\displaystyle 2$};
			\draw (103,83) node [anchor=north west][inner sep=0.75pt]   [align=left] {$\displaystyle 1$};
			\draw (129,84) node [anchor=north west][inner sep=0.75pt]   [align=left] {$\displaystyle 2$};
			\draw (356,84) node [anchor=north west][inner sep=0.75pt]   [align=left] {$\displaystyle 1$};
			\draw (383,84) node [anchor=north west][inner sep=0.75pt]   [align=left] {$\displaystyle 2$};
			\draw (233,154) node [anchor=north west][inner sep=0.75pt]   [align=left] {$\displaystyle 1$};
			\draw (259,155) node [anchor=north west][inner sep=0.75pt]   [align=left] {$\displaystyle 2$};
			\draw (233,14) node [anchor=north west][inner sep=0.75pt]   [align=left] {$\displaystyle 1$};
			\draw (259,15) node [anchor=north west][inner sep=0.75pt]   [align=left] {$\displaystyle 2$};
			\draw (30,86.8) node [anchor=north west][inner sep=0.75pt]   [align=left] {$\displaystyle \mathcal{C}_{2} =$\\};
			\draw (217,215) node [anchor=north west][inner sep=0.75pt]   [align=left] {$\displaystyle 1$};
			\draw (249,215) node [anchor=north west][inner sep=0.75pt]   [align=left] {$\displaystyle 2$};
			\draw (282,215) node [anchor=north west][inner sep=0.75pt]   [align=left] {$\displaystyle 3$};
			\draw (217,265) node [anchor=north west][inner sep=0.75pt]   [align=left] {$\displaystyle 1$};
			\draw (249,265) node [anchor=north west][inner sep=0.75pt]   [align=left] {$\displaystyle 2$};
			\draw (282,265) node [anchor=north west][inner sep=0.75pt]   [align=left] {$\displaystyle 3$};
			\draw (156,258) node [anchor=north west][inner sep=0.75pt]   [align=left] {$\displaystyle \partial $};

		\end{tikzpicture}
		\caption{On top there is a full description of $\mathcal{C}_2\cong S^1$: there are two zero cells $(2,1)$ (on the left) and $(1,2)$ (on the right). The $1$-cells are $(2,1,2)$ (on the top) and $(1,2,1)$ (on the bottom). Below we see the cell $(2,3,2,1,2)\cong \Delta^0\times\Delta^2\times\Delta^0$ of $\mathcal{C}_3$ and the codimension one cells in its boundary.}
		\label{fig:esempi di cactus e loro bordi}
	\end{figure}
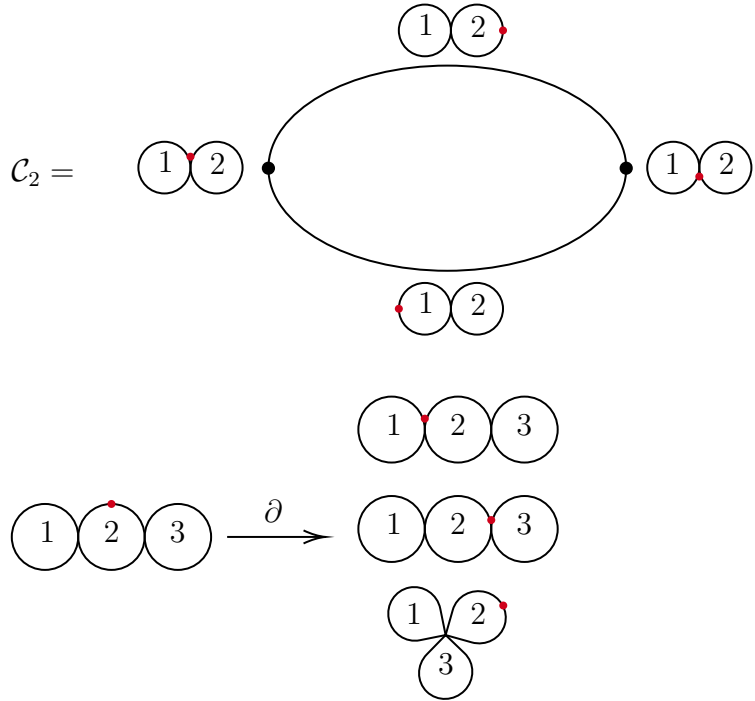
	\begin{oss}
		There are two relevant groups acting on $\mathcal{C}_n$: $S^1$ acts by rotating the base point of a cactus, $\Sigma_n$ acts by relabelling the lobes.
	\end{oss}
	\begin{prop}[\cite{Salvatore}]
		\label{prop: trivial circle bundle}
		The projection $p:\mathcal{C}_n\to \mathcal{C}_n/S^1$ is a trivial principal $S^1$-bundle.
	\end{prop}
	The important thing about cacti is that they are a very small cellular model for the configuration space $F_n(\C)$:
	
	\begin{thm}[ \cite{Salvatore}]\label{thm:cactus sono deformation retract dello spazio di configurazioni}
		The space of cacti $\mathcal{C}_n$ is $(S^1\times\Sigma_n)$-equivariantly homotopy equivalent to $F_n(\C)$.
	\end{thm}
	\subsection{Operads and cacti}\label{subsec:operads and cacti}
	As we have just seen the space of cacti is homotopy equivalent to the ordered configuration space (Theorem \ref{thm:cactus sono deformation retract dello spazio di configurazioni}). Actually there is more, since it is possible to define cellular maps
	\[
	\theta_{n_1,\dots,n_k}:\mathcal{C}_k\times\mathcal{C}_{n_1}\times\cdots\times \mathcal{C}_{n_k}\to \mathcal{C}_{n_1+\cdots +n_k}
	\]
	which give a structure of operad up to homotopy on the sequence of spaces $\{\mathcal{C}_n\}_{n\geq 1}$. As usual we can also define partial compositions $\circ_i:\mathcal{C}_n\times\mathcal{C}_m\to \mathcal{C}_{n+m-1}$ by the formula
	\[
	x\circ_iy\coloneqq \theta(1,\dots,1,y,1,\dots,1)
	\]
	where $1$ denotes the unique point of $\mathcal{C}_1$. In what follows we construct  $\theta_{n_1,\dots,n_k}$, following \cite[Sec. 4]{Salvatore}. First we need a preliminary definition:
	\begin{defn}
		Let $x\in\mathcal{C}_n$ be a cactus and let $(X_1,\dots,X_l)$ be its associated sequence. If we think $S^1$ as the quotient of $[0,n]$ by its endpoints we can pullback the partition $x$ along the quotient map $p:[0,n]\to [0,n]/\sim$, obtaining a decomposition of $[0,n]$ into $l$ closed intervals $[0,y_1], [y_1,y_2], \dots,[y_{l-1},n]$. The \textbf{cactus map } $c_x=(c_x^1,\dots,c_x^n):[0,n]\to [0,1]^n$ is defined as follows:
		\begin{itemize}
			\item $c_x(0)=(0,\dots,0)$
			\item If $y\in[y_i,y_{i+1}]$ then 
			$
			c_x^j(y)=\begin{cases}
				c_x^j(y_i) \text{ if } j\neq X_i\\
				c_x^{X_i}(y_i)+(y-y_i) \text{ if } j=X_i
			\end{cases}
			$
		\end{itemize}
		Intuitively the curve $c_x$ describes the motion of a point that moves along the cactus clockwise (starting at the base point): when the points passes through the $i$-th lobe the $i$-th coordinate of $c_x$ increases, while the other components remain constant. In particular $c_x(n)=(1,\dots,1)$. For an example see Figure \ref{fig: cactus map}. Note that the cactus map $c_x$ uniquely determines the cactus $x$.
		
	\end{defn}
	\begin{figure}
		\centering

		\tikzset{every picture/.style={line width=0.75pt}} 
		
		\begin{tikzpicture}[x=0.75pt,y=0.75pt,yscale=-1,xscale=1]
			
			\draw   (100,143.33) .. controls (100,133.76) and (107.76,126) .. (117.33,126) .. controls (126.91,126) and (134.67,133.76) .. (134.67,143.33) .. controls (134.67,152.91) and (126.91,160.67) .. (117.33,160.67) .. controls (107.76,160.67) and (100,152.91) .. (100,143.33) -- cycle ;
			\draw   (134.67,143.33) .. controls (134.67,133.76) and (142.43,126) .. (152,126) .. controls (161.57,126) and (169.33,133.76) .. (169.33,143.33) .. controls (169.33,152.91) and (161.57,160.67) .. (152,160.67) .. controls (142.43,160.67) and (134.67,152.91) .. (134.67,143.33) -- cycle ;
			\draw [color={rgb, 255:red, 208; green, 2; blue, 27 }  ,draw opacity=1 ][line width=1.5]    (169.33,143.33) -- (175.67,143.33) ;
			\draw  [fill={rgb, 255:red, 0; green, 0; blue, 0 }  ,fill opacity=1 ] (161,156.67) .. controls (161,155.65) and (161.82,154.83) .. (162.83,154.83) .. controls (163.85,154.83) and (164.67,155.65) .. (164.67,156.67) .. controls (164.67,157.68) and (163.85,158.5) .. (162.83,158.5) .. controls (161.82,158.5) and (161,157.68) .. (161,156.67) -- cycle ;
			\draw    (163.83,164.83) .. controls (152.4,169.79) and (145.53,167.44) .. (136.44,162.99) ;
			\draw [shift={(134.67,162.11)}, rotate = 26.57] [color={rgb, 255:red, 0; green, 0; blue, 0 }  ][line width=0.75]    (10.93,-3.29) .. controls (6.95,-1.4) and (3.31,-0.3) .. (0,0) .. controls (3.31,0.3) and (6.95,1.4) .. (10.93,3.29)   ;
			\draw  [dash pattern={on 0.84pt off 2.51pt}] (237,97.11) -- (316.89,97.11) -- (316.89,177) -- (237,177) -- cycle ;
			\draw [line width=1.5]    (237,135.89) -- (237,177) ;
			\draw [shift={(237,148.14)}, rotate = 90] [fill={rgb, 255:red, 0; green, 0; blue, 0 }  ][line width=0.08]  [draw opacity=0] (13.4,-6.43) -- (0,0) -- (13.4,6.44) -- (8.9,0) -- cycle    ;
			\draw [line width=1.5]    (316.89,97.11) -- (316.89,135.89) ;
			\draw [shift={(316.89,108.2)}, rotate = 90] [fill={rgb, 255:red, 0; green, 0; blue, 0 }  ][line width=0.08]  [draw opacity=0] (13.4,-6.43) -- (0,0) -- (13.4,6.44) -- (8.9,0) -- cycle    ;
			\draw [line width=1.5]    (237,135.89) -- (316.89,135.89) ;
			\draw [shift={(283.74,135.89)}, rotate = 180] [fill={rgb, 255:red, 0; green, 0; blue, 0 }  ][line width=0.08]  [draw opacity=0] (13.4,-6.43) -- (0,0) -- (13.4,6.44) -- (8.9,0) -- cycle    ;
			
			\draw (113,134) node [anchor=north west][inner sep=0.75pt]   [align=left] {$\displaystyle 1$};
			\draw (146,134) node [anchor=north west][inner sep=0.75pt]   [align=left] {$\displaystyle 2$};
			\draw (271,113) node [anchor=north west][inner sep=0.75pt]   [align=left] {$\displaystyle c_{x}$};

		\end{tikzpicture}

		\caption{On the left we see a cactus $x$, whose associated sequence is  $(2,1,2)$. On the right we see the image of the corresponding cactus map $c_x:[0,2]\to [0,1]^2$}
		\label{fig: cactus map}
	\end{figure}
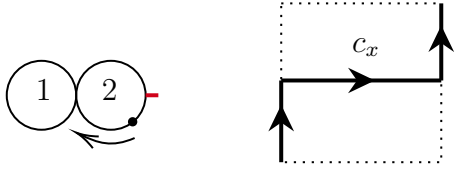
	Now we are ready to define $\theta_{n_1,\dots,n_k}:\mathcal{C}_k\times\mathcal{C}_{n_1}\times\cdots\times \mathcal{C}_{n_k}\to \mathcal{C}_{n_1+\cdots +n_k}$.
	\begin{defn}
		For $k,n_1,\dots,n_k\in N$ we define embeddings
		\[
		\theta_{n_1,\dots,n_k}:\mathcal{C}_k\times\mathcal{C}_{n_1}\times\cdots\times \mathcal{C}_{n_k}\to \mathcal{C}_{n_1+\cdots +n_k}
		\]
		as follows: fix $x\in\mathcal{C}_k$ and $x_j\in\mathcal{C}_{n_j}$ for $j=1,\dots,k$. Then $\theta_{n_1,\dots,n_k}(x,x_1,\dots,x_k)$ will be described by the corresponding cactus map: let $n\coloneqq n_1+\dots+n_k$ and consider the product of dilations
		\begin{align*}
			D:[0,1]^k&\to \prod_{j=1}^k[0,n_j]\\
			(t_1,\dots,t_k)&\mapsto (t_1n_1,\dots,t_kn_k)
		\end{align*} 
		There is a unique piecewise linear homomorphism $\alpha:[0,k]\to[0,n]$ and a unique piecewise oriented isometry onto its image $c:[0,n]\to \prod_{j=1}^k[0,n_j]$ such that the following square is commutative:
		\[
		\begin{tikzcd}
			& \text{$[0,k]$} \arrow[r,"c_x"] \arrow[d,"\alpha"]& \text{$[0,1]^k$} \arrow[d,"D"]\\
			& \text{$[0,n]$} \arrow[r,"c"] &\prod_{j=1}^k[0,n_j]
		\end{tikzcd}
		\]
		Then $\theta_{n_1,\dots,n_k}(x,x_1,\dots,x_k)$ is uniquely determined by the cactus map 
		\[
		\left(\prod_{j=1}^kc_{x_j}\right)\circ c:[0,n]\to [0,1]^n
		\]
		Intuitively the composition of cacti works as follows: fix $x\in\mathcal{C}_k$ and $x_j\in\mathcal{C}_{n_j}$ for $j=1,\dots,k$. Observe that each lobe $l$ of $x$ has a \emph{local base point}: if $l$ contains the base point then the local base point coincides with it. Otherwise the local base point is the intersection point of $l$ with the connected component of $x-l$ containing the base point. $\theta_{n_1,\dots,n_k}(x,x_1,\dots,x_k)$ is obtained by inserting each  cactus $x_i$ into the $i$-th lobe of $x$ so that the base point of $x_i$ coincides with the local base point of the $i$-th lobe of $x$. For a picture see Figure \ref{fig: composizione di cactus puntati}.
	\end{defn}
	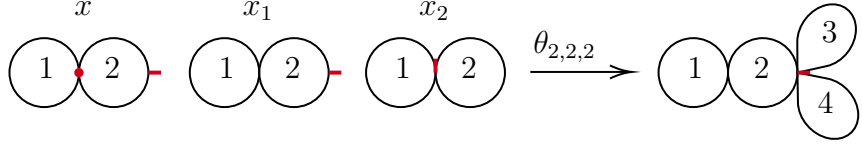
\begin{figure}
		\centering

		\tikzset{every picture/.style={line width=0.75pt}} 
		
		\begin{tikzpicture}[x=0.75pt,y=0.75pt,yscale=-1,xscale=1]
			
			\draw   (429.39,166.37) .. controls (429.39,166.37) and (429.39,166.37) .. (429.39,166.37) .. controls (437.7,168.05) and (444.55,176.13) .. (444.68,184.41) .. controls (444.82,192.69) and (438.19,198.05) .. (429.88,196.37) .. controls (421.56,194.69) and (414.72,186.61) .. (414.58,178.33) .. controls (414.58,178.33) and (414.58,178.33) .. (414.58,178.33) .. controls (414.41,168.33) and (414.33,163.33) .. (414.33,163.33) .. controls (414.33,163.33) and (419.35,164.35) .. (429.39,166.37) -- cycle ;
			\draw   (21,163.33) .. controls (21,153.76) and (28.76,146) .. (38.33,146) .. controls (47.91,146) and (55.67,153.76) .. (55.67,163.33) .. controls (55.67,172.91) and (47.91,180.67) .. (38.33,180.67) .. controls (28.76,180.67) and (21,172.91) .. (21,163.33) -- cycle ;
			\draw   (55.67,163.33) .. controls (55.67,153.76) and (63.43,146) .. (73,146) .. controls (82.57,146) and (90.33,153.76) .. (90.33,163.33) .. controls (90.33,172.91) and (82.57,180.67) .. (73,180.67) .. controls (63.43,180.67) and (55.67,172.91) .. (55.67,163.33) -- cycle ;
			\draw [color={rgb, 255:red, 208; green, 2; blue, 27 }  ,draw opacity=1 ][line width=1.5]    (90.33,163.33) -- (96.67,163.33) ;
			\draw   (111,163.33) .. controls (111,153.76) and (118.76,146) .. (128.33,146) .. controls (137.91,146) and (145.67,153.76) .. (145.67,163.33) .. controls (145.67,172.91) and (137.91,180.67) .. (128.33,180.67) .. controls (118.76,180.67) and (111,172.91) .. (111,163.33) -- cycle ;
			\draw   (145.67,163.33) .. controls (145.67,153.76) and (153.43,146) .. (163,146) .. controls (172.57,146) and (180.33,153.76) .. (180.33,163.33) .. controls (180.33,172.91) and (172.57,180.67) .. (163,180.67) .. controls (153.43,180.67) and (145.67,172.91) .. (145.67,163.33) -- cycle ;
			\draw [color={rgb, 255:red, 208; green, 2; blue, 27 }  ,draw opacity=1 ][line width=1.5]    (180.33,163.33) -- (186.67,163.33) ;
			\draw   (199,163.33) .. controls (199,153.76) and (206.76,146) .. (216.33,146) .. controls (225.91,146) and (233.67,153.76) .. (233.67,163.33) .. controls (233.67,172.91) and (225.91,180.67) .. (216.33,180.67) .. controls (206.76,180.67) and (199,172.91) .. (199,163.33) -- cycle ;
			\draw   (233.67,163.33) .. controls (233.67,153.76) and (241.43,146) .. (251,146) .. controls (260.57,146) and (268.33,153.76) .. (268.33,163.33) .. controls (268.33,172.91) and (260.57,180.67) .. (251,180.67) .. controls (241.43,180.67) and (233.67,172.91) .. (233.67,163.33) -- cycle ;
			\draw [color={rgb, 255:red, 208; green, 2; blue, 27 }  ,draw opacity=1 ][line width=1.5]    (233.67,163.33) -- (233.67,156.33) ;
			\draw    (281,163) -- (329.67,163) ;
			\draw [shift={(331.67,163)}, rotate = 180] [color={rgb, 255:red, 0; green, 0; blue, 0 }  ][line width=0.75]    (10.93,-3.29) .. controls (6.95,-1.4) and (3.31,-0.3) .. (0,0) .. controls (3.31,0.3) and (6.95,1.4) .. (10.93,3.29)   ;
			\draw   (345,163.33) .. controls (345,153.76) and (352.76,146) .. (362.33,146) .. controls (371.91,146) and (379.67,153.76) .. (379.67,163.33) .. controls (379.67,172.91) and (371.91,180.67) .. (362.33,180.67) .. controls (352.76,180.67) and (345,172.91) .. (345,163.33) -- cycle ;
			\draw   (379.67,163.33) .. controls (379.67,153.76) and (387.43,146) .. (397,146) .. controls (406.57,146) and (414.33,153.76) .. (414.33,163.33) .. controls (414.33,172.91) and (406.57,180.67) .. (397,180.67) .. controls (387.43,180.67) and (379.67,172.91) .. (379.67,163.33) -- cycle ;
			\draw   (414.25,147.98) .. controls (414.25,147.98) and (414.25,147.98) .. (414.25,147.98) .. controls (414.21,139.5) and (420.72,131.15) .. (428.81,129.33) .. controls (436.89,127.52) and (443.48,132.92) .. (443.52,141.4) .. controls (443.57,149.88) and (437.05,158.23) .. (428.97,160.05) .. controls (428.97,160.05) and (428.97,160.05) .. (428.97,160.05) .. controls (419.21,162.24) and (414.33,163.33) .. (414.33,163.33) .. controls (414.33,163.33) and (414.3,158.22) .. (414.25,147.98) -- cycle ;
			\draw [color={rgb, 255:red, 208; green, 2; blue, 27 }  ,draw opacity=1 ][line width=1.5]    (414.33,163.33) -- (420.67,163.33) ;
			\draw  [color={rgb, 255:red, 208; green, 2; blue, 27 }  ,draw opacity=1 ][fill={rgb, 255:red, 208; green, 2; blue, 27 }  ,fill opacity=1 ] (53.83,163.33) .. controls (53.83,162.32) and (54.65,161.5) .. (55.67,161.5) .. controls (56.68,161.5) and (57.5,162.32) .. (57.5,163.33) .. controls (57.5,164.35) and (56.68,165.17) .. (55.67,165.17) .. controls (54.65,165.17) and (53.83,164.35) .. (53.83,163.33) -- cycle ;
			
			\draw (34,154) node [anchor=north west][inner sep=0.75pt]   [align=left] {$\displaystyle 1$};
			\draw (67,154) node [anchor=north west][inner sep=0.75pt]   [align=left] {$\displaystyle 2$};
			\draw (124,154) node [anchor=north west][inner sep=0.75pt]   [align=left] {$\displaystyle 1$};
			\draw (157,154) node [anchor=north west][inner sep=0.75pt]   [align=left] {$\displaystyle 2$};
			\draw (212,154) node [anchor=north west][inner sep=0.75pt]   [align=left] {$\displaystyle 1$};
			\draw (245,154) node [anchor=north west][inner sep=0.75pt]   [align=left] {$\displaystyle 2$};
			\draw (281,141) node [anchor=north west][inner sep=0.75pt]   [align=left] {$\displaystyle \theta _{2,2,2}$};
			\draw (358,154) node [anchor=north west][inner sep=0.75pt]   [align=left] {$\displaystyle 1$};
			\draw (391,154) node [anchor=north west][inner sep=0.75pt]   [align=left] {$\displaystyle 2$};
			\draw (425,136) node [anchor=north west][inner sep=0.75pt]   [align=left] {$\displaystyle 3$};
			\draw (423,172) node [anchor=north west][inner sep=0.75pt]   [align=left] {$\displaystyle 4$};
			\draw (52,126) node [anchor=north west][inner sep=0.75pt]   [align=left] {$\displaystyle x$};
			\draw (136,126) node [anchor=north west][inner sep=0.75pt]   [align=left] {$\displaystyle x_{1}$};
			\draw (224,126) node [anchor=north west][inner sep=0.75pt]   [align=left] {$\displaystyle x_{2}$};

		\end{tikzpicture}
		
		\caption{On the left we see three cacti with two lobes. The red bullet on the leftmost cactus denotes the local base point of lobe $1$. On the right we see a picture of the composite $\theta_{2,2,2}(x,x_1,x_2)$.}
		\label{fig: composizione di cactus puntati}
	\end{figure}

	\begin{thm}[\cite{Salvatore}]
		The maps $\theta_{n_1,\dots,n_k}$ induce an operad structure on the sequence of chain complexes $Cact\coloneqq\{C_*^{cell}(\mathcal{C}_n)\}_{n\in\N}$, where $C_*^{cell}(\mathcal{C}_n)$ are the cellular chains. Therefore, $Cact$ is a chain model for the little two disk operad $\mathcal{D}_2$.
	\end{thm}
	\begin{oss}
		The operad $Cact$ is isomorphic to $\mathcal{S}_2$, a natural $E_2$ suboperad of the $E_{\infty}$ surjection operad $\mathcal{S}$ by McClure and Smith \cite{McClure-Smith3}. The reader can find some details on the comparison between $Cact$ and $\mathcal{S}_2$ in \cite{Salvatore} and \cite{Salvatore2}.
	\end{oss}
	\begin{oss}
		There are many variants of the space of cacti:
		\begin{enumerate}
			\item One can do the same construction as before but with the additional data of a base point (also called a \emph{spine}) for each lobe. We call the resulting space $f\mathcal{C}_n$ the \textbf{space of cacti with spines} (or \emph{framed cacti}). From the point of view of configuration spaces this corresponds to assigning an element of $S^1$ to each point of the configuration. Therefore $f\mathcal{C}_n$ is homotopy equivalent to the space of framed configurations $fF_n(\C)\coloneqq F_n(\C)\times (S^1)^n$ and the family of cellular chains $fCact\coloneqq \{C_*^{cell}(f\mathcal{C}_n)\}_{n\in\N}$ is a chain model for the framed little two disks operad $f\mathcal{D}_2$.
			\item  The quotient $\mathcal{C}_n/S^1$ is still a regular CW-complex, and its cells are described by the same combinatorics of $\mathcal{C}_n$ except that we do not have a base point on our cacti (see Figure \ref{fig:esempio di cactus senza punto base} for an example). We call $\mathcal{C}_n/S^1$ the \textbf{space of unbased cacti}. By Theorem \ref{thm:cactus sono deformation retract dello spazio di configurazioni} $\mathcal{C}_n/S^1$ is homotopy equivalent to $F_n(\C)/S^1\simeq \M_{0,n+1}$. The operadic structure of $Cact$ induces an operad structure on the cellular chains $C_*^{cell}(\mathcal{C}_n/S^1)$, giving a chain model for the Gravity operad. More precisely, the collection of chain complexes $grav\coloneqq \{sC_*^{cell}(\mathcal{C}_n/S^1)\}$ is an operad, with partial compositions defined by the following diagram:
			\[
			\begin{tikzcd}
				&C_*^{cell}(\mathcal{C}_n)\otimes C_*^{cell}(\mathcal{C}_m)\arrow[r," \circ_i"] &C_*^{cell}(\mathcal{C}_{n+m-1})\\
				&sC_*^{cell}(\mathcal{C}_n/S^1)\otimes sC_*^{cell}(\mathcal{C}_m/S^1)\arrow[r," \circ_i",dashed]\arrow[u,"\tau\otimes \tau"] &sC_*^{cell}(\mathcal{C}_{n+m-1}/S^1)\arrow[u,"\tau"]
			\end{tikzcd}
			\]
			
			Here $\tau: sC_*^{cell}(\mathcal{C}_n/S^1)\to C_*^{cell}(\mathcal{C}_n)$ is a chain model for the transfer 
			\[
			H_*(F_n(\C)/S^1)\to H_{*+1}(F_n(\C))
			\]
			and takes an (unbased) cactus to the sum (with signs, depending on the orientations) of all cacti one can obtain by adjoining a base point in one of the lobes (see Figure \ref{fig:chain model for the transfer}). It is possible to see that if $C_1\in sC_*^{cell}(\mathcal{C}_n/S^1)$ and $C_2\in sC_*^{cell}(\mathcal{C}_m/S^1)$, then $\tau(C_1)\circ_i\tau(C_2)=\tau(C)$ for a unique $C\in sC_*^{cell}(\mathcal{C}_{n+m-1}/S^1)$, so we set $C_1\circ_i C_2\coloneqq C$. To be more explicit, the composition $C_1\circ_i C_2$ is the sum of all (unbased) cacti which one obtains by inserting $C_2$ into the $i$-th lobe of $C_1$ in all possible ways. For an explicit example see Figure \ref{composition of Grav }. This chain model for the gravity operad is described (with slightly different terms) in \cite[Paragraph 4.2.1]{Ward}, where it is denoted by $T_{\circlearrowright}$.
			
			\begin{figure}
				\centering
				
				\tikzset{every picture/.style={line width=0.75pt}} 
				
				\begin{tikzpicture}[x=0.75pt,y=0.75pt,yscale=-1,xscale=1]
					
					\draw   (19.25,111.01) .. controls (19.25,104.21) and (24.76,98.7) .. (31.56,98.7) .. controls (38.35,98.7) and (43.86,104.21) .. (43.86,111.01) .. controls (43.86,117.8) and (38.35,123.31) .. (31.56,123.31) .. controls (24.76,123.31) and (19.25,117.8) .. (19.25,111.01) -- cycle ;
					\draw   (43.86,111.01) .. controls (43.86,104.21) and (49.37,98.7) .. (56.17,98.7) .. controls (62.96,98.7) and (68.47,104.21) .. (68.47,111.01) .. controls (68.47,117.8) and (62.96,123.31) .. (56.17,123.31) .. controls (49.37,123.31) and (43.86,117.8) .. (43.86,111.01) -- cycle ;
					\draw   (68.47,111.01) .. controls (68.47,104.21) and (73.98,98.7) .. (80.78,98.7) .. controls (87.58,98.7) and (93.08,104.21) .. (93.08,111.01) .. controls (93.08,117.8) and (87.58,123.31) .. (80.78,123.31) .. controls (73.98,123.31) and (68.47,117.8) .. (68.47,111.01) -- cycle ;
					\draw    (107.08,111.01) -- (157.17,111.01) ;
					\draw [shift={(159.17,111.01)}, rotate = 180] [color={rgb, 255:red, 0; green, 0; blue, 0 }  ][line width=0.75]    (10.93,-3.29) .. controls (6.95,-1.4) and (3.31,-0.3) .. (0,0) .. controls (3.31,0.3) and (6.95,1.4) .. (10.93,3.29)   ;
					\draw   (170.25,111.01) .. controls (170.25,104.21) and (175.76,98.7) .. (182.56,98.7) .. controls (189.35,98.7) and (194.86,104.21) .. (194.86,111.01) .. controls (194.86,117.8) and (189.35,123.31) .. (182.56,123.31) .. controls (175.76,123.31) and (170.25,117.8) .. (170.25,111.01) -- cycle ;
					\draw   (194.86,111.01) .. controls (194.86,104.21) and (200.37,98.7) .. (207.17,98.7) .. controls (213.96,98.7) and (219.47,104.21) .. (219.47,111.01) .. controls (219.47,117.8) and (213.96,123.31) .. (207.17,123.31) .. controls (200.37,123.31) and (194.86,117.8) .. (194.86,111.01) -- cycle ;
					\draw   (219.47,111.01) .. controls (219.47,104.21) and (224.98,98.7) .. (231.78,98.7) .. controls (238.58,98.7) and (244.08,104.21) .. (244.08,111.01) .. controls (244.08,117.8) and (238.58,123.31) .. (231.78,123.31) .. controls (224.98,123.31) and (219.47,117.8) .. (219.47,111.01) -- cycle ;
					\draw   (267.25,111.01) .. controls (267.25,104.21) and (272.76,98.7) .. (279.56,98.7) .. controls (286.35,98.7) and (291.86,104.21) .. (291.86,111.01) .. controls (291.86,117.8) and (286.35,123.31) .. (279.56,123.31) .. controls (272.76,123.31) and (267.25,117.8) .. (267.25,111.01) -- cycle ;
					\draw   (291.86,111.01) .. controls (291.86,104.21) and (297.37,98.7) .. (304.17,98.7) .. controls (310.96,98.7) and (316.47,104.21) .. (316.47,111.01) .. controls (316.47,117.8) and (310.96,123.31) .. (304.17,123.31) .. controls (297.37,123.31) and (291.86,117.8) .. (291.86,111.01) -- cycle ;
					\draw   (316.47,111.01) .. controls (316.47,104.21) and (321.98,98.7) .. (328.78,98.7) .. controls (335.58,98.7) and (341.08,104.21) .. (341.08,111.01) .. controls (341.08,117.8) and (335.58,123.31) .. (328.78,123.31) .. controls (321.98,123.31) and (316.47,117.8) .. (316.47,111.01) -- cycle ;
					\draw   (365.25,111.01) .. controls (365.25,104.21) and (370.76,98.7) .. (377.56,98.7) .. controls (384.35,98.7) and (389.86,104.21) .. (389.86,111.01) .. controls (389.86,117.8) and (384.35,123.31) .. (377.56,123.31) .. controls (370.76,123.31) and (365.25,117.8) .. (365.25,111.01) -- cycle ;
					\draw   (389.86,111.01) .. controls (389.86,104.21) and (395.37,98.7) .. (402.17,98.7) .. controls (408.96,98.7) and (414.47,104.21) .. (414.47,111.01) .. controls (414.47,117.8) and (408.96,123.31) .. (402.17,123.31) .. controls (395.37,123.31) and (389.86,117.8) .. (389.86,111.01) -- cycle ;
					\draw   (414.47,111.01) .. controls (414.47,104.21) and (419.98,98.7) .. (426.78,98.7) .. controls (433.58,98.7) and (439.08,104.21) .. (439.08,111.01) .. controls (439.08,117.8) and (433.58,123.31) .. (426.78,123.31) .. controls (419.98,123.31) and (414.47,117.8) .. (414.47,111.01) -- cycle ;
					\draw   (461.25,111.01) .. controls (461.25,104.21) and (466.76,98.7) .. (473.56,98.7) .. controls (480.35,98.7) and (485.86,104.21) .. (485.86,111.01) .. controls (485.86,117.8) and (480.35,123.31) .. (473.56,123.31) .. controls (466.76,123.31) and (461.25,117.8) .. (461.25,111.01) -- cycle ;
					\draw   (485.86,111.01) .. controls (485.86,104.21) and (491.37,98.7) .. (498.17,98.7) .. controls (504.96,98.7) and (510.47,104.21) .. (510.47,111.01) .. controls (510.47,117.8) and (504.96,123.31) .. (498.17,123.31) .. controls (491.37,123.31) and (485.86,117.8) .. (485.86,111.01) -- cycle ;
					\draw   (510.47,111.01) .. controls (510.47,104.21) and (515.98,98.7) .. (522.78,98.7) .. controls (529.58,98.7) and (535.08,104.21) .. (535.08,111.01) .. controls (535.08,117.8) and (529.58,123.31) .. (522.78,123.31) .. controls (515.98,123.31) and (510.47,117.8) .. (510.47,111.01) -- cycle ;
					\draw  [color={rgb, 255:red, 208; green, 2; blue, 27 }  ,draw opacity=1 ][fill={rgb, 255:red, 208; green, 2; blue, 27 }  ,fill opacity=1 ] (180.97,98.7) .. controls (180.97,97.83) and (181.68,97.12) .. (182.56,97.12) .. controls (183.43,97.12) and (184.14,97.83) .. (184.14,98.7) .. controls (184.14,99.57) and (183.43,100.28) .. (182.56,100.28) .. controls (181.68,100.28) and (180.97,99.57) .. (180.97,98.7) -- cycle ;
					\draw  [color={rgb, 255:red, 208; green, 2; blue, 27 }  ,draw opacity=1 ][fill={rgb, 255:red, 208; green, 2; blue, 27 }  ,fill opacity=1 ] (302.58,98.7) .. controls (302.58,97.83) and (303.29,97.12) .. (304.17,97.12) .. controls (305.04,97.12) and (305.75,97.83) .. (305.75,98.7) .. controls (305.75,99.57) and (305.04,100.28) .. (304.17,100.28) .. controls (303.29,100.28) and (302.58,99.57) .. (302.58,98.7) -- cycle ;
					\draw  [color={rgb, 255:red, 208; green, 2; blue, 27 }  ,draw opacity=1 ][fill={rgb, 255:red, 208; green, 2; blue, 27 }  ,fill opacity=1 ] (425.2,98.7) .. controls (425.2,97.83) and (425.9,97.12) .. (426.78,97.12) .. controls (427.65,97.12) and (428.36,97.83) .. (428.36,98.7) .. controls (428.36,99.57) and (427.65,100.28) .. (426.78,100.28) .. controls (425.9,100.28) and (425.2,99.57) .. (425.2,98.7) -- cycle ;
					\draw  [color={rgb, 255:red, 208; green, 2; blue, 27 }  ,draw opacity=1 ][fill={rgb, 255:red, 208; green, 2; blue, 27 }  ,fill opacity=1 ] (496.58,123.31) .. controls (496.58,122.44) and (497.29,121.73) .. (498.17,121.73) .. controls (499.04,121.73) and (499.75,122.44) .. (499.75,123.31) .. controls (499.75,124.19) and (499.04,124.89) .. (498.17,124.89) .. controls (497.29,124.89) and (496.58,124.19) .. (496.58,123.31) -- cycle ;
					
					\draw (26.95,102.95) node [anchor=north west][inner sep=0.75pt]   [align=left] {$\displaystyle 1$};
					\draw (50.89,101.84) node [anchor=north west][inner sep=0.75pt]   [align=left] {$\displaystyle 2$};
					\draw (75.89,101.84) node [anchor=north west][inner sep=0.75pt]   [align=left] {$\displaystyle 3$};
					\draw (122,92) node [anchor=north west][inner sep=0.75pt]   [align=left] {$\displaystyle \tau $};
					\draw (177.95,102.95) node [anchor=north west][inner sep=0.75pt]   [align=left] {$\displaystyle 1$};
					\draw (201.89,101.84) node [anchor=north west][inner sep=0.75pt]   [align=left] {$\displaystyle 2$};
					\draw (226.89,101.84) node [anchor=north west][inner sep=0.75pt]   [align=left] {$\displaystyle 3$};
					\draw (274.95,102.95) node [anchor=north west][inner sep=0.75pt]   [align=left] {$\displaystyle 1$};
					\draw (298.89,101.84) node [anchor=north west][inner sep=0.75pt]   [align=left] {$\displaystyle 2$};
					\draw (323.89,101.84) node [anchor=north west][inner sep=0.75pt]   [align=left] {$\displaystyle 3$};
					\draw (372.95,102.95) node [anchor=north west][inner sep=0.75pt]   [align=left] {$\displaystyle 1$};
					\draw (396.89,101.84) node [anchor=north west][inner sep=0.75pt]   [align=left] {$\displaystyle 2$};
					\draw (421.89,101.84) node [anchor=north west][inner sep=0.75pt]   [align=left] {$\displaystyle 3$};
					\draw (468.95,102.95) node [anchor=north west][inner sep=0.75pt]   [align=left] {$\displaystyle 1$};
					\draw (492.89,101.84) node [anchor=north west][inner sep=0.75pt]   [align=left] {$\displaystyle 2$};
					\draw (517.89,101.84) node [anchor=north west][inner sep=0.75pt]   [align=left] {$\displaystyle 3$};
					\draw (249,101) node [anchor=north west][inner sep=0.75pt]   [align=left] {$\displaystyle +$};
					\draw (347,101) node [anchor=north west][inner sep=0.75pt]   [align=left] {$\displaystyle +$};
					\draw (445,101) node [anchor=north west][inner sep=0.75pt]   [align=left] {$\displaystyle +$};

				\end{tikzpicture}

				\caption{An example of how the transfer $\tau$ works (with $\F_2$ coefficients).}
				\label{fig:chain model for the transfer}
			\end{figure}
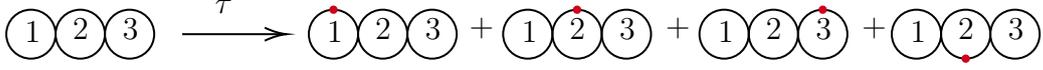
			
			\begin{figure}
				\centering

				\tikzset{every picture/.style={line width=0.75pt}} 
				
				\begin{tikzpicture}[x=0.75pt,y=0.75pt,yscale=-1,xscale=1]
					
					\draw   (16.8,110.01) .. controls (16.8,103.21) and (22.31,97.7) .. (29.11,97.7) .. controls (35.9,97.7) and (41.41,103.21) .. (41.41,110.01) .. controls (41.41,116.8) and (35.9,122.31) .. (29.11,122.31) .. controls (22.31,122.31) and (16.8,116.8) .. (16.8,110.01) -- cycle ;
					\draw   (41.41,110.01) .. controls (41.41,103.21) and (46.92,97.7) .. (53.72,97.7) .. controls (60.51,97.7) and (66.02,103.21) .. (66.02,110.01) .. controls (66.02,116.8) and (60.51,122.31) .. (53.72,122.31) .. controls (46.92,122.31) and (41.41,116.8) .. (41.41,110.01) -- cycle ;
					\draw   (100.25,110.01) .. controls (100.25,103.21) and (105.76,97.7) .. (112.56,97.7) .. controls (119.35,97.7) and (124.86,103.21) .. (124.86,110.01) .. controls (124.86,116.8) and (119.35,122.31) .. (112.56,122.31) .. controls (105.76,122.31) and (100.25,116.8) .. (100.25,110.01) -- cycle ;
					\draw   (124.86,110.01) .. controls (124.86,103.21) and (130.37,97.7) .. (137.17,97.7) .. controls (143.96,97.7) and (149.47,103.21) .. (149.47,110.01) .. controls (149.47,116.8) and (143.96,122.31) .. (137.17,122.31) .. controls (130.37,122.31) and (124.86,116.8) .. (124.86,110.01) -- cycle ;
					\draw   (149.47,110.01) .. controls (149.47,103.21) and (154.98,97.7) .. (161.78,97.7) .. controls (168.58,97.7) and (174.08,103.21) .. (174.08,110.01) .. controls (174.08,116.8) and (168.58,122.31) .. (161.78,122.31) .. controls (154.98,122.31) and (149.47,116.8) .. (149.47,110.01) -- cycle ;
					\draw   (193.88,110.79) .. controls (193.88,103.99) and (199.39,98.48) .. (206.19,98.48) .. controls (212.98,98.48) and (218.49,103.99) .. (218.49,110.79) .. controls (218.49,117.58) and (212.98,123.09) .. (206.19,123.09) .. controls (199.39,123.09) and (193.88,117.58) .. (193.88,110.79) -- cycle ;
					\draw   (218.49,110.79) .. controls (218.49,103.99) and (224,98.48) .. (230.8,98.48) .. controls (237.6,98.48) and (243.1,103.99) .. (243.1,110.79) .. controls (243.1,117.58) and (237.6,123.09) .. (230.8,123.09) .. controls (224,123.09) and (218.49,117.58) .. (218.49,110.79) -- cycle ;
					\draw   (243.1,110.79) .. controls (243.1,103.99) and (248.61,98.48) .. (255.41,98.48) .. controls (262.21,98.48) and (267.72,103.99) .. (267.72,110.79) .. controls (267.72,117.58) and (262.21,123.09) .. (255.41,123.09) .. controls (248.61,123.09) and (243.1,117.58) .. (243.1,110.79) -- cycle ;
					\draw   (267.72,110.79) .. controls (267.72,103.99) and (273.23,98.48) .. (280.02,98.48) .. controls (286.82,98.48) and (292.33,103.99) .. (292.33,110.79) .. controls (292.33,117.58) and (286.82,123.09) .. (280.02,123.09) .. controls (273.23,123.09) and (267.72,117.58) .. (267.72,110.79) -- cycle ;
					\draw   (310.3,110.79) .. controls (310.3,103.99) and (315.81,98.48) .. (322.6,98.48) .. controls (329.4,98.48) and (334.91,103.99) .. (334.91,110.79) .. controls (334.91,117.58) and (329.4,123.09) .. (322.6,123.09) .. controls (315.81,123.09) and (310.3,117.58) .. (310.3,110.79) -- cycle ;
					\draw   (334.91,110.79) .. controls (334.91,103.99) and (340.42,98.48) .. (347.21,98.48) .. controls (354.01,98.48) and (359.52,103.99) .. (359.52,110.79) .. controls (359.52,117.58) and (354.01,123.09) .. (347.21,123.09) .. controls (340.42,123.09) and (334.91,117.58) .. (334.91,110.79) -- cycle ;
					\draw   (359.52,110.79) .. controls (359.52,103.99) and (365.03,98.48) .. (371.83,98.48) .. controls (378.62,98.48) and (384.13,103.99) .. (384.13,110.79) .. controls (384.13,117.58) and (378.62,123.09) .. (371.83,123.09) .. controls (365.03,123.09) and (359.52,117.58) .. (359.52,110.79) -- cycle ;
					\draw   (384.13,110.79) .. controls (384.13,103.99) and (389.64,98.48) .. (396.44,98.48) .. controls (403.23,98.48) and (408.74,103.99) .. (408.74,110.79) .. controls (408.74,117.58) and (403.23,123.09) .. (396.44,123.09) .. controls (389.64,123.09) and (384.13,117.58) .. (384.13,110.79) -- cycle ;
					\draw   (194.27,176.42) .. controls (194.27,169.62) and (199.78,164.11) .. (206.58,164.11) .. controls (213.38,164.11) and (218.88,169.62) .. (218.88,176.42) .. controls (218.88,183.21) and (213.38,188.72) .. (206.58,188.72) .. controls (199.78,188.72) and (194.27,183.21) .. (194.27,176.42) -- cycle ;
					\draw   (218.88,176.42) .. controls (218.88,169.62) and (224.39,164.11) .. (231.19,164.11) .. controls (237.99,164.11) and (243.5,169.62) .. (243.5,176.42) .. controls (243.5,183.21) and (237.99,188.72) .. (231.19,188.72) .. controls (224.39,188.72) and (218.88,183.21) .. (218.88,176.42) -- cycle ;
					\draw   (243.5,176.42) .. controls (243.5,169.62) and (249,164.11) .. (255.8,164.11) .. controls (262.6,164.11) and (268.11,169.62) .. (268.11,176.42) .. controls (268.11,183.21) and (262.6,188.72) .. (255.8,188.72) .. controls (249,188.72) and (243.5,183.21) .. (243.5,176.42) -- cycle ;
					\draw   (218.88,151.81) .. controls (218.88,145.01) and (224.39,139.5) .. (231.19,139.5) .. controls (237.99,139.5) and (243.5,145.01) .. (243.5,151.81) .. controls (243.5,158.6) and (237.99,164.11) .. (231.19,164.11) .. controls (224.39,164.11) and (218.88,158.6) .. (218.88,151.81) -- cycle ;
					\draw   (289.59,176.42) .. controls (289.59,169.62) and (295.1,164.11) .. (301.9,164.11) .. controls (308.69,164.11) and (314.2,169.62) .. (314.2,176.42) .. controls (314.2,183.21) and (308.69,188.72) .. (301.9,188.72) .. controls (295.1,188.72) and (289.59,183.21) .. (289.59,176.42) -- cycle ;
					\draw   (314.2,176.42) .. controls (314.2,169.62) and (319.71,164.11) .. (326.51,164.11) .. controls (333.31,164.11) and (338.81,169.62) .. (338.81,176.42) .. controls (338.81,183.21) and (333.31,188.72) .. (326.51,188.72) .. controls (319.71,188.72) and (314.2,183.21) .. (314.2,176.42) -- cycle ;
					\draw   (338.81,176.42) .. controls (338.81,169.62) and (344.32,164.11) .. (351.12,164.11) .. controls (357.92,164.11) and (363.43,169.62) .. (363.43,176.42) .. controls (363.43,183.21) and (357.92,188.72) .. (351.12,188.72) .. controls (344.32,188.72) and (338.81,183.21) .. (338.81,176.42) -- cycle ;
					\draw   (314.2,201.03) .. controls (314.2,194.23) and (319.71,188.72) .. (326.51,188.72) .. controls (333.31,188.72) and (338.81,194.23) .. (338.81,201.03) .. controls (338.81,207.82) and (333.31,213.33) .. (326.51,213.33) .. controls (319.71,213.33) and (314.2,207.82) .. (314.2,201.03) -- cycle ;
					
					\draw (75,105) node [anchor=north west][inner sep=0.75pt]   [align=left] {$\displaystyle \circ _{2} \ $};
					\draw (24,101) node [anchor=north west][inner sep=0.75pt]   [align=left] {$\displaystyle 1$};
					\draw (48.44,101) node [anchor=north west][inner sep=0.75pt]   [align=left] {$\displaystyle 2$};
					\draw (107.95,101) node [anchor=north west][inner sep=0.75pt]   [align=left] {$\displaystyle 1$};
					\draw (131.89,101) node [anchor=north west][inner sep=0.75pt]   [align=left] {$\displaystyle 2$};
					\draw (156.89,101) node [anchor=north west][inner sep=0.75pt]   [align=left] {$\displaystyle 3$};
					\draw (201.58,101) node [anchor=north west][inner sep=0.75pt]   [align=left] {$\displaystyle 1$};
					\draw (225.52,101) node [anchor=north west][inner sep=0.75pt]   [align=left] {$\displaystyle 2$};
					\draw (250.52,101) node [anchor=north west][inner sep=0.75pt]   [align=left] {$\displaystyle 3$};
					\draw (274.74,101) node [anchor=north west][inner sep=0.75pt]   [align=left] {$\displaystyle 4$};
					\draw (317.99,101) node [anchor=north west][inner sep=0.75pt]   [align=left] {$\displaystyle 1$};
					\draw (341.94,101) node [anchor=north west][inner sep=0.75pt]   [align=left] {$\displaystyle 4$};
					\draw (366.94,101) node [anchor=north west][inner sep=0.75pt]   [align=left] {$\displaystyle 3$};
					\draw (391.16,101) node [anchor=north west][inner sep=0.75pt]   [align=left] {$\displaystyle 2$};
					\draw (201.3,167.25) node [anchor=north west][inner sep=0.75pt]   [align=left] {$\displaystyle 2$};
					\draw (226.3,167.25) node [anchor=north west][inner sep=0.75pt]   [align=left] {$\displaystyle 3$};
					\draw (250.52,168.03) node [anchor=north west][inner sep=0.75pt]   [align=left] {$\displaystyle 4$};
					\draw (226.58,143.36) node [anchor=north west][inner sep=0.75pt]   [align=left] {$\displaystyle 1$};
					\draw (296.62,167.25) node [anchor=north west][inner sep=0.75pt]   [align=left] {$\displaystyle 2$};
					\draw (321.62,167.25) node [anchor=north west][inner sep=0.75pt]   [align=left] {$\displaystyle 3$};
					\draw (345.84,168.03) node [anchor=north west][inner sep=0.75pt]   [align=left] {$\displaystyle 4$};
					\draw (321.9,191.8) node [anchor=north west][inner sep=0.75pt]   [align=left] {$\displaystyle 1$};
					\draw (175,105) node [anchor=north west][inner sep=0.75pt]   [align=left] {$\displaystyle =\ $};
					\draw (295,105) node [anchor=north west][inner sep=0.75pt]   [align=left] {$\displaystyle +\ $};
					\draw (177.46,167.5) node [anchor=north west][inner sep=0.75pt]   [align=left] {$\displaystyle +\ $};
					\draw (272,167.5) node [anchor=north west][inner sep=0.75pt]   [align=left] {$\displaystyle +\ $};

				\end{tikzpicture}

				\caption{An example of how the partial composition of grav works (with $\mathbb{F}_2$ coefficients). }
				\label{composition of Grav }
			\end{figure}
			
		\end{enumerate} 
		
		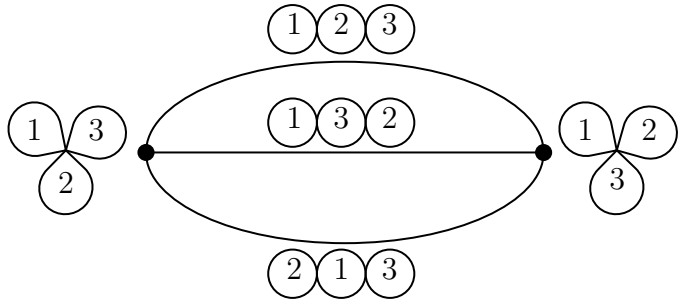
\begin{figure}
			\centering
			\tikzset{every picture/.style={line width=0.75pt}} 
			
			\begin{tikzpicture}[x=0.75pt,y=0.75pt,yscale=-1,xscale=1]
				
				\draw  [fill={rgb, 255:red, 0; green, 0; blue, 0 }  ,fill opacity=1 ] (295.87,96.73) .. controls (295.87,94.65) and (297.56,92.96) .. (299.64,92.96) .. controls (301.73,92.96) and (303.41,94.65) .. (303.41,96.73) .. controls (303.41,98.82) and (301.73,100.51) .. (299.64,100.51) .. controls (297.56,100.51) and (295.87,98.82) .. (295.87,96.73) -- cycle ;
				\draw   (48.32,98.23) .. controls (48.32,98.23) and (48.32,98.23) .. (48.32,98.23) .. controls (48.32,98.23) and (48.32,98.23) .. (48.32,98.23) .. controls (41.23,99.66) and (34.28,94.89) .. (32.81,87.58) .. controls (31.33,80.28) and (35.89,73.19) .. (42.98,71.76) .. controls (50.07,70.33) and (57.02,75.09) .. (58.49,82.4) .. controls (60.27,91.22) and (61.16,95.63) .. (61.16,95.63) .. controls (61.16,95.63) and (56.88,96.5) .. (48.32,98.23) -- cycle ;
				\draw   (64.72,83.03) .. controls (64.72,83.03) and (64.72,83.03) .. (64.72,83.03) .. controls (66.68,76.06) and (74.09,72.06) .. (81.27,74.08) .. controls (88.44,76.11) and (92.67,83.39) .. (90.7,90.36) .. controls (88.74,97.32) and (81.33,101.32) .. (74.15,99.3) .. controls (65.49,96.86) and (61.16,95.63) .. (61.16,95.63) .. controls (61.16,95.63) and (62.35,91.43) .. (64.72,83.03) -- cycle ;
				\draw   (70.57,104.74) .. controls (70.57,104.74) and (70.57,104.74) .. (70.57,104.74) .. controls (75.77,109.78) and (75.79,118.2) .. (70.6,123.56) .. controls (65.42,128.91) and (57,129.18) .. (51.8,124.15) .. controls (46.6,119.12) and (46.59,110.69) .. (51.77,105.34) .. controls (51.77,105.34) and (51.77,105.34) .. (51.77,105.34) .. controls (58.04,98.87) and (61.17,95.64) .. (61.16,95.63) .. controls (61.17,95.64) and (64.3,98.67) .. (70.57,104.74) -- cycle ;
				\draw  [fill={rgb, 255:red, 0; green, 0; blue, 0 }  ,fill opacity=1 ] (97.39,96.7) .. controls (97.39,94.62) and (99.08,92.93) .. (101.17,92.93) .. controls (103.25,92.93) and (104.94,94.62) .. (104.94,96.7) .. controls (104.94,98.78) and (103.25,100.47) .. (101.17,100.47) .. controls (99.08,100.47) and (97.39,98.78) .. (97.39,96.7) -- cycle ;
				\draw   (323.32,98.23) .. controls (323.32,98.23) and (323.32,98.23) .. (323.32,98.23) .. controls (323.32,98.23) and (323.32,98.23) .. (323.32,98.23) .. controls (316.23,99.66) and (309.28,94.89) .. (307.81,87.58) .. controls (306.33,80.28) and (310.89,73.19) .. (317.98,71.76) .. controls (325.07,70.33) and (332.02,75.09) .. (333.49,82.4) .. controls (335.27,91.22) and (336.16,95.63) .. (336.16,95.63) .. controls (336.16,95.63) and (331.88,96.5) .. (323.32,98.23) -- cycle ;
				\draw   (339.72,83.03) .. controls (339.72,83.03) and (339.72,83.03) .. (339.72,83.03) .. controls (341.68,76.06) and (349.09,72.06) .. (356.27,74.08) .. controls (363.44,76.11) and (367.67,83.39) .. (365.7,90.36) .. controls (363.74,97.32) and (356.33,101.32) .. (349.15,99.3) .. controls (340.49,96.86) and (336.16,95.63) .. (336.16,95.63) .. controls (336.16,95.63) and (337.35,91.43) .. (339.72,83.03) -- cycle ;
				\draw   (345.57,104.74) .. controls (345.57,104.74) and (345.57,104.74) .. (345.57,104.74) .. controls (350.77,109.78) and (350.79,118.2) .. (345.6,123.56) .. controls (340.42,128.91) and (332,129.18) .. (326.8,124.15) .. controls (321.6,119.12) and (321.59,110.69) .. (326.77,105.34) .. controls (333.04,98.87) and (336.17,95.64) .. (336.16,95.63) .. controls (336.17,95.64) and (339.3,98.67) .. (345.57,104.74) -- cycle ;
				\draw   (101.17,96.7) .. controls (101.17,71.5) and (145.6,51.07) .. (200.4,51.07) .. controls (255.21,51.07) and (299.64,71.5) .. (299.64,96.7) .. controls (299.64,121.9) and (255.21,142.33) .. (200.4,142.33) .. controls (145.6,142.33) and (101.17,121.9) .. (101.17,96.7) -- cycle ;
				\draw   (186.48,34.71) .. controls (186.48,27.99) and (191.92,22.55) .. (198.63,22.55) .. controls (205.34,22.55) and (210.79,27.99) .. (210.79,34.71) .. controls (210.79,41.42) and (205.34,46.86) .. (198.63,46.86) .. controls (191.92,46.86) and (186.48,41.42) .. (186.48,34.71) -- cycle ;
				\draw   (210.79,34.71) .. controls (210.79,27.99) and (216.23,22.55) .. (222.94,22.55) .. controls (229.65,22.55) and (235.1,27.99) .. (235.1,34.71) .. controls (235.1,41.42) and (229.65,46.86) .. (222.94,46.86) .. controls (216.23,46.86) and (210.79,41.42) .. (210.79,34.71) -- cycle ;
				\draw   (162.17,34.71) .. controls (162.17,27.99) and (167.61,22.55) .. (174.32,22.55) .. controls (181.03,22.55) and (186.48,27.99) .. (186.48,34.71) .. controls (186.48,41.42) and (181.03,46.86) .. (174.32,46.86) .. controls (167.61,46.86) and (162.17,41.42) .. (162.17,34.71) -- cycle ;
				\draw    (101.17,96.7) -- (299.64,96.7) ;
				\draw   (186.48,81.71) .. controls (186.48,74.99) and (191.92,69.55) .. (198.63,69.55) .. controls (205.34,69.55) and (210.79,74.99) .. (210.79,81.71) .. controls (210.79,88.42) and (205.34,93.86) .. (198.63,93.86) .. controls (191.92,93.86) and (186.48,88.42) .. (186.48,81.71) -- cycle ;
				\draw   (210.79,81.71) .. controls (210.79,74.99) and (216.23,69.55) .. (222.94,69.55) .. controls (229.65,69.55) and (235.1,74.99) .. (235.1,81.71) .. controls (235.1,88.42) and (229.65,93.86) .. (222.94,93.86) .. controls (216.23,93.86) and (210.79,88.42) .. (210.79,81.71) -- cycle ;
				\draw   (162.17,81.71) .. controls (162.17,74.99) and (167.61,69.55) .. (174.32,69.55) .. controls (181.03,69.55) and (186.48,74.99) .. (186.48,81.71) .. controls (186.48,88.42) and (181.03,93.86) .. (174.32,93.86) .. controls (167.61,93.86) and (162.17,88.42) .. (162.17,81.71) -- cycle ;
				\draw   (186.48,157.71) .. controls (186.48,150.99) and (191.92,145.55) .. (198.63,145.55) .. controls (205.34,145.55) and (210.79,150.99) .. (210.79,157.71) .. controls (210.79,164.42) and (205.34,169.86) .. (198.63,169.86) .. controls (191.92,169.86) and (186.48,164.42) .. (186.48,157.71) -- cycle ;
				\draw   (210.79,157.71) .. controls (210.79,150.99) and (216.23,145.55) .. (222.94,145.55) .. controls (229.65,145.55) and (235.1,150.99) .. (235.1,157.71) .. controls (235.1,164.42) and (229.65,169.86) .. (222.94,169.86) .. controls (216.23,169.86) and (210.79,164.42) .. (210.79,157.71) -- cycle ;
				\draw   (162.17,157.71) .. controls (162.17,150.99) and (167.61,145.55) .. (174.32,145.55) .. controls (181.03,145.55) and (186.48,150.99) .. (186.48,157.71) .. controls (186.48,164.42) and (181.03,169.86) .. (174.32,169.86) .. controls (167.61,169.86) and (162.17,164.42) .. (162.17,157.71) -- cycle ;
				
				\draw (71,78.46) node [anchor=north west][inner sep=0.75pt]   [align=left] {$\displaystyle 3$};
				\draw (40,78.46) node [anchor=north west][inner sep=0.75pt]   [align=left] {$\displaystyle 1$};
				\draw (55.77,105.34) node [anchor=north west][inner sep=0.75pt]   [align=left] {$\displaystyle 2$};
				\draw (331,103.46) node [anchor=north west][inner sep=0.75pt]   [align=left] {$\displaystyle 3$};
				\draw (315,78.46) node [anchor=north west][inner sep=0.75pt]   [align=left] {$\displaystyle 1$};
				\draw (347,78.46) node [anchor=north west][inner sep=0.75pt]   [align=left] {$\displaystyle 2$};
				\draw (169.79,25.5) node [anchor=north west][inner sep=0.75pt]   [align=left] {$\displaystyle 1$};
				\draw (193.24,25.5) node [anchor=north west][inner sep=0.75pt]   [align=left] {$\displaystyle 2$};
				\draw (217.43,25.5) node [anchor=north west][inner sep=0.75pt]   [align=left] {$\displaystyle 3$};
				\draw (169.79,72.5) node [anchor=north west][inner sep=0.75pt]   [align=left] {$\displaystyle 1$};
				\draw (193.24,72.5) node [anchor=north west][inner sep=0.75pt]   [align=left] {$\displaystyle 3$};
				\draw (217.43,72.5) node [anchor=north west][inner sep=0.75pt]   [align=left] {$\displaystyle 2$};
				\draw (169.79,148.5) node [anchor=north west][inner sep=0.75pt]   [align=left] {$\displaystyle 2$};
				\draw (193.24,148.5) node [anchor=north west][inner sep=0.75pt]   [align=left] {$\displaystyle 1$};
				\draw (217.43,148.5) node [anchor=north west][inner sep=0.75pt]   [align=left] {$\displaystyle 3$};

			\end{tikzpicture}

			\caption{This picture shows the CW-complex $\mathcal{C}_3/S^1\simeq \M_{0,4}$. There are two zero cells and three edges. }
			\label{fig:esempio di cactus senza punto base}
		\end{figure}

		We summarize these observations in the next table:
		\[
		\begin{tabularx}{\textwidth}{lXXX}
			
			\toprule
			\textbf{Space} & $F_n(\C)/S^1$  & $F_n(\C)$ & $fF_n(\C)$\\
			\midrule
			\textbf{Chain model } & $\mathcal{C}_n/S^1$   &  $\mathcal{C}_n$ & $f\mathcal{C}_n$\\
			\midrule
			\textbf{Operad in Top} &  & $\mathcal{D}_2$ & $ f\mathcal{D}_2$\\
			\midrule
			\textbf{Operad in $Ch(\Z)$} &  $grav$  & $Cact$ & $fCact$ \\
			\bottomrule
		\end{tabularx}
		\]
	\end{oss}

	\section{CW-decompositions for $\overline{\mathcal{M}}_{0,n+1}$ }\label{sec: combinatorial models for Deligne-Mumford}
	In \cite{Salvatore} the second author constructs a CW-decomposition of $FM(n)$, the Fulton-Pherson compactification of $F_n(\C)/\C\rtimes \R^{>0}$. The main result of this section is the construction of a similar cell decomposition for $\overline{\mathcal{M}}_{0,n+1}$ (Theorem \ref{thm:omeo tra la compattificazione e lo spazio dei nested cactus}).

	\subsection{Nested trees}\label{subsec:nested trees}
	In this paragraph we introduce some combinatorial notions that will be useful for the rest of the section. 
	\begin{defn}
		Let $R$ be a finite set. A \textbf{nested tree with leaves labelled in $R$} is a collection $\mathcal{S}$ of subsets of $R$ of cardinality at least $2$, called vertices, such that:
		\begin{itemize}
			
			\item $R\in\mathcal{S}$. This vertex is called the \textbf{root}.
			\item If $S_1,S_2\in\mathcal{S}$ then either $S_1\cap S_2=\emptyset$ or $S_1\subseteq S_2$ or $S_2\subseteq S_1$.
		\end{itemize}
		We denote by $N_R$ the set of all nested trees with leaves labelled in $R$. A pair $(S_1,S_2)\in \mathcal{S}\times\mathcal{S}$ is called an \textbf{internal edge} if  $S_1\subseteq S_2$ and there is no vertex $S_3\in\mathcal{S}$ such that $S_1\subseteq S_3\subseteq S_2$. We say that the internal edge $(S_1,S_2)$ goes out of $S_1$ and into $S_2$. For each $i\in R$ let $S_i$ be
		the minimal element of $\mathcal{S}$ containing $i$. We say that $i$ is an \textbf{open edge} going into $S_i$. The \textbf{valence} $\abs{S}$ of a vertex $S\in\mathcal{S}$ is the number of edges (either internal or open) going into it. For an example see Figure \ref{fig:esempio di nested tree}. Finally, observe that $N_R$ is partially ordered by inclusion. Geometrically, if $\mathcal{S}\subseteq \mathcal{T}$ then  $\mathcal{S}$ can be obtained from $\mathcal{T}$ by collapsing some edges.
	\end{defn}
	
	\begin{figure}
		
		\centering
		
		\tikzset{every picture/.style={line width=0.75pt}} 
		
		\begin{tikzpicture}[x=0.75pt,y=0.75pt,yscale=-1,xscale=1]
			
			\draw   (43,162.83) .. controls (43,131.63) and (87.92,106.33) .. (143.33,106.33) .. controls (198.75,106.33) and (243.67,131.63) .. (243.67,162.83) .. controls (243.67,194.04) and (198.75,219.33) .. (143.33,219.33) .. controls (87.92,219.33) and (43,194.04) .. (43,162.83) -- cycle ;
			\draw   (68,164.67) .. controls (68,145.52) and (90.76,130) .. (118.83,130) .. controls (146.91,130) and (169.67,145.52) .. (169.67,164.67) .. controls (169.67,183.81) and (146.91,199.33) .. (118.83,199.33) .. controls (90.76,199.33) and (68,183.81) .. (68,164.67) -- cycle ;
			\draw   (86,156.33) .. controls (86,145.84) and (97.12,137.33) .. (110.83,137.33) .. controls (124.55,137.33) and (135.67,145.84) .. (135.67,156.33) .. controls (135.67,166.83) and (124.55,175.33) .. (110.83,175.33) .. controls (97.12,175.33) and (86,166.83) .. (86,156.33) -- cycle ;
			\draw  [fill={rgb, 255:red, 0; green, 0; blue, 0 }  ,fill opacity=1 ] (330.83,146.33) .. controls (330.83,144.77) and (332.1,143.5) .. (333.67,143.5) .. controls (335.23,143.5) and (336.5,144.77) .. (336.5,146.33) .. controls (336.5,147.9) and (335.23,149.17) .. (333.67,149.17) .. controls (332.1,149.17) and (330.83,147.9) .. (330.83,146.33) -- cycle ;
			\draw    (297.67,110.33) -- (333.67,146.33) ;
			\draw    (365.67,114.33) -- (333.67,146.33) ;
			\draw    (333.67,146.33) -- (369.67,182.33) ;
			\draw    (369.67,182.33) -- (405.67,218.33) ;
			\draw    (401.67,150.33) -- (369.67,182.33) ;
			\draw    (437.67,186.33) -- (405.67,218.33) ;
			\draw  [fill={rgb, 255:red, 0; green, 0; blue, 0 }  ,fill opacity=1 ] (366.83,182.33) .. controls (366.83,180.77) and (368.1,179.5) .. (369.67,179.5) .. controls (371.23,179.5) and (372.5,180.77) .. (372.5,182.33) .. controls (372.5,183.9) and (371.23,185.17) .. (369.67,185.17) .. controls (368.1,185.17) and (366.83,183.9) .. (366.83,182.33) -- cycle ;
			\draw  [fill={rgb, 255:red, 0; green, 0; blue, 0 }  ,fill opacity=1 ] (402.83,218.33) .. controls (402.83,216.77) and (404.1,215.5) .. (405.67,215.5) .. controls (407.23,215.5) and (408.5,216.77) .. (408.5,218.33) .. controls (408.5,219.9) and (407.23,221.17) .. (405.67,221.17) .. controls (404.1,221.17) and (402.83,219.9) .. (402.83,218.33) -- cycle ;
			
			\draw (92,149) node [anchor=north west][inner sep=0.75pt]   [align=left] {$\displaystyle 1$};
			\draw (116,145) node [anchor=north west][inner sep=0.75pt]   [align=left] {$\displaystyle 2$};
			\draw (140,171) node [anchor=north west][inner sep=0.75pt]   [align=left] {$\displaystyle 3$};
			\draw (187,149) node [anchor=north west][inner sep=0.75pt]   [align=left] {$\displaystyle 4$};
			\draw (293,93) node [anchor=north west][inner sep=0.75pt]   [align=left] {$\displaystyle 1$};
			\draw (362,97) node [anchor=north west][inner sep=0.75pt]   [align=left] {$\displaystyle 2$};
			\draw (397,132) node [anchor=north west][inner sep=0.75pt]   [align=left] {$\displaystyle 3$};
			\draw (433,168) node [anchor=north west][inner sep=0.75pt]   [align=left] {$\displaystyle 4$};

		\end{tikzpicture}
		\caption{In this picture we see two graphical representations of the nested tree $\mathcal{S}=\{\{1,2,3,4\},\{1,2,3\},\{1,2\}\}$. In this example there are two internal edges, three vertices and each vertex has valence two.}
		\label{fig:esempio di nested tree}
	\end{figure}
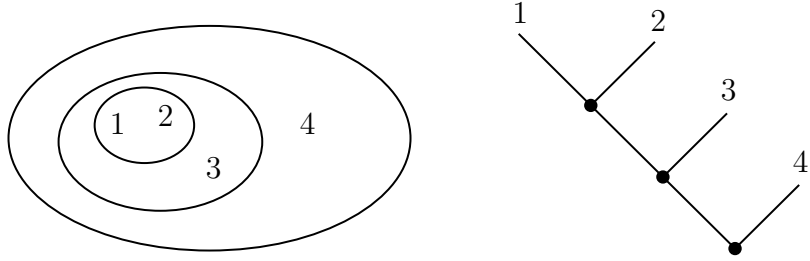
	\begin{oss}
		When $R=\{1,\dots,n\}$ we will use the notation $N_n$ instead of $N_{\{1,\dots,n\}}$ and we will call an element of $N_n$ a \textbf{nested tree with $n$-leaves}. If $R$ is a finite set with $n$ elements, the data of a total order on $R$ induces a bijection between $N_R$ and $N_n$.
	\end{oss}
	\begin{es}
		There are four nested trees on three leaves: the parenthesis indicates the elements of $\{1,2,3\}$ which are contained in the same $S\in\mathcal{S}$.
		
		\begin{tikzpicture}
			
			\draw[black, thick] (0,0) -- (-1,1) node[anchor=south]{$1$};
			\draw[black, thick] (0,0) -- (0,1) node[anchor=south]{$2$};
			\draw[black, thick] (0,0) -- (1,1) node[anchor=south]{$3$};
			\filldraw[color=black, fill=white, thick] (0,0) circle (2pt) node[anchor=north]{$(123)$};

			\draw[black, thick] (3.5,0) -- (3,1) node[anchor=south]{$1$};
			\draw[black, thick] (3.5,0) -- (4,1);
			\draw[black, thick] (4,1) -- (4.5,2) node[anchor=south]{$3$};
			\draw[black, thick] (4,1) -- (3.5,2) node[anchor=south]{$2$};
			\filldraw[color=black, fill=white, thick] (3.5,0) circle (2pt) node[anchor=north]{$(1(23))$};
			\filldraw[color=black, fill=white, thick] (4,1) circle (2pt);

			\draw[black, thick] (7,0) -- (6.5,1);
			\draw[black, thick] (7,0) -- (7.5,1) node[anchor=south]{$3$};
			\draw[black, thick] (6.5,1) -- (6,2) node[anchor=south]{$1$};
			\draw[black, thick] (6.5,1) -- (7,2) node[anchor=south]{$2$};
			\filldraw[color=black, fill=white, thick] (7,0) circle (2pt) node[anchor=north]{$((12)3)$};
			\filldraw[color=black, fill=white, thick] (6.5,1) circle (2pt);
			
			\draw[black, thick] (10.5,0) -- (10,1);
			\draw[black, thick] (10.5,0) -- (11,1) node[anchor=south]{$2$};
			\draw[black, thick] (10,1) -- (9.5,2) node[anchor=south]{$1$};
			\draw[black, thick] (10,1) -- (10.5,2) node[anchor=south]{$3$};
			\filldraw[color=black, fill=white, thick] (10.5,0) circle (2pt) node[anchor=north]{$((13)2)$};
			\filldraw[color=black, fill=white, thick] (10,1) circle (2pt);
			
		\end{tikzpicture}
	\end{es}
	
	\begin{defn}
		Let $\mathcal{S}$ be a nested tree with $n$-leaves. The \textbf{composition of cacti} associated to $\mathcal{S}$ is the  map
		\[
		\theta_{\mathcal{S}}:\prod_{S\in\mathcal{S}}\mathcal{C}_{\abs{S}}\to \mathcal{C}_n
		\]
		defined as follows: the collapse of any internal edge $(S,T)$ of $\mathcal{S}$ corresponds to a composition $\circ_i:\mathcal{C}_{\abs{T}}\times \mathcal{C}_{\abs{S}}\to\mathcal{C}_{\abs{T}+\abs{S}-1}$.  Then $\theta_{\mathcal{S}}$ is defined as the composition of such $\circ_i$, where we first collapse the internal edges which are closer to the leaves ad then proceed iteratively towards the root. 
		
	\end{defn}
	\begin{defn}
		For any $k\geq 2$ let $\sigma_k:\mathcal{C}_k/S^1\to \mathcal{C}_k$ be a fixed section of the trivial bundle $p:\mathcal{C}_k\to \mathcal{C}_k/S^1$ and $\mathcal{S}$ be a nested tree with $n$-leaves. We define a map 
		\[
		\gamma^{\sigma}_{\mathcal{S}}:\mathcal{C}_{\abs{R}}/S^1\times \prod_{S\in\mathcal{S}-R}\left(S^1\times \mathcal{C}_{\abs{S}}/S^1\right)\to \mathcal{C}_n
		\]
		iteratively on $n$ as follows: 
		\begin{description}
			\item[$n=2$] The only nested tree with two leaves is the corolla $\mathcal{S}=(12)$, and in this case we set $\gamma_{\mathcal{S}}^{\sigma}\coloneqq \sigma_2$.
			\item[$n\geq 3$] If $\mathcal{S}$ is the corolla with $n$-leaves, put $\gamma_{\mathcal{S}}^{\sigma}\coloneqq \sigma_n$. Now suppose that $\mathcal{S}$ is not the corolla, and denote by $(S_1,R),\dots,(S_m,R)$ the internal edges going into the root $R\coloneqq\{1,\dots,n\}$. If we delete the root $R$, we obtain a collection of nested trees  $\mathcal{S}_1\dots,\mathcal{S}_m$ with leaves labelled (respectively) by $S_1,\dots,S_m$. We will denote by $n_i$ the cardinality of $S_i$ and by $\mathcal{T}$ the nested tree given by $\{S_1,\dots,S_m,R\}$. Then $\gamma_{\mathcal{S}}^{\sigma}$ is defined to be the following composition:
			\begin{equation}\label{diag: diagramma che definisce le embeddings}
				\begin{tikzcd}
					& \mathcal{C}_{\abs{R}}/S^1\times \prod_{\substack{S\in\mathcal{S}\\ S\neq R}}\left(S^1\times \mathcal{C}_{\abs{S}}/S^1\right) \arrow[d] \arrow[r,dashed,"\gamma^{\sigma}_{\mathcal{S}}"] & \mathcal{C}_n\\
					& \mathcal{C}_{\abs{R}}/S^1\times\prod_{i=1}^m \prod_{S\in\mathcal{S}_i}\left(S^1\times \mathcal{C}_{\abs{S}}/S^1\right) \arrow[d,"\sigma_{\abs{R}}\times \gamma^{\sigma}_{\mathcal{S}_1}\times\cdots\times\gamma^{\sigma}_{\mathcal{S}_m}" ]& \\
					& \mathcal{C}_{\abs{R}}\times \prod_{i=1}^m S^1\times\mathcal{C}_{n_i}\arrow[r,"1\times \rho^m"] & \mathcal{C}_{\abs{R}}\times \prod_{i=1}^m \mathcal{C}_{n_i}\arrow[uu,"\theta_{\mathcal{T}}"]
				\end{tikzcd}
			\end{equation}
		\end{description}
		In this diagram the first vertical arrow is just a permutation of the factors in the cartesian product, $\rho: S^1\times \mathcal{C}_k\to \mathcal{C}_k$ is the circle action and $\theta_{\mathcal{T}}$ is the composition of cacti associated to the nested tree $\mathcal{T}$.
	\end{defn}
	\begin{es}
		Let $\mathcal{S}\in N_5$ given by $R=\{1,2,3,4,5\}$, $S=\{2,3,4,5\}$, $T=\{4,5\}$. Then 
		\begin{equation*}
			\gamma^{\sigma}_{\mathcal{S}}:\mathcal{C}_{\abs{R}}/S^1\times (S^1\times  \mathcal{C}_{\abs{S}}/S^1)\times (S^1\times  \mathcal{C}_{\abs{T}}/S^1)\to \mathcal{C}_5
		\end{equation*}
		sends $	(x_R,\alpha_S,x_S,\alpha_T,x_T)$ to $ \sigma_2(x_R)\circ_2(\alpha_S\cdot(\sigma_3(x_S)\circ_3(\alpha_T\cdot \sigma_2(x_T))))$. For a picture see Figure \ref{fig:esempio di applicazione della mappa strana}.
	\end{es}
	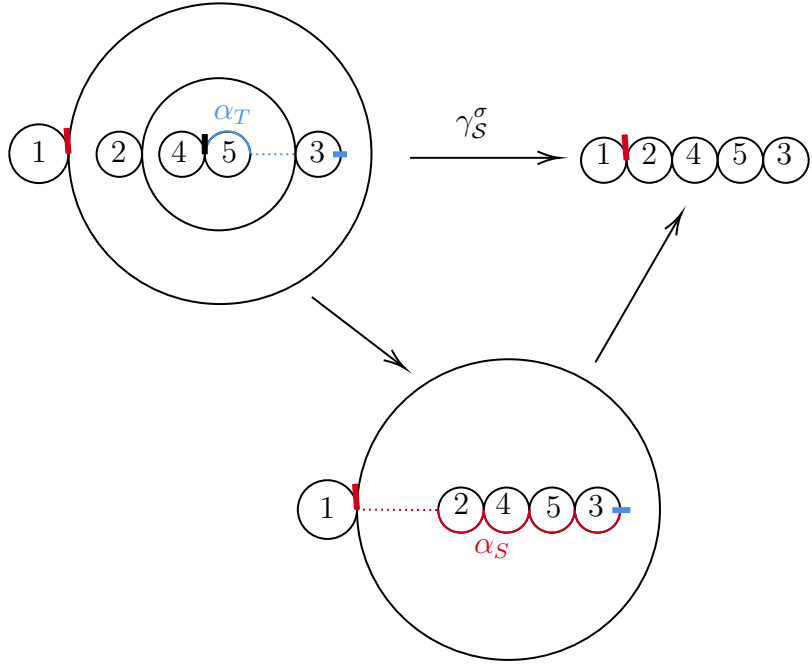
\begin{figure}
		\centering
		
		\tikzset{every picture/.style={line width=0.75pt}} 
		
		\begin{tikzpicture}[x=0.75pt,y=0.75pt,yscale=-1,xscale=1]
			
			\draw   (75.68,160.61) .. controls (75.68,152.45) and (82.29,145.84) .. (90.45,145.84) .. controls (98.61,145.84) and (105.23,152.45) .. (105.23,160.61) .. controls (105.23,168.77) and (98.61,175.39) .. (90.45,175.39) .. controls (82.29,175.39) and (75.68,168.77) .. (75.68,160.61) -- cycle ;
			\draw   (105.23,160.61) .. controls (105.23,118.85) and (139.08,85) .. (180.84,85) .. controls (222.6,85) and (256.45,118.85) .. (256.45,160.61) .. controls (256.45,202.37) and (222.6,236.23) .. (180.84,236.23) .. controls (139.08,236.23) and (105.23,202.37) .. (105.23,160.61) -- cycle ;
			\draw   (119.23,160.89) .. controls (119.23,154.61) and (124.31,149.52) .. (130.59,149.52) .. controls (136.86,149.52) and (141.95,154.61) .. (141.95,160.89) .. controls (141.95,167.16) and (136.86,172.25) .. (130.59,172.25) .. controls (124.31,172.25) and (119.23,167.16) .. (119.23,160.89) -- cycle ;
			\draw   (141.95,160.89) .. controls (141.95,139.75) and (159.09,122.61) .. (180.23,122.61) .. controls (201.36,122.61) and (218.5,139.75) .. (218.5,160.89) .. controls (218.5,182.03) and (201.36,199.16) .. (180.23,199.16) .. controls (159.09,199.16) and (141.95,182.03) .. (141.95,160.89) -- cycle ;
			\draw   (218.5,160.89) .. controls (218.5,154.61) and (223.59,149.52) .. (229.86,149.52) .. controls (236.14,149.52) and (241.23,154.61) .. (241.23,160.89) .. controls (241.23,167.16) and (236.14,172.25) .. (229.86,172.25) .. controls (223.59,172.25) and (218.5,167.16) .. (218.5,160.89) -- cycle ;
			\draw [color={rgb, 255:red, 208; green, 2; blue, 27 }  ,draw opacity=1 ][line width=2.25]    (104.23,147.68) -- (105.23,160.61) ;
			\draw [color={rgb, 255:red, 74; green, 144; blue, 226 }  ,draw opacity=1 ] [dash pattern={on 0.75pt off 1.5pt}]  (195.95,160.89) -- (218.5,160.89) ;
			\draw  [color={rgb, 255:red, 0; green, 0; blue, 0 }  ,draw opacity=1 ] (173.23,160.89) .. controls (173.23,154.61) and (178.31,149.52) .. (184.59,149.52) .. controls (190.86,149.52) and (195.95,154.61) .. (195.95,160.89) .. controls (195.95,167.16) and (190.86,172.25) .. (184.59,172.25) .. controls (178.31,172.25) and (173.23,167.16) .. (173.23,160.89) -- cycle ;
			\draw   (150.5,160.89) .. controls (150.5,154.61) and (155.59,149.52) .. (161.86,149.52) .. controls (168.14,149.52) and (173.23,154.61) .. (173.23,160.89) .. controls (173.23,167.16) and (168.14,172.25) .. (161.86,172.25) .. controls (155.59,172.25) and (150.5,167.16) .. (150.5,160.89) -- cycle ;
			\draw [color={rgb, 255:red, 74; green, 144; blue, 226 }  ,draw opacity=1 ][line width=2.25]    (237.23,160.89) -- (244.23,160.89) ;
			\draw    (275.68,162.68) -- (348.23,162.68) ;
			\draw [shift={(350.23,162.68)}, rotate = 180] [color={rgb, 255:red, 0; green, 0; blue, 0 }  ][line width=0.75]    (10.93,-3.29) .. controls (6.95,-1.4) and (3.31,-0.3) .. (0,0) .. controls (3.31,0.3) and (6.95,1.4) .. (10.93,3.29)   ;
			\draw   (383.77,163.89) .. controls (383.77,157.61) and (388.86,152.52) .. (395.14,152.52) .. controls (401.41,152.52) and (406.5,157.61) .. (406.5,163.89) .. controls (406.5,170.16) and (401.41,175.25) .. (395.14,175.25) .. controls (388.86,175.25) and (383.77,170.16) .. (383.77,163.89) -- cycle ;
			\draw   (451.95,163.89) .. controls (451.95,157.61) and (457.04,152.52) .. (463.31,152.52) .. controls (469.59,152.52) and (474.68,157.61) .. (474.68,163.89) .. controls (474.68,170.16) and (469.59,175.25) .. (463.31,175.25) .. controls (457.04,175.25) and (451.95,170.16) .. (451.95,163.89) -- cycle ;
			\draw   (429.23,163.89) .. controls (429.23,157.61) and (434.31,152.52) .. (440.59,152.52) .. controls (446.86,152.52) and (451.95,157.61) .. (451.95,163.89) .. controls (451.95,170.16) and (446.86,175.25) .. (440.59,175.25) .. controls (434.31,175.25) and (429.23,170.16) .. (429.23,163.89) -- cycle ;
			\draw   (406.5,163.89) .. controls (406.5,157.61) and (411.59,152.52) .. (417.86,152.52) .. controls (424.14,152.52) and (429.23,157.61) .. (429.23,163.89) .. controls (429.23,170.16) and (424.14,175.25) .. (417.86,175.25) .. controls (411.59,175.25) and (406.5,170.16) .. (406.5,163.89) -- cycle ;
			\draw   (361.05,163.89) .. controls (361.05,157.61) and (366.14,152.52) .. (372.41,152.52) .. controls (378.69,152.52) and (383.77,157.61) .. (383.77,163.89) .. controls (383.77,170.16) and (378.69,175.25) .. (372.41,175.25) .. controls (366.14,175.25) and (361.05,170.16) .. (361.05,163.89) -- cycle ;
			\draw [color={rgb, 255:red, 208; green, 2; blue, 27 }  ,draw opacity=1 ][line width=2.25]    (382.77,150.95) -- (383.77,163.89) ;
			\draw   (219.68,339.61) .. controls (219.68,331.45) and (226.29,324.84) .. (234.45,324.84) .. controls (242.61,324.84) and (249.23,331.45) .. (249.23,339.61) .. controls (249.23,347.77) and (242.61,354.39) .. (234.45,354.39) .. controls (226.29,354.39) and (219.68,347.77) .. (219.68,339.61) -- cycle ;
			\draw   (249.23,339.61) .. controls (249.23,297.85) and (283.08,264) .. (324.84,264) .. controls (366.6,264) and (400.45,297.85) .. (400.45,339.61) .. controls (400.45,381.37) and (366.6,415.23) .. (324.84,415.23) .. controls (283.08,415.23) and (249.23,381.37) .. (249.23,339.61) -- cycle ;
			\draw   (289.77,339.89) .. controls (289.77,333.61) and (294.86,328.52) .. (301.14,328.52) .. controls (307.41,328.52) and (312.5,333.61) .. (312.5,339.89) .. controls (312.5,346.16) and (307.41,351.25) .. (301.14,351.25) .. controls (294.86,351.25) and (289.77,346.16) .. (289.77,339.89) -- cycle ;
			\draw   (357.95,339.89) .. controls (357.95,333.61) and (363.04,328.52) .. (369.31,328.52) .. controls (375.59,328.52) and (380.68,333.61) .. (380.68,339.89) .. controls (380.68,346.16) and (375.59,351.25) .. (369.31,351.25) .. controls (363.04,351.25) and (357.95,346.16) .. (357.95,339.89) -- cycle ;
			\draw [color={rgb, 255:red, 208; green, 2; blue, 27 }  ,draw opacity=1 ][line width=2.25]    (248.23,326.68) -- (249.23,339.61) ;
			\draw [color={rgb, 255:red, 208; green, 2; blue, 27 }  ,draw opacity=1 ] [dash pattern={on 0.75pt off 1.5pt}]  (249.23,339.61) -- (289.77,339.89) ;
			\draw  [color={rgb, 255:red, 0; green, 0; blue, 0 }  ,draw opacity=1 ] (335.23,339.89) .. controls (335.23,333.61) and (340.31,328.52) .. (346.59,328.52) .. controls (352.86,328.52) and (357.95,333.61) .. (357.95,339.89) .. controls (357.95,346.16) and (352.86,351.25) .. (346.59,351.25) .. controls (340.31,351.25) and (335.23,346.16) .. (335.23,339.89) -- cycle ;
			\draw   (312.5,339.89) .. controls (312.5,333.61) and (317.59,328.52) .. (323.86,328.52) .. controls (330.14,328.52) and (335.23,333.61) .. (335.23,339.89) .. controls (335.23,346.16) and (330.14,351.25) .. (323.86,351.25) .. controls (317.59,351.25) and (312.5,346.16) .. (312.5,339.89) -- cycle ;
			\draw [color={rgb, 255:red, 208; green, 2; blue, 27 }  ,draw opacity=1 ]   (289.77,339.89) .. controls (289.77,354.32) and (311.3,355.51) .. (312.5,339.89) ;
			\draw [color={rgb, 255:red, 208; green, 2; blue, 27 }  ,draw opacity=1 ]   (312.5,339.89) .. controls (312.5,354.32) and (334.03,355.51) .. (335.23,339.89) ;
			\draw    (226.68,232.68) -- (271.64,267.34) ;
			\draw [shift={(273.23,268.57)}, rotate = 217.63] [color={rgb, 255:red, 0; green, 0; blue, 0 }  ][line width=0.75]    (10.93,-3.29) .. controls (6.95,-1.4) and (3.31,-0.3) .. (0,0) .. controls (3.31,0.3) and (6.95,1.4) .. (10.93,3.29)   ;
			\draw    (368.23,265.57) -- (410.23,192.3) ;
			\draw [shift={(411.23,190.57)}, rotate = 119.83] [color={rgb, 255:red, 0; green, 0; blue, 0 }  ][line width=0.75]    (10.93,-3.29) .. controls (6.95,-1.4) and (3.31,-0.3) .. (0,0) .. controls (3.31,0.3) and (6.95,1.4) .. (10.93,3.29)   ;
			\draw [color={rgb, 255:red, 74; green, 144; blue, 226 }  ,draw opacity=1 ]   (195.77,160.89) .. controls (196.23,148.68) and (176.23,142.68) .. (173.23,160.89) ;
			\draw [line width=2.25]    (173.23,150.68) -- (173.23,160.89) ;
			\draw [color={rgb, 255:red, 208; green, 2; blue, 27 }  ,draw opacity=1 ]   (335.23,339.89) .. controls (335.23,354.32) and (356.76,355.51) .. (357.95,339.89) ;
			\draw [color={rgb, 255:red, 208; green, 2; blue, 27 }  ,draw opacity=1 ]   (357.95,339.89) .. controls (357.95,354.32) and (379.48,355.51) .. (380.68,339.89) ;
			\draw [color={rgb, 255:red, 74; green, 144; blue, 226 }  ,draw opacity=1 ][line width=2.25]    (376.68,339.89) -- (385.68,339.89) ;
			
			\draw (85,153) node [anchor=north west][inner sep=0.75pt]   [align=left] {$\displaystyle 1$};
			\draw (124.82,153) node [anchor=north west][inner sep=0.75pt]   [align=left] {$\displaystyle 2$};
			\draw (224.09,153) node [anchor=north west][inner sep=0.75pt]   [align=left] {$\displaystyle 3$};
			\draw (154.9,153) node [anchor=north west][inner sep=0.75pt]   [align=left] {$\displaystyle 4$};
			\draw (180.01,153) node [anchor=north west][inner sep=0.75pt]   [align=left] {$\displaystyle 5$};
			\draw (176.23,135) node [anchor=north west][inner sep=0.75pt]  [color={rgb, 255:red, 74; green, 144; blue, 226 }  ,opacity=1 ] [align=left] {$\displaystyle \alpha _{T}$};
			\draw (367,154) node [anchor=north west][inner sep=0.75pt]   [align=left] {$\displaystyle 1$};
			\draw (389.82,154) node [anchor=north west][inner sep=0.75pt]   [align=left] {$\displaystyle 2$};
			\draw (457.09,154) node [anchor=north west][inner sep=0.75pt]   [align=left] {$\displaystyle 3$};
			\draw (411.9,154) node [anchor=north west][inner sep=0.75pt]   [align=left] {$\displaystyle 4$};
			\draw (435.01,154) node [anchor=north west][inner sep=0.75pt]   [align=left] {$\displaystyle 5$};
			\draw (297,136) node [anchor=north west][inner sep=0.75pt]   [align=left] {$\displaystyle \gamma _{\mathcal{S}}^{\sigma }$};
			\draw (229,332) node [anchor=north west][inner sep=0.75pt]   [align=left] {$\displaystyle 1$};
			\draw (295.82,330.06) node [anchor=north west][inner sep=0.75pt]   [align=left] {$\displaystyle 2$};
			\draw (363.09,331.06) node [anchor=north west][inner sep=0.75pt]   [align=left] {$\displaystyle 3$};
			\draw (316.9,330.25) node [anchor=north west][inner sep=0.75pt]   [align=left] {$\displaystyle 4$};
			\draw (342.01,331.45) node [anchor=north west][inner sep=0.75pt]   [align=left] {$\displaystyle 5$};
			\draw (306.14,353) node [anchor=north west][inner sep=0.75pt]  [color={rgb, 255:red, 208; green, 2; blue, 27 }  ,opacity=1 ] [align=left] {$\displaystyle \alpha _{S}$};

		\end{tikzpicture}

		\caption{In this picture the element $(x_R,\alpha_S,x_S,\alpha_T,x_T)$ is depicted on the left: $x_R$ is the bigger cactus, $x_S$ is the middle one and $x_T$ is the smaller. They are nested one into the other according to the nested tree $\mathcal{S}=(1(23(45)))$. The base points depicted on the cacti are those obtained using the sections $\sigma_k:\mathcal{C}_k/S^1\to \mathcal{C}_k$. To compute $\gamma^{\sigma}_{\mathcal{S}}$ we proceed as follows: start from the cactus $x_T$, use the section $\sigma_2:\mathcal{C}_2/S^1\to\mathcal{C}_2$ to get a base point (the black one) and then rotate by $\alpha_T$: the result is $\alpha_T\cdot\sigma_2(x_T)$. Then glue this cactus into the third lobe of $\sigma_3(x_S)$, obtaining the lower nested cactus. Now use $\alpha_S$ to rotate the base point and glue the resulting cactus in the second lobe of $\sigma_2(x_R)$. The dotted lines in this picture shows how to glue a smaller cactus into a bigger one.}
		\label{fig:esempio di applicazione della mappa strana}
	\end{figure}
	\begin{lem}
		For any nested tree $\mathcal{S}\in N_n$ the composite 
		\[
		\mathcal{C}_{\abs{R}}/S^1\times \prod_{S\in\mathcal{S}-R}\left(S^1\times \mathcal{C}_{\abs{S}}/S^1\right)\xrightarrow{\gamma^{\sigma}_{\mathcal{S}}}\mathcal{C}_n\xrightarrow{p} \mathcal{C}_n/S^1
		\]
		is an embedding.
	\end{lem}
	\begin{proof}
		Since the domain is compact and the target is Hausdorff it suffices to show the injectivity. Let us proceed by induction on the number of leaves $n$: 
		\begin{description}
			\item[$n=2$] the only nested tree with two leaves is the corolla $\mathcal{S}=(12)$. By definition $\gamma^{\sigma}_{\mathcal{S}}\coloneqq \sigma_2$, so $p\circ \gamma^{\sigma}_{\mathcal{S}}=Id$.
			\item[$n\geq 3$] if $\mathcal{S}$ is the corolla then by the same argument as before $p\circ \gamma^{\sigma}_{\mathcal{S}}=Id$, so it is injective. Now suppose $\mathcal{S}$ is not the corolla: if we have
			\[
			p\circ\gamma^{\sigma}_{\mathcal{S}}(x_R,(\alpha_S,x_S)_{S\in\mathcal{S}-R})=p\circ\gamma^{\sigma}_{\mathcal{S}}(y_R,(\beta_S,y_S)_{S\in\mathcal{S}-R})
			\]
			then $\gamma^{\sigma}_{\mathcal{S}}(x_R,(\alpha_S,x_S)_{S\in\mathcal{S}-R})$ and $\gamma^{\sigma}_{\mathcal{S}}(y_R,(\beta_S,y_S)_{S\in\mathcal{S}-R})$ differ by a rotation. Let us denote by $(S_1,R),\dots,(S_m,R)$ the internal edges going into the root $R\coloneqq\{1,\dots,n\}$. If we delete the root $R$, we obtain a collection of nested trees $\mathcal{S}_1,\dots,\mathcal{S}_m$ with leaves labelled (respectively) by $S_1,\dots,S_m$. We will denote by $n_i$ the cardinality of $S_i$ and by $\mathcal{T}$ the nested tree given by $\{S_1,\dots,S_m,R\}$. The map $\gamma_{\mathcal{S}}^{\sigma}$ is defined to be the following composition:
			\[
			\begin{tikzcd}
				& \mathcal{C}_{\abs{R}}/S^1\times \prod_{\substack{S\in\mathcal{S}\\ S\neq R}}\left(S^1\times \mathcal{C}_{\abs{S}}/S^1\right) \arrow[rr,"\gamma^{\sigma}_{\mathcal{S}}"] \arrow[rd,"f"]& & \mathcal{C}_n\\
				& & \mathcal{C}_{\abs{R}}\times \prod_{i=1}^m \mathcal{C}_{n_i}\arrow[ur,"\theta_{\mathcal{T}}"] &
			\end{tikzcd}
			\]
			where $f$ is the composite of the first three arrows of diagram \ref{diag: diagramma che definisce le embeddings} and $\theta_{\mathcal{T}}$ is the composition of cacti associated to the nested tree $\mathcal{T}$. By definition 
			\[
			\begin{cases}
				f(x_R,(\alpha_S,x_S)_{S\in\mathcal{S}-R})=(\sigma_{m}(x_R),a_1,\dots,a_m)\\
				f(y_R,(\beta_S,y_S)_{S\in\mathcal{S}-R})=(\sigma_m(y_R),b_1,\dots,b_m)
			\end{cases}
			\]
			for some elements $a_i,b_i\in\mathcal{C}_{n_i}$, $i=1,\dots,m$. Since $\theta_{\mathcal{T}}(\sigma_m(y_R),b_1,\dots,b_m)$ and $\theta_{\mathcal{T}}(\sigma_m(x_R),a_1,\dots,a_m)$ differ by a rotation, we have that
			\[
			x_R=p(\sigma_{m}(x_R))=p(\sigma_{m}(y_R))=y_R
			\]
			Therefore $\sigma_{m}(x_R)=\sigma_{m}(y_R)$ and we get that
			\[
			\theta_{\mathcal{T}}(\sigma_{m}(y_R),b_1,\dots,b_m)=\theta_{\mathcal{T}}(\sigma_{m}(x_R),a_1,\dots,a_m)
			\]
			But $\theta_{\mathcal{T}}$ is an embedding, so $a_i=b_i$ for any $i=1,\dots,m$. Now observe that
			\begin{align*}
				p\circ\gamma^{\sigma}_{\mathcal{S}_i}(x_{S_i},(\alpha_S,x_S)_{S\in\mathcal{S}_i-S_i})&=p(a_i)\\
				&=p(b_i)\\
				&=p\circ\gamma^{\sigma}_{\mathcal{S}_i}(y_{S_i},(\beta_S,y_S)_{S\in\mathcal{S}_i-S_i})
			\end{align*}
			and use the inductive hypothesis to conclude that for any $i=1,\dots,m$
			\[
			\begin{cases}
				x_S=y_S \text{ for all } S\in\mathcal{S}_i\\
				\alpha_S=\beta_S \text{ for all } S\in\mathcal{S}_i-S_i\\
				
			\end{cases}
			\]
			Finally, observe that 
			\[
			\alpha_{S_i}\cdot\gamma_{\mathcal{S}_i}^{\sigma}(x_{S_i}, (\alpha_S,x_S)_{S\in\mathcal{S}_i-S_i})=a_i=b_i=\beta_{S_i}\cdot\gamma_{\mathcal{S}_i}^{\sigma}(y_{S_i}, (\beta_S,y_S)_{S\in\mathcal{S}_i-S_i})
			\]
			and since the circle acts freely on the space of cacti we conclude that $\alpha_{S_i}=\beta_{S_i}$, concluding the proof.
		\end{description} 
	\end{proof}
	\subsection{The space of nested cacti}\label{subsec: space of nested cacti}
	In this paragraph we introduce a regular CW-complex $N_n^{\sigma}(\mathcal{C}/S^1)$, called the \emph{space of nested cacti}. This space turns out to be homeomorphic to $\overline{\M}_{0,n+1}$, giving a CW-decomposition of $\overline{\M}_{0,n+1}$.
	\begin{defn}
		Let $\mathcal{S}$ be a nested tree with $n$-leaves and root $R$. We define 
		\[
		St(\mathcal{S})\coloneqq \mathcal{C}_{\abs{R}}/S^1\times\prod_{S\in\mathcal{S}-R}(D_2\times\mathcal{C}_{\abs{S}}/S^1) 
		\]
		We will call $St(\mathcal{S})$ the \textbf{stratum} associated to $\mathcal{S}$. In what follows we will think of $D_2=\{z\in\C\mid \abs{z}\leq 1\}$ as the quotient of the cylinder $[0,1]\times S^1$ by collapsing to a point the base $\{0\}\times S^1$. So we will represent a point in $D_2$ by a pair $(t,\alpha)\in[0,1]\times S^1$. 
	\end{defn}
	\begin{defn}
		For any $k\geq 2$ fix a section $\sigma_k:\mathcal{C}_k/S^1\to \mathcal{C}_k$. The \textbf{space of nested cacti} (with $n$ leaves) is defined as the quotient 
		\[
		N^{\sigma}_{n}(\mathcal{C}/S^1)\coloneqq \frac{{\bigsqcup_{\mathcal{S}\in N_n}\mathcal{C}_{\abs{R}}/S^1\times\prod_{S\in\mathcal{S}-R}(D_2\times\mathcal{C}_{\abs{S}}/S^1 })}{\sim}
		\] 
		The equivalence relation $\sim$ is defined as follows: fix a point $(x_R,(t_S,\alpha_S,x_S)_{S\in\mathcal{S}-R})$, with $t_S\in[0,1]$, $\alpha_S\in S^1$, $x_S\in\mathcal{C}_{\abs{S}}$. Suppose $t_{S_0}=1$, and let $\mathcal{T}\subseteq \mathcal{S}$ be the maximal nested tree containing $S_0$ such that $t_S\neq 0$ for any $S\in\mathcal{T}-U$, where $U$ is the root of $\mathcal{T}$. Then $(x_R,(t_S,\alpha_S,x_S)_{S\in\mathcal{S}-R}))\sim (x'_R,(t'_S,\alpha'_S,x'_S)_{S\in\mathcal{S}-\{R,S_0\}})$ if and only if 
		\[
		\begin{cases}
			x_S=x'_S \text{ for all } S\in\mathcal{S}-\mathcal{T}\\
			t_S=t'_S \text{ for all } S\in\mathcal{S}-S_0\\
			\alpha_S=\alpha'_S \text{ for all } S\in\mathcal{S}-\mathcal{T}\\
			p\circ\gamma^{\sigma}_{\mathcal{T}}(x_U,(\alpha_S,x_S)_{S\in\mathcal{T}-U})=p\circ\gamma^{\sigma}_{\mathcal{T}-S_0} (x'_U,(\alpha'_S,x'_S)_{S\in\mathcal{T}-\{U,S_0\}})
		\end{cases}
		\]
		For an example of nested cacti that are identified under this relation see Figure \ref{fig:esempio di nested cactus che vengono identificati}. To sum up, a point in $N_n^{\sigma}(\mathcal{C}/S^1)$ is specified by the following data:
		\begin{itemize}
			\item A nested tree $\mathcal{S}$ with $n$-leaves.
			\item For any vertex $S\in\mathcal{S}$ an unbased cactus $x_S\in\mathcal{C}_{\abs{S}}/S^1$.
			\item For any vertex $S\in\mathcal{S}-R$ a parameter $t_S\in[0,1]$. We will call it the \textbf{radial parameter} and we think of it as a decoration of the unique edge going out the vertex $S$.
			\item For any vertex $S\in\mathcal{S}-R$ a parameter $\alpha_S\in S^1$. We will call it the \textbf{angular parameter}. Note that if $t_S=0$ the angular parameter $\alpha_S$ is superfluous, so we will omit it.
			
		\end{itemize}
		We will call a point in $N_n^{\sigma}(\mathcal{C}/S^1)$ a \textbf{nested cactus}.
	\end{defn}
	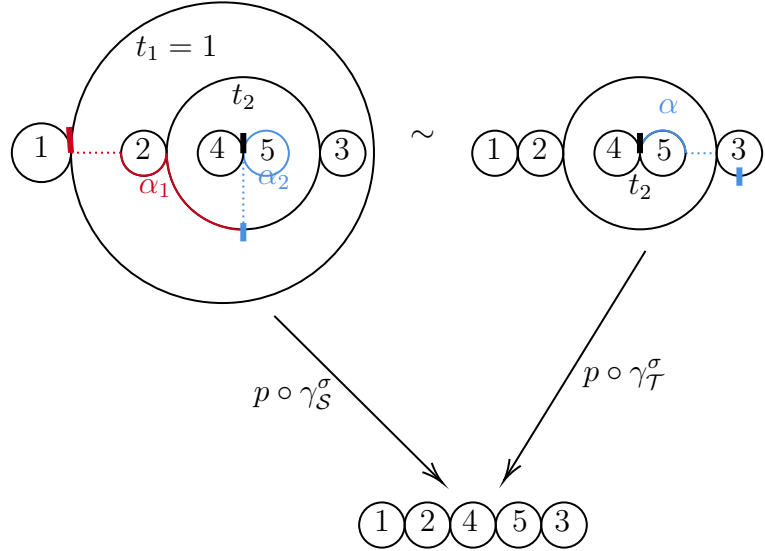
\begin{figure}
		\centering

		\tikzset{every picture/.style={line width=0.75pt}} 
		
		\begin{tikzpicture}[x=0.75pt,y=0.75pt,yscale=-1,xscale=1]
			
			\draw [color={rgb, 255:red, 74; green, 144; blue, 226 }  ,draw opacity=1 ] [dash pattern={on 0.75pt off 1.5pt}]  (391.95,140.89) -- (407.5,140.89) ;
			\draw   (55.68,140.61) .. controls (55.68,132.45) and (62.29,125.84) .. (70.45,125.84) .. controls (78.61,125.84) and (85.23,132.45) .. (85.23,140.61) .. controls (85.23,148.77) and (78.61,155.39) .. (70.45,155.39) .. controls (62.29,155.39) and (55.68,148.77) .. (55.68,140.61) -- cycle ;
			\draw   (85.23,140.61) .. controls (85.23,98.85) and (119.08,65) .. (160.84,65) .. controls (202.6,65) and (236.45,98.85) .. (236.45,140.61) .. controls (236.45,182.37) and (202.6,216.23) .. (160.84,216.23) .. controls (119.08,216.23) and (85.23,182.37) .. (85.23,140.61) -- cycle ;
			\draw   (110.23,140.89) .. controls (110.23,134.61) and (115.31,129.52) .. (121.59,129.52) .. controls (127.86,129.52) and (132.95,134.61) .. (132.95,140.89) .. controls (132.95,147.16) and (127.86,152.25) .. (121.59,152.25) .. controls (115.31,152.25) and (110.23,147.16) .. (110.23,140.89) -- cycle ;
			\draw   (132.95,140.89) .. controls (132.95,119.75) and (150.09,102.61) .. (171.23,102.61) .. controls (192.36,102.61) and (209.5,119.75) .. (209.5,140.89) .. controls (209.5,162.03) and (192.36,179.16) .. (171.23,179.16) .. controls (150.09,179.16) and (132.95,162.03) .. (132.95,140.89) -- cycle ;
			\draw   (209.5,140.89) .. controls (209.5,134.61) and (214.59,129.52) .. (220.86,129.52) .. controls (227.14,129.52) and (232.23,134.61) .. (232.23,140.89) .. controls (232.23,147.16) and (227.14,152.25) .. (220.86,152.25) .. controls (214.59,152.25) and (209.5,147.16) .. (209.5,140.89) -- cycle ;
			\draw [color={rgb, 255:red, 208; green, 2; blue, 27 }  ,draw opacity=1 ][line width=2.25]    (84.23,127.68) -- (85.23,140.61) ;
			\draw [color={rgb, 255:red, 208; green, 2; blue, 27 }  ,draw opacity=1 ] [dash pattern={on 0.75pt off 1.5pt}]  (85.23,140.61) -- (111.23,140.61) ;
			\draw [color={rgb, 255:red, 74; green, 144; blue, 226 }  ,draw opacity=1 ] [dash pattern={on 0.75pt off 1.5pt}]  (171.23,179.16) -- (171.23,140.89) ;
			\draw  [color={rgb, 255:red, 74; green, 144; blue, 226 }  ,draw opacity=1 ] (171.23,140.89) .. controls (171.23,134.61) and (176.31,129.52) .. (182.59,129.52) .. controls (188.86,129.52) and (193.95,134.61) .. (193.95,140.89) .. controls (193.95,147.16) and (188.86,152.25) .. (182.59,152.25) .. controls (176.31,152.25) and (171.23,147.16) .. (171.23,140.89) -- cycle ;
			\draw   (148.5,140.89) .. controls (148.5,134.61) and (153.59,129.52) .. (159.86,129.52) .. controls (166.14,129.52) and (171.23,134.61) .. (171.23,140.89) .. controls (171.23,147.16) and (166.14,152.25) .. (159.86,152.25) .. controls (153.59,152.25) and (148.5,147.16) .. (148.5,140.89) -- cycle ;
			\draw [color={rgb, 255:red, 208; green, 2; blue, 27 }  ,draw opacity=1 ]   (110.23,140.89) .. controls (110.23,155.32) and (131.76,156.51) .. (132.95,140.89) ;
			\draw [color={rgb, 255:red, 208; green, 2; blue, 27 }  ,draw opacity=1 ]   (132.95,140.89) .. controls (134.15,173.26) and (164.05,180.43) .. (171.23,179.16) ;
			\draw [color={rgb, 255:red, 74; green, 144; blue, 226 }  ,draw opacity=1 ][line width=2.25]    (171.23,175.96) -- (171.23,185.22) ;
			\draw [line width=2.25]    (171.23,130.68) -- (171.23,140.89) ;
			\draw   (308.23,140.89) .. controls (308.23,134.61) and (313.31,129.52) .. (319.59,129.52) .. controls (325.86,129.52) and (330.95,134.61) .. (330.95,140.89) .. controls (330.95,147.16) and (325.86,152.25) .. (319.59,152.25) .. controls (313.31,152.25) and (308.23,147.16) .. (308.23,140.89) -- cycle ;
			\draw   (330.95,140.89) .. controls (330.95,119.75) and (348.09,102.61) .. (369.23,102.61) .. controls (390.36,102.61) and (407.5,119.75) .. (407.5,140.89) .. controls (407.5,162.03) and (390.36,179.16) .. (369.23,179.16) .. controls (348.09,179.16) and (330.95,162.03) .. (330.95,140.89) -- cycle ;
			\draw   (407.5,140.89) .. controls (407.5,134.61) and (412.59,129.52) .. (418.86,129.52) .. controls (425.14,129.52) and (430.23,134.61) .. (430.23,140.89) .. controls (430.23,147.16) and (425.14,152.25) .. (418.86,152.25) .. controls (412.59,152.25) and (407.5,147.16) .. (407.5,140.89) -- cycle ;
			\draw   (369.23,140.89) .. controls (369.23,134.61) and (374.31,129.52) .. (380.59,129.52) .. controls (386.86,129.52) and (391.95,134.61) .. (391.95,140.89) .. controls (391.95,147.16) and (386.86,152.25) .. (380.59,152.25) .. controls (374.31,152.25) and (369.23,147.16) .. (369.23,140.89) -- cycle ;
			\draw   (346.5,140.89) .. controls (346.5,134.61) and (351.59,129.52) .. (357.86,129.52) .. controls (364.14,129.52) and (369.23,134.61) .. (369.23,140.89) .. controls (369.23,147.16) and (364.14,152.25) .. (357.86,152.25) .. controls (351.59,152.25) and (346.5,147.16) .. (346.5,140.89) -- cycle ;
			\draw [color={rgb, 255:red, 74; green, 144; blue, 226 }  ,draw opacity=1 ][line width=2.25]    (418.86,147.99) -- (418.86,157.25) ;
			\draw [color={rgb, 255:red, 74; green, 144; blue, 226 }  ,draw opacity=1 ][line width=0.75]    (391.95,140.89) .. controls (392.23,127.68) and (371.23,123.68) .. (369.23,140.89) ;
			\draw [line width=2.25]    (369.23,130.68) -- (369.23,140.89) ;
			\draw   (285.5,140.89) .. controls (285.5,134.61) and (290.59,129.52) .. (296.86,129.52) .. controls (303.14,129.52) and (308.23,134.61) .. (308.23,140.89) .. controls (308.23,147.16) and (303.14,152.25) .. (296.86,152.25) .. controls (290.59,152.25) and (285.5,147.16) .. (285.5,140.89) -- cycle ;
			\draw    (186.68,222.68) -- (269.26,305.26) ;
			\draw [shift={(270.68,306.68)}, rotate = 225] [color={rgb, 255:red, 0; green, 0; blue, 0 }  ][line width=0.75]    (10.93,-3.29) .. controls (6.95,-1.4) and (3.31,-0.3) .. (0,0) .. controls (3.31,0.3) and (6.95,1.4) .. (10.93,3.29)   ;
			\draw    (372,190) -- (302.27,303.97) ;
			\draw [shift={(301.23,305.68)}, rotate = 301.46] [color={rgb, 255:red, 0; green, 0; blue, 0 }  ][line width=0.75]    (10.93,-3.29) .. controls (6.95,-1.4) and (3.31,-0.3) .. (0,0) .. controls (3.31,0.3) and (6.95,1.4) .. (10.93,3.29)   ;
			\draw   (251.77,327.89) .. controls (251.77,321.61) and (256.86,316.52) .. (263.14,316.52) .. controls (269.41,316.52) and (274.5,321.61) .. (274.5,327.89) .. controls (274.5,334.16) and (269.41,339.25) .. (263.14,339.25) .. controls (256.86,339.25) and (251.77,334.16) .. (251.77,327.89) -- cycle ;
			\draw   (319.95,327.89) .. controls (319.95,321.61) and (325.04,316.52) .. (331.31,316.52) .. controls (337.59,316.52) and (342.68,321.61) .. (342.68,327.89) .. controls (342.68,334.16) and (337.59,339.25) .. (331.31,339.25) .. controls (325.04,339.25) and (319.95,334.16) .. (319.95,327.89) -- cycle ;
			\draw   (297.23,327.89) .. controls (297.23,321.61) and (302.31,316.52) .. (308.59,316.52) .. controls (314.86,316.52) and (319.95,321.61) .. (319.95,327.89) .. controls (319.95,334.16) and (314.86,339.25) .. (308.59,339.25) .. controls (302.31,339.25) and (297.23,334.16) .. (297.23,327.89) -- cycle ;
			\draw   (274.5,327.89) .. controls (274.5,321.61) and (279.59,316.52) .. (285.86,316.52) .. controls (292.14,316.52) and (297.23,321.61) .. (297.23,327.89) .. controls (297.23,334.16) and (292.14,339.25) .. (285.86,339.25) .. controls (279.59,339.25) and (274.5,334.16) .. (274.5,327.89) -- cycle ;
			\draw   (229.05,327.89) .. controls (229.05,321.61) and (234.14,316.52) .. (240.41,316.52) .. controls (246.69,316.52) and (251.77,321.61) .. (251.77,327.89) .. controls (251.77,334.16) and (246.69,339.25) .. (240.41,339.25) .. controls (234.14,339.25) and (229.05,334.16) .. (229.05,327.89) -- cycle ;
			
			\draw (65,131) node [anchor=north west][inner sep=0.75pt]   [align=left] {$\displaystyle 1$};
			\draw (115.82,132.06) node [anchor=north west][inner sep=0.75pt]   [align=left] {$\displaystyle 2$};
			\draw (215.09,132.06) node [anchor=north west][inner sep=0.75pt]   [align=left] {$\displaystyle 3$};
			\draw (152.9,131.25) node [anchor=north west][inner sep=0.75pt]   [align=left] {$\displaystyle 4$};
			\draw (178.01,132.45) node [anchor=north west][inner sep=0.75pt]   [align=left] {$\displaystyle 5$};
			\draw (116,80) node [anchor=north west][inner sep=0.75pt]   [align=left] {$\displaystyle t_{1} =1$};
			\draw (164,105) node [anchor=north west][inner sep=0.75pt]   [align=left] {$\displaystyle t_{2}{}$};
			\draw (177,148) node [anchor=north west][inner sep=0.75pt]  [color={rgb, 255:red, 74; green, 144; blue, 226 }  ,opacity=1 ] [align=left] {$\displaystyle \alpha _{2}{}$};
			\draw (117.59,152.25) node [anchor=north west][inner sep=0.75pt]  [color={rgb, 255:red, 208; green, 2; blue, 27 }  ,opacity=1 ] [align=left] {$\displaystyle \alpha _{1}{}$};
			\draw (291,132) node [anchor=north west][inner sep=0.75pt]   [align=left] {$\displaystyle 1$};
			\draw (313.82,132.06) node [anchor=north west][inner sep=0.75pt]   [align=left] {$\displaystyle 2$};
			\draw (413.09,132.06) node [anchor=north west][inner sep=0.75pt]   [align=left] {$\displaystyle 3$};
			\draw (350.9,131.25) node [anchor=north west][inner sep=0.75pt]   [align=left] {$\displaystyle 4$};
			\draw (376.01,132.45) node [anchor=north west][inner sep=0.75pt]   [align=left] {$\displaystyle 5$};
			\draw (362,150) node [anchor=north west][inner sep=0.75pt]   [align=left] {$\displaystyle t_{2}{}$};
			\draw (377,112) node [anchor=north west][inner sep=0.75pt]  [color={rgb, 255:red, 74; green, 144; blue, 226 }  ,opacity=1 ] [align=left] {$\displaystyle \alpha $};
			\draw (253,128) node [anchor=north west][inner sep=0.75pt]   [align=left] {$\displaystyle \sim $};
			\draw (235,319) node [anchor=north west][inner sep=0.75pt]   [align=left] {$\displaystyle 1$};
			\draw (257.82,319) node [anchor=north west][inner sep=0.75pt]   [align=left] {$\displaystyle 2$};
			\draw (325.09,319) node [anchor=north west][inner sep=0.75pt]   [align=left] {$\displaystyle 3$};
			\draw (278.9,319) node [anchor=north west][inner sep=0.75pt]   [align=left] {$\displaystyle 4$};
			\draw (304.01,319) node [anchor=north west][inner sep=0.75pt]   [align=left] {$\displaystyle 5$};
			\draw (175,253) node [anchor=north west][inner sep=0.75pt]   [align=left] {$\displaystyle p\circ \gamma _{\mathcal{S}}^{\sigma }$};
			\draw (339.61,242.84) node [anchor=north west][inner sep=0.75pt]   [align=left] {$\displaystyle p\circ \gamma _{\mathcal{T}}^{\sigma }$};

		\end{tikzpicture}
		
		\caption{In this picture we see two nested cactus that are identified in $N_5^{\sigma}(\mathcal{C}/S^1)$: the nested tree underlying the leftmost (resp. rightmost) nested cactus is $\mathcal{S}=(1(23(45)))$ (resp. $\mathcal{T}=(123(45))$). $t_1,t_2\in(0,1]$ are radial parameters, $\alpha,\alpha_1,\alpha_2\in S^1$ are angular parameters. The base points depicted on the cacti are those obtained using the sections $\sigma_k:\mathcal{C}_k/S^1\to\mathcal{C}_k$. In the leftmost cactus $t_1,\alpha_1$ and $t_2,\alpha_2$ are the parameters associated respectively to $\{2,3,4,5\}$ and $\{4,5\}$. Similarly, in the right cactus $t_2$ and $\alpha$ are the parameters associated to the vertex $\{4,5\}$. If we compute $p\circ\gamma^{\sigma}_{\mathcal{S}}$ on the left nested cactus and $p\circ\gamma^{\sigma}_{\mathcal{T}}$ on the right nested cactus we end up with the same unbased cactus, which is depicted below.  }
		\label{fig:esempio di nested cactus che vengono identificati}
	\end{figure}
	\subsubsection{Cell decomposition}\label{subsec: CW decomposition of nested cacti}
	The space of nested cacti $N_n^{\sigma}(\mathcal{C}/S^1)$ has a natural cell structure that we describe next. A cell $\tau$ is determined by:
	\begin{enumerate}
		\item A nested tree $\mathcal{S}$ with $n$-leaves.
		\item A cell $C_{S}\subseteq \mathcal{C}_{\abs{S}}/S^1$ for any vertex $S\in\mathcal{S}$.
	\end{enumerate}
	and contains all the nested cacti $(x_R,(t_S,\alpha_S,x_S)_{S\in\mathcal{S}-R}))$ such that $x_S\in C_S$ and $t_S<1$ for all $S\in\mathcal{S}$. This subspace of $N_n^{\sigma}(\mathcal{C}/S^1)$ is homeomorphic to
	\[
	C_{R}\times \prod_{S\in\mathcal{S}-R}(\mathring{D}_2\times C_S)
	\]
	The dimension of such a cell is
	\[
	dim(\tau)= 2E+ \sum_{{S\in \mathcal{S}}}dim(C_S)
	\]
	where $E$ is the number of internal edges of $\mathcal{S}$. The boundary of $\tau$ is the union of cells with a similar description that are  obtained by an iteration of the following moves (for a concrete example see Figure \ref{bordo di una cella }):
	\begin{figure}
		\centering
		\includegraphics[width=13 cm]{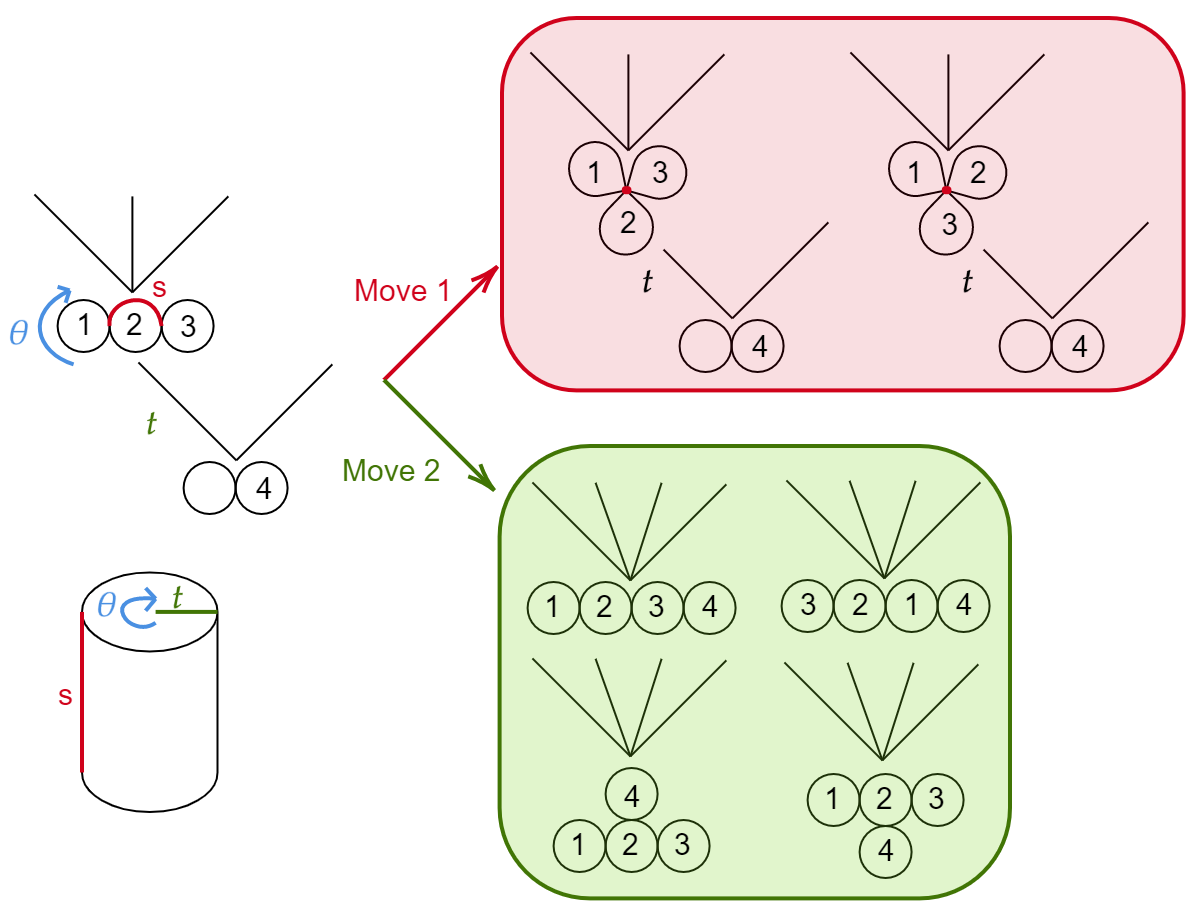}
		\caption{The cell in this picture has dimension three (it is a cylinder). If we compute the boundary of the cactus with three lobes (move $1$) we obtain the cells in the red rectangle, corresponding to the base and the top of the cylinder. If we contract the internal edge (move $2$) we obtain the cells in the green rectangle.  }
		\label{bordo di una cella }
	\end{figure}
	\begin{description}
		\item[Move 1 (cacti boundary)] for a fixed vertex $S\in\mathcal{S}$ replace the cell $C_S$ with a cell $C'_S\subseteq \partial C_S$.
		\item[Move 2 (edge contraction)] replace an edge $(S,T)$ of the nested tree $\mathcal{S}$ with a vertex (i.e. contract the edge) labelled by a cell of $\mathcal{C}_{\abs{S}+\abs{T}-1}/S^1$ which is obtained by inserting the cactus $C_S$ in the lobe of $C_T$ corresponding to the edge $(S,T)$.    
	\end{description}
	
	\
	When we want to highlight the combinatorial data defining a cell $\tau$ we will write $\tau=(\mathcal{S},(C_S)_{S\in\mathcal{S}})$.
	\begin{prop}
		The space $N_n^{\sigma}(\mathcal{C}/S^1)$ is a regular CW-complex.
	\end{prop}
	\begin{proof}
		The cell decomposition just described is regular, because $\mathcal{C}_k/S^1$ is a regular CW-complex and the maps $p\circ\gamma^{\sigma}_{\mathcal{S}}$ are embeddings.
	\end{proof}
	
	\subsection{Nested cacti vs labelled trees}\label{subsec:nested cacti vs labelled trees}
	In this paragraph we show that the space of labelled treed $Tr_n$ is homeomorphic to a subspace of $N_n^{\sigma}(\mathcal{C}/S^1)$. First we remind a result from \cite{Salvatore}:
	\begin{prop}[ \cite{Salvatore}]\label{prop:def retraction }
		There is a deformation retraction $\beta_n:Tr_n\to \mathcal{C}_n/S^1$.
	\end{prop}
	\begin{lem}\label{lem:omeo tra nested tree e alberi b/w}
		Consider the subspace of $N_n^{\sigma}(\mathcal{C}/S^1)$ defined as follows: 
		\[
		N_n^{\sigma}(\mathcal{C}/S^1)^{>0}\coloneqq \{(x_R,(t_S,\alpha_S,x_S)_{S\in\mathcal{S}-R})\in N_n^{\sigma}(\mathcal{C}/S^1)\mid t_S>0 \text{ for any } S\in\mathcal{S}-R\}
		\]
		Then the space of labelled trees $Tr_n$ is homeomorphic to $N_n^{\sigma}(\mathcal{C}/S^1)^{>0}$.
	\end{lem}
	\begin{proof}
		Fix a point $x\coloneqq(T,f:B\to(0,1],\{g_w:E_w\to(0,1)\}_{w\in W})$ in $Tr_n$. Here $T$ is an admissible tree with set of black (resp. white) vertices $B$ (resp. $W$), $f$ are the parameters of the black vertices, $g_w$ are the parameters associated to a white vertex $w\in W$ (for more details see Section \ref{sec: combinatorial models for the open moduli space}). We want to associate to $x$ a point $f(x)\in N_n^{\sigma}(\mathcal{C}/S^1)^{>0}$, so we need to specify a nested tree $\mathcal{S}$ with $n$-leaves, a cactus $x_S\in\mathcal{C}_{\abs{S}}/S^1$ for any $S\in\mathcal{S}$ and a point $(t_S,\alpha_S)\in D_2$ for any $S\in\mathcal{S}-R$. Let us do this step by step:
		\begin{description}
			\item[Nested tree] we obtain the nested tree $\mathcal{S}$ by the following procedure:
			\begin{enumerate}
				\item Remove from $T$ all black vertices with label $1$ together with the edges that are incident to them. In this way we obtain a forest. If $T_i$ is a connected component of this forest, let $S_i$ be the set of labels of the white vertices of $T_i$.
				\item For each connected component $T_i$ of this forest which contains at least one black vertex, let $\lambda_i$ the maximum label of a black vertex in that component. Now divide by $\lambda_i$ the labels of the black vertices in $T_i$ and iterate the procedure in each component. The procedure stops when there are no black vertices left.
			\end{enumerate}
			At the end of this procedure we get a bunch of subsets of $\{1,\dots,n\}$ (the $S_i$ defined at the first step of the algorithm) which form a nested tree $\mathcal{S}$ with $n$-leaves. See Figure \ref{fig: associo a un labelled tree un nested tree} for an example.
			\item [Radial parameters] any $S\in\mathcal{S}$ corresponds to one of the trees $T_i$ obtained from $T$ by the previous procedure. Then we define $t_S\in(0,1]$ to be $\lambda_i$, the maximum label of a black vertex of $T_i$. For an explicit example see Figure \ref{fig: associo a un labelled tree un nested tree}. 
			\item [Decorative cacti and angular parameters] consider the deformation retraction $\beta_n:Tr_n\to \mathcal{C}_n/S^1$ and the embedding
			\[    p\circ\gamma^{\sigma}_{\mathcal{S}}:\mathcal{C}_{\abs{R}}/S^1\times \prod_{S\in\mathcal{S}-R}\left(S^1\times \mathcal{C}_{\abs{S}}/S^1\right)\to \mathcal{C}_n/S^1
			\]
			We claim that $\beta_n(x)\in Im(p\circ \gamma^{\sigma}_{\mathcal{S}})$: this is true because for any $S\in\mathcal{S}$ the union of the lobes of $\beta_n(x)$ labelled by $S$ is again a cactus. Therefore we take $(p\circ\gamma^{\sigma}_{\mathcal{S}})^{-1}(\beta_n(x))$ and we get a decorative cactus $x_S\in\mathcal{C}_{\abs{S}}/S^1$ for any $S\in\mathcal{S}$, together with an angular parameter $\alpha_S\in S^1$ when $S$ is not the root.
		\end{description}   
		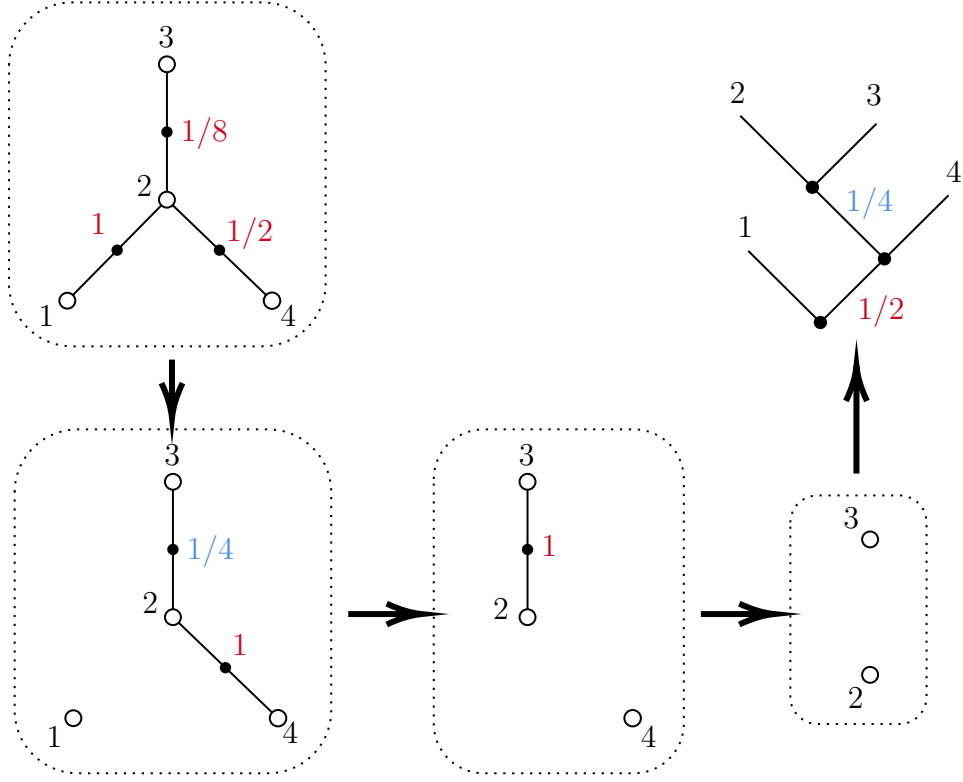
\begin{figure}
			\centering

		\tikzset{every picture/.style={line width=0.75pt}} 
		
		\begin{tikzpicture}[x=0.75pt,y=0.75pt,yscale=-1,xscale=1]
			
			\draw  [fill={rgb, 255:red, 255; green, 255; blue, 255 }  ,fill opacity=1 ] (122.5,58.46) .. controls (122.5,56.26) and (124.28,54.48) .. (126.48,54.48) .. controls (128.68,54.48) and (130.46,56.26) .. (130.46,58.46) .. controls (130.46,60.66) and (128.68,62.44) .. (126.48,62.44) .. controls (124.28,62.44) and (122.5,60.66) .. (122.5,58.46) -- cycle ;
			\draw  [fill={rgb, 255:red, 0; green, 0; blue, 0 }  ,fill opacity=1 ] (124.15,92.52) .. controls (124.15,91.23) and (125.19,90.18) .. (126.48,90.18) .. controls (127.77,90.18) and (128.82,91.23) .. (128.82,92.52) .. controls (128.82,93.81) and (127.77,94.85) .. (126.48,94.85) .. controls (125.19,94.85) and (124.15,93.81) .. (124.15,92.52) -- cycle ;
			\draw    (126.48,62.44) -- (126.48,90.18) ;
			\draw    (126.48,94.85) -- (126.48,122.59) ;
			\draw  [fill={rgb, 255:red, 0; green, 0; blue, 0 }  ,fill opacity=1 ] (150.38,151.98) .. controls (150.38,150.69) and (151.42,149.64) .. (152.71,149.64) .. controls (154,149.64) and (155.05,150.69) .. (155.05,151.98) .. controls (155.05,153.27) and (154,154.31) .. (152.71,154.31) .. controls (151.42,154.31) and (150.38,153.27) .. (150.38,151.98) -- cycle ;
			\draw    (126.48,126.57) -- (101.63,151.98) ;
			\draw    (126.48,126.57) -- (152.71,151.98) ;
			\draw  [fill={rgb, 255:red, 0; green, 0; blue, 0 }  ,fill opacity=1 ] (99.29,151.98) .. controls (99.29,150.69) and (100.34,149.64) .. (101.63,149.64) .. controls (102.92,149.64) and (103.96,150.69) .. (103.96,151.98) .. controls (103.96,153.27) and (102.92,154.31) .. (101.63,154.31) .. controls (100.34,154.31) and (99.29,153.27) .. (99.29,151.98) -- cycle ;
			\draw  [fill={rgb, 255:red, 255; green, 255; blue, 255 }  ,fill opacity=1 ] (122.5,126.57) .. controls (122.5,124.37) and (124.28,122.59) .. (126.48,122.59) .. controls (128.68,122.59) and (130.46,124.37) .. (130.46,126.57) .. controls (130.46,128.77) and (128.68,130.56) .. (126.48,130.56) .. controls (124.28,130.56) and (122.5,128.77) .. (122.5,126.57) -- cycle ;
			\draw    (152.71,151.98) -- (178.94,177.38) ;
			\draw    (101.63,151.98) -- (76.77,177.38) ;
			\draw  [fill={rgb, 255:red, 255; green, 255; blue, 255 }  ,fill opacity=1 ] (72.79,177.38) .. controls (72.79,175.18) and (74.57,173.4) .. (76.77,173.4) .. controls (78.97,173.4) and (80.75,175.18) .. (80.75,177.38) .. controls (80.75,179.58) and (78.97,181.37) .. (76.77,181.37) .. controls (74.57,181.37) and (72.79,179.58) .. (72.79,177.38) -- cycle ;
			\draw  [fill={rgb, 255:red, 255; green, 255; blue, 255 }  ,fill opacity=1 ] (174.82,177.38) .. controls (174.82,175.11) and (176.66,173.26) .. (178.94,173.26) .. controls (181.21,173.26) and (183.06,175.11) .. (183.06,177.38) .. controls (183.06,179.66) and (181.21,181.5) .. (178.94,181.5) .. controls (176.66,181.5) and (174.82,179.66) .. (174.82,177.38) -- cycle ;
			\draw  [fill={rgb, 255:red, 255; green, 255; blue, 255 }  ,fill opacity=1 ] (125.5,268.46) .. controls (125.5,266.26) and (127.28,264.48) .. (129.48,264.48) .. controls (131.68,264.48) and (133.46,266.26) .. (133.46,268.46) .. controls (133.46,270.66) and (131.68,272.44) .. (129.48,272.44) .. controls (127.28,272.44) and (125.5,270.66) .. (125.5,268.46) -- cycle ;
			\draw  [fill={rgb, 255:red, 0; green, 0; blue, 0 }  ,fill opacity=1 ] (127.15,302.52) .. controls (127.15,301.23) and (128.19,300.18) .. (129.48,300.18) .. controls (130.77,300.18) and (131.82,301.23) .. (131.82,302.52) .. controls (131.82,303.81) and (130.77,304.85) .. (129.48,304.85) .. controls (128.19,304.85) and (127.15,303.81) .. (127.15,302.52) -- cycle ;
			\draw    (129.48,272.44) -- (129.48,300.18) ;
			\draw    (129.48,304.85) -- (129.48,332.59) ;
			\draw  [fill={rgb, 255:red, 0; green, 0; blue, 0 }  ,fill opacity=1 ] (153.38,361.98) .. controls (153.38,360.69) and (154.42,359.64) .. (155.71,359.64) .. controls (157,359.64) and (158.05,360.69) .. (158.05,361.98) .. controls (158.05,363.27) and (157,364.31) .. (155.71,364.31) .. controls (154.42,364.31) and (153.38,363.27) .. (153.38,361.98) -- cycle ;
			\draw    (129.48,336.57) -- (155.71,361.98) ;
			\draw  [fill={rgb, 255:red, 255; green, 255; blue, 255 }  ,fill opacity=1 ] (125.5,336.57) .. controls (125.5,334.37) and (127.28,332.59) .. (129.48,332.59) .. controls (131.68,332.59) and (133.46,334.37) .. (133.46,336.57) .. controls (133.46,338.77) and (131.68,340.56) .. (129.48,340.56) .. controls (127.28,340.56) and (125.5,338.77) .. (125.5,336.57) -- cycle ;
			\draw    (155.71,361.98) -- (181.94,387.38) ;
			\draw  [fill={rgb, 255:red, 255; green, 255; blue, 255 }  ,fill opacity=1 ] (75.79,387.38) .. controls (75.79,385.18) and (77.57,383.4) .. (79.77,383.4) .. controls (81.97,383.4) and (83.75,385.18) .. (83.75,387.38) .. controls (83.75,389.58) and (81.97,391.37) .. (79.77,391.37) .. controls (77.57,391.37) and (75.79,389.58) .. (75.79,387.38) -- cycle ;
			\draw  [fill={rgb, 255:red, 255; green, 255; blue, 255 }  ,fill opacity=1 ] (177.82,387.38) .. controls (177.82,385.11) and (179.66,383.26) .. (181.94,383.26) .. controls (184.21,383.26) and (186.06,385.11) .. (186.06,387.38) .. controls (186.06,389.66) and (184.21,391.5) .. (181.94,391.5) .. controls (179.66,391.5) and (177.82,389.66) .. (177.82,387.38) -- cycle ;
			\draw  [fill={rgb, 255:red, 255; green, 255; blue, 255 }  ,fill opacity=1 ] (302.5,268.46) .. controls (302.5,266.26) and (304.28,264.48) .. (306.48,264.48) .. controls (308.68,264.48) and (310.46,266.26) .. (310.46,268.46) .. controls (310.46,270.66) and (308.68,272.44) .. (306.48,272.44) .. controls (304.28,272.44) and (302.5,270.66) .. (302.5,268.46) -- cycle ;
			\draw  [fill={rgb, 255:red, 0; green, 0; blue, 0 }  ,fill opacity=1 ] (304.15,302.52) .. controls (304.15,301.23) and (305.19,300.18) .. (306.48,300.18) .. controls (307.77,300.18) and (308.82,301.23) .. (308.82,302.52) .. controls (308.82,303.81) and (307.77,304.85) .. (306.48,304.85) .. controls (305.19,304.85) and (304.15,303.81) .. (304.15,302.52) -- cycle ;
			\draw    (306.48,272.44) -- (306.48,300.18) ;
			\draw    (306.48,304.85) -- (306.48,332.59) ;
			\draw  [fill={rgb, 255:red, 255; green, 255; blue, 255 }  ,fill opacity=1 ] (302.5,336.57) .. controls (302.5,334.37) and (304.28,332.59) .. (306.48,332.59) .. controls (308.68,332.59) and (310.46,334.37) .. (310.46,336.57) .. controls (310.46,338.77) and (308.68,340.56) .. (306.48,340.56) .. controls (304.28,340.56) and (302.5,338.77) .. (302.5,336.57) -- cycle ;
			\draw  [fill={rgb, 255:red, 255; green, 255; blue, 255 }  ,fill opacity=1 ] (354.82,387.38) .. controls (354.82,385.11) and (356.66,383.26) .. (358.94,383.26) .. controls (361.21,383.26) and (363.06,385.11) .. (363.06,387.38) .. controls (363.06,389.66) and (361.21,391.5) .. (358.94,391.5) .. controls (356.66,391.5) and (354.82,389.66) .. (354.82,387.38) -- cycle ;
			\draw  [fill={rgb, 255:red, 255; green, 255; blue, 255 }  ,fill opacity=1 ] (473.5,297.46) .. controls (473.5,295.26) and (475.28,293.48) .. (477.48,293.48) .. controls (479.68,293.48) and (481.46,295.26) .. (481.46,297.46) .. controls (481.46,299.66) and (479.68,301.44) .. (477.48,301.44) .. controls (475.28,301.44) and (473.5,299.66) .. (473.5,297.46) -- cycle ;
			\draw  [fill={rgb, 255:red, 255; green, 255; blue, 255 }  ,fill opacity=1 ] (473.5,365.57) .. controls (473.5,363.37) and (475.28,361.59) .. (477.48,361.59) .. controls (479.68,361.59) and (481.46,363.37) .. (481.46,365.57) .. controls (481.46,367.77) and (479.68,369.56) .. (477.48,369.56) .. controls (475.28,369.56) and (473.5,367.77) .. (473.5,365.57) -- cycle ;
			\draw  [dash pattern={on 0.84pt off 2.51pt}] (47.67,58.93) .. controls (47.67,41.48) and (61.81,27.33) .. (79.27,27.33) -- (174.07,27.33) .. controls (191.52,27.33) and (205.67,41.48) .. (205.67,58.93) -- (205.67,168.73) .. controls (205.67,186.19) and (191.52,200.33) .. (174.07,200.33) -- (79.27,200.33) .. controls (61.81,200.33) and (47.67,186.19) .. (47.67,168.73) -- cycle ;
			\draw  [dash pattern={on 0.84pt off 2.51pt}] (50.48,273.69) .. controls (50.48,256.24) and (64.63,242.09) .. (82.08,242.09) -- (176.88,242.09) .. controls (194.33,242.09) and (208.48,256.24) .. (208.48,273.69) -- (208.48,383.49) .. controls (208.48,400.94) and (194.33,415.09) .. (176.88,415.09) -- (82.08,415.09) .. controls (64.63,415.09) and (50.48,400.94) .. (50.48,383.49) -- cycle ;
			\draw  [dash pattern={on 0.84pt off 2.51pt}] (259.67,267.05) .. controls (259.67,253.27) and (270.84,242.09) .. (284.63,242.09) -- (359.52,242.09) .. controls (373.31,242.09) and (384.48,253.27) .. (384.48,267.05) -- (384.48,390.13) .. controls (384.48,403.91) and (373.31,415.09) .. (359.52,415.09) -- (284.63,415.09) .. controls (270.84,415.09) and (259.67,403.91) .. (259.67,390.13) -- cycle ;
			\draw  [dash pattern={on 0.84pt off 2.51pt}] (437.67,288.93) .. controls (437.67,281.42) and (443.76,275.33) .. (451.27,275.33) -- (492.07,275.33) .. controls (499.58,275.33) and (505.67,281.42) .. (505.67,288.93) -- (505.67,376.73) .. controls (505.67,384.24) and (499.58,390.33) .. (492.07,390.33) -- (451.27,390.33) .. controls (443.76,390.33) and (437.67,384.24) .. (437.67,376.73) -- cycle ;
			\draw  [fill={rgb, 255:red, 0; green, 0; blue, 0 }  ,fill opacity=1 ] (445.83,120.33) .. controls (445.83,118.77) and (447.1,117.5) .. (448.67,117.5) .. controls (450.23,117.5) and (451.5,118.77) .. (451.5,120.33) .. controls (451.5,121.9) and (450.23,123.17) .. (448.67,123.17) .. controls (447.1,123.17) and (445.83,121.9) .. (445.83,120.33) -- cycle ;
			\draw    (412.67,84.33) -- (448.67,120.33) ;
			\draw    (480.67,88.33) -- (448.67,120.33) ;
			\draw    (448.67,120.33) -- (484.67,156.33) ;
			\draw    (416.67,152.33) -- (452.67,188.33) ;
			\draw    (516.67,124.33) -- (484.67,156.33) ;
			\draw    (484.67,156.33) -- (452.67,188.33) ;
			\draw  [fill={rgb, 255:red, 0; green, 0; blue, 0 }  ,fill opacity=1 ] (481.83,156.33) .. controls (481.83,154.77) and (483.1,153.5) .. (484.67,153.5) .. controls (486.23,153.5) and (487.5,154.77) .. (487.5,156.33) .. controls (487.5,157.9) and (486.23,159.17) .. (484.67,159.17) .. controls (483.1,159.17) and (481.83,157.9) .. (481.83,156.33) -- cycle ;
			\draw  [fill={rgb, 255:red, 0; green, 0; blue, 0 }  ,fill opacity=1 ] (449.83,188.33) .. controls (449.83,186.77) and (451.1,185.5) .. (452.67,185.5) .. controls (454.23,185.5) and (455.5,186.77) .. (455.5,188.33) .. controls (455.5,189.9) and (454.23,191.17) .. (452.67,191.17) .. controls (451.1,191.17) and (449.83,189.9) .. (449.83,188.33) -- cycle ;
			\draw [line width=2.25]    (217,335) -- (249.67,335) ;
			\draw [shift={(253.67,335)}, rotate = 180] [color={rgb, 255:red, 0; green, 0; blue, 0 }  ][line width=2.25]    (17.49,-5.26) .. controls (11.12,-2.23) and (5.29,-0.48) .. (0,0) .. controls (5.29,0.48) and (11.12,2.23) .. (17.49,5.26)   ;
			\draw [line width=2.25]    (393,335) -- (425.67,335) ;
			\draw [shift={(429.67,335)}, rotate = 180] [color={rgb, 255:red, 0; green, 0; blue, 0 }  ][line width=2.25]    (17.49,-5.26) .. controls (11.12,-2.23) and (5.29,-0.48) .. (0,0) .. controls (5.29,0.48) and (11.12,2.23) .. (17.49,5.26)   ;
			\draw [line width=2.25]    (129,207) -- (129,232.33) ;
			\draw [shift={(129,236.33)}, rotate = 270] [color={rgb, 255:red, 0; green, 0; blue, 0 }  ][line width=2.25]    (17.49,-5.26) .. controls (11.12,-2.23) and (5.29,-0.48) .. (0,0) .. controls (5.29,0.48) and (11.12,2.23) .. (17.49,5.26)   ;
			\draw [line width=2.25]    (470.67,264.33) -- (470.67,214.33) ;
			\draw [shift={(470.67,210.33)}, rotate = 90] [color={rgb, 255:red, 0; green, 0; blue, 0 }  ][line width=2.25]    (17.49,-5.26) .. controls (11.12,-2.23) and (5.29,-0.48) .. (0,0) .. controls (5.29,0.48) and (11.12,2.23) .. (17.49,5.26)   ;
			
			\draw (60.86,178.31) node [anchor=north west][inner sep=0.75pt]   [align=left] {$\displaystyle 1$};
			\draw (109.82,112.69) node [anchor=north west][inner sep=0.75pt]   [align=left] {$\displaystyle 2$};
			\draw (120.71,38.06) node [anchor=north west][inner sep=0.75pt]   [align=left] {$\displaystyle 3$};
			\draw (181.62,178.27) node [anchor=north west][inner sep=0.75pt]   [align=left] {$\displaystyle 4$};
			\draw (86.79,132.27) node [anchor=north west][inner sep=0.75pt]   [align=left] {$\displaystyle \textcolor[rgb]{0.82,0.01,0.11}{1}$};
			\draw (131.95,83.76) node [anchor=north west][inner sep=0.75pt]   [align=left] {$\displaystyle \textcolor[rgb]{0.82,0.01,0.11}{1/8}$};
			\draw (154.47,133.62) node [anchor=north west][inner sep=0.75pt]   [align=left] {$\displaystyle \textcolor[rgb]{0.82,0.01,0.11}{1/2}$};
			\draw (64.86,390.31) node [anchor=north west][inner sep=0.75pt]   [align=left] {$\displaystyle 1$};
			\draw (112.82,322.69) node [anchor=north west][inner sep=0.75pt]   [align=left] {$\displaystyle 2$};
			\draw (123.71,248.06) node [anchor=north west][inner sep=0.75pt]   [align=left] {$\displaystyle 3$};
			\draw (182.62,388.27) node [anchor=north west][inner sep=0.75pt]   [align=left] {$\displaystyle 4$};
			\draw (134.95,293.76) node [anchor=north west][inner sep=0.75pt]   [align=left] {$\displaystyle \textcolor[rgb]{0.29,0.56,0.89}{1/4}$};
			\draw (157.47,343.62) node [anchor=north west][inner sep=0.75pt]   [align=left] {$\displaystyle \textcolor[rgb]{0.82,0.01,0.11}{1}$};
			\draw (287.82,325.69) node [anchor=north west][inner sep=0.75pt]   [align=left] {$\displaystyle 2$};
			\draw (300.71,248.06) node [anchor=north west][inner sep=0.75pt]   [align=left] {$\displaystyle 3$};
			\draw (361.62,390.27) node [anchor=north west][inner sep=0.75pt]   [align=left] {$\displaystyle 4$};
			\draw (311.95,293.76) node [anchor=north west][inner sep=0.75pt]   [align=left] {$\displaystyle \textcolor[rgb]{0.82,0.01,0.11}{1}$};
			\draw (464.82,370.69) node [anchor=north west][inner sep=0.75pt]   [align=left] {$\displaystyle 2$};
			\draw (462.71,280.06) node [anchor=north west][inner sep=0.75pt]   [align=left] {$\displaystyle 3$};
			\draw (406,66) node [anchor=north west][inner sep=0.75pt]   [align=left] {$\displaystyle 2$};
			\draw (474,67) node [anchor=north west][inner sep=0.75pt]   [align=left] {$\displaystyle 3$};
			\draw (514,106) node [anchor=north west][inner sep=0.75pt]   [align=left] {$\displaystyle 4$};
			\draw (409.86,133.31) node [anchor=north west][inner sep=0.75pt]   [align=left] {$\displaystyle 1$};
			\draw (469.67,173.33) node [anchor=north west][inner sep=0.75pt]   [align=left] {$\displaystyle \textcolor[rgb]{0.82,0.01,0.11}{1/2}$};
			\draw (463.95,118.76) node [anchor=north west][inner sep=0.75pt]   [align=left] {$\displaystyle \textcolor[rgb]{0.29,0.56,0.89}{1/4}$};

		\end{tikzpicture}
		
			\caption{An explicit example of the procedure that assigns a nested tree to a labelled tree.}
			\label{fig: associo a un labelled tree un nested tree}
		\end{figure}  
		To sum up, we have constructed a continuous map $f:Tr_n\to N_n^{\sigma}(\mathcal{C}/S^1)^{>0}$. We show that it is a homeomorphism by exhibiting an explicit continuous inverse. If $x\coloneqq(x_R,(t_S,\alpha_S,x_S)_{S\in\mathcal{S}-R})\in N_n^{\sigma}(\mathcal{C}/S^1)^{>0}$, we want to construct $f^{-1}(x)$, so we need to specify an admissible tree $T$, a function $f:B\to(0,1]$ (parameters of the black vertices) and for any white vertex $w\in W$ a map $g_w:E_w\to (0,1)$ (parameters of the white vertices): the admissible tree $T$ will have $n$ white vertices and a black vertex for each intersection point of the lobes of a cactus $x_S$, $S\in\mathcal{S}$. If $b$ is a black vertex corresponding to an intersection point of the lobes of some cactus $x_S$, we put
		\[
		f(b)\coloneqq \prod_{\substack{T\in \mathcal{S}-R\\T\supseteq S}} t_{T}
		\]
		Now we define the edges of $T$ and the parameters of the white vertices inductively: start form the vertices $S\in\mathcal{S}$ that are furthest from the root. Take the decorative cactus $x_S\in\mathcal{C}_{\abs{S}}/S^1$ and consider the pointed cactus $\alpha_S\sigma_{\abs{S}}(x_S)\in\mathcal{C}_{\abs{S}}$. Passing to the dual graph we obtain a labelled tree (with a base point). Now let us move towards the root: suppose $S\in\mathcal{S}$ is a fixed vertex and that all the labelled trees corresponding to edges into $S$ have been constructed. If $x_S$ is the decorative cactus associated to $S$, then consider $y_S\coloneqq\sigma_{\abs{S}}(x_S)\in\mathcal{C}_{\abs{S}}$. This is a pointed cactus, so each lobe of $y_S$ has a canonical base point. 
		Let $b$ be a black vertex associated to the intersection of some lobes of $y_S$. We will denote this lobes as $l_1,\dots,l_m$. Each lobe corresponds to an edge into $S$ (either internal or open). If $l_j$ is associated to an open edge (labelled by $k\in\{1,\dots,n\}$), then draw an edge from $b$ to a white vertex labelled by $k$. If $l_j$ corresponds to an internal edge $(S_j,S)$, then we will draw an edge connecting $b$ to the labelled tree (with base point) $T_i$ associated to $S_i$ as follows: using the deformation retraction of Proposition \ref{prop:def retraction } we can associate to $T_i$ a cactus (with base point) $C_i$. Now we can use the composition of based cacti to identify $C_i$ with the corresponding lobe $l_i$ of $y_S$. Therefore our black vertex $b$ will correspond to a point $p$ in the cactus $C_i$. If this point is an intersection of lobes, then connect $b$ to the black vertex of $T_i$ associated to $p$ which has the largest label. If $p$ is not an intersection of lobes, then connect $b$ to the corresponding white vertex of $T_i$. After we have done this for any black vertex associated to $x_S$ we get a labelled tree with a base point; then we rotate the base point by the angular parameter $\alpha_{S}$ and we get a labelled tree with base point $T_S$. Now we iterate this procedure until we get to the root. At the end we obtain a labelled tree with base point; forgetting this base point we get $f^{-1}(x)$. See Figure \ref{fig:corrispondenza labelled trees vs nested cacti} for an example.
	\end{proof}
	\begin{figure}
		\centering
	
	\tikzset{every picture/.style={line width=0.75pt}} 
	
	\begin{tikzpicture}[x=0.75pt,y=0.75pt,yscale=-1,xscale=1]
		
		\draw   (300.33,219.71) .. controls (300.33,205.85) and (311.56,194.62) .. (325.42,194.62) .. controls (339.27,194.62) and (350.5,205.85) .. (350.5,219.71) .. controls (350.5,233.56) and (339.27,244.79) .. (325.42,244.79) .. controls (311.56,244.79) and (300.33,233.56) .. (300.33,219.71) -- cycle ;
		\draw  [fill={rgb, 255:red, 255; green, 255; blue, 255 }  ,fill opacity=1 ] (134,85.83) .. controls (134,83.16) and (136.16,81) .. (138.83,81) .. controls (141.5,81) and (143.67,83.16) .. (143.67,85.83) .. controls (143.67,88.5) and (141.5,90.67) .. (138.83,90.67) .. controls (136.16,90.67) and (134,88.5) .. (134,85.83) -- cycle ;
		\draw  [fill={rgb, 255:red, 0; green, 0; blue, 0 }  ,fill opacity=1 ] (136,127.17) .. controls (136,125.6) and (137.27,124.33) .. (138.83,124.33) .. controls (140.4,124.33) and (141.67,125.6) .. (141.67,127.17) .. controls (141.67,128.73) and (140.4,130) .. (138.83,130) .. controls (137.27,130) and (136,128.73) .. (136,127.17) -- cycle ;
		\draw    (138.83,90.67) -- (138.83,124.33) ;
		\draw    (138.83,130) -- (138.83,163.67) ;
		\draw  [fill={rgb, 255:red, 0; green, 0; blue, 0 }  ,fill opacity=1 ] (167.83,199.33) .. controls (167.83,197.77) and (169.1,196.5) .. (170.67,196.5) .. controls (172.23,196.5) and (173.5,197.77) .. (173.5,199.33) .. controls (173.5,200.9) and (172.23,202.17) .. (170.67,202.17) .. controls (169.1,202.17) and (167.83,200.9) .. (167.83,199.33) -- cycle ;
		\draw    (138.83,168.5) -- (108.67,199.33) ;
		\draw    (138.83,168.5) -- (170.67,199.33) ;
		\draw  [fill={rgb, 255:red, 0; green, 0; blue, 0 }  ,fill opacity=1 ] (105.83,199.33) .. controls (105.83,197.77) and (107.1,196.5) .. (108.67,196.5) .. controls (110.23,196.5) and (111.5,197.77) .. (111.5,199.33) .. controls (111.5,200.9) and (110.23,202.17) .. (108.67,202.17) .. controls (107.1,202.17) and (105.83,200.9) .. (105.83,199.33) -- cycle ;
		\draw  [fill={rgb, 255:red, 255; green, 255; blue, 255 }  ,fill opacity=1 ] (134,168.5) .. controls (134,165.83) and (136.16,163.67) .. (138.83,163.67) .. controls (141.5,163.67) and (143.67,165.83) .. (143.67,168.5) .. controls (143.67,171.17) and (141.5,173.33) .. (138.83,173.33) .. controls (136.16,173.33) and (134,171.17) .. (134,168.5) -- cycle ;
		\draw    (170.67,199.33) -- (202.5,230.17) ;
		\draw    (108.67,199.33) -- (78.5,230.17) ;
		\draw  [fill={rgb, 255:red, 255; green, 255; blue, 255 }  ,fill opacity=1 ] (73.67,230.17) .. controls (73.67,227.5) and (75.83,225.33) .. (78.5,225.33) .. controls (81.17,225.33) and (83.33,227.5) .. (83.33,230.17) .. controls (83.33,232.84) and (81.17,235) .. (78.5,235) .. controls (75.83,235) and (73.67,232.84) .. (73.67,230.17) -- cycle ;
		\draw  [fill={rgb, 255:red, 255; green, 255; blue, 255 }  ,fill opacity=1 ] (197.5,230.17) .. controls (197.5,227.41) and (199.74,225.17) .. (202.5,225.17) .. controls (205.26,225.17) and (207.5,227.41) .. (207.5,230.17) .. controls (207.5,232.93) and (205.26,235.17) .. (202.5,235.17) .. controls (199.74,235.17) and (197.5,232.93) .. (197.5,230.17) -- cycle ;
		\draw  [dash pattern={on 0.84pt off 2.51pt}] (35.35,230.17) .. controls (35.35,206.34) and (54.67,187.02) .. (78.5,187.02) .. controls (102.33,187.02) and (121.65,206.34) .. (121.65,230.17) .. controls (121.65,254) and (102.33,273.31) .. (78.5,273.31) .. controls (54.67,273.31) and (35.35,254) .. (35.35,230.17) -- cycle ;
		\draw  [dash pattern={on 0.84pt off 2.51pt}] (94.17,54.78) .. controls (128.27,27.4) and (190.06,47.72) .. (232.17,100.17) .. controls (274.29,152.63) and (280.78,217.35) .. (246.67,244.73) .. controls (212.57,272.11) and (150.78,251.79) .. (108.67,199.33) .. controls (66.55,146.88) and (60.06,82.16) .. (94.17,54.78) -- cycle ;
		\draw  [dash pattern={on 0.84pt off 2.51pt}] (170.67,199.33) .. controls (180.29,189.11) and (199.28,191.35) .. (213.08,204.35) .. controls (226.88,217.34) and (230.26,236.16) .. (220.63,246.38) .. controls (211,256.61) and (192.01,254.36) .. (178.22,241.37) .. controls (164.42,228.38) and (161.04,209.56) .. (170.67,199.33) -- cycle ;
		\draw  [dash pattern={on 0.84pt off 2.51pt}] (112.42,56.32) .. controls (140.64,44.68) and (176.61,66.95) .. (192.76,106.08) .. controls (208.91,145.2) and (199.13,186.36) .. (170.91,198.01) .. controls (142.69,209.66) and (106.73,187.38) .. (90.58,148.26) .. controls (74.43,109.13) and (84.21,67.97) .. (112.42,56.32) -- cycle ;
		\draw  [dash pattern={on 0.84pt off 2.51pt}] (139.98,59.34) .. controls (153.29,59.56) and (163.82,74.93) .. (163.5,93.66) .. controls (163.19,112.39) and (152.14,127.39) .. (138.83,127.17) .. controls (125.53,126.94) and (115,111.57) .. (115.31,92.84) .. controls (115.63,74.11) and (126.67,59.11) .. (139.98,59.34) -- cycle ;
		\draw  [dash pattern={on 0.84pt off 2.51pt}] (138.83,127.17) .. controls (151.04,127.37) and (160.71,141.31) .. (160.42,158.29) .. controls (160.14,175.28) and (150,188.88) .. (137.79,188.67) .. controls (125.58,188.47) and (115.91,174.53) .. (116.2,157.55) .. controls (116.49,140.56) and (126.62,126.96) .. (138.83,127.17) -- cycle ;
		\draw   (392.94,174.29) .. controls (392.94,165.11) and (400.38,157.67) .. (409.56,157.67) .. controls (418.74,157.67) and (426.18,165.11) .. (426.18,174.29) .. controls (426.18,183.47) and (418.74,190.91) .. (409.56,190.91) .. controls (400.38,190.91) and (392.94,183.47) .. (392.94,174.29) -- cycle ;
		\draw  [color={rgb, 255:red, 74; green, 144; blue, 226 }  ,draw opacity=1 ] (392.94,141.05) .. controls (392.94,131.88) and (400.38,124.44) .. (409.56,124.44) .. controls (418.74,124.44) and (426.18,131.88) .. (426.18,141.05) .. controls (426.18,150.23) and (418.74,157.67) .. (409.56,157.67) .. controls (400.38,157.67) and (392.94,150.23) .. (392.94,141.05) -- cycle ;
		\draw   (341.15,176.09) .. controls (341.15,217.38) and (374.62,250.85) .. (415.91,250.85) .. controls (457.2,250.85) and (490.67,217.38) .. (490.67,176.09) .. controls (490.67,134.8) and (457.2,101.33) .. (415.91,101.33) .. controls (374.62,101.33) and (341.15,134.8) .. (341.15,176.09) -- cycle ;
		\draw   (437.07,212.08) .. controls (437.07,202.58) and (444.77,194.89) .. (454.26,194.89) .. controls (463.76,194.89) and (471.45,202.58) .. (471.45,212.08) .. controls (471.45,221.57) and (463.76,229.27) .. (454.26,229.27) .. controls (444.77,229.27) and (437.07,221.57) .. (437.07,212.08) -- cycle ;
		\draw   (364.49,162.61) .. controls (364.49,136.36) and (385.77,115.07) .. (412.03,115.07) .. controls (438.29,115.07) and (459.57,136.36) .. (459.57,162.61) .. controls (459.57,188.87) and (438.29,210.16) .. (412.03,210.16) .. controls (385.77,210.16) and (364.49,188.87) .. (364.49,162.61) -- cycle ;
		\draw [color={rgb, 255:red, 0; green, 0; blue, 0 }  ,draw opacity=1 ][line width=1.5]    (340.25,195.79) -- (346.97,205.67) ;
		\draw [color={rgb, 255:red, 208; green, 2; blue, 27 }  ,draw opacity=1 ][line width=1.5]    (440.63,200.92) -- (434.7,208.52) ;
		\draw [color={rgb, 255:red, 74; green, 144; blue, 226 }  ,draw opacity=1 ][line width=1.5]    (396.71,157.67) -- (407.19,157.67) ;
		\draw [color={rgb, 255:red, 74; green, 144; blue, 226 }  ,draw opacity=1 ]   (407.19,157.67) .. controls (429.96,157.58) and (428.77,180.11) .. (421.66,186.17) ;
		\draw [color={rgb, 255:red, 74; green, 144; blue, 226 }  ,draw opacity=1 ] [dash pattern={on 0.84pt off 2.51pt}]  (422.84,186.17) -- (440.63,200.92) ;
		\draw [color={rgb, 255:red, 208; green, 2; blue, 27 }  ,draw opacity=1 ] [dash pattern={on 0.84pt off 2.51pt}]  (374.24,191.97) -- (346.97,205.67) ;
		\draw [color={rgb, 255:red, 208; green, 2; blue, 27 }  ,draw opacity=1 ]   (439.79,200.77) .. controls (422.84,215.41) and (390.83,213.04) .. (374.24,191.97) ;
		
		\draw (64,232) node [anchor=north west][inner sep=0.75pt]   [align=left] {$\displaystyle 1$};
		\draw (122,157) node [anchor=north west][inner sep=0.75pt]   [align=left] {$\displaystyle 2$};
		\draw (133,63) node [anchor=north west][inner sep=0.75pt]   [align=left] {$\displaystyle 3$};
		\draw (204.5,233.17) node [anchor=north west][inner sep=0.75pt]   [align=left] {$\displaystyle 4$};
		\draw (319.93,210.5) node [anchor=north west][inner sep=0.75pt]   [align=left] {$\displaystyle 1$};
		\draw (402.98,166.2) node [anchor=north west][inner sep=0.75pt]   [align=left] {$\displaystyle 2$};
		\draw (449.75,203.58) node [anchor=north west][inner sep=0.75pt]   [align=left] {$\displaystyle 4$};
		\draw (404.01,134.2) node [anchor=north west][inner sep=0.75pt]   [align=left] {$\displaystyle 3$};
		\draw (397.57,212.61) node [anchor=north west][inner sep=0.75pt]   [align=left] {$\displaystyle \textcolor[rgb]{0.82,0.01,0.11}{\alpha }\textcolor[rgb]{0.82,0.01,0.11}{_{T}}$};
		\draw (426.18,145.05) node [anchor=north west][inner sep=0.75pt]  [color={rgb, 255:red, 208; green, 2; blue, 27 }  ,opacity=1 ] [align=left] {$\displaystyle \textcolor[rgb]{0.29,0.56,0.89}{\alpha }\textcolor[rgb]{0.29,0.56,0.89}{_{S}}$};

	\end{tikzpicture}
	
		\caption{In this picture we see the nested cactus associated to the labelled tree on the left. We assume that the label of the black vertices are the same as in Figure \ref{fig: associo a un labelled tree un nested tree}. The base point on the cacti are those obtained using the sections $\sigma_k:\mathcal{C}_k/S^1\to \mathcal{C}_k$. The angular parameters $\alpha_S,\alpha_T$ are represented as distances between the base point given by the section and the base point induced by the bigger cactus (dotted lines). The angular parameters are omitted.}
		\label{fig:corrispondenza labelled trees vs nested cacti}
	\end{figure}
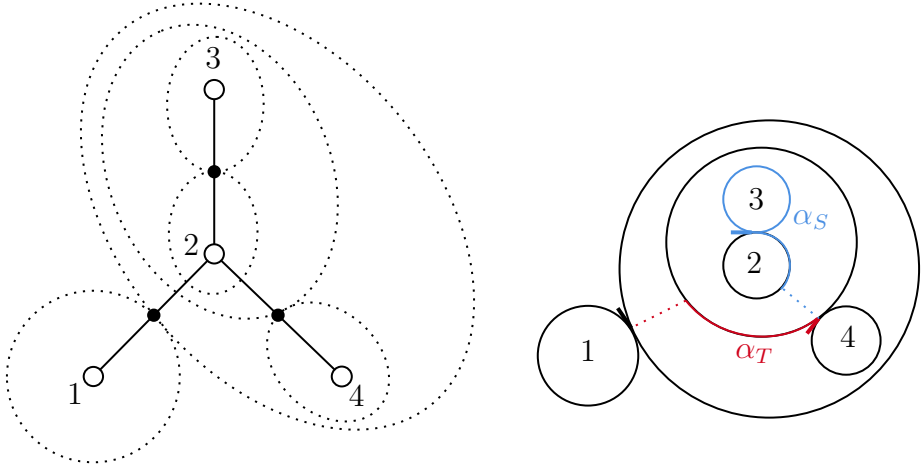
	
	\subsection{A CW-decomposition of $\overline{\M}_{0,n+1}$}\label{subsec: CW decomposition of Deligne mumford}
	Combining Theorem \ref{thm:omeo tra spazio di moduli e labelled trees} and Lemma \ref{lem:omeo tra nested tree e alberi b/w}, for each choice of weights $a_1,\dots,a_n$ we get an embedding
	\begin{equation}\label{embedding dello spazio dei moduli nello spazio dei nested trees decorati}
		\phi_{a_1\dots,a_n}:\mathcal{M}_{0,n+1}\to Tr_n\cong N_n^{\sigma}(\mathcal{C}/S^1)^{>0}\subseteq  N_n^{\sigma}(\mathcal{C}/S^1)    
	\end{equation}
	We will see that this embedding extends to a homeomorphism between the Deligne-Mumford compactification $\overline{\mathcal{M}}_{0,n+1}$ and $N_n^{\sigma}(\mathcal{C}/S^1)$, giving a regular CW decomposition of $\overline{\mathcal{M}}_{0,n+1}$. In what follows it is convenient to use a coordinate free approach: for any finite set $R$ with $r$ elements we define $\mathcal{M}_{0,R+1}$ to be the space of injective maps from $R$ to $\C$, modulo the action of $\C\rtimes \C^*$ on the target. A total order on $R$ induces a homeomorphism between $\mathcal{M}_{0,R+1}$ and $\mathcal{M}_{0,r+1}$. Similarly it is possible to define $N_R^{\sigma}(\mathcal{C}/S^1)\cong N_r^{\sigma}(\mathcal{C}/S^1)$, the space of nested cacti with leaves labelled by $R$. The choice of a function $a:R\to\R^{>0}$ gives an embedding $\phi_{a}:\M_{0,R+1}\to N^{\sigma}_{R}(\mathcal{C}/S^1)$ analogous to \ref{embedding dello spazio dei moduli nello spazio dei nested trees decorati}. When $R=\{y_1<y_2<\cdots<y_r\}$ is equipped with a total order $\phi_a$ corresponds to $\phi_{a(y_1)/A,\dots,a(y_r)/A}$, with $A=\sum_{i=1}^ra(y_i)$. 
	\begin{defn}\label{defn: relazione equivalenza vertici nested tree}
		Let $\mathcal{S}$ be a nested tree with $n$-leaves. For any $S\in\mathcal{S}$, $\hat{S}$ will denote the quotient of $S$ by the equivalence relation given by
		\[
		s_1\sim s_2 \iff \exists T\in\mathcal{S}, T\subseteq S \mid s_1,s_2\in T 
		\]
		We will denote by $\pi_S:S\to\hat{S}$ the quotient map.
	\end{defn}
	Recall that $\overline{\mathcal{M}}_{0,n+1}$ is a stratified space: for each nested tree $\mathcal{S}\in N_n$ we have $\overline{\mathcal{M}}_{0,n+1}=\bigcup_{\mathcal{S}\in N_n} \mathcal{M}(\mathcal{S})$, where $\mathcal{M}(\mathcal{S})$ contains all the stable curves whose dual graph is $\mathcal{S}$. For example, if $\mathcal{S}$ is the corolla, then $\M(\mathcal{S})=\mathcal{M}_{0,n+1}$. In general we have a homeomorphism
	\[
	\beta_{\mathcal{S}}: \prod_{S\in \mathcal{S}}\mathcal{M}_{0,\hat{S}+1}\to \mathcal{M}(\mathcal{S})
	\]
	\begin{defn}
		Let $a_1,\dots,a_n\in\R^{>0}$ such that $\sum_{i=1}^na_i=1$. We will define a function
		\[
		\overline{\phi}_{a_1,\dots,a_n}:\overline{\mathcal{M}}_{0,n+1}\to N_n(\mathcal{C}/S^1)
		\]
		stratum-wise as follows: given a nested tree $\mathcal{S}\in N_n$, take $(y_S)_{S\in\mathcal{S}}\in \prod_{S\in \mathcal{S}}\mathcal{M}_{0,\hat{S}+1}\cong\mathcal{M}(\mathcal{S})$. In order to define 
		\[
		\overline{\phi}_{a_1,\dots,a_n}(\beta_{\mathcal{S}}((y_S)_{S\in\mathcal{S}}))
		\]
		we have to specify a nested tree $\mathcal{T}$ with $n$-leaves, an unbased cactus for any vertex and radial/angular parameters for any vertex which is not the root. For any $S\in\mathcal{S}$, consider the labelled nested tree
		\[
		\phi_{a_{S}}(y_S)\in N_{\hat{S}}^{\sigma}(\mathcal{C}/S^1)
		\]
		where $ a_S:\hat{S}\to \R^{>0}$ is defined as $a_S(x)\coloneqq\sum_{i\in\pi^{-1}_S(x)}a_i$. Let $\mathcal{T}_S$ be the underlying nested tree. Then $\mathcal{T}\in N_n$ is defined to be the collection of subsets of the form $\pi_S^{-1}(U)\subseteq\{1,\dots,n\}$, with $U\in \mathcal{T}_S$ and $S\in\mathcal{S}$. Intuitively $\mathcal{T}$ is obtained by grafting together the nested trees $\{T_S\}_{S\in\mathcal{S}}$. Now we specify the unbased cacti and the radial/angular parameters decorating each vertex: given $\pi^{-1}_S(U)\in \mathcal{T}$, with $U\in\mathcal{T}_S$ and $S\in\mathcal{S}$ we put:
		\begin{enumerate}
			\item The unbased cactus associated to $\pi^{-1}_S(U)$ will be the one decorating $U\in \mathcal{T}_S$ in $\phi_{a_{S}}(y_S)$.
			\item If  $U\neq \hat{S}$ the radial parameter associated to $\pi^{-1}_S(U)$ will be the one associated to $U\in \mathcal{T}_S$ in $\phi_{a_{S}}(y_S)$. It will be zero if $U=\hat{S}$.
			\item If  $U\neq \hat{S}$ the angular parameter associated to $\pi^{-1}_S(U)$ will be the one associated to $U\in \mathcal{T}_S$ in $\phi_{a_{S}}(y_S)$ if $U\neq \hat{S}$. If $U=\hat{S}$ we do not have to specify it since the radial parameter is zero.
			
		\end{enumerate}
	\end{defn}
	\begin{es}
		Consider the nested tree $\mathcal{S}\coloneqq((1,2,3),4)$ and let us denote by $R$ the root and by $S$ the other vertex. Fix some weights $a_1,a_2,a_3,a_4$ and take a point $(y_R,y_S)\in\M_{0,\hat{R}+1}\times\M_{0,\hat{S}+1}\cong \M(\mathcal{S})$. Then $\overline{\phi}_{a_1,a_2,a_3,a_4}(\beta_{\mathcal{S}}(y_R,y_S))$ is obtained by grafting $\phi_{a_1,a_2,a_3}(y_S)$ to the first leaf of $\phi_{a_1+a_2+a_3,a_4}(y_R)$ and setting the radial parameter $t_S=0$.
	\end{es}
	\begin{thm}\label{thm:omeo tra la compattificazione e lo spazio dei nested cactus}
		For each choice of weights $a_1,\dots,a_n>0$, $\sum_{i=1}^na_i=1$, the map
		\[
		\overline{\phi}_{a_1\dots,a_n}:\overline{\mathcal{M}}_{0,n+1}\to N_n^{\sigma}(\mathcal{C}/S^1)
		\]
		is a homeomorphism.
	\end{thm}
	\begin{proof}
		$\overline{\phi}_{a_1\dots,a_n}$ is bijective, the domain is compact and the target is Hausdorff. So it is enough to prove the continuity of $\overline{\phi}_{a_1\dots,a_n}$. By definition $\overline{\phi}_{a_1\dots,a_n}$ restricted to $\mathcal{M}_{0,n+1}\subseteq\overline{\M}_{0,n+1}$ is the map $\phi_{a_1,\dots,a_n}$ of \ref{embedding dello spazio dei moduli nello spazio dei nested trees decorati}, which is continuous. Thus we only need to understand what happens when two or more points collide. Let us consider for example the configuration $[z_0-\epsilon w,z_0+\epsilon w,z_3]\in\mathcal{M}_{0,4}$ for $\epsilon>0$ (see Figure \ref{limit of config}) and fix some weights $a_1,a_2,a_3$ (again the physical intuition is useful, so think the weights as the value of some charges placed in  $z_0-\epsilon w,z_0+\epsilon w,z_3$).
		\begin{figure}
			\centering

		\tikzset{every picture/.style={line width=0.75pt}} 
		
		\begin{tikzpicture}[x=0.75pt,y=0.75pt,yscale=-1,xscale=1]
			
			\draw    (168.48,62.46) -- (149.67,84.33) ;
			\draw    (149.67,84.33) -- (130.85,106.21) ;
			\draw  [fill={rgb, 255:red, 255; green, 255; blue, 255 }  ,fill opacity=1 ] (164.5,62.46) .. controls (164.5,60.26) and (166.28,58.48) .. (168.48,58.48) .. controls (170.68,58.48) and (172.46,60.26) .. (172.46,62.46) .. controls (172.46,64.66) and (170.68,66.44) .. (168.48,66.44) .. controls (166.28,66.44) and (164.5,64.66) .. (164.5,62.46) -- cycle ;
			\draw  [fill={rgb, 255:red, 0; green, 0; blue, 0 }  ,fill opacity=1 ] (147.33,84.33) .. controls (147.33,83.04) and (148.38,82) .. (149.67,82) .. controls (150.96,82) and (152,83.04) .. (152,84.33) .. controls (152,85.62) and (150.96,86.67) .. (149.67,86.67) .. controls (148.38,86.67) and (147.33,85.62) .. (147.33,84.33) -- cycle ;
			\draw  [fill={rgb, 255:red, 255; green, 255; blue, 255 }  ,fill opacity=1 ] (126.87,106.21) .. controls (126.87,104.01) and (128.65,102.22) .. (130.85,102.22) .. controls (133.05,102.22) and (134.83,104.01) .. (134.83,106.21) .. controls (134.83,108.4) and (133.05,110.19) .. (130.85,110.19) .. controls (128.65,110.19) and (126.87,108.4) .. (126.87,106.21) -- cycle ;
			\draw  [fill={rgb, 255:red, 255; green, 255; blue, 255 }  ,fill opacity=1 ] (49.79,135.38) .. controls (49.79,133.18) and (51.57,131.4) .. (53.77,131.4) .. controls (55.97,131.4) and (57.75,133.18) .. (57.75,135.38) .. controls (57.75,137.58) and (55.97,139.37) .. (53.77,139.37) .. controls (51.57,139.37) and (49.79,137.58) .. (49.79,135.38) -- cycle ;
			\draw  [dash pattern={on 0.84pt off 2.51pt}] (20.67,55.53) .. controls (20.67,39.96) and (33.29,27.33) .. (48.87,27.33) -- (177.47,27.33) .. controls (193.04,27.33) and (205.67,39.96) .. (205.67,55.53) -- (205.67,140.13) .. controls (205.67,155.71) and (193.04,168.33) .. (177.47,168.33) -- (48.87,168.33) .. controls (33.29,168.33) and (20.67,155.71) .. (20.67,140.13) -- cycle ;
			\draw [line width=2.25]    (225,101) -- (306.67,101) ;
			\draw [shift={(310.67,101)}, rotate = 180] [color={rgb, 255:red, 0; green, 0; blue, 0 }  ][line width=2.25]    (17.49,-5.26) .. controls (11.12,-2.23) and (5.29,-0.48) .. (0,0) .. controls (5.29,0.48) and (11.12,2.23) .. (17.49,5.26)   ;
			\draw    (461.3,81.67) -- (455.43,88.5) ;
			\draw    (455.43,88.5) -- (449.56,95.32) ;
			\draw  [fill={rgb, 255:red, 255; green, 255; blue, 255 }  ,fill opacity=1 ] (460.06,81.67) .. controls (460.06,80.99) and (460.62,80.43) .. (461.3,80.43) .. controls (461.99,80.43) and (462.54,80.99) .. (462.54,81.67) .. controls (462.54,82.36) and (461.99,82.92) .. (461.3,82.92) .. controls (460.62,82.92) and (460.06,82.36) .. (460.06,81.67) -- cycle ;
			\draw  [fill={rgb, 255:red, 0; green, 0; blue, 0 }  ,fill opacity=1 ] (454.71,88.5) .. controls (454.71,88.09) and (455.03,87.77) .. (455.43,87.77) .. controls (455.84,87.77) and (456.16,88.09) .. (456.16,88.5) .. controls (456.16,88.9) and (455.84,89.23) .. (455.43,89.23) .. controls (455.03,89.23) and (454.71,88.9) .. (454.71,88.5) -- cycle ;
			\draw  [fill={rgb, 255:red, 255; green, 255; blue, 255 }  ,fill opacity=1 ] (448.32,95.32) .. controls (448.32,94.63) and (448.88,94.08) .. (449.56,94.08) .. controls (450.25,94.08) and (450.81,94.63) .. (450.81,95.32) .. controls (450.81,96.01) and (450.25,96.56) .. (449.56,96.56) .. controls (448.88,96.56) and (448.32,96.01) .. (448.32,95.32) -- cycle ;
			\draw  [fill={rgb, 255:red, 255; green, 255; blue, 255 }  ,fill opacity=1 ] (355.79,135.38) .. controls (355.79,133.18) and (357.57,131.4) .. (359.77,131.4) .. controls (361.97,131.4) and (363.75,133.18) .. (363.75,135.38) .. controls (363.75,137.58) and (361.97,139.37) .. (359.77,139.37) .. controls (357.57,139.37) and (355.79,137.58) .. (355.79,135.38) -- cycle ;
			\draw  [dash pattern={on 0.84pt off 2.51pt}] (326.67,55.53) .. controls (326.67,39.96) and (339.29,27.33) .. (354.87,27.33) -- (483.47,27.33) .. controls (499.04,27.33) and (511.67,39.96) .. (511.67,55.53) -- (511.67,140.13) .. controls (511.67,155.71) and (499.04,168.33) .. (483.47,168.33) -- (354.87,168.33) .. controls (339.29,168.33) and (326.67,155.71) .. (326.67,140.13) -- cycle ;
			\draw  [dash pattern={on 0.84pt off 2.51pt}] (442.81,88.5) .. controls (442.81,81.93) and (448.13,76.6) .. (454.71,76.6) .. controls (461.28,76.6) and (466.6,81.93) .. (466.6,88.5) .. controls (466.6,95.07) and (461.28,100.39) .. (454.71,100.39) .. controls (448.13,100.39) and (442.81,95.07) .. (442.81,88.5) -- cycle ;
			
			\draw (37.86,139.31) node [anchor=north west][inner sep=0.75pt]   [align=left] {$\displaystyle z_{3}$};
			\draw (103.82,110.69) node [anchor=north west][inner sep=0.75pt]   [align=left] {$\displaystyle z_{0} +\epsilon w$};
			\draw (138.71,40.06) node [anchor=north west][inner sep=0.75pt]   [align=left] {$\displaystyle z_{0} -\epsilon w$};
			\draw (154.71,79.06) node [anchor=north west][inner sep=0.75pt]   [align=left] {$\displaystyle z_{0}$};
			\draw (243.71,76.06) node [anchor=north west][inner sep=0.75pt]   [align=left] {$\displaystyle \epsilon \rightarrow 0$};
			\draw (343.86,139.31) node [anchor=north west][inner sep=0.75pt]   [align=left] {$\displaystyle z_{3}$};
			\draw (428.16,64.66) node [anchor=north west][inner sep=0.75pt]   [align=left] {$\displaystyle z_{0}$};

		\end{tikzpicture}
		
			\caption{This picture represent the collision of two charges.}
			\label{limit of config}
		\end{figure}
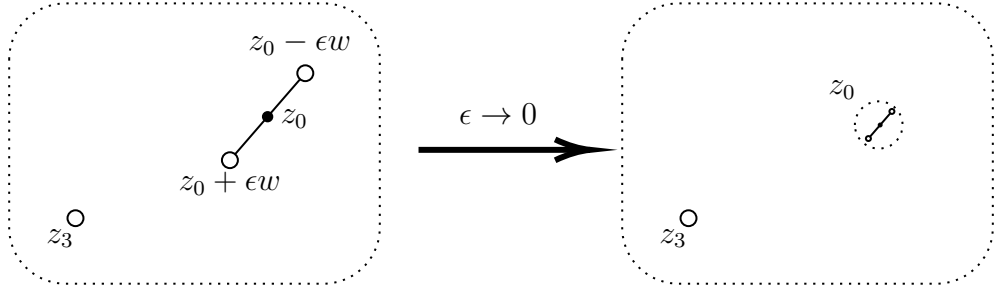
		When $\epsilon\to 0$ the configuration $[z_0-\epsilon w,z_0+\epsilon w,z_3]$ tends to the following infinitesimal configuration: the third charge remains in $z_3$, with value $a_3$. The other two charges collide and became infinitesimal at the point $z_0$. So, looking at this configuration from "far away" it seems that there are only two charges: one in $z_0$ with value $a_1+a_2$, and the other in $z_3$ with value $a_3$; but when we zoom in on $z_0$ we realize that in fact there are two charges of values $a_1$ and $a_2$ which are very close to each other. In other words, the limit of this configuration when $\epsilon\to 0$ is a stable curve in $\M(\mathcal{S})$, where $\mathcal{S}\coloneqq\{R,S\}$ is the nested tree given by  $R=\{1,2,3\}$ and $S=\{1,2\}$. More precisely, the limit is  $\beta_{\mathcal{S}}(y_R,y_S)$ where $y_R=[z_0,z_3]$ and $y_S=[-w,w]$. The black vertices of the labelled tree which represents $[z_0-\epsilon w,z_0+\epsilon w,z_3]$ are the two (possibly coinciding) critical points $w_1^{\epsilon}$, $w_2^{\epsilon}$ of
		\[
		h(z)=(z-z_0+\epsilon w)^{a_1}(z-z_0-\epsilon w)^{a_2}(z-z_3)^{a_3}
		\]
		that are the zeros of
		\[
		p_{\epsilon}(z)=a_1(z-z_0-\epsilon w)(z-z_3)+a_2(z-z_0+\epsilon w)(z-z_3)+a_3((z-z_0)^2-\epsilon^2w^2)
		\]
		These zeros depend continuously on $\epsilon$, and in the limit tend to the zeros of
		\[
		p_0(z)=(z-z_0)((a_1+a_2)(z-z_3)+a_3(z-z_0))
		\]
		that are $w_1^0=z_0$ and $w_2^0$, which is the critical point associated to the configuration $[z_0,z_3]$ with weights $a_1+a_2,a_3$. For $\epsilon\to 0$ we have $\abs{h(w_1^{\epsilon})}\to 0$ and $\abs{h(w_2^{\epsilon})}\to\abs{h(w_2^0)}\neq 0$, so the ratio of their values tends to $0$. This implies that the radial parameter associated to $S\in\mathcal{S}$ in the nested cactus $\phi_{a_1,a_2,a_3}([z_0-\epsilon w,z_0+\epsilon w,z_3])$ tends to zero. Moreover, the cactus labelling the root of $\phi_{a_1,a_2,a_3}([z_0-\epsilon w,z_0+\epsilon w,z_3])$ tends to the cactus $\phi_{a_1+a_2,a_3}([z_0,z_3])$ and the one labelling $S=\{1,2\}$ tends to $\phi_{a_1,a_2}([-w,w])$. In the general case we can do a similar computation to show the continuity of $\overline{\phi}_{a_1,\dots,a_n}$.
	\end{proof}
	\begin{oss}
		By definition the space $N_n^{\sigma}(\mathcal{C}/S^1)$ depends on the choice of sections $\sigma_k:\mathcal{C}/S^1\to \mathcal{C}$. Theorem \ref{thm:omeo tra la compattificazione e lo spazio dei nested cactus} tells us that this choice is not so important, since in any case we obtain a CW-complex which is homeomorphic to $\overline{\M}_{0,n+1}$.
	\end{oss}
	\begin{corollario}\label{cor: decomposizione cellulare spazio compattificato}
		Each choice of weights $a_1,\dots,a_n>0$, $\sum_{i=1}^na_i=1$ and sections $\sigma_k:\mathcal{C}_k/S^1\to\mathcal{C}_k$ determines a regular CW-decomposition of $\overline{\M}_{0,n+1}$. A cell is determined by 
		\begin{enumerate}
			\item A nested tree $\mathcal{S}$ with $n$-leaves.
			\item A cell $C_{S}\subseteq \mathcal{C}_{\abs{S}}/S^1$ for any vertex $S\in\mathcal{S}$.
		\end{enumerate}
		and its closure is homeomorphic to 
		\[
		\overline{C}_R\times\prod_{S\in\mathcal{S}-R}D_2\times\overline{C}_S
		\]
		
	\end{corollario}
	\begin{proof}
		Just use Theorem \ref{thm:omeo tra la compattificazione e lo spazio dei nested cactus} to transfer the regular cell structure of $N_n^{\sigma}(\mathcal{C}/S^1)$ to $\overline{\M}_{0,n+1}$.
	\end{proof}
	\begin{es}
		Let us describe explicitly the cell structure of $\overline{\M}_{0,4}$: there are two $0$-cells $C_1,C_2$, three $1$-cells $B_1,B_2,B_3$ and three $2$-cells $A_1,A_2,A_3$. The cells of dimension $0$ and $1$ are described by a corolla with three leaves decorated by an unbased cactus with three lobes. The two dimensional cells are associated to the three nested trees $((12)3))$, $((13)2)$ and $(1(23))$. Topologically $A_1, A_2$ and $A_3$ are homeomorphic to $D_2$. The center of each such a disk represents the stable curve whose dual graph is the nested tree representing the cell. See Figure \ref{fig:esempio decomposizione cellulare con 3 punti}. 
	\end{es}
	\begin{oss}
		As we know $Tr_n\cong \M_{0,n+1}$. If $n=3$ we get the (open) cell decomposition of $\M_{0,4}$ depicted in Figure \ref{fig:decomposizione cellulare alberi con tre foglie}. Adding the three missing points (i.e. the three stable curves of $\overline{\M}_{0,4}$) we get a CW-decomposition of $\overline{\M}_{0,4}$. However it is different from the one we obtained in this paragraph: in this decomposition the stable curves lie in the boundary of the $2$-cells (see Figure \ref{fig:decomposizione cellulare alberi con tre foglie}), while in the decomposition given by nested cacti they are the centers of the $2$-cells. 
	\end{oss}
	\begin{figure}
		\centering

		\tikzset{every picture/.style={line width=0.75pt}} 
		
		\begin{tikzpicture}[x=0.75pt,y=0.75pt,yscale=-1,xscale=1]
			
			\draw   (31,222.78) .. controls (31,217.38) and (35.38,213) .. (40.78,213) .. controls (46.18,213) and (50.56,217.38) .. (50.56,222.78) .. controls (50.56,228.18) and (46.18,232.56) .. (40.78,232.56) .. controls (35.38,232.56) and (31,228.18) .. (31,222.78) -- cycle ;
			\draw   (50.56,222.78) .. controls (50.56,217.38) and (54.93,213) .. (60.33,213) .. controls (65.73,213) and (70.11,217.38) .. (70.11,222.78) .. controls (70.11,228.18) and (65.73,232.56) .. (60.33,232.56) .. controls (54.93,232.56) and (50.56,228.18) .. (50.56,222.78) -- cycle ;
			\draw   (70.11,222.78) .. controls (70.11,217.38) and (74.49,213) .. (79.89,213) .. controls (85.29,213) and (89.67,217.38) .. (89.67,222.78) .. controls (89.67,228.18) and (85.29,232.56) .. (79.89,232.56) .. controls (74.49,232.56) and (70.11,228.18) .. (70.11,222.78) -- cycle ;
			\draw    (26.67,175.67) -- (60,209) ;
			\draw    (60,209) -- (93.67,175.33) ;
			\draw    (60,209) -- (60,174.33) ;
			\draw    (31.67,57.67) -- (65,91) ;
			\draw    (65,91) -- (98.67,57.33) ;
			\draw    (65,91) -- (65,56.33) ;
			\draw   (61.22,125.26) .. controls (61.22,125.26) and (61.22,125.26) .. (61.22,125.26) .. controls (59.96,130.21) and (54.93,133.2) .. (49.98,131.93) .. controls (45.04,130.67) and (42.05,125.64) .. (43.31,120.69) .. controls (44.58,115.75) and (49.61,112.76) .. (54.56,114.02) .. controls (60.52,115.55) and (63.51,116.31) .. (63.51,116.31) .. controls (63.51,116.31) and (62.75,119.3) .. (61.22,125.26) -- cycle ;
			\draw   (72.44,113.94) .. controls (72.44,113.94) and (72.44,113.94) .. (72.44,113.94) .. controls (72.44,113.94) and (72.44,113.94) .. (72.44,113.94) .. controls (77.38,112.63) and (82.44,115.57) .. (83.74,120.51) .. controls (85.05,125.44) and (82.11,130.5) .. (77.18,131.81) .. controls (72.24,133.12) and (67.18,130.18) .. (65.88,125.24) .. controls (64.3,119.29) and (63.51,116.31) .. (63.51,116.31) .. controls (63.51,116.31) and (66.48,115.52) .. (72.44,113.94) -- cycle ;
			\draw   (57.12,109.63) .. controls (57.12,109.63) and (57.12,109.63) .. (57.12,109.63) .. controls (53.59,105.95) and (53.71,100.1) .. (57.4,96.57) .. controls (61.09,93.04) and (66.94,93.16) .. (70.47,96.85) .. controls (74,100.54) and (73.87,106.39) .. (70.19,109.92) .. controls (65.74,114.18) and (63.51,116.31) .. (63.51,116.31) .. controls (63.51,116.31) and (61.38,114.09) .. (57.12,109.63) -- cycle ;
			\draw    (108.67,57.67) -- (142,91) ;
			\draw    (142,91) -- (175.67,57.33) ;
			\draw    (142,91) -- (142,56.33) ;
			\draw   (138.22,125.26) .. controls (138.22,125.26) and (138.22,125.26) .. (138.22,125.26) .. controls (136.96,130.21) and (131.93,133.2) .. (126.98,131.93) .. controls (122.04,130.67) and (119.05,125.64) .. (120.31,120.69) .. controls (121.58,115.75) and (126.61,112.76) .. (131.56,114.02) .. controls (137.52,115.55) and (140.51,116.31) .. (140.51,116.31) .. controls (140.51,116.31) and (139.75,119.3) .. (138.22,125.26) -- cycle ;
			\draw   (149.44,113.94) .. controls (149.44,113.94) and (149.44,113.94) .. (149.44,113.94) .. controls (149.44,113.94) and (149.44,113.94) .. (149.44,113.94) .. controls (154.38,112.63) and (159.44,115.57) .. (160.74,120.51) .. controls (162.05,125.44) and (159.11,130.5) .. (154.18,131.81) .. controls (149.24,133.12) and (144.18,130.18) .. (142.88,125.24) .. controls (141.3,119.29) and (140.51,116.31) .. (140.51,116.31) .. controls (140.51,116.31) and (143.48,115.52) .. (149.44,113.94) -- cycle ;
			\draw   (134.12,109.63) .. controls (134.12,109.63) and (134.12,109.63) .. (134.12,109.63) .. controls (130.59,105.95) and (130.71,100.1) .. (134.4,96.57) .. controls (138.09,93.04) and (143.94,93.16) .. (147.47,96.85) .. controls (151,100.54) and (150.87,106.39) .. (147.19,109.92) .. controls (147.19,109.92) and (147.19,109.92) .. (147.19,109.92) .. controls (142.74,114.18) and (140.51,116.31) .. (140.51,116.31) .. controls (140.51,116.31) and (138.38,114.09) .. (134.12,109.63) -- cycle ;
			\draw   (111,222.78) .. controls (111,217.38) and (115.38,213) .. (120.78,213) .. controls (126.18,213) and (130.56,217.38) .. (130.56,222.78) .. controls (130.56,228.18) and (126.18,232.56) .. (120.78,232.56) .. controls (115.38,232.56) and (111,228.18) .. (111,222.78) -- cycle ;
			\draw   (130.56,222.78) .. controls (130.56,217.38) and (134.93,213) .. (140.33,213) .. controls (145.73,213) and (150.11,217.38) .. (150.11,222.78) .. controls (150.11,228.18) and (145.73,232.56) .. (140.33,232.56) .. controls (134.93,232.56) and (130.56,228.18) .. (130.56,222.78) -- cycle ;
			\draw   (150.11,222.78) .. controls (150.11,217.38) and (154.49,213) .. (159.89,213) .. controls (165.29,213) and (169.67,217.38) .. (169.67,222.78) .. controls (169.67,228.18) and (165.29,232.56) .. (159.89,232.56) .. controls (154.49,232.56) and (150.11,228.18) .. (150.11,222.78) -- cycle ;
			\draw    (106.67,175.67) -- (140,209) ;
			\draw    (140,209) -- (173.67,175.33) ;
			\draw    (140,209) -- (140,174.33) ;
			\draw   (189,222.78) .. controls (189,217.38) and (193.38,213) .. (198.78,213) .. controls (204.18,213) and (208.56,217.38) .. (208.56,222.78) .. controls (208.56,228.18) and (204.18,232.56) .. (198.78,232.56) .. controls (193.38,232.56) and (189,228.18) .. (189,222.78) -- cycle ;
			\draw   (208.56,222.78) .. controls (208.56,217.38) and (212.93,213) .. (218.33,213) .. controls (223.73,213) and (228.11,217.38) .. (228.11,222.78) .. controls (228.11,228.18) and (223.73,232.56) .. (218.33,232.56) .. controls (212.93,232.56) and (208.56,228.18) .. (208.56,222.78) -- cycle ;
			\draw   (228.11,222.78) .. controls (228.11,217.38) and (232.49,213) .. (237.89,213) .. controls (243.29,213) and (247.67,217.38) .. (247.67,222.78) .. controls (247.67,228.18) and (243.29,232.56) .. (237.89,232.56) .. controls (232.49,232.56) and (228.11,228.18) .. (228.11,222.78) -- cycle ;
			\draw    (184.67,175.67) -- (218,209) ;
			\draw    (218,209) -- (251.67,175.33) ;
			\draw    (218,209) -- (218,174.33) ;
			\draw    (33.67,276.22) -- (51,309) ;
			\draw    (51,309) -- (67.67,277.22) ;
			\draw    (54.67,337) -- (69.67,373.56) ;
			\draw    (69.67,373.56) -- (109.67,281.56) ;
			\draw   (32,321.78) .. controls (32,316.38) and (36.38,312) .. (41.78,312) .. controls (47.18,312) and (51.56,316.38) .. (51.56,321.78) .. controls (51.56,327.18) and (47.18,331.56) .. (41.78,331.56) .. controls (36.38,331.56) and (32,327.18) .. (32,321.78) -- cycle ;
			\draw   (51.56,321.78) .. controls (51.56,316.38) and (55.93,312) .. (61.33,312) .. controls (66.73,312) and (71.11,316.38) .. (71.11,321.78) .. controls (71.11,327.18) and (66.73,331.56) .. (61.33,331.56) .. controls (55.93,331.56) and (51.56,327.18) .. (51.56,321.78) -- cycle ;
			\draw   (50,383.78) .. controls (50,378.38) and (54.38,374) .. (59.78,374) .. controls (65.18,374) and (69.56,378.38) .. (69.56,383.78) .. controls (69.56,389.18) and (65.18,393.56) .. (59.78,393.56) .. controls (54.38,393.56) and (50,389.18) .. (50,383.78) -- cycle ;
			\draw   (69.56,383.78) .. controls (69.56,378.38) and (73.93,374) .. (79.33,374) .. controls (84.73,374) and (89.11,378.38) .. (89.11,383.78) .. controls (89.11,389.18) and (84.73,393.56) .. (79.33,393.56) .. controls (73.93,393.56) and (69.56,389.18) .. (69.56,383.78) -- cycle ;
			\draw    (123.67,276.22) -- (141,309) ;
			\draw    (141,309) -- (157.67,277.22) ;
			\draw    (144.67,337) -- (159.67,373.56) ;
			\draw    (159.67,373.56) -- (199.67,281.56) ;
			\draw   (122,321.78) .. controls (122,316.38) and (126.38,312) .. (131.78,312) .. controls (137.18,312) and (141.56,316.38) .. (141.56,321.78) .. controls (141.56,327.18) and (137.18,331.56) .. (131.78,331.56) .. controls (126.38,331.56) and (122,327.18) .. (122,321.78) -- cycle ;
			\draw   (141.56,321.78) .. controls (141.56,316.38) and (145.93,312) .. (151.33,312) .. controls (156.73,312) and (161.11,316.38) .. (161.11,321.78) .. controls (161.11,327.18) and (156.73,331.56) .. (151.33,331.56) .. controls (145.93,331.56) and (141.56,327.18) .. (141.56,321.78) -- cycle ;
			\draw   (140,383.78) .. controls (140,378.38) and (144.38,374) .. (149.78,374) .. controls (155.18,374) and (159.56,378.38) .. (159.56,383.78) .. controls (159.56,389.18) and (155.18,393.56) .. (149.78,393.56) .. controls (144.38,393.56) and (140,389.18) .. (140,383.78) -- cycle ;
			\draw   (159.56,383.78) .. controls (159.56,378.38) and (163.93,374) .. (169.33,374) .. controls (174.73,374) and (179.11,378.38) .. (179.11,383.78) .. controls (179.11,389.18) and (174.73,393.56) .. (169.33,393.56) .. controls (163.93,393.56) and (159.56,389.18) .. (159.56,383.78) -- cycle ;
			\draw    (218.67,276.22) -- (236,309) ;
			\draw    (236,309) -- (252.67,277.22) ;
			\draw    (239.67,337) -- (254.67,373.56) ;
			\draw    (254.67,373.56) -- (294.67,281.56) ;
			\draw   (217,321.78) .. controls (217,316.38) and (221.38,312) .. (226.78,312) .. controls (232.18,312) and (236.56,316.38) .. (236.56,321.78) .. controls (236.56,327.18) and (232.18,331.56) .. (226.78,331.56) .. controls (221.38,331.56) and (217,327.18) .. (217,321.78) -- cycle ;
			\draw   (236.56,321.78) .. controls (236.56,316.38) and (240.93,312) .. (246.33,312) .. controls (251.73,312) and (256.11,316.38) .. (256.11,321.78) .. controls (256.11,327.18) and (251.73,331.56) .. (246.33,331.56) .. controls (240.93,331.56) and (236.56,327.18) .. (236.56,321.78) -- cycle ;
			\draw   (235,383.78) .. controls (235,378.38) and (239.38,374) .. (244.78,374) .. controls (250.18,374) and (254.56,378.38) .. (254.56,383.78) .. controls (254.56,389.18) and (250.18,393.56) .. (244.78,393.56) .. controls (239.38,393.56) and (235,389.18) .. (235,383.78) -- cycle ;
			\draw   (254.56,383.78) .. controls (254.56,378.38) and (258.93,374) .. (264.33,374) .. controls (269.73,374) and (274.11,378.38) .. (274.11,383.78) .. controls (274.11,389.18) and (269.73,393.56) .. (264.33,393.56) .. controls (258.93,393.56) and (254.56,389.18) .. (254.56,383.78) -- cycle ;
			\draw  [fill={rgb, 255:red, 248; green, 231; blue, 28 }  ,fill opacity=0.33 ] (315.67,180.67) .. controls (315.67,125.62) and (360.29,81) .. (415.33,81) .. controls (470.38,81) and (515,125.62) .. (515,180.67) .. controls (515,235.71) and (470.38,280.33) .. (415.33,280.33) .. controls (360.29,280.33) and (315.67,235.71) .. (315.67,180.67) -- cycle ;
			\draw  [color={rgb, 255:red, 0; green, 0; blue, 0 }  ,draw opacity=1 ][fill={rgb, 255:red, 0; green, 0; blue, 0 }  ,fill opacity=1 ] (412.67,81) .. controls (412.67,79.53) and (413.86,78.33) .. (415.33,78.33) .. controls (416.81,78.33) and (418,79.53) .. (418,81) .. controls (418,82.47) and (416.81,83.67) .. (415.33,83.67) .. controls (413.86,83.67) and (412.67,82.47) .. (412.67,81) -- cycle ;
			\draw  [color={rgb, 255:red, 0; green, 0; blue, 0 }  ,draw opacity=1 ][fill={rgb, 255:red, 0; green, 0; blue, 0 }  ,fill opacity=1 ] (412.67,280.33) .. controls (412.67,278.86) and (413.86,277.67) .. (415.33,277.67) .. controls (416.81,277.67) and (418,278.86) .. (418,280.33) .. controls (418,281.81) and (416.81,283) .. (415.33,283) .. controls (413.86,283) and (412.67,281.81) .. (412.67,280.33) -- cycle ;
			
			\draw (35.62,216) node [anchor=north west][inner sep=0.75pt]   [align=left] {$\displaystyle 1$};
			\draw (54.72,216) node [anchor=north west][inner sep=0.75pt]   [align=left] {$\displaystyle 2$};
			\draw (74.5,216) node [anchor=north west][inner sep=0.75pt]   [align=left] {$\displaystyle 3$};
			\draw (46.9,116) node [anchor=north west][inner sep=0.75pt]   [align=left] {$\displaystyle 1$};
			\draw (68.63,116) node [anchor=north west][inner sep=0.75pt]   [align=left] {$\displaystyle 3$};
			\draw (58.35,95.04) node [anchor=north west][inner sep=0.75pt]   [align=left] {$\displaystyle 2$};
			\draw (123.9,116) node [anchor=north west][inner sep=0.75pt]   [align=left] {$\displaystyle 1$};
			\draw (145.63,116) node [anchor=north west][inner sep=0.75pt]   [align=left] {$\displaystyle 2$};
			\draw (135.35,95.04) node [anchor=north west][inner sep=0.75pt]   [align=left] {$\displaystyle 3$};
			\draw (115.62,216) node [anchor=north west][inner sep=0.75pt]   [align=left] {$\displaystyle 1$};
			\draw (134.72,216) node [anchor=north west][inner sep=0.75pt]   [align=left] {$\displaystyle 3$};
			\draw (154.5,216) node [anchor=north west][inner sep=0.75pt]   [align=left] {$\displaystyle 2$};
			\draw (193.62,216) node [anchor=north west][inner sep=0.75pt]   [align=left] {$\displaystyle 2$};
			\draw (212.72,216) node [anchor=north west][inner sep=0.75pt]   [align=left] {$\displaystyle 1$};
			\draw (232.5,216) node [anchor=north west][inner sep=0.75pt]   [align=left] {$\displaystyle 3$};
			\draw (36.62,315.23) node [anchor=north west][inner sep=0.75pt]   [align=left] {$\displaystyle 1$};
			\draw (55.72,315.55) node [anchor=north west][inner sep=0.75pt]   [align=left] {$\displaystyle 2$};
			\draw (73.72,376.55) node [anchor=north west][inner sep=0.75pt]   [align=left] {$\displaystyle 3$};
			\draw (126.62,315.23) node [anchor=north west][inner sep=0.75pt]   [align=left] {$\displaystyle 1$};
			\draw (145.72,315.55) node [anchor=north west][inner sep=0.75pt]   [align=left] {$\displaystyle 3$};
			\draw (163.72,376.55) node [anchor=north west][inner sep=0.75pt]   [align=left] {$\displaystyle 2$};
			\draw (221.62,315.23) node [anchor=north west][inner sep=0.75pt]   [align=left] {$\displaystyle 2$};
			\draw (240.72,315.55) node [anchor=north west][inner sep=0.75pt]   [align=left] {$\displaystyle 3$};
			\draw (258.72,3746.55) node [anchor=north west][inner sep=0.75pt]   [align=left] {$\displaystyle 1$};
			\draw (406,166) node [anchor=north west][inner sep=0.75pt]   [align=left] {$\displaystyle A_{1}$};
			\draw (406,58) node [anchor=north west][inner sep=0.75pt]   [align=left] {$\displaystyle C_{1}$};
			\draw (408,289) node [anchor=north west][inner sep=0.75pt]   [align=left] {$\displaystyle C_{2}$};
			\draw (294,170) node [anchor=north west][inner sep=0.75pt]   [align=left] {$\displaystyle B_{1}$};
			\draw (519,167) node [anchor=north west][inner sep=0.75pt]   [align=left] {$\displaystyle B_{3}$};
			\draw (53,139) node [anchor=north west][inner sep=0.75pt]   [align=left] {$\displaystyle C_{1}$};
			\draw (130,139) node [anchor=north west][inner sep=0.75pt]   [align=left] {$\displaystyle C_{2}$};
			\draw (52,240) node [anchor=north west][inner sep=0.75pt]   [align=left] {$\displaystyle B_{1}$};
			\draw (129,240) node [anchor=north west][inner sep=0.75pt]   [align=left] {$\displaystyle B_{2}$};
			\draw (209,240) node [anchor=north west][inner sep=0.75pt]   [align=left] {$\displaystyle B_{3}$};
			\draw (62,402) node [anchor=north west][inner sep=0.75pt]   [align=left] {$\displaystyle A_{1}$};
			\draw (150,402) node [anchor=north west][inner sep=0.75pt]   [align=left] {$\displaystyle A_{2}$};
			\draw (245,402) node [anchor=north west][inner sep=0.75pt]   [align=left] {$\displaystyle A_{3}$};

		\end{tikzpicture}
		
		\caption{In this picture we see the cell decomposition of $\overline{\M}_{0,4}$. On the left there are the cells. On the right it is depicted the cell $A_1$, together with its boundary.}
		\label{fig:esempio decomposizione cellulare con 3 punti}
	\end{figure}
	\begin{oss}\label{oss:celle non sono compatibili con l'operad}
		The regular cell decomposition of $\overline{\M}_{0,n+1}$ obtained in this way is not compatible with the operad structure. Indeed the operadic composition maps
		\[
		\circ_i:\overline{\M}_{0,n+1}\times \overline{\M}_{0,m+1}\to \overline{\M}_{0,m+n}
		\]
		are not cellular maps (in the domain we put the product cell structure). For example consider $\circ_1:\overline{\M}_{0,3}\times \overline{\M}_{0,3}\to \overline{\M}_{0,4}$:  its image is the stable curve associated to the nested tree $((12)3)$ which is not a $0$-cell of $\overline{\M}_{0,4}$ (actually it is the center of the $2$-cell $A_1$ of Figure \ref{fig:esempio decomposizione cellulare con 3 punti}).
	\end{oss}

	
	\section{An operad built from the dual cell decomposition}\label{sec:operad of dual cells}
	In this Section we present an operad of chain complexes which will be our chain model for the Hypercommutative operad.  The idea is the following: in Paragraph \ref{subsec: relations with the bar construction} we will identify (up to a shift in degrees) the operadic bar construction $B(grav)$ with the collection of chain complexes
	\[
	C_*^{cell}(\overline{\M})\coloneqq\{C_*^{cell}(\overline{\M}_{0,n+1}) \}_{n\geq 2}
	\]
	where $C_*^{cell}(\overline{\M}_{0,n+1})$ is the complex of cellular chains associated to the CW decomposition of Paragraph \ref{subsec: CW decomposition of Deligne mumford}.
	By definition $B(grav)$ is a cooperad, so we get a cooperad structure on $C_*^{cell}(\overline{\M})$. Dualizing, i.e. replacing chains with cochains we get an operad in cochain complexes
	\[
	C^*_{cell}(\overline{\M})\coloneqq\{C^*_{cell}(\overline{\M}_{0,n+1})\}_{n\geq 2}
	\]
	Since the Gravity and Hypercommutative operad are Koszul dual to each other (\cite{Getzler}, \cite{Ginzburg-Kapranov}) this will be a good candidate as a (co)chain model for the Hypercommutative operad. In this section we will expand a bit this observation, giving a more down to earth explanation of what we just said using Poincaré duality instead of Koszul duality: in the eyes of Poincaré duality, passing to cochains is the same as taking the dual cell decomposition. So, if $C_*^{dual}(\overline{\M}_{0,n+1})$ denotes the cellular chains respect to the dual cell decomposition, we expect that
	\[
	C_*^{dual}(\overline{\M})\coloneqq\{C_*^{dual}(\overline{\M}_{0,n+1})\}_{n\geq 2}
	\]
	has an operad structure, giving a chain model for the Hypercommutative operad. The fact that $C_*^{dual}(\overline{\M})$ is the linear dual of $B(grav)$ (up to shifts) shows that the Hypercommutative and Gravity operad are Koszul dual to each other also at the level of chains, improving the classical result of Getzler \cite{Getzler}. 
	
	\subsection{Relation with the bar construction}\label{subsec: relations with the bar construction}
	In Section \ref{sec: combinatorial models for Deligne-Mumford} we described a cell decomposition of $\overline{\M}_{0,n+1}$ where the cells are described by the following combinatorics:
	\begin{itemize}
		\item A nested cactus $\mathcal{S}$ with $n$-leaves.
		\item For each vertex $S\in \mathcal{S}$ we have a cell $C_S\subseteq\mathcal{C}_{\abs{S}}/S^1$.
	\end{itemize}
	If we want to highlight this combinatorial data we will write $(\mathcal{S},(C_S)_{S\in\mathcal{S}})$. 
	\begin{oss}
		These cells resemble a bar construction: elements of $C_*^{cell}(\overline{\M}_{0,n+1})$ are linear combinations of nested trees with $n$-leaves decorated by unbased cacti, and this is exactly what we would obtain after performing a bar construction to an \emph{operad of unbased cacti}. But we already encountered such an operad: indeed, in Section \ref{sec: cactus models} we saw that the 
		\[
		grav\coloneqq \{sC_*^{cell}(\mathcal{C}_n/S^1)\}_{n\geq 2}
		\]
		is a chain model for the Gravity operad. 
	\end{oss}
	After this observation, the next Proposition seems to be more plausible:
	\begin{prop}
		\label{prop:bar construction vs chains}
		$B(grav)(n)=s^2C_*^{cell}(\overline{\M}_{0,n+1})$ for any $n\geq 2$, where $B$ is the operadic bar construction.
	\end{prop}
	\begin{proof}
		By definition $B(grav)=(T^c(sgrav),d)$, where $T^c$ denotes the free co-operad on the $\Sigma$-sequence $sgrav\coloneqq\{s^2C_*^{cell}(\mathcal{C}_n/S^1)\}_{n\geq 2}$. Forgetting the differential for a moment, 
		\[
		B(grav)(n)=\bigoplus_{\mathcal{S}\in N_n}\bigotimes_{S\in \mathcal{S}}sgrav(\abs{S})=\bigoplus_{\mathcal{S}\in N_n}\bigotimes_{S\in \mathcal{S}}s^2C_*^{cell}(\mathcal{C}_{\abs{S}}/S^1)
		\]
		so we can think of an element in the chain complex $B(grav)(n)$ as a sum of nested trees decorated by unbased cacti whose dimension is raised by two. Now consider the following isomorphism of graded abelian groups:
		\begin{align*}
			f:C_*^{cell}(\overline{\M}_{0,n+1})&\to B(grav)(n)\\
			(\mathcal{S},(C_S)_{S\in\mathcal{S}})&\mapsto \otimes_{S\in\mathcal{S}}s^2C_S
		\end{align*}
		If $\tau= (\mathcal{S},(C_S)_{S\in\mathcal{S}})$ is a cell in $C_*^{cell}(\overline{\M}_{0,n+1})$, then
		\[
		dim(\tau)=2(V-1)+ \sum_{{S\in \mathcal{S}}}dim(C_S)
		\]
		where $V$ is the number of vertices of $\mathcal{S}$. The corresponding element $f(\tau)$ has degree
		\[
		deg (f(\tau))=2V+\sum_{S\in \mathcal{S}}dim(C_S)
		\]
		therefore $f$ is a degree two map and induces an isomorphism 
		\[
		B(grav)(n)\cong s^2C_*^{cell}(\overline{\M}_{0,n+1})
		\]
		of graded abelian groups. Finally, let us prove that this identification is indeed an isomorphism of chain complexes: the differential of $B(grav)(n)$ splits into two parts $d=d_1+d_2$: $d_1$ acts on an element $f(\tau)\in B(grav)(n)$ by performing the differential of $grav$ one vertex at the time, and then summing all the results; $d_2$ acts on $f(\tau)$ by contracting edges and composing the corresponding cacti. But this is exactly how we compute the boundary of the cell $\tau$ (see Paragraph \ref{subsec: CW decomposition of nested cacti}), and this proves the claim.
	\end{proof}
	
	\subsection{The dual cell decomposition}\label{subsec: dual cells}
	Suppose from now on that we have fixed some weights $(a_1,\dots a_n)\in\mathring{\Delta}^{n-1}$. In this way we can identify $\overline{\M}_{0,n+1}$ with the space of nested cacti $N_n^{\sigma}(\mathcal{C}/S^1)$ (Theorem \ref{thm:omeo tra la compattificazione e lo spazio dei nested cactus}) and get a regular CW-decomposition of $\overline{\M}_{0,n+1}$. In this paragraph we give an explicit construction of the dual cell decomposition. We will work on $N_n^{\sigma}(\mathcal{C}/S^1)$, and then use Theorem \ref{thm:omeo tra la compattificazione e lo spazio dei nested cactus} to define the dual cells on $\overline{\M}_{0,n+1}$.
	\begin{costr}[Dual cell decomposition] Let $X$ be a regular cell decomposition of a compact oriented $n$-dimensional manifold $M$. The dual cell decomposition $X^*$ of $M$ is constructed as follows:
		\begin{enumerate}
			\item For each cell $\tau\in X$ choose a point $B\tau$ in its interior, which we will call the \textbf{barycenter}.
			\item For any strictly increasing chain $\tau_0\subset\tau_1\subset\cdots\subset \tau_k$ of cells, we define inductively $(\tau_0,\dots,\tau_k)\subseteq \tau_k$ as follows: $(\tau_0)\coloneqq B\tau_0$. Suppose that $(\tau_0,\dots,\tau_{k-1})\subseteq \tau_{k-1}$ is defined, then $(\tau_0,\dots,\tau_{k})$ is will be the cone on $(\tau_0,\dots,\tau_{k-1})$ with vertex $B\tau_k$. Observe that $(\tau_0,\dots,\tau_k)$ is homeomorphic to the $k$-simplex.
			\item Given a cell $\tau$, the dual cell $\tau^*$ is defined as the union of all simplices $(\tau_0,\dots,\tau_k)$ with $\tau_0=\tau$, $k\in\N$.
		\end{enumerate}
	\end{costr}
	In the case of $N_n^{\sigma}(\mathcal{C}/S^1)$ we have a natural choice for the barycenters:
	\begin{defn}
		Let $\tau=(\mathcal{S},(C_S)_{S\in\mathcal{S}})$ be a cell of $N_n^{\sigma}(\mathcal{C}/S^1)$. The \textbf{barycenter} $B\tau=(x_R,(t_S,\alpha_S,x_S)_{S\in\mathcal{S}-R})$ is the point specified by the following parameters:
		\begin{itemize}
			\item For any $S\in\mathcal{S}-R$ put $t_S=0$. Note that since the radial parameter is zero we do not have to specify any angular parameter.
			\item For any $S\in\mathcal{S}$,  $x_S$ will be the barycenter of $C_S$ ( $C_S$ is a product of simplices, so its barycenter is just the product of the barycenters of each simplex).
		\end{itemize}
	\end{defn}
	Now that we have chosen the barycenters we can construct the dual cells as just described. Let us prove some properties of the dual cells:
	\begin{defn}
		Fix a nested tree $\mathcal{S}$ with $n$-leaves. We will indicate with $\overline{N_n(\mathcal{S})}$ the following subspace of $N_n^{\sigma}(\mathcal{C}/S^1)$:
		\[
		\overline{N_n(\mathcal{S})}\coloneqq\{ (x_R,(t_T,\alpha_T,x_T)_{T\in\mathcal{T}-R})\in N_n^{\sigma}(\mathcal{C}/S^1))\mid \mathcal{S}\subseteq\mathcal{T} \text{ and } t_S=0 \text{ for all } S\in\mathcal{S}\}
		\]
	\end{defn}
	\begin{oss}
		Under the homeomorphism of Theorem \ref{thm:omeo tra la compattificazione e lo spazio dei nested cactus} the subspace $\overline{N_n(\mathcal{S})}$ corresponds to the closure of the stratum $\mathcal{M}(\mathcal{S})$.
	\end{oss}
	\begin{lem}\label{lem: celle duali vs strati1}
		Let $\tau=(\mathcal{S},(C_S)_{S\in\mathcal{S}})$ be a cell of $N_n^{\sigma}(\mathcal{C}/S^1)$. Then $\tau^*\subseteq \overline{N_n(\mathcal{S})}$.
	\end{lem}
	\begin{proof}
		By the construction  of $\tau^*$ it suffices to show that any simplex of the form $(\tau,\tau_1,\dots,\tau_k)$ is contained in $\overline{N_n(\mathcal{S})}$. We proceed by induction on $k$:
		\begin{description}
			\item[$k=0$] in this case we have the zero simplex $(\tau)\coloneqq B\tau$, which is a point of $\overline{N_n(\mathcal{S})}$ by construction.
			\item[$k\geq 1$] by definition $(\tau,\tau_1,\dots,\tau_k)$ is the cone on $(\tau,\tau_1,\dots,\tau_{k-1})$ with vertex $B\tau_k$. Since $\tau\subseteq \tau_k$, the cell $\tau_k$ will be of the form $\tau_k=(\mathcal{T}_k,(C'_T)_{T\in\mathcal{
					T}_k})$ for some nested tree $\mathcal{T}_k\supseteq\mathcal{S}$. The barycenter of $\tau_k$ is then $B\tau_k=(b_R,(0,b_T)_{T\in\mathcal{T}_k-R})$, where $b_T$ is the barycenter of $C'_T$.  Now take a point $p\in(\tau,\dots,\tau_{k-1})$, where $(\tau,\dots,\tau_{k-1})$ is seen as a subspace of $\partial\tau_k$. Then $p$ will be of the form 
			\[
			p=(x_R,(t_T,\alpha_T,x_T)_{T\in\mathcal{T}_k-R})
			\]
			By induction $(\tau,\tau_1,\dots,\tau_{k-1})\subseteq\overline{N_n(\mathcal{S})}$, so $t_S=0$ for all $S\in\mathcal{S}$. The segment connecting $B\tau_k$ and $p$ interpolates between the coordinates of $B\tau_k$ and $p$ which are different from each other. Since also $B\tau_k$ has $t_S=0$ for all $S\in\mathcal{S}$, any point on such a segment will have $t_S=0$ for any $S\in\mathcal{S}$. Finally we observe that $(\tau,\tau_1,\dots,\tau_k)$ is the union of all such segments, so we get the statement. 
		\end{description}
	\end{proof}
	A converse statement also holds:
	\begin{lem}\label{lem: celle duali vs strati2}
		Fix a nested tree $\mathcal{S}$ and let $x\in \overline{N_n(\mathcal{S})}$. If $\tau^*$ is the unique dual cell containing $x$, then $\mathcal{S}\subseteq tree(\tau)$, where $tree(\tau)$ is the nested tree underlying the cell $\tau$.
	\end{lem}
	\begin{proof}
		Since $x\in \overline{N_n(\mathcal{S})}$ it has the form $x=(x_R,(t_T,\alpha_T,x_T)_{T\in\mathcal{T}-R})$ for some $\mathcal{T}\supseteq\mathcal{S}$, and $t_S=0$ for any $S\in\mathcal{S}$. Without loss of generality we can suppose that all the radial parameters are different from $1$. Let $\tau_x$ be the unique (open) cell containing $x$. Then $\tau_x=(\mathcal{T},(C_T)_{T\in\mathcal{T}})$, with $x_T\in C_T$ for any $T\in\mathcal{T}$. We prove the statement by induction on $k\coloneqq \dim(\tau_x)$:
		\begin{description}
			\item[$k=0$] in this case $\tau_x$ is just the corolla with $n$-leaves decorated by a cactus $C_R$ of dimension zero. Since $\mathcal{S}\subseteq\mathcal{T}$, we have that also $\mathcal{S}$ is the corolla and the statement becomes trivial.
			\item[$k\geq 1$] If $x=B\tau_x$ then $x\in\tau_x^*$ and the statement follows. Now suppose $x\neq B\tau_x$: in this case there must be a radial parameter $t_{T_0}\neq 0$, for some $T_0\in\mathcal{T}$, $T_0\notin\mathcal{S}$. Consider the point $y\in\partial \tau_x$ which has the same coordinates as $x$ except for the radial parameter of $T_0$, which we set to be $1$. Let $\tau_y\subseteq \partial \tau_x$ be the unique (open) cell containing $y$ and $\sigma^*$ be the unique (open) dual cell containing $y$. By induction the nested tree $tree(\sigma)$ contains $\mathcal{S}$. Now let $(\sigma,\sigma_1,\dots,\sigma_m)$ be the unique simplex containing $y$ in its interior. Then $y\in\sigma_m$, but since there is a unique (open) cell containing $y$ we conclude that $\sigma_m=\tau_y\subseteq \partial \tau_x$. Therefore $x\in(\sigma,\sigma_1,\dots,\sigma_m,\tau_x)$, so $x\in\sigma^*$ as well concluding the proof.
		\end{description}
	\end{proof}
	\begin{prop}\label{prop:celle duali che formano uno strato}
		Fix a nested tree $\mathcal{S}$. Then  
		\[
		\overline{N_n(\mathcal{S})}=\bigcup_{\substack{\tau \text{ cell} \\ tree(\tau)\supseteq \mathcal{S}}}\tau^*
		\]
	\end{prop}
	In other words this Proposition tells us that if we put the dual cell decomposition on $N_n^{\sigma}(\mathcal{C}/S^1)$, then $\overline{N_n(\mathcal{S})}$ becomes a subcomplex.
	\begin{proof}
		Fix a cell $\tau$ such that $\mathcal{T}\coloneqq tree(\tau)\supseteq \mathcal{S}$. By Lemma \ref{lem: celle duali vs strati1} we have that $\tau^*\subseteq \overline{N_n(\mathcal{T})}$. This space is contained in $\overline{N_n(\mathcal{S})}$ because $\mathcal{T}\supseteq \mathcal{S}$. Therefore
		\[
		\bigcup_{\substack{\tau \text{ cell} \\ tree(\tau)\supseteq \mathcal{S}}}\tau^*\subseteq  \overline{N_n(\mathcal{S})}
		\]
		The other inclusion is given by Lemma \ref{lem: celle duali vs strati2}.
	\end{proof}
	\begin{oss}
		$\overline{N_n(\mathcal{S})}$ is a subcomplex of dimension $2(n-2)-2E(\mathcal{S})$. 
	\end{oss}
	\begin{es}[Dual cell decomposition of $\overline{\M}_{0,4}$]
		In Paragraph \ref{subsec: CW decomposition of Deligne mumford} we described in detail the CW-decomposition of $\overline{\M}_{0,4}$ given by nested cacti: we had two $0$-cells $C_1,C_2$, three $1$-cells $B_1,B_2,B_3$ and three $2$-cells $A_1,A_2,A_3$ (see Figure \ref{fig:esempio decomposizione cellulare con 3 punti}). The dual cell decomposition will have three $0$-cells $A_1^*,A_2^*,A_3^*$, (corresponding to the three stable curves), three $1$-cells $B_1^*,B_2^*,B_3^*$ and two $2$-cells $C_1^*,C_2^*$. See Figure \ref{decomposizioni della sfera } for a picture.
		\begin{figure}
			\centering
			\includegraphics[width=6 cm]{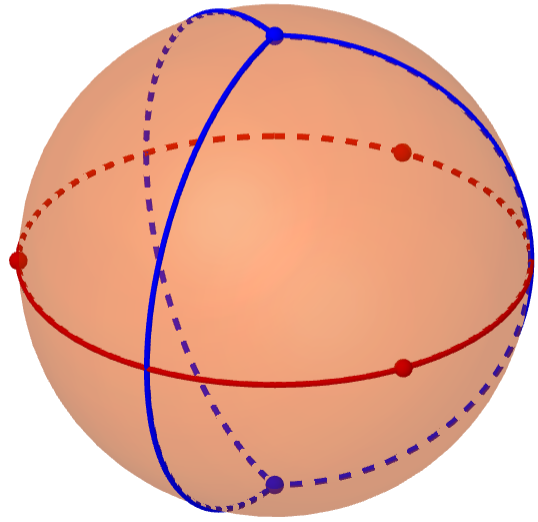}
			\caption{In blue we see the CW-decomposition of $\overline{\M}_{0,4}$ given by nested cacti. In red we see the dual cells. }
			\label{decomposizioni della sfera }
		\end{figure}
	\end{es}
	
	\subsection{Operad stucture on the dual cells}\label{subsec: operad structure on the dual cells}
	In the previous paragraph we constructed explicitly the dual cell decomposition of $N_n^{\sigma}(\mathcal{C}/S^1)\cong \overline{\M}_{0,n+1}$. We will denote by $C_*^{dual}(\overline{\M})$ the collection of chain complexes 
	\[
	C_*^{dual}(\overline{\M})\coloneqq\{C_*^{dual}(N_n(\mathcal{C}/S^1))\}_{n\geq 2}
	\]
	where $C_*^{dual}(N_n(\mathcal{C}/S^1))$ is the cellular chain complex of the dual cell decomposition. 	The purpose of this paragraph is to define an operad structure on  $C^{dual}_*(\overline{\M})$. This operad structure exists already at the topological level: indeed we are going to define operadic compositions $\circ_i:N_{n}^{\sigma}(\mathcal{C}/S^1)\times N_{m}^{\sigma}(\mathcal{C}/S^1)\to N_{n+m-1}^{\sigma}(\mathcal{C}/S^1)$ that are cellular maps if we put the dual cell decomposition on $N_{n}^{\sigma}(\mathcal{C}/S^1)$ for any $n\in\N$. If we take the cellular chain complexes we will obtain an operad structure on $C^{dual}_*(\overline{\M})$. 
	\begin{defn}
		Let $\mathcal{P}$ be a topological operad. If there are CW-decompositions of $\mathcal{P}(n)$, $n\in\N$ such that the operadic compositions 
		\[
		\circ_i:\mathcal{P}(n)\times \mathcal{P}(m)\to \mathcal{P}(n+m-1)
		\]
		are cellular maps, then we will call $\mathcal{P}$ a \textbf{cellular operad}.
	\end{defn}
	\begin{oss}
		Let us recall briefly the combinatorics of the dual cells on $N_n^{\sigma}(\mathcal{C}/S^1)$: a (dual) cell $\tau^*$ is determined by:
		\begin{enumerate}
			\item A nested tree $\mathcal{S}$ on $n$-leaves.
			\item A cell $C_S\subseteq \mathcal{C}_{\abs{S}}/S^1$ for any $S\in\mathcal{S}$.
		\end{enumerate}
		In the case we want to highlight this data we will write $\tau^*=(\mathcal{S},(C_S)_{S\in\mathcal{S}})$. Such a cell has dimension
		\[
		dim(\tau^*)=2(n-2)-2E(\mathcal{S})- \sum_{{S\in \mathcal{S}}}dim(C_S)
		\]
		where $E(\mathcal{S})$ is the number of internal edges of $\mathcal{S}$. The boundary of $\tau^*=(\mathcal{S},(C_S)_{S\in\mathcal{S}})$ is the sum (with signs depending on the orientations) of all the codimension one cells obtained from $\tau^*$ by one of the following moves:
		\begin{description}
			\item[Move 1 (cacti co-boundary)] for a fixed vertex $S\in\mathcal{S}$ replace $C_S$ with a cell $C'_S$ such that $C_S\subseteq \partial C'_S$.
			\item[Move 2 (vertex expansion)] for a fixed vertex $S_0\in\mathcal{S}$, chose a subset $S_1\subseteq S_0 $ with at least two elements such that $\mathcal{S}\cup S_1$ is still a nested tree. Then consider the following cell $\sigma^*=(\mathcal{T},(C'_T)_{T\in\mathcal{T}})$: 
			\begin{itemize}
				\item As underlying nested tree we take $\mathcal{T}\coloneqq\mathcal{S}\cup S_1$.
				\item As decorative cacti we put $C'_T=C_T$ for $T\neq S_0,S_1$. $C'_{S_0}$ and $C'_{S_1}$ will be two cacti such that $C_{S_0}$ is obtained from $C'_{S_0}$ gluing $C'_{S_1}$ in the lobe corresponding to the edge $(S_1,S_0)$. 
			\end{itemize}
		\end{description}
		
	\end{oss}
	
	\begin{defn}
		Let $\mathcal{S}\in N_n$ and $\mathcal{T}\in N_m$. For any $i=1,\dots,n$ we define $\mathcal{S}\circ_i\mathcal{T}\in N_{n+m-1}$ as follows:
		\begin{itemize}
			\item  For any $T\in\mathcal{T}$,  $\mathcal{S}\circ_i\mathcal{T}$ contains $T+i\coloneqq\{t+i-1\mid t\in T\}$.
			\item If $S\in\mathcal{S}$ does not contain $i$, then $\mathcal{S}\circ_i\mathcal{T}$ contains 
			\[
			\overline{S}\coloneqq\{s\in S\mid s<i\}\cup\{s+m-1\mid s\in S, s\geq i+1\}
			\]
			\item If $S\in\mathcal{S}$ is a vertex containing $i$, then $\mathcal{S}\circ_i\mathcal{T}$ contains 
			\[
			\overline{S}\coloneqq\{s\in S\mid s<i\}\cup\{i,\dots,i+m-1\}\cup\{s+m-1\mid s\in S, s\geq i+1\}
			\]
		\end{itemize}
		Intuitively $\mathcal{S}\circ_i\mathcal{T}$ is obtained by grafting $\mathcal{T}$ to the $i$-th leaf of $\mathcal{S}$. See Figure \ref{fig:grafting of trees} for an example.        
	\end{defn}
	Now we are ready to define an operad structure on the collection of topological spaces $N(\mathcal{C}/S^1)\coloneqq\{N_n^{\sigma}(\mathcal{C}/S^1)\}_{n\geq 2}$. 
	\begin{defn}
		Consider the map
		\[
		\circ_i:N_n^{\sigma}(\mathcal{C}/S^1)\times N_m^{\sigma}(\mathcal{C}/S^1)\to N_{n+m-1}^{\sigma}(\mathcal{C}/S^1)
		\]
		which sends a pair of nested cactus $x\times x'$ to the nested cactus obtained by grafting the root of $x'$ to the $i$-th leaf of $x$, setting the radial parameter of the new edge to be zero. More precisely, if $x=(x_R,(t_S,\alpha_S,x_S)_{S\in \mathcal{S}-R})$ and	$x'=(x'_R,(t'_T,\alpha'_T,x'_T)_{T\in \mathcal{T}-R})$ then $x\circ_ix'$ is specified by the following data: 
		\begin{enumerate}
			\item As underlying nested tree we put $\mathcal{S}\circ_i\mathcal{T}$.
			\item Given $U\in \mathcal{S}\circ_i\mathcal{T}$, the unbased cactus labelling this vertex is given by:
			\begin{itemize}
				\item $x'_T$ if $U=T+i$ for some $T\in\mathcal{T}$.
				\item $x_S$ if $U=\overline{S}$ for some $S\in\mathcal{S}$.
				
			\end{itemize}
			\item Given $U\in \mathcal{S}\circ_i\mathcal{T}$ which is not the root, the radial and angular parameters are given by:
			\begin{itemize}
				\item $t'_T$ and $\alpha'_T$ if $U=T+i$ for some $T\in\mathcal{T}$. In the case $U=R+i$ (where $R$ is the root of $\mathcal{T}$) we set the radial parameter to be $0$ (and therefore we do not have to specify an angular parameter).
				\item $t_S$ and $\alpha_S$ if $U=\overline{S}$ for some $S\in\mathcal{S}$.
			\end{itemize}
		\end{enumerate}
		These maps defines an operad structure on the collection of topological spaces $N(\mathcal{C}/S^1)\coloneqq \{N_n^{\sigma}(\mathcal{C}/S^1)\}_{n\geq 2}$. 
	\end{defn}
	Now we prove that if we put the dual cell structure on $N_n^{\sigma}(\mathcal{C}/S^1)$ these maps are cellular.
	\begin{figure}
		\centering
		
		\tikzset{every picture/.style={line width=0.75pt}} 
		
		\begin{tikzpicture}[x=0.75pt,y=0.75pt,yscale=-1,xscale=1]
			
			\draw    (53.23,139.68) -- (90.23,188.68) ;
			\draw    (119.23,139.68) -- (90.23,188.68) ;
			\draw    (142.23,108.68) -- (119.23,139.68) ;
			\draw    (119.23,104.68) -- (119.23,139.68) ;
			\draw    (99.23,107.68) -- (119.23,139.68) ;
			\draw    (67.23,108.68) -- (53.23,139.68) ;
			\draw    (37.23,108.68) -- (53.23,139.68) ;
			\draw    (293.23,140.68) -- (330.23,189.68) ;
			\draw    (359.23,140.68) -- (330.23,189.68) ;
			\draw    (382.23,109.68) -- (359.23,140.68) ;
			\draw    (359.23,105.68) -- (359.23,140.68) ;
			\draw    (339.23,108.68) -- (359.23,140.68) ;
			\draw    (307.23,109.68) -- (293.23,140.68) ;
			\draw    (277.23,109.68) -- (293.23,140.68) ;
			\draw    (201.23,157.68) -- (219.73,182.18) ;
			\draw    (236.23,148.68) -- (219.73,182.18) ;
			\draw    (215.23,126.68) -- (201.23,157.68) ;
			\draw    (185.23,126.68) -- (201.23,157.68) ;
			\draw    (320.73,84.18) -- (339.23,108.68) ;
			\draw    (355.73,75.18) -- (339.23,108.68) ;
			\draw    (334.73,53.18) -- (320.73,84.18) ;
			\draw    (304.73,53.18) -- (320.73,84.18) ;
			\draw  [fill={rgb, 255:red, 0; green, 0; blue, 0 }  ,fill opacity=1 ] (88.11,188.68) .. controls (88.11,187.51) and (89.06,186.56) .. (90.23,186.56) .. controls (91.39,186.56) and (92.34,187.51) .. (92.34,188.68) .. controls (92.34,189.84) and (91.39,190.79) .. (90.23,190.79) .. controls (89.06,190.79) and (88.11,189.84) .. (88.11,188.68) -- cycle ;
			\draw  [fill={rgb, 255:red, 0; green, 0; blue, 0 }  ,fill opacity=1 ] (51.11,139.68) .. controls (51.11,138.51) and (52.06,137.56) .. (53.23,137.56) .. controls (54.39,137.56) and (55.34,138.51) .. (55.34,139.68) .. controls (55.34,140.84) and (54.39,141.79) .. (53.23,141.79) .. controls (52.06,141.79) and (51.11,140.84) .. (51.11,139.68) -- cycle ;
			\draw  [fill={rgb, 255:red, 0; green, 0; blue, 0 }  ,fill opacity=1 ] (117.11,139.68) .. controls (117.11,138.51) and (118.06,137.56) .. (119.23,137.56) .. controls (120.39,137.56) and (121.34,138.51) .. (121.34,139.68) .. controls (121.34,140.84) and (120.39,141.79) .. (119.23,141.79) .. controls (118.06,141.79) and (117.11,140.84) .. (117.11,139.68) -- cycle ;
			\draw  [fill={rgb, 255:red, 0; green, 0; blue, 0 }  ,fill opacity=1 ] (199.11,157.68) .. controls (199.11,156.51) and (200.06,155.56) .. (201.23,155.56) .. controls (202.39,155.56) and (203.34,156.51) .. (203.34,157.68) .. controls (203.34,158.84) and (202.39,159.79) .. (201.23,159.79) .. controls (200.06,159.79) and (199.11,158.84) .. (199.11,157.68) -- cycle ;
			\draw  [fill={rgb, 255:red, 0; green, 0; blue, 0 }  ,fill opacity=1 ] (217.61,182.18) .. controls (217.61,181.01) and (218.56,180.06) .. (219.73,180.06) .. controls (220.89,180.06) and (221.84,181.01) .. (221.84,182.18) .. controls (221.84,183.34) and (220.89,184.29) .. (219.73,184.29) .. controls (218.56,184.29) and (217.61,183.34) .. (217.61,182.18) -- cycle ;
			\draw  [fill={rgb, 255:red, 0; green, 0; blue, 0 }  ,fill opacity=1 ] (291.11,140.68) .. controls (291.11,139.51) and (292.06,138.56) .. (293.23,138.56) .. controls (294.39,138.56) and (295.34,139.51) .. (295.34,140.68) .. controls (295.34,141.84) and (294.39,142.79) .. (293.23,142.79) .. controls (292.06,142.79) and (291.11,141.84) .. (291.11,140.68) -- cycle ;
			\draw  [fill={rgb, 255:red, 0; green, 0; blue, 0 }  ,fill opacity=1 ] (328.11,189.68) .. controls (328.11,188.51) and (329.06,187.56) .. (330.23,187.56) .. controls (331.39,187.56) and (332.34,188.51) .. (332.34,189.68) .. controls (332.34,190.84) and (331.39,191.79) .. (330.23,191.79) .. controls (329.06,191.79) and (328.11,190.84) .. (328.11,189.68) -- cycle ;
			\draw  [fill={rgb, 255:red, 0; green, 0; blue, 0 }  ,fill opacity=1 ] (357.11,140.68) .. controls (357.11,139.51) and (358.06,138.56) .. (359.23,138.56) .. controls (360.39,138.56) and (361.34,139.51) .. (361.34,140.68) .. controls (361.34,141.84) and (360.39,142.79) .. (359.23,142.79) .. controls (358.06,142.79) and (357.11,141.84) .. (357.11,140.68) -- cycle ;
			\draw  [fill={rgb, 255:red, 0; green, 0; blue, 0 }  ,fill opacity=1 ] (337.11,108.68) .. controls (337.11,107.51) and (338.06,106.56) .. (339.23,106.56) .. controls (340.39,106.56) and (341.34,107.51) .. (341.34,108.68) .. controls (341.34,109.84) and (340.39,110.79) .. (339.23,110.79) .. controls (338.06,110.79) and (337.11,109.84) .. (337.11,108.68) -- cycle ;
			\draw  [fill={rgb, 255:red, 0; green, 0; blue, 0 }  ,fill opacity=1 ] (318.61,84.18) .. controls (318.61,83.01) and (319.56,82.06) .. (320.73,82.06) .. controls (321.89,82.06) and (322.84,83.01) .. (322.84,84.18) .. controls (322.84,85.34) and (321.89,86.29) .. (320.73,86.29) .. controls (319.56,86.29) and (318.61,85.34) .. (318.61,84.18) -- cycle ;
			
			\draw (30,88) node [anchor=north west][inner sep=0.75pt]   [align=left] {$\displaystyle 1$};
			\draw (62,88) node [anchor=north west][inner sep=0.75pt]   [align=left] {$\displaystyle 2$};
			\draw (93,88) node [anchor=north west][inner sep=0.75pt]   [align=left] {$\displaystyle 3$};
			\draw (114,87) node [anchor=north west][inner sep=0.75pt]   [align=left] {$\displaystyle 4$};
			\draw (137,90) node [anchor=north west][inner sep=0.75pt]   [align=left] {$\displaystyle 5$};
			\draw (271,89) node [anchor=north west][inner sep=0.75pt]   [align=left] {$\displaystyle 1$};
			\draw (303,91) node [anchor=north west][inner sep=0.75pt]   [align=left] {$\displaystyle 2$};
			\draw (298,35) node [anchor=north west][inner sep=0.75pt]   [align=left] {$\displaystyle 3$};
			\draw (329,36) node [anchor=north west][inner sep=0.75pt]   [align=left] {$\displaystyle 4$};
			\draw (153,133) node [anchor=north west][inner sep=0.75pt]   [align=left] {$\displaystyle \circ _{3}$};
			\draw (256,133) node [anchor=north west][inner sep=0.75pt]   [align=left] {$\displaystyle =$};
			\draw (379.73,90.18) node [anchor=north west][inner sep=0.75pt]   [align=left] {$\displaystyle 7$};
			\draw (211,108) node [anchor=north west][inner sep=0.75pt]   [align=left] {$\displaystyle 2$};
			\draw (234,129) node [anchor=north west][inner sep=0.75pt]   [align=left] {$\displaystyle 3$};
			\draw (178,108) node [anchor=north west][inner sep=0.75pt]   [align=left] {$\displaystyle 1$};
			\draw (351,57) node [anchor=north west][inner sep=0.75pt]   [align=left] {$\displaystyle 5$};
			\draw (355,86) node [anchor=north west][inner sep=0.75pt]   [align=left] {$\displaystyle 6$};

		\end{tikzpicture}

		\caption{In this picture we see the grafting $\mathcal{S}\circ_3\mathcal{T}$, where $\mathcal{S}=((12)(345))$ and $\mathcal{T}=((12)3)$.}
		\label{fig:grafting of trees}
	\end{figure}
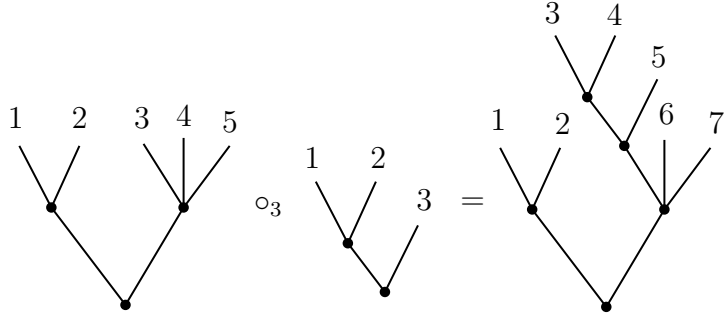
	\begin{thm}\label{thm: cellular operad of nested cacti}
		Consider the dual cell structure on $N_n^{\sigma}(\mathcal{C}/S^1)$ for any $n\in\N$. Then the maps
		\[
		\circ_i: N_n^{\sigma}(\mathcal{C}/S^1)\times N_m^{\sigma}(\mathcal{C}/S^1)\to N_{n+m-1}^{\sigma}(\mathcal{C}/S^1)
		\]
		are cellular embeddings. Moreover if $\tau^*\subseteq N_n^{\sigma}(\mathcal{C}/S^1)$ and $\sigma^*\subseteq N_m^{\sigma}(\mathcal{C}/S^1)$ are two dual cells, with $\tau^*=(\mathcal{S},(C_S)_{S\in\mathcal{S}})$ and $\sigma^*=(\mathcal{T},(C'_T)_{T\in\mathcal{T}})$, then $\tau^*\circ_i\sigma^*$ is specified by the following combinatorial data:
		\begin{enumerate}
			\item As nested tree we put $\mathcal{S}\circ_i\mathcal{T}$.
			\item Given $U\in \mathcal{S}\circ_i\mathcal{T}$, the decorative cactus of this vertex is given by:
			\begin{itemize}
				\item $C'_T$ if $U=T+i$ for some $T\in\mathcal{T}$.
				\item $C_S$ if $U=\overline{S}$ for some $S\in\mathcal{S}$.
			\end{itemize}
		\end{enumerate}
		
	\end{thm}
	\begin{proof}
		First of all observe that $\circ_i$ are embeddings because they are injective maps with compact domain and Haursdorff target. We prove the rest of the statement by induction on the dimension $k$ of the skeleton. In what follows we will denote by Greek letters the cells of the space of nested cacti, while the dual cells are denoted by Greek letters with an asterisk. Moreover if $\tau=(\mathcal{S},(C_S)_{S\in\mathcal{S}})$ and $\sigma=(\mathcal{T},(C'_T)_{T\in\mathcal{T}})$, are two cells (dual or not), we will denote by $\tau\circ_i^c\sigma$ the cell given by:
		\begin{enumerate}
			\item The nested tree $\mathcal{S}\circ_i\mathcal{T}$.
			\item Given $U\in \mathcal{S}\circ_i\mathcal{T}$, the decorative cactus of this vertex is given by:
			\begin{itemize}
				\item $C'_T$ if $U=T+i$ for some $T\in\mathcal{T}$.
				\item $C_S$ if $U=\overline{S}$ for some $S\in\mathcal{S}$.
			\end{itemize}
		\end{enumerate}
		\begin{description}
			\item[$k=0$] let $\tau^*\times \sigma^*$ be a zero dimensional cell in $N_n^{\sigma}(\mathcal{C}/S^1)\times N_m^{\sigma}(\mathcal{C}/S^1)$. In this case $\tau^*=B(\tau)$ and $\sigma^*=B(\sigma)$, where $\tau$ and $\sigma$ are top dimensional cells. Therefore the underlying nested trees of $\tau$ and $\sigma$ are binary trees. Now observe that the composition of barycenters is a barycenter, i.e. $B\tau\circ_iB\sigma$ is the barycenter of the cell $\gamma\coloneqq \tau\circ_i^c\sigma$. Now observe that $\mathcal{S}\circ_i\mathcal{T}$ is a binary tree with $n+m-1$ leaves, therefore $\gamma$ is a top dimensional cell inside $N_{n+m-1}^{\sigma}(\mathcal{C}/S^1)$. It follows the the dual cell  $\gamma^*$ is zero dimensional, thus it coincides with its barycenter. Therefore
			\begin{equation*}
				\tau^*\circ_i\sigma^*=B\tau\circ_iB\sigma=B\gamma=\gamma^*
			\end{equation*}
			is a zero dimensional cell, and this proves the claim for $k=0$.
			\item[$k\geq 1$] supposing that the statement is true up to the $k$-skeleton, let us prove it for the $k+1$-skeleton. Let  $\tau^*\times \sigma^*$ be a $k+1$ dimensional cell of $N_n^{\sigma}(\mathcal{C}/S^1)\times N_m^{\sigma}(\mathcal{C}/S^1)$. By construction this cell is the cone on $\partial (\tau^*\times \sigma^*)$ with vertex $B\tau\times B\sigma$. By induction we know that $\circ_i$ sends $\partial (\tau^*\times \sigma^*)$ to  $\partial((\tau\circ_i^c\sigma)^*)$. Moreover $B\tau\circ_iB\sigma$ is the barycenter of $\tau\circ_i^c\sigma$. It follows that $\tau^*\circ_i \sigma^*$ is the cone on $\partial((\tau\circ_i^c\sigma)^*)$ with vertex $B(\tau\circ_i^c\sigma)$. Therefore $\tau^*\circ_i\sigma^*=(\tau\circ_i^c\sigma)^*$.
		\end{description}
	\end{proof}
	
	\section{A chain model for the Hypercommutative operad}\label{sec: chain model for hycomm}
	
	In this section we prove that $C^{dual}_*(\overline{\M})$ is a chain model for the Hypercommutative operad, i.e. that there is a zig-zag of operadic quasi isomoprhisms between $C^{dual}_*(\overline{\M})$ and the operad of singular chains $C_*(\overline{\M})$. 
	
	The strategy of the proof will be the following: let $W$ be the Boardman-Vogt construction for topological operads \cite{Boardman-Vogt}. By Theorem \ref{thm: cellular operad of nested cacti} the operad $N(\mathcal{C}/S^1)$ is cellular. This implies that $WN(\mathcal{C}/S^1)$ is a cellular operad as well. We will prove that $WN(\mathcal{C}/S^1)$ is isomorphic (as operad) to $W\overline{\M}$, so we can give a cellular operad structure to $W\overline{\M}$. In this case the operad of singular chains $C_*(W\overline{\M})$ and that of cellular chains $C_*^{cell}(W\overline{\M})$ are quasi-isomorphic, therefore we can construct a zig-zag of weak equivalences
	\[
	\begin{tikzcd}[column sep=small]
		& C^{dual}_*(\overline{\M}) &C^{cell}_*(WN(\mathcal{C}/S^1))\cong C_*^{cell}(W\overline{\M}) \arrow[l] \arrow[r,"\sim"] & C_*(W\overline{\M}) \arrow[r,"p_*"] &C_*(\overline{\M})
	\end{tikzcd}
	\]
	where the first and last arrows are induced by the retractions $W\mathcal{P}\to \mathcal{P}$. 
	\subsection{A cellular decomposition of $W\overline{\M}$}
	In what follows we will think of an element of $W\overline{\M}$ as a nested tree with vertices labelled by elements of $\overline{\M}$, and internal edges with length in $[0,1]$, modulo the relation that collapses an internal edge of length $0$ and composes the labels of its vertices. We are going to define a cellular operad structure on $W\overline{\M}$ following this strategy: we will construct an isomorphism of operads 
	\[
	\phi:W\overline{\M}\to WN(\mathcal{C}/S^1)
	\]
	where $N(\mathcal{C}/S^1)$ is the operad of nested cacti defined in the Paragraph \ref{subsec: operad structure on the dual cells}. Since $N(\mathcal{C}/S^1)$ is a cellular topological operad we get that $WN(\mathcal{C}/S^1)$ is also a cellular topological operad. Using the isomorphism $\phi$ we will get that $W\overline{\M}$ is cellular as well. Before going on we need the following definition:
	\begin{defn}
		We define an operad map $c:WCom\to \mathring{\Delta}$, where $\mathring{\Delta}$ is the open-simplex operad of Arone and Kankaarinta \cite{Arone}. We consider an element of $WCom(n)$ as a nested tree with $n$-leaves and edges labelled by numbers between $0$ and $1$. We call these numbers the lengths of the edges. For each $n$ let $P_n$ be the corolla with $n$-leaves. We set $c(P_n)\coloneqq b_n$, where $b_n\coloneqq(1/n,\dots,1/n)$ is the baricenter of $\Delta^{n-1}$. Using the operadic composition $c$ is uniquely determined by its value on trees $\mathcal{T}\in WCom(n)$ such that any internal edge of $\mathcal{T}$ has length strictly less than $1$. We define $c$ on such a tree as follows: let $s$ the maximum of the lengths of the internal edges of $\mathcal{T}$ and let $T/s$ denote the tree obtained from $\mathcal{T}$ by dividing the lengths of its edges by $s$. We then define $c(\mathcal{T})$ inductively by the formula
		\[
		c(\mathcal{T})\coloneqq s\cdot c(T/s)+(1-s)b_n
		\]
	\end{defn}
	\begin{oss}
		The unique map of operads $\overline{\M}\to Com$ induces a map between the $W$-constructions
		\[
		u: W\overline{\M}\to WCom
		\]
		If we think an element of $W\overline{\M}$ as a nested tree with labelled vertices and internal edges of length in $[0,1]$, this map just forget the labels of the vertices. 
	\end{oss}
	\begin{defn}
		We define $\phi:W\overline{\M}\to WN(\mathcal{C}/S^1)$ stratum-wise as follows: let $\mathcal{S}$ be a nested tree with $n$ leaves. If $\mathcal{P}$ is an operad we denote by $W\mathcal{P}(\mathcal{S})$ the subset of $W\mathcal{P}(n)$ consisting of elements whose underlying nested tree is $\mathcal{S}$. We will denote such an element by $x\coloneqq(x_R,(l_S,x_S)_{S\in\mathcal{S}-R})$, where $R$ is the root of $\mathcal{S}$, $l_S\in[0,1]$ is the length of the unique internal edge of $\mathcal{S}$ going out from $S\in\mathcal{S}$, $x_S\in \mathcal{P}(\abs{\hat{S}})$ is the labeling of the vertex $S$ (remember that $\hat{S}$ is the quotient of $S$ by the equivalence relation described in Definition \ref{defn: relazione equivalenza vertici nested tree}). Take $x\coloneqq(x_R,(l_S,x_S)_{S\in\mathcal{S}-R})\in W\overline{\M}(\mathcal{S})$ and consider $c\circ u(x)\coloneqq(a_1(x),\dots,a_n(x))\in \mathring{\Delta}^{n-1}$. We define
		\[
		\phi_{\mathcal{S}}(x)\in WN(\mathcal{C}/S^1)(\mathcal{S})
		\]
		by the following data:
		\begin{enumerate}
			\item The underlying nested tree of $\phi_{\mathcal{S}}(x)$ is $\mathcal{S}$.
			\item For any $S\in\mathcal{S}-R$ the length of the unique internal edge going out from $S$ is $l_S$.
			\item The nested cactus labeling a vertex $S\in\mathcal{S}$ is given by $\overline{\phi}_{a_S}(x_S)$. Here $a_S:\hat{S}\to \R^{>0}$ is given by $a_S(\iota)=\sum_{i\in \pi_S^{-1}(\iota)}a_i(x)$, where $\pi_S:S\to \hat{S}$ is the projection, and $\overline{\phi}_{a_S}:\overline{\M}_{0,\hat{S}+1}\to N^{\sigma}_{\hat{S}}(\mathcal{C}/S^1)$ is the coordinate free version of the homeomorphism of Theorem \ref{thm:omeo tra la compattificazione e lo spazio dei nested cactus} associated to the weights $a_S$.   
		\end{enumerate} 
	\end{defn}
	We now prove that $\phi_{S}:W\overline{\M}(\mathcal{S})\to WN(\mathcal{C}/S^1)(\mathcal{S})$ are compatible respect to the intersection of strata and therefore we can glue them together obtaining a well defined function $\phi_n:W\overline{\M}(n)\to WN(\mathcal{C}/S^1)(n)$ for any $n\in\N$. Then we prove that the collection of all these $\phi_n$ give an isomorphism of operads $\phi:W\overline{\M}\to WN(\mathcal{C}/S^1)$, as claimed at the beginning.
	\begin{lem}
		$\phi_n:W\overline{\M}(n)\to WN(\mathcal{C}/S^1)(n)$ is well defined for any $n\in\N$.
	\end{lem}
	\begin{proof}
		Fix $x\coloneqq(x_R,(l_S,x_S)_{S\in\mathcal{S}-R})\in W\overline{\M}(\mathcal{S})$ and suppose that $l_{S_1}=0$ for a certain vertex $S_1\in\mathcal{S}$. Let us denote by $e=(S_1,S_2)$ the unique edge going out from $S_1$ and by $\mathcal{S}/e$ the nested tree obtained from $\mathcal{S}$ by collapsing the edge $e$. More precisely $\mathcal{S}/e=\mathcal{S}-\{S_1\}$. Since $l_{S_1}=0$ the point $x\in W\overline{\M}(\mathcal{S})$ is equivalent to the point $y\coloneqq (y_R,(l_S,y_S)_{S\in\mathcal{S}/e-R})\in W\overline{\M}(\mathcal{S}/e)$ given by:
		\begin{itemize}
			\item $y_S=x_S$ if $S\neq S_2$.
			\item $y_S=x_{S_2}\circ_i x_{S_1}$ if $S=S_2$, for some $i=0,\dots,\abs{\hat{S_2}}$.
		\end{itemize} 
		In order to prove that $\phi$ is well defined we need to check that
		\begin{equation}\label{eq: compatibilità strati}
			\phi_{\mathcal{S}}(x)=\phi_{\mathcal{S}/e}(y) \quad \text{ in }  WN(\mathcal{C}/S^1)(n)
		\end{equation}
		First of all observe that $c\circ u(x)=c\circ u(y)$ because $x$ and $y$ represent the same point of $W\overline{\M}(n)$. In particular if $S\neq S_2$ then the nested cactus labelling $S$ in both $\phi_{\mathcal{S}}(x)$ and $\phi_{\mathcal{S}/e}(y)$ is the same. The only potential nested cactus which is different is the one labelling the vertex $S_2$. However the following compatibility between the weights and the operadic compositions of $\overline{\M}$ and $N(\mathcal{C}/S^1)$ tells us that also in this case the nested cacti labelling $S_2$ in $\phi_{\mathcal{S}}(x)$ and $\phi_{\mathcal{S}/e}(y)$ are equal. More precisely, let $m$ (resp. $k$) be the cardinality of $\hat{S_1}$ (resp. $\hat{S_2}$). Equation \ref{eq: compatibilità strati} follows from the commutativity of this diagram:
		\[
		\begin{tikzcd}
			& \overline{\M}_{0,k+1}\times \overline{\M}_{0,m+1}\arrow[r,"\overline{\phi}_a\times \overline{\phi}_b"] \arrow[d,"\circ_i"] & N^{\sigma}_k(\mathcal{C}/S^1)\times N^{\sigma}_m(\mathcal{C}/S^1) \arrow[d,"\circ_i"]\\
			& \overline{\M}_{0,k+m}\arrow[r,"\overline{\phi}_{a\circ_i b}"] & N^{\sigma}_{k+m}(\mathcal{C}/S^1) 
		\end{tikzcd}
		\]
		where $a=(a_1,\dots,a_k)\in \mathring{\Delta}^{k-1}$,  $b=(b_1,\dots,b_m)\in \mathring{\Delta}^{m-1}$ and
		\[
		a\circ_i b=(a_1,\dots,a_ib_1,\dots,a_ib_m,\dots,a_k)
		\]
		is the composition of $a$ with $b$ in the open simplex operad $\mathring{\Delta}$.
	\end{proof}
	\begin{thm}\label{thm:iso di operad tra le W constructions}
		$\phi:W\overline{\M}\to W N(\mathcal{C}/S^1)$ is an isomorphism of operad.
	\end{thm}
	\begin{proof}
		We prove the claim in two steps:
		\begin{itemize}
			\item First we show that for any $n\in\N$ the map $\phi_n:W\overline{\M}(n)\to WN(\mathcal{C}/S^1)(n)$ is a homeomorphism: it is a bijection by definition. The domain is compact and the target is Hausdorff, so it is enough to prove that $\phi_n$ is continuous. Observe that the restriction of $\phi_n$ to each stratum $W\overline{\M}(\mathcal{S})$ is continuous: indeed for each $S\in\mathcal{S}$ the homeomorphisms $\overline{\phi}_{a_S}:\overline{\M}_{0,\hat{S}+1}\to N^{\sigma}_{\hat{S}}(\mathcal{C}/S^1)$ depend continuously on $a_S:\hat{S}\to \R^{>0}$, and $a_S$ depends continuously on the point $x\in W\overline{\M}(\mathcal{S})$. Therefore $\phi_n$ is continuous because its restrictions to each closed strata is continuous. 
			\item $\phi:W\overline{\M}\to N(\mathcal{C}/S^1)$ is a map of operads: $\phi_n$ is obviously $\Sigma_n $-equivariant. The compatibility with the operadic compositions follows from the fact that both $c:WCom\to \mathring{\Delta}$ and $u:W\overline{\M}\to WCom$ are maps of operads.  
		\end{itemize}
	\end{proof}
	\begin{corollario}\label{cor: WM is a celular operad}
		$W\overline{\M}$ is a cellular operad. More precisely, each $W\overline{\M}(n)$ can be equipped with a CW-decomposition such that the operadic compositions
		\[
		\circ_i:W\overline{\M}(n)\times W\overline{\M}(n)\to W\overline{\M}(n+m-1)
		\]
		are cellular embeddings.
	\end{corollario}
	\begin{proof}
		We use the isomorphism $phi: W\overline{\M}\to WN(\mathcal{C}/S^1)$ of Theorem \ref{thm:iso di operad tra le W constructions} to define the CW-decomposition of $W\overline{\M}(n)$ for any $n\in\N$. The statement then follows from Theorem \ref{thm: cellular operad of nested cacti}.
	\end{proof}
	\begin{oss}
		The Boardman-Vogt $W$ construction is compatible with the cellular chain functor, in the sense that if $\mathcal{P}$ is a cellular topological operad, then $W\mathcal{P}$ is a cellular topological operad as well. Moreover the operad of cellular chains $C_*^{cell}(W\mathcal{P})$ can be identified with the Bar-Cobar resolution $\Omega BC_*^{cell}(\mathcal{P})$. Recall that the operad $C_*^{dual}(\overline{\M})$ was the operad of cellular chains on $N(\mathcal{C}/S^1)$. Therefore the operad of cellular chains on $W\overline{\M}\cong WN(\mathcal{C}/S^1)$ is just $\Omega B C_*^{dual}(\overline{\M})$.  
	\end{oss}
	
	\subsection{A zig-zag of quasi-isomorphisms}
	In this paragraph we finally prove the main result of this paper:
	\begin{thm}\label{thm: chain model for hycomm}
		The operad $C_*^{dual}(\overline{\M})$ is a chain model for the Hypercommutative operad.
	\end{thm}
	\begin{proof}
		Consider the operad of singular chains $C_*(\overline{\M})$ over $\Z$. We claim that there is a zig-zag of quasi-isomorphisms connecting $C_*(\overline{\M})$ and $C_*^{dual}(\overline{\M})$. By definition $C_*^{dual}(\overline{\M})$ is the operad of cellular chains on $N(\mathcal{C}/S^1)$. The retraction $r:WN(\mathcal{C}/S^1)\to N(\mathcal{C}/S^1)$ is a cellular map, therefore it induces a quasi-isomorphism
		\[
		r_*:C_*^{cell}(WN(\mathcal{C}/S^1))\to C_*^{dual}(\overline{\M})
		\]
		Theorem \ref{thm:iso di operad tra le W constructions} tells us that we can identify the Boardman-Vogt construction $WN(\mathcal{C}/S^1)$ with $W\overline{\M}$, therefore 
		\[
		C_*^{cell}(WN(\mathcal{C}/S^1))\cong C_*^{cell}(W\overline{\M})
		\]
		where we consider $W\overline{\M}$ as a cellular operad as in Corollary \ref{cor: WM is a celular operad}. Since $W\overline{\M}$ is a cellular operad we have a quasi-isomorphism
		\[
		C_*^{cell}(W\overline{\M})\to C_*(W\overline{\M})
		\]
		given by the realization of each cell as a sum of simplices in the usual way. Finally observe that the retraction $p: W\overline{\M}\to \overline{\M}$ induces a quasi-isomorphism between the operads of singular chains $C_*(W\overline{\M})$ and $C_*(\overline{\M})$. If we put all these things together we get the desired zig-zag of weak equivalences:
		\[
		\begin{tikzcd}[column sep=small]
			&C_*^{dual}(\overline{\M}) & C_*^{cell}(WN(\mathcal{C}/S^1))\cong C_*^{cell}(W\overline{\M}) \arrow[l,"r_*"'] \arrow[r] &C_*(W\overline{\M})\arrow[r,"p_*"] &C_*(\overline{\M})
		\end{tikzcd}
		\]
	\end{proof}
	
	\subsection{An open problem}\label{subsec:an open problem}
	In \cite{Salvatore} the second author constructed a cellular decomposition of the two dimensional Fulton MacPherson operad $FM_2$. More precisely, he defined a CW-decomposition for each space $FM_2(n)$ and proved that the operad compositions $\circ_i:FM_2(n)\times FM_2(m)\to FM_2(m+n-1)$ are cellular maps. The original purpose of this project was to prove the same statement in the case of the Deligne-Mumford operad $\overline{\M}$:
	\medskip
	
	\textbf{Conjecture:} $\overline{\M}$ is a cellular operad.
	\medskip
	
	The first step is to construct a CW-decomposition of $\overline{\M}_{0,n+1}$ for any $n\geq 2$: for example we can choose the cell structures described in Corollary \ref{cor: decomposizione cellulare spazio compattificato}. However these cells are not good candidates to solve our problem, because they are not compatible with the operad structure (see Remark \ref{oss:celle non sono compatibili con l'operad}).  A better choice would be to put on $\overline{\M}_{0,n+1}$ the dual cell decomposition. Indeed, at least at the level of chains we get the right operad structure (Theorem \ref{thm: chain model for hycomm}). However even in this case we encounter some issues, basically because the dual cell decompositions depends on the choice of some weights $(a_1,\dots,a_n)$. Let us explain this last sentence: in order to solve our problem one could try to work inductively. $\overline{\M}_{0,3}$ is just a point, and we put on it the obvious cell structure. Now consider the operadic composition
	\[
	\circ_1:\overline{\M}_{0,3}\times \overline{\M}_{0,3}\to \overline{\M}_{0,4}
	\]
	It is an embedding whose image is the stable curve associated to the nested tree $((12)3)$. Similarly, the image of $\circ_2$ is the stable curve associated to $(1(23))$. If we put on $\overline{\M}_{0,4}$ the dual cell structure (in this case any choice of weights is fine) we get that $\circ_1,\circ_2$ are cellular maps (see Figure \ref{decomposizioni della sfera }). By induction suppose we have constructed CW-decompositions of $\overline{\M}_{0,n}$ for $n\leq k-1$ which are compatible with the operad compositions. We would like to define a CW-decomposition of $\overline{\M}_{0,k}$ compatible with the operad structure. Let $n,m\in\N$ such that $n+m=k$. Then for $i=1,\dots,n$ consider 
	\[
	\circ_i: \overline{\M}_{0,n+1}\times \overline{\M}_{0,m+1}\to \overline{\M}_{0,k}
	\]
	We would like that $\circ_i$ is a cellular map. Since $\circ_i$ is an embedding we can put on $Im(\circ_i)$ the product cell structure of $\overline{\M}_{0,n+1}\times \overline{\M}_{0,m+1}$. In this way we obtain a cell decomposition of the closed strata $\overline{\M}(\mathcal{S})\subseteq \overline{\M}_{0,k}$, with $\mathcal{S}$ a nested tree on $(k-1)$-leaves which is not the corolla. The challenge is to extend this CW-decomposition to the whole $\overline{\M}_{0,k}$. The problem is that if we take the dual cell decomposition of $\overline{\M}_{0,k}$ associated to some weights $a_1,\dots,a_k$, in general it does not match the cell decomposition of the strata described above.
	\begin{oss}
		In \cite{Salvatore} there was the same problem. In that case the solution was the following: firstly the author proved that $WFM_2$ is a cellular operad. Then he constructed an isomorphism of operads from $FM_2$ to $WFM_2$, proving the claim. In the case of the Deligne-Mumford operad we can not use the same trick, indeed $W\overline{\M}$ is not isomorphic to $\overline{\M}$ (it is only homotopy equivalent).
	\end{oss}
	\section{Chain level Koszul duality}\label{sec: Koszul duality}
	In \cite{Getzler} Getzler proved that $Hycom$ and $Grav$ are Koszul dual as operads of graded vector spaces (see also Remark \ref{oss: koszul duality}). In this final section we lift this result to the category of chain complexes. Before stating the result let us recall a few definitions:
	\begin{defn}[Operadic suspension]                                                  
		Let $\mathcal{P}$ be an operad of chain complexes over some ring $R$. The \textbf{operadic suspension} $\Lambda\mathcal{P}$ is defined by
		\[
		(\Lambda\mathcal{P})(n)\coloneqq s^{n-1}\mathcal{P}(n)\otimes sgn_n
		\] 
		where $sgn_n$ is the sign representation of $\Sigma_n$.
	\end{defn} 
	\begin{defn}
			Let $(C=\{C_n\}_{n\in\Z},d_C)$ be a chain complex over some ring $R$. The \textbf{linear dual} $C^*$ is the chain complex given by
			\[
			(C^*)_n\coloneqq Hom_R(C_{-n},R)
			\]
			with differential the adjoint map to $d_C:C_n\to C_{n-1}$. 
	\end{defn} 
	\begin{defn}[Koszul dual]
		Let $\mathcal{P}$ be an operad of chain complexes. The \textbf{Koszul dual operad} $\mathbf{D}(\mathcal{P})$ is the linear dual of the cooperad $B(\mathcal{P})$, where $B$ is the operadic bar construction. In particular the space of $n$-ary operations is given by
		\[
		\mathbf{D}(\mathcal{P})(n)\coloneqq B(\mathcal{P})(n)^*
		\]
	\end{defn}
	Now we can state the main Theorem of this paper.
	\begin{thm}[chain level Koszul duality]\label{thm: chain level Koszul duality} Let $grav$ be the chain model for the gravity operad described in Paragraph \ref{subsec:operads and cacti}. Then
		\[
		\mathbf{D}(grav)=\Lambda^{-2}C_*^{dual}(\overline{\M})
		\]
	\end{thm}
	\begin{proof}
	 By definition $\mathbf{D}(grav)$ is the linear dual of the bar construciton $B(grav)$. Proposition \ref{prop:bar construction vs chains} tells us that we can identify $B(grav)(n)$ with $s^2C_*^{cell}(\overline{\M}_{0,n+1})$. The space of $n$-ary operations of $\mathbf{D}(grav)=B(grav)^*$ can be identified with the linear dual of  $s^2C_*^{cell}(\overline{\M}_{0,n+1})$, i.e. with $s^{-2}C^*_{cell}(\overline{\M}_{0,n+1})$. As usual we can see any cochain complex $C$ as a chain complex through the rule $C_n\coloneqq C^{-n}$. Therefore $C^*_{cell}(\overline{\M}_{0,n+1})$ is nothing but $s^{4-2n}C_*^{dual}(\overline{\M}_{0,n+1})$ by Poincarè duality (observe that the real dimension of $\overline{\M}_{0,n+1})$ is $2n-4$). To sum up we have that 
	 \begin{align*}
	 B(grav)^*(n)&=s^{-2}C^*_{cell}(\overline{\M}_{0,n+1})\\
	 &=s^{-2}s^{4-2n}C_*^{dual}(\overline{\M}_{0,n+1})\\
	 &=s^{2-2n}C_*^{dual}(\overline{\M}_{0,n+1})\\
	 &=(\Lambda^{-2}C_*^{dual}(\overline{\M}))(n)
	 \end{align*} 
	 This proves the statement.
	\end{proof}

	\printbibliography
	
\end{document}